\documentclass[10pt]{article}
\pdfoutput=1
\usepackage{amsmath}
\usepackage{amsfonts}
\usepackage{amssymb}
\usepackage{amsthm}
\usepackage{mathrsfs}
\usepackage{mathtools}
\usepackage{enumerate}
\usepackage{amscd}
\usepackage[all]{xy}
\usepackage[margin=1.4in]{geometry}
\usepackage{graphicx}
\usepackage{ifpdf}
\ifpdf
  \DeclareGraphicsExtensions{.pdf,.png,.jpg}
\else
  \DeclareGraphicsExtensions{.eps}
\fi
\usepackage{caption}
\usepackage{xcolor}
\usepackage{wrapfig}
\usepackage{verbatim}
\usepackage{pinlabel}
\usepackage[toc, page]{appendix}
\usepackage{tikz}
\usepackage{tikz-3dplot}
\usetikzlibrary{arrows}

\newtheoremstyle{forADefinitionInsideARemark}
{3pt}
{3pt}
{}
{}
{}
{:}
{.5em}
{}

\newtheorem{theorem}{Theorem}[section]
\newtheorem{corollary}[theorem]{Corollary}
\newtheorem{lemma}[theorem]{Lemma}
\newtheorem{proposition}[theorem]{Proposition}

\theoremstyle{forADefinitionInsideARemark}
\newtheorem{specialDefinition}[theorem]{\textbf{Definition}}

\theoremstyle{remark}
\newtheorem{remark}[theorem]{Remark}
\newtheorem{observation}[theorem]{Observation}

\theoremstyle{definition}
\newtheorem{definition}[theorem]{Definition}
\theoremstyle{definition}
\newtheorem{example}[theorem]{Example}

\newcommand{\at}[2]{\left.#1\right|_{#2}}
\newcommand{\RR}{\mathbb{R}}
\newcommand{\PP}{\mathbb{P}}
\newcommand{\CC}{\mathbb{C}}
\newcommand{\CP}{\mathbb{C}P}
\newcommand{\RP}{\mathbb{R}P}
\newcommand{\FF}{\mathbb{F}}
\newcommand{\ZZ}{\mathbb{Z}}
\newcommand{\NN}{\mathbb{N}_{\ge 0}}
\newcommand{\M}{\mathcal{M}}
\renewcommand{\L}{\mathcal{L}}
\newcommand{\PM}{\widetilde{\mathcal{M}}}
\renewcommand{\P}{\mathcal{P}}
\newcommand{\C}{\mathcal{C}}
\newcommand{\J}{\mathcal{J}}
\newcommand{\Jreg}{\mathcal{J}_{reg}}
\newcommand{\F}{\mathcal{F}}
\newcommand{\D}{\mathcal{D}}
\newcommand{\CO}{\mathcal{CO}}
\newcommand{\OC}{\mathcal{OC}}
\newcommand{\A}{\mathcal{A}}
\newcommand{\X}{\mathcal{X}}

\newcommand{\R}{\mathcal{R}}
\newcommand{\N}{\mathcal{N}}
\renewcommand{\H}{\mathcal{H}}
\newcommand{\del}{\partial}

\renewcommand{\O}{\mathcal{O}}
\renewcommand{\hom}{\operatorname{Hom}}
\newcommand{\End}{\operatorname{End}}
\newcommand{\Aut}{\operatorname{Aut}}

\newcommand{\spaN}{\operatorname{Span}}
\newcommand{\ord}{\operatorname{ord}}
\newcommand{\cl}[2][3]{%
  {}\mkern#1mu\overline{\mkern-#1mu#2}}

\newcommand{\st}{\;\colon\;}
\renewcommand{\Re}{\operatorname{Re}}
\renewcommand{\Im}{\operatorname{Im}}
\newcommand{\Id}{\operatorname{Id}}
\newcommand{\ind}{\operatorname{ind}}
\newcommand{\ev}{\operatorname{ev}}
\renewcommand{\angle}{\measuredangle}
\renewcommand{\S}{\frac{\FF_2[c]}{(1+c+c^2)}}
\newcommand{\s}{\FF_2[c]/(1+c+c^2)}

\def\presuper#1#2%
  {\mathop{}%
   \mathopen{\vphantom{#2}}^{#1}%
   \kern-\scriptspace%
   #2}

\newcommand{\chiang}{L_{\Delta}}
\newcommand{\bindih}{\Gamma_{\Delta}}
\newcommand{\ain}{A^{-1}}
\renewcommand{\a}{A}
\renewcommand{\aa}{A^2}

\newcommand{\aaaa}{A^4}
\newcommand{\aaaaa}{A^5}
\newcommand{\bin}{B^{-1}}
\renewcommand{\b}{B}
\newcommand{\abin}{(AB)^{-1}}
\newcommand{\ab}{(AB)}
\newcommand{\aabin}{(A^2B)^{-1}}
\newcommand{\aab}{(A^2B)}
\newcommand{\aaabin}{(A^3B)^{-1}}
\newcommand{\aaab}{(A^3B)}
\newcommand{\aaaabin}{(A^4B)^{-1}}
\newcommand{\aaaab}{(A^4B)}
\newcommand{\aaaaabin}{(A^5B)^{-1}}
\newcommand{\aaaaab}{(A^5B)}
\title{Higher rank local systems in Lagrangian Floer theory}
\date{}
\author{Momchil Konstantinov}

\begin{document}
\maketitle

\begin{abstract}
We extend Floer theory for monotone Lagrangians to allow coefficients in local systems of arbitrary rank. Unlike the rank 1 case, this is often obstructed by Maslov 2 discs. We study exactly what the obstruction is and define some natural unobstructed subcomplexes. To illustrate these constructions we do some explicit calculations for the Chiang Lagrangian $\chiang \subseteq \CP^3$. For example, we equip $\chiang$ with a particular rank 2 local system $W$ over $\FF_2$ for which the resulting Floer complex $CF^*(W,W)$ is unobstructed despite the presence of Maslov 2 discs. We compute that the cohomology $HF^*(W,W)$ is non-zero and deduce that $\chiang$ cannot be disjoined from $\RP^3$ by a Hamiltonian isotopy.
\end{abstract}
 \tableofcontents
\section{Introduction}
Lagrangian Floer cohomology is a very powerful tool for studying the topology and relative position of Lagrangian submanifolds in a symplectic manifold. The even more elaborate machinery of Fukaya categories encodes an immense amount of information about them. 
These invariants are all defined by ``counting'' (perturbed) punctured pseudoholomorphic discs whose boundary components are restricted to lie on a collection of Lagrangians. One can enrich this structure even further by recording some information about the homotopy classes of paths that these boundaries trace. One way to accomplish this is by using local systems of coefficients for the theory.\footnote{Another related approach to recording similar homotopy data is to call upon techniques from the topology of loop spaces. This idea was pioneered by Viterbo in \cite{viterbo1997exact} and has been extensively explored. For a small sample see e.g. \cite{barraud2007lagrangian}, \cite{fukaya2006application}, \cite{abouzaid2012wrapped}.}

Rank 1 local systems are a classical tool in Lagrangian Floer theory and are widely applied, specifically in the context of Mirror Symmetry. A natural question to ask then is whether higher rank local systems would yield useful invariants. In his paper \cite{damian2012floer}, Damian was the first to realise that in the monotone setting, it is not always possible to define Floer cohomology with coefficients in a local system of rank higher than 1 since Maslov 2 discs with boundary on the Lagrangian may obstruct $d^2 = 0$. In fact, concentrating on this obstruction for the local system arising from the universal cover, he derives 
many interesting restrictions on the topology of monotone Lagrangians in linear symplectic spaces, including a proof of the monotone Audin conjecture for Lagrangians with contractible universal cover (\cite{damian2012floer}).
From another point of view, a systematic treatment of high rank local systems in the exact case (when no such obstructions arise) was done by Abouzaid in \cite{abouzaid2012nearby}, where he constructs an extended Fukaya category which contains Lagrangians with higher rank local systems.  The powerful results of that paper then follow by an application of (a modified version of) the split-generation criterion, again due to Abouzaid (\cite{abouzaid2010geometric}). 
The goal of the present paper is to combine the above two approaches in their simplest forms: we study exactly how Maslov 2 discs obstruct $d^2 = 0$ in the monotone case and give an application of Abouzaid's criterion to an example in which this obstruction vanishes despite the existence of Maslov 2 disc bubbles. 

Regarding the first point, recall
that in the absence of local systems, the differential on the Floer complex $CF^*(L^0, L^1)$ of a monotone pair (see \ref{monotonicity} below for a quick review of monotonicity) squares to zero, if and only if 
\begin{equation}\label{obstruction numbers}
m_0(L^0) = m_0(L^1),
\end{equation}
where these are the so-called \emph{obstruction numbers} of the Lagrangians, defined as algebraic counts of Maslov 2 discs passing through a generic point. Exactly the same condition carries over when one considers local systems of rank 1, except that in the algebraic count each disc is weighted by the monodromy of the local system around its boundary. Applying the same proof to the case when the Lagrangians are equipped with  local systems $E^j \to L^j$ of higher rank, one arrives at the notion of an \emph{obstruction section} $m_0(L^j, E^j) \colon L^j \to \End(E^j)$ and a generalised form of condition \eqref{obstruction numbers} \-- see equation \eqref{obstruction equation} below. When this condition is satisfied (for example when $m_0(L^j, E^j)=0$ for $j \in \{0, 1\}$), one obtains a well-defined invariant $HF^*((L^0, E^0), (L^1, E^1))$. This is not surprising and is well-known to experts but, to our knowledge, is not explicitly written anywhere, so we devote sections \ref{The Obstruction Section} and \ref{Definition, Obstruction and Invariance} to reviewing these facts. 

In section \ref{The Pearl Complex and the Obstruction Revisited} we examine how the same obstruction appears in the setting of a computationally powerful theory of Biran and Cornea - the pearl complex \cite{biran2007quantum}. This machinery allows one to efficiently calculate Floer cohomology of a Lagrangian with itself and applies equally well when higher rank local systems come into play. 
Section \ref{honest subcomplexes} is devoted to the observation that the possibly obstructed Floer complex $CF^*((L^0, E^0), (L^1, E^1))$ contains a maximal unobstructed subcomplex $\overline{CF}^*((L^0, E^0), (L^1, E^1))$ which we call the \emph{central Floer complex}. In the case when $L^0=L^1=L$ and $E^0=E^1=E$ there is a further unobstructed subcomplex: the \emph{monodromy Floer complex} $CF^*_{mon}(E)$. We observe that when $E_{reg}$ is the local system arising from the universal cover of $L$, the differential of $CF^*_{mon}(E_{reg})$ can be seen as a deformation of the differential on the complex $C^*(L; E_{conj})$ of singular cochains with coefficients in the local system $E_{conj}$, where $E_{conj}$ arises from the conjugation action of $\pi_1(L)$ on the group algebra $\FF_2[\pi_1(L)]$. 

Next, in section \ref{The Monotone Fukaya Category} we explain how to enlarge the monotone Fukaya category over $\FF_2$ so as to include Lagrangians equipped with local systems of rank higher than 1. This mimics \cite{abouzaid2012nearby}.  

In section \ref{application to the Chiang Lagrangian} we illustrate all of the above techniques in some explicit computations for the Chiang Lagrangian $\chiang \subseteq \CP^3$. This Lagrangian is defined as the manifold of equilateral triangles inscribed equatorially in the unit sphere $S^2 \subseteq \RR^3$ and it embeds naturally in $\CP^3$ under the identification $\CP^3 \cong Sym^3(\CP^1)$. The fact that this embedding is Lagrangian was first noticed by River Chiang in \cite{chiang2004new}. The Floer theory of $\chiang$ was computed by Evans and Lekili in \cite{evans2014floer} where 
they prove that $HF^*(\chiang, \chiang)$ is non-zero if and only if one works over a field of characteristic 5. Further, by equipping the Clifford torus $T_{Cl}$ and $\chiang$ with suitable $\FF_5-$local systems $\alpha_{\zeta}$, $\beta_{\zeta}$ of rank 1, they also obtain $HF^*((T_{Cl}, \alpha_{\zeta}),(\chiang, \beta_{\zeta})) \neq 0$. As a consequence, $\chiang$ cannot be displaced from itself, or from the Clifford torus by a Hamiltonian isotopy (in fact Evans and Lekili show something much stronger: $(\chiang, \beta_{\zeta})$ generates its summand of the Fukaya category over $\FF_5$ and by varying $\beta_{\zeta}$ one can place it in any of the four summands; see \cite[Section 8]{evans2014floer}). 
One can then turn to studying the relationship between $\chiang$ and another classical Lagrangian in $\CP^3$, namely $\RP^3$. One easily checks that when they are in their standard positions, they intersect cleanly in two disjoint circles and it is then a natural question to ask whether they can be displaced by a Hamiltonian isotopy. A standard obstruction to such displaceability would be the non-vanishing of the Floer cohomology $HF(\chiang, \RP^3; \FF)$ for some field $\FF$. However, the obstruction numbers of the two Lagrangians are $m_0(\RP^3) = 0$ (since $RP^3$ has minimal Maslov number 4) and $m_0(\chiang) = \pm3$ (see \cite{evans2014floer}) and so it follows from \eqref{obstruction numbers} that the Floer cohomology $HF^*(\chiang, \RP^3;\FF)$ is well-defined only when $\text{char}(\FF)= 3$. Now, by the Auroux-Kontsevich-Seidel criterion (see e.g. \cite[Lemma 2.7]{sheridan2013fukaya}) one sees that when $\text{char}(\FF) = 3$ both $HF^*(\RP^3, \RP^3;\FF)$ and $HF^*(\chiang, \chiang;\FF)$ must vanish ($m_0 = 0$ is not an eigenvalue of quantum multiplication by $c_1(\CP^3)$ in characteristic 3). Since $HF^*(\chiang, \RP^3;\FF)$ is a (respectively left and right) module over these rings, it must also vanish. Thus, standard Floer cohomology with field coefficients cannot be used to address the question of non-displaceability.  
 In the present work we apply the idea of Floer theory with coefficients in a local system to this situation. We establish the following result:
\begin{theorem}\label{mytheorem}
There exists an $\FF_2-$local system $W \to \chiang$ of rank 2 such that $m_0(\chiang, W)=0$ and $HF^*((\chiang, W), (\chiang, W)) \neq 0$.
\end{theorem}
The main corollary which we derive from this is:
\begin{corollary}\label{chiangRp3nonDisp}
The Chiang Lagrangian $\chiang$ cannot be displaced from $\RP^3$ by a Hamiltonian isotopy.
\end{corollary}
Theorem \ref{mytheorem} is proved in section \ref{Computation of Floer Cohomology with Local Coefficients} by an explicit calculation using the pearl complex. Our starting point are the results from \cite{evans2014floer}. For the numerical calculations we use Wolfram Mathematica (the code can be found on the author's home page). To establish Corollary \ref{chiangRp3nonDisp} one enlarges the monotone Fukaya category of $\CP^3$ in order to include Lagrangians with finite rank local systems and in particular $(\chiang, W)$. 
In this enlarged category, $(\chiang, W)$ and $\RP^3$ (with no local system) lie in the same summand. 
A result by Tonkonog  (\cite[Corollary 1.2]{tonkonog2015closed}) states that $\RP^3$ split-generates this summand and thus we must have  
$HF^*((\chiang, W), \RP^3) \neq 0$, since by Theorem \ref{mytheorem} we know that $(\chiang, W)$ is an essential object in this category. 

Finally, in section \ref{some additional calculations} we compute the cohomology of the central and monodromy Floer complexes of $\chiang$ with coefficients in some other $\FF_2$ local systems. In particular we find that in many cases one can obtain a non-zero Floer invariant which does not agree with the corresponding Morse invariant. Specifically we make the following
\begin{observation} \label{narrow-wide breaks}
If $W_{reg}$ is the rank 12 local system coming from the universal cover $S^3 \to \chiang$ and $W_{conj}$ is the local system induced by the conjugation action of $\pi_1(\chiang)$ on $\FF_2[\pi_1(\chiang)]$ one has 
$$0 \neq HF^*_{mon}(W_{reg}) \neq H^*(\chiang; W_{conj}).$$
That is, the ``quantum corrections'' to the differential on $C^*(\chiang; W_{conj})$ are non-trivial but they do not kill the cohomology.
\end{observation}
\noindent In other words, when one uses local coefficients, the ``narrow\---wide'' dichotomy, introduced by Biran-Cornea in \cite{biran2009rigidity}, does not hold.

\vspace{5mm}
To deduce the above results we only need to work over the field $\FF_2$ and do not require any grading on Lagrangians. It is reasonable to assume that the results below should hold in the more general setting when one considers graded spin Lagrangians in order to define graded invariants over fields of arbitrary characteristic (of course, then a substantial amount of signs would need to be introduced). For simplicity, we restrict all our discussions to the setting modulo 2. 

\subsection*{Acknowledgements}
First and foremost I would like to thank my supervisor Jonny Evans for suggesting the local systems approach to the problem of displacing $\chiang$ form $\RP^3$
and for his many patient explanations and general encouragement. My understanding of holomorphic curves has greatly benefited also from the excellent teaching of Chris Wendl. I thank also Yank{\i} Lekili and my fellow students at the LSGNT for numerous useful conversations. Special thanks to Agustin Moreno, Tobias Sodoge, Emily Maw and Antonio Cauchi for their help and for being such a faithful audience.

Wolfram Mathematica was used to produce the plots in Figures \ref{actual plot b11aller} and \ref{actual plot b11retour} and for numerical calculations. Most other plots were produced using the tikz-3dplot package of Jeff Hein. 

This work was supported by the Engineering and Physical Sciences Research Council [EP/L015234/1] 
as the author is part of the EPSRC Centre for Doctoral Training in Geometry and Number Theory (The London School of Geometry and Number Theory), University College London.

\section{Floer Cohomology and Local Systems}\label{floer cohomology and local systems}
\subsection{Terminology and Notation}
\subsubsection{Monotonicity}\label{monotonicity}

Let $(M, \omega)$ be a symplectic manifold and $L \subseteq M$ \--- a Lagrangian submanifold. 
Recall that monotonicity concerns the following standard maps:
\begin{itemize}
\item the symplectic area homomorphism:
$$I_{\omega} \colon \pi_2(M) \to \RR; \quad I_{\omega}(A) \coloneqq \int_{A}\omega \; ,$$
\item the Chern homomorphism:
$$I_{c_1} \colon \pi_2(M) \to \ZZ; \quad I_{c_1}(A) \coloneqq \int_{A}c_1(TM) \;  ,$$
\item the Maslov homomorphism: $I_{\mu, L} \colon \pi_2(M, L)  \to \ZZ$, defined by evaluation of the Maslov class $\mu \in H^2(M,L;\ZZ)$.\label{maslov homomorphism}
\end{itemize}
We call the symplectic manifold $M$ \emph{monotone}, if there exists a constant $\lambda > 0$ such that for every $A \in \pi_2(M)$ one has $I_{\omega}(A) = 2\lambda I_{c_1}(A).$
We then call a Lagrangian submanifold $L \subseteq M$ \emph{monotone} if one has
$I_{\omega, L}(A) = \lambda I_{\mu, L}(A)$ for every $A\in \pi_2(M, L)$, where $I_{\omega, L} \colon \pi_2(M, L) \to \RR$ is the map induced by $I_{\omega}$. 
\begin{remark}\label{finite pi1 => monotone}
One usually calls a Lagrangian monotone whenever $I_{\omega}$ and $I_{\mu, L}$ are positively proportional. A standard result by Viterbo (\cite{viterbo1987intersection}) asserts that the pullback map
$j^* \colon H^2(M, L;\ZZ) \to H^2(M;\ZZ)$ satisfies $j^*\mu = 2c_1(TM)$ and this implies, firstly, that such Lagrangians can exist only in monotone symplectic manifolds and, secondly (at least when $I_{\omega}$ doesn't vanish identically, i.e. $M$ is not symplectically aspherical), that the constants of proportionality are related by a factor of 2. This shows that with our definition above we don't lose any generality in the non-aspherical case. Note further that if $M$ is monotone and $L\subseteq M$ is any Lagrangian with $H^1(L;\ZZ) = 0$ (e.g. if $H_1(L;\ZZ)$ is finite) then, since the map $j^*$ is injective in this case, the Lagrangian $L$ must also be monotone.
\end{remark}

The concept of monotonicity also extends to pairs of Lagrangian submanifolds (cf. \cite{pozniak1999floer} 3.3.2). Given two Lagrangians $L^0$, $L^1$ in $M$, the area and Maslov homomorphisms can also be evaluated on (homotopy classes of) continuous maps $u \colon S^1 \times [0, 1] \to M$ with $u(S^1 \times \{0\}) \subseteq L^0$ and $u(S^1 \times \{1\}) \subseteq L^1$. We denote these extensions by $I_{\omega, L^0, L^1}$ and $I_{\mu, L^0, L^1}$ respectively. Then we call $(L^0, L^1)$ a \emph{monotone pair} of Lagrangians if each of them is monotone and further $I_{\omega, L^0, L^1} = \lambda I_{\mu, L^0, L^1}$. It is not hard to see that if $(L^0, L^1)$ is a monotone pair, then so is $(\psi(L^0), L^1)$ for any Hamiltonian diffeomorphism $\psi$. Another useful fact is that if $L^0$ and $L^1$ are monotone Lagrangians and for at least one $j \in \{0, 1\}$ the image of $\pi_1(L^j)$ in $\pi_1(M)$ under the map induced by inclusion is trivial, then $(L^0, L^1)$ is a monotone pair (see \cite{pozniak1999floer} Remark 3.3.2 or \cite{oh1993floer}, Proposition 2.7). Finally, if $L$ is monotone, then the pair $(L, L)$ is always monotone (\cite{oh1993floer}, Proposition 2.10).

\subsubsection{Local Systems}\label{local systems}
In this paper, we will be concerned with compact, connected, monotone Lagrangians equipped with local systems of vector spaces over the field of two elements $\FF_2$. For us such a local system would just be a functor $E \colon \Pi_1L \to \mathbf{Vect}_{\FF_2}$, where $\Pi_1L$ is the fundamental groupoid of the Lagrangian $L$. More concretely, it is an assignment of an $\FF_2-$vector space $E_x$ for each point $x \in L$ and an isomorphism $P_{\gamma} \colon E_{s(\gamma)} \to E_{t(\gamma)}$ for each homotopy class $\gamma$ of paths in $L$ with source $s(\gamma)$ and target $t(\gamma)$, in a manner which is compatible with concatenation of paths. As is customary, we call these isomorphisms \emph{parallel transport maps}.

We will sometimes write $E \to L$ to denote such a local system, the notation being a shorthand for the map
\begin{eqnarray*}
\coprod_{x} E_x &\to & L \\
v \in E_x & \mapsto & x.
\end{eqnarray*}
Similarly, by a section $\sigma \colon L \to E$ we will mean a section of this map. We will call such a section \emph{parallel} if for every path $\gamma$ on $L$ one has $P_\gamma(\sigma(s(\gamma))) = \sigma(t(\gamma))$. 

Note that, since $L$ is path-connected, the inclusion $\operatorname{B}\pi_1(L, x)^{Opp} \hookrightarrow \Pi_1L$ (where $\operatorname{B}\pi_1(L, x)$ denotes the category with one object and morphisms in bijection with $\pi_1(L, x)$ whose composition follows the group law in $\pi_1(L, x)$) induces an equivalence of categories and so we get an equivalence
\begin{equation}\label{equivalence of categories for local systems}
\operatorname{Fun}(\Pi_1L, \mathbf{Vect}_{\FF_2}) \simeq \operatorname{Fun}(\operatorname{B}\pi_1(L, x)^{Opp}, \mathbf{Vect}_{\FF_2}).
\end{equation}
 Note that our conventions are such that concatenation of paths will be written from left to right, while compositions of maps, as usual, from right to left. Since this can cause headaches in explicit computations, let us spell-out concretely how the above equivalence plays out in practice. Given two points $x, y \in L$ we write $\Pi_1L(x,y) \coloneqq \hom_{\Pi_1L}(x,y)$ for the set of homotopy classes of paths connecting $x$ to $y$. To go from left to right in \eqref{equivalence of categories for local systems}, one can associate to each local system $E \to L$, a right representation of the fundamental group $\pi_1(L,x)$, by considering the action of $\Pi_1L(x,x) \cong \pi_1(L, x)^{Opp}$ on the fibre $E_x$. To go the other way, since $L$ admits a universal cover $p \colon \widetilde{L} \to L$, we can use the following construction. Suppose we have a linear right action of $\pi_1(L, x)$ on some $\FF_2-$vector space $V$. Then we consider the quotient $V \times_{\pi_1(L, x)} \widetilde{L} = (V \times \widetilde{L})\slash\pi_1(L, x)$, where the right action on $\widetilde{L}$ is by inverse deck transformations (to identify the group of deck transformations with $\pi_1(L, x)$ we have chosen a fixed basepoint $\tilde{x} \in p^{-1}(x)$). This quotient comes equipped with a projection to $L = \widetilde{L}\slash\pi_1(L, x)$ and for any $y \in L$ we define $E_y$ to be the fibre of this map over $y$. The parallel transport map along any path $\gamma \in \Pi_1L(y, z)$ is then defined to be the one induced by 
\begin{eqnarray*}
V \times p^{-1}(y) & \to &  V \times p^{-1}(z)\\
(v,\tilde{y}) & \mapsto & (v,t(\widetilde{\gamma}_{\tilde{y}})),
\end{eqnarray*}
where $\widetilde{\gamma}_{\tilde{y}}$ is any lift of the path $\gamma$ to $\widetilde{L}$, satisfying $s(\widetilde{\gamma}_{\tilde{y}}) = \tilde{y}$. This map is equivariant with respect to the right $\pi_1(L, x)-$action and so descends to a well-defined map $P_{\gamma} \colon E_y \to E_z$ on the quotients.
 
 \textbf{Wording and notation:} Note that above we viewed the path $\gamma$ slightly ambiguously as both a morphism in the fundamental groupoid and (when we mentioned lifting) an actual continuous map $\gamma \colon [a,b] \to L$. This ambiguous use of the word \emph{path} persists throughout this work since eliminating it would cause notational clutter. Hopefully the meaning is clear from the context. Given two paths $\gamma$ and $\delta$ in $L$ with $t(\gamma) = s(\delta)$, their concatenation will be written as $\gamma \cdot \delta \in \Pi_1L(s(\gamma), t(\delta))$. Since sometimes we will use two different local systems $E^j \to L$, $j \in \{0, 1\}$ on the same space, we shall write $P_{j, \gamma} \colon E^j_{s(\gamma)} \to E^j_{t(\gamma)}$ to distinguish between the parallel transport maps.

\subsection{The Obstruction Section}\label{The Obstruction Section}
Given a monotone pair of Lagrangians $(L^0, L^1)$ equipped with local systems $E^0 \to L^0$ and $E^1 \to L^1$, we define in section \ref{Definition, Obstruction and Invariance} below a Floer-theoretic invariant 
$HF^*((L^0,E^0),(L^1,E^1))$. However, the existence of this invariant is often obstructed by Maslov 2 disc bubbles and this is captured by the \emph{obstruction sections} $m_0(E^j) \colon L^j \to \End(E^j)$. This claim is made precise in Theorem \ref{main theorem of floer theory with local coefficients} below. In order to describe the obstruction section, we concentrate on a single monotone Lagrangian $L \subseteq M$, equipped with a local system $E \to L$. 
 We begin by making our setup precise and establishing some notation.
 
  Let $(M, \omega)$ be a monotone symplectic manifold and let $L \subseteq M$ be a monotone Lagrangian submanifold with minimal Maslov number $N_L \ge 2$. Let $E \to L$ be a local system of $\FF_2$-vector spaces. 
  Let $\J(M, \omega)$ denote the 
  space of $\omega$-compatible almost complex structures on $M$, that is, the space of sections $J$ of $\End(TM)$ which satisfy $J^2 = -\Id$ and such that $g_J(\cdot, \cdot) \coloneqq \omega(\cdot, J \cdot )$ is a Riemannian metric on $M$. We denote by $D$ the closed unit disc in $\CC$. Given $J \in \J(M, \omega)$, we will be concerned with $J$-holomorphic discs with boundary on $L$, i.e. smooth maps $u \colon (D, \del D) \to (M, L)$, which satisfy the Cauchy-Riemann equation
 \begin{equation}
 d u + J(u)\circ d u \circ i =0. \label{Cauchy-Riemann for discs}
 \end{equation}
Such a disc is called \textit{simple} if there exists an open and dense subset $S \subseteq D^2$ such that for all $z \in S$ one has $u^{-1}(u(z)) = \{z\}$ and $du(z) \neq 0$.
Now let us introduce the following pieces of notation.
\begin{itemize}
\item For any class $C \in \pi_2(M, L)$ and any $k \in \ZZ$ we set
\begin{eqnarray*}
\PM^C(L;J) & \coloneqq & \left\lbrace u \in C^{\infty}((D, \del D),(M, L)) \st 
  d u + J(u)\circ d u \circ i =0, \; [u] = C \right\rbrace,\\
\PM(k, L;J) & \coloneqq & \bigcup_{\substack{C \in \pi_2(M, L)\\ \mu(C)=k}}\widetilde{\M}^C(L;J),\\
\M^C(L;J) & \coloneqq & \PM^C(L;J) / G, \\
\M(k, L;J) &\coloneqq & \PM(k, L;J) / G,
\end{eqnarray*}
where $G = \operatorname{PSL(2, \RR)}$ is the reparametrisation group of the disc acting by precomposition. We will write $q_G \colon \PM^C(L;J) \to \M^C(L;J)$ for the quotient map.
\item We further set
\begin{eqnarray*}
\M^C_{0, 1}(L;J) & \coloneqq & \PM^C(L;J) \times_G S^1,\label{discs in fixed class relative one lagrangian}\\
\M_{0, 1}(k, L; J) &\coloneqq & \PM(k,L;J) \times_G S^1,\label{discs with fixed maslov}
\end{eqnarray*}
where we view $S^1 = \del D$ and an element $\phi \in  G$ acts by $\phi \cdot (u, z) = (u \circ \phi^{-1}, \phi(z))$. We shall denote the corresponding quotient map again by $q_G$.
\item The above moduli spaces come with natural evaluation maps, 
 $$\widetilde{\ev} \colon \PM^C(L;J) \times S^1 \to L, \quad \widetilde{\ev}(u, z)  \coloneqq u(z),$$
 which clearly descend to maps $\ev \colon \M^C_{0, 1}(L;J) \to L$.
 \item For any point $p \in L$ we then write $\M^C_{0, 1}(p, L;J)$ and $\M_{0, 1}(p,k, L;J)$ for the set $\ev^{-1}(\{p\})$, where the evaluation map is restricted to the appropriate domain and we set
 \begin{eqnarray*}
 \PM^C_{0, 1}(p, L;J) &\coloneqq & q_G^{-1}(\M^C_{0, 1}(p, L;J)) \subseteq \PM^C(L;J) \times S^1 \\
 \PM_{0, 1}(p,k,L;J) &\coloneqq & q_G^{-1}(\M_{0, 1}(p,k, L;J)) \subseteq \PM(k,L;J) \times S^1.
 \end{eqnarray*}
 \item We shall decorate any of the above sets with a superscript $*$ to denote the subset, consisting of simple discs. For example $\PM^{C, *}(L;J)  \coloneqq  \left\lbrace u \in \PM^{C}(L;J) \st 
   u\text{ is simple}\right\rbrace$ and $\M^{C, *}_{0, 1}(L;J)  \coloneqq  \PM^{C, *}(L;J) \times_G S^1$.
 \end{itemize}
 All these spaces are equipped with the $C^{\infty}$--topology. When there is no danger of confusion we shall sometimes simply write $u \in \M_{0, 1}(p, k, L;J)$ for the equivalence class $[u, z] = q_G(u, z)$ and $\del u \in \pi_1(L, p)$ for the homotopy class of the loop $s \mapsto u\left(z e^{2\pi i s}\right)$. 
 
By standard transversality arguments (see \cite{mcduff2012j}, Chapter 3) it follows that there exists a Baire subset $\Jreg(L) \subseteq \J(M, \omega)$ such that for all $J \in \Jreg(L)$ and any class $C \in \pi_2(M,L)$, the space $\PM^{C, *}(L;J)$ has the structure of a smooth manifold of dimension $n+\mu(C)$ and the evaluation map $\widetilde{\ev}\colon \PM^{C, *}(L;J) \times S^1 \to L$ is smooth. Since the reparametrisation action on $\PM^{C, *}(L;J)$ is free and proper, one has that $\M^C(L;J)$ is a smooth manifold of dimension $n + \mu(C) - 3$ and the quotient map $q_G$ is everywhere a submersion (in particular the map $\ev\colon \M^{C,*}_{0, 1}(L;J) \to L$ is also smooth).
Further transversality arguments (i.e. the Lagrangian boundry analogue of \cite[Proposition 3.4.2]{mcduff2012j}) show that for any smooth map of manifolds $F \colon X \to L$, there exists a Baire subset $\Jreg(L \vert F) \subseteq \Jreg(L)$ such that for every $J \in \Jreg(L \vert F)$ the maps $F \colon X \to L$ and $\ev \colon \M^*_{0, 1}(k, L;J) \to L$ are everywhere transverse. When $X$ is a submanifold of $L$ and $F$ is the inclusion map we shall write simply $\Jreg(L\vert X)$. Results by Kwon-Oh and Lazzarini (\cite{kwon2000structure, lazzarini2000existence})
yield that when $L$ is monotone one has $\M(N_L, L;J)=\M^*(N_L, L;J)$ and a further application of Gromov compactness ensures that the manifold $\M(N_L, L;J)$ is actually compact. In particular if $N_L \ge 2$ then $\M_{0, 1}(2, L;J)$ is a compact manifold (possibly empty) of dimension $\dim(\PM(2, L;J) \times S^1) - \dim(G) = n +2 + 1-3 =n$. Therefore for any $p \in L$ and $J^p \in \Jreg(L\vert p)$ the manifold $\M_{0, 1}(p, 2, L;J^p)$ consists of a finite number of points. We are now ready to define the obstruction section.
\begin{definition}\label{definition of m0}
Let $E$ be an $\FF_2$-local system on a monotone Lagrangian submanifold $L \subseteq (M, \omega)$ with $N_L \ge 2$. The \emph{obstruction section} for $E$ is a section of the local system $\End(E)$, defined as follows. For every point $p \in L$ we choose an almost complex structure $J^p \in \Jreg(L \vert p)$ and set 
$$m_0(p, E; J^p) \coloneqq \sum_{\substack{u \in \M_{0, 1}(p,2, L;J^p)}} P_{\del u}  \in \End(E_p).$$
The obstruction section is then
\begin{eqnarray*}
m_0(E) \colon L & \to & \End(E) \\
p & \mapsto & m_0(p, E; J^p).
\end{eqnarray*}
\end{definition}

\begin{remark}
Note that when $E$ is trivial and of rank one $m_0(p, E; J^p)$ is the (local) $\FF_2-$degree of the map $\ev \colon \M_{0, 1}(2, L;J^p) \to L$.
\end{remark}
As stated, the obstruction section appears to depend on the choices of almost complex structures $J^p$. This is not the case, as the following proposition shows.

\begin{proposition} \label{invariance of m0}
The following invariance properties hold:
\begin{enumerate}[i)]
\item For any $p \in L$ and $J, J' \in \J_{reg}(L\vert p)$ one has $m_0(p, E; J) = m_0(p, E; J')$;\label{m0(p) is independent of J}
\item $m_0(E)$ is a parallel section of $\End(E)$,\label{m0 is a parallel section}
\item \label{m0 is invariant under symplectos} if $\psi \colon M \to M$ is any symplectomorphism, then for every point $p \in L$, one has 
$$m_0(E)(p) = m_0(\psi_*E)(\psi(p)).$$ 
\end{enumerate}
\end{proposition}
In the remaining part of this section we prove this proposition. We will make use of the following lemma.
\begin{lemma}\label{homotopy}
Let $\C = C^0((D, \del D), (M, L))$ equipped with the compact-open topology. Let $\gamma \colon [0, 1] \to L$ be a continuous path and let $\C_{\gamma} \coloneqq \{(t, u, z) \in [0, 1] \times \C \times S^1 \st \gamma(t) = u(z)\}$. Further let $\nu \colon [0, 1] \to \C_{\gamma}$ be a continuous path and write $\nu(s)=(t(s), u_s, z_s)$. Then the loops
$\delta_{\nu(0)} \colon [0, 1]  \to   L $, 
$\delta_{\nu(0)}(s) \coloneqq u_0(z_0 e^{2\pi i s})$
and 
$\delta_{\nu(1)} \colon [0, 1]  \to  L $, 
$$
\delta_{\nu(1)}(s) \coloneqq \begin{cases}
\gamma(t(3s)), \quad s \in [0, 1/3] \\
u_1\left(z_1 e^{2 \pi i(3s-1)}\right), \quad s \in [1/3, 2/3]\\
\gamma(t(3-3s)), \quad s \in [2/3, 1]
\end{cases}
$$
are homotopic based at $u_0(z_0)$.
\end{lemma}
\begin{proof}
An explicit homotopy is given by $H \colon [0, 1] \times [0, 1] \to  L$,
$$
H(s, r) =
\begin{cases}
\gamma(t(3s)), \quad s \in [0, r/3],\, r \in [0, 1] \\
u_r \left( z_r e^{2 \pi i \frac{3s-r}{3-2r}} \right), \quad s \in [r/3, 1-r/3],\, r \in [0, 1] \\
\gamma(t(3 - 3s)), \quad s \in [1-r/3, 1],\, r \in [0, 1]
\end{cases}
$$
Continuity of $H$ follows from that of $\nu$ and of the evaluation map $\C \times S^1 \to L$.
\end{proof}
To establish part \ref{m0(p) is independent of J}) of proposition \ref{invariance of m0} we need to consider a homotopy of almost complex structures, interpolating between $J$ and $J'$. Standard transversality and compactness arguments imply the following.
\begin{theorem}\label{cobordism for maslov 2 discs}
Suppose $L$ is monotone with $N_L \ge 2$ and let $p \in L$ and $J$, $J' \in \Jreg(L\vert p)$. Then there exists a Baire subset $\Jreg(J, J') \subseteq C^{\infty}([0, 1], \J(M, \omega))$ such that for every $\hat{J} \in \Jreg(J, J')$ one has $\hat{J}(0) = J$, $\hat{J}(1) = J'$ and if we set
\begin{eqnarray*}
\PM_{0, 1}(p, 2, L;\hat{J}) & \coloneqq & \lbrace(\lambda, u, z) \in [0, 1] \times C^{\infty}((D, \del D), (M, L)) \times S^1 \st du + \hat{J}(\lambda) \circ du \circ i = 0,\\
&& \mu([u]) = 2, u(z) = p \rbrace \; \text{and}\\
\M_{0, 1}(p, 2, L;\hat{J}) &\coloneqq & \PM_{0, 1}(p, 2, L;\hat{J}) / G,
\end{eqnarray*}
then 
$\PM_{0, 1}(p, 2, L;\hat{J})$ is a smooth $4$-dimensional manifold with boundary. Further, 
$\M_{0, 1}(p, 2, L; \hat{J})$ is a \textbf{compact} $1$-dimensional manifold with boundary
$$\del \M_{0, 1}(p, 2, L; \hat{J}) = \{0\} \times \M_{0, 1}(p, 2, L; J) \; \coprod \;\{1\}\times\M_{0, 1}(p, 2, L;J')$$
and the quotient map
$q_G \colon \PM_{0, 1}(p, 2, L;\hat{J}) \to \M_{0, 1}(p, 2, L; \hat{J})$ is everywhere a submersion.
\end{theorem}
From this theorem it follows that the elements of $\del \M_{0, 1}(p, 2, L; \hat{J})$ are naturally paired up as opposite endpoints of closed intervals. Let $(\lambda, [u, z])$ and $(\lambda', [u', z'])$ be such a pair with $\lambda \le \lambda'$ (note that $\lambda$, $\lambda' \in \{0, 1\}$) and let $\bar{\nu} \colon [0, 1] \to \M_{0, 1}(p, 2, L; \hat{J})$ be any parametrisation of the interval which connects them. Choose a lift $\nu \colon [0, 1] \to \PM_{0, 1}(p, 2, L; \hat{J})$ of $\bar{\nu}$. Since $\PM_{0, 1}(p, 2, L; \hat{J})$ embeds continuously into $\C_p$ (this is notation from Lemma \ref{homotopy}, where we let $\gamma$ be the constant path at $p$), we can apply Lemma \ref{homotopy} to obtain $P_{\del u} = P_{\del u'}$. 
We then have 
$$m_0(p, E; J) + m_0(p, E; J') = \sum_{\substack{(\lambda, [u, z]) \in \del\M_{0, 1}(p,2, L;\hat{J})}} P_{\del u} = \; 0,$$
because every term in the sum appears an even number of times. This proves part \ref{m0(p) is independent of J}) of Proposition \ref{invariance of m0}. We shall henceforth use the notation $m_0(E)(p)$ as in Definition \ref{definition of m0}.

We now move on to proving part \ref{m0 is a parallel section}). Let $p, \, q \in L$ and let $\gamma \colon [0, 1] \to L$ be any smooth path with $\gamma(0) = p$, $\gamma(1) = q$. Then, by what we explained above about achieving transversality of the evaluation map with any other map, there exists a Baire subset $\Jreg(L\vert \gamma) \subseteq \Jreg(L\vert p) \cap \Jreg(L\vert q)$ such that for every $J \in \Jreg(L\vert \gamma)$, the space
$$\PM_{0,1}(\gamma, 2, L;J) \coloneqq \left\lbrace(s, u, z) \in [0, 1] \times \PM(2, L; J) \times S^1 \st u(z) = \gamma(s)\right\rbrace$$
is a smooth $4$-dimensional manifold with boundary.
Further, the manifold $\M_{0,1}(\gamma, 2, L;J) \coloneqq \PM_{0,1}(\gamma, 2, L;J) / G$ is a $1$-dimensional compact manifold with boundary 
$$\del \M_{0,1}(\gamma, 2, L;J) = \{0\} \times \M_{0, 1}(p, 2, L; J) \;\coprod\; \{1\} \times \M_{0, 1}(q,2,  L;J).$$
Thus again the elements of $\M_{0, 1}(p, 2, L; J)\;\coprod\; \M_{0, 1}(q,2,  L;J)$ are naturally paired up as endpoints of intervals. 
Let $N$ be the number of such intervals and choose parametrisations
 $\bar{\nu}_1, \ldots, \bar{\nu}_N \colon [0, 1] \to \M_{0, 1}(\gamma, 2, L; J)$ and corresponding lifts 
 $\nu_1, \ldots, \nu_N \colon [0, 1] \to \PM_{0, 1}(\gamma, 2, L; J)$ with $\nu_i(s) = (t^i(s), u^i_s, z^i_s)$ such that $t^i(0) \le t^i(1)$ $\forall i$ (recall $t^i(0), t^i(1) \in \{0, 1\}$). Since $\PM_{0, 1}(\gamma, 2, L;J)$ embeds continuously in $\C_{\gamma}$ then by applying Lemma \ref{homotopy} to $\nu_i$, we obtain that
 \begin{equation}
 P_{\del u^i_0} = P_{\delta_{\nu_i(1)}} \quad \text{for all } 1 \le i \le N.
 \end{equation}
  Let $N_1,\, N_2 \in \{1, \ldots, N+1\}$ be such that 
$t^i(0) = 0$, $t^i(1) = 0$ for all $1 \le i \le N_1 - 1$, 
$t^i(0) = 0$, $t^i(1) = 1$ for all $N_1 \le i \le N_2 - 1$ and 
$t^i(0) = 1$, $t^i(1) = 1$ for all $N_2 \le i \le N$.
Note then that if $1 \le i \le N_1 - 1$, the loop $\delta_{\nu_i(1)}$ is based at $p$ and lies in the homotopy class $\del u^i_1 \in \pi_1(L, p)$; if $N_1 \le i \le N_2 - 1$, then $\delta_{\nu_i(1)}$ is again based at $p$ and lies in the class $\gamma \cdot \del u^i_1 \cdot \gamma^{-1} \in \pi_1(L, p)$; if $N_2 \le i \le N$, then $\delta_{\nu_i(1)}$ is based at $q$ and lies in the class $\del u^i_1 \in \pi_1(L, q)$. We thus have
\begin{eqnarray*}
m_0(E)(p) + P^{-1}_{\gamma} \circ m_0(E)(q) \circ P_{\gamma} & = & \sum_{i=1}^{N_1 -1} \left(P_{\del u^i_0} + P_{\del u^i_1}\right) + \sum_{i=N_1}^{N_2 - 1} P_{\del u^i_0} \\
& + & P^{-1}_{\gamma} \circ \left(\sum_{i=N_1}^{N_2-1} P_{\del u^i_1} + \sum_{i=N_2-1}^N\left(P_{\del u^i_0} + P_{\del u^i_1}\right)\right)\circ P_{\gamma} \\
&=& \sum_{i=1}^{N_1 -1}\left( P_{\del u^i_0} + P_{\delta_{\nu_i(1)}} \right) + \sum_{i=N_1}^{N_2-1} \left(P_{\del u^i_0} + P^{-1}_{\gamma} \circ P_{\del u^i_1} \circ P_{\gamma} \right) \\
&+& 
P^{-1}_{\gamma} \circ \left( \sum_{i = N_2}^N \left( P_{\del u^i_0} + P_{\delta_{\nu_i(1)}} \right)\right) \circ P_{\gamma}\\
&=& 0 + \sum_{i=N_1}^{N_2-1} \left( P_{\del u^i_0} + P_{\delta_{\nu_i(1)}} \right) + P^{-1}_{\gamma} \circ 0 \circ P_{\gamma}\\
&=&0.
\end{eqnarray*}
This concludes the proof of part \ref{m0 is a parallel section}) of Proposition \ref{invariance of m0}.

Finally, part \ref{m0 is invariant under symplectos}) is an easy consequence of \ref{m0(p) is independent of J}). Indeed, we know that we are free to choose $J \in \Jreg(L\vert p)$ to compute $m_0(E)(p)$ and $J' \in \Jreg(\psi(L)\vert \psi(p))$ to compute $m_0(\psi_*E)(\psi(p))$. So let $J$ be any element of $\Jreg(L\vert p)$ and set $J' = \psi_* J$. Then, almost tautologically, we have that $J' \in \Jreg(\psi(L) \vert \psi(p))$ (compatibility with $\omega$ is ensured by the fact the $\psi$ is a symplectomorphism). It is then clear that $m_0(\psi(p), \psi_*E;\psi_* J)=m_0(p, E;J)$ and this completes the proof of Proposition \ref{invariance of m0}. \hfill \qed

\subsection{Definition, Obstruction and Invariance}\label{Definition, Obstruction and Invariance}
Let us now move on to defining the Floer cohomology groups $HF^*((L^0,E^0),(L^1, E^1))$ and seeing what role the obstruction section plays. First we recall some basics of Floer theory. Let $(M, \omega)$ be closed, symplectic manifold and let $L^0$, $L^1$ be two compact Lagrangian submanifolds. Let us also assume for now that $L^0$ and $L^1$ intersect transversely. In fact, it is more natural to remove this assumption and instead consider a \emph{regular Hamiltonian} $H$ for the pair $(L^0, L^1)$ as part of the data. This is, by definition, a smooth function $H \colon [0, 1] \times M \to \RR$ whose Hamiltonian flow\footnote{
 Defined by the ODE 
  $\dot{\psi_t} = X_t \circ \psi_t$,
  where $i_{X_t}\omega = -d H_t$ and the initial condition $\psi_0= \Id$.
 } 
$\psi_t \colon [0, 1] \times M \to M$ 
satisfies $\psi_1(L^0) \pitchfork L^1$. Then one can replace $L^0$ by $\psi_1(L^0)$ and conduct the same argument. We come back to this point of view in point \ref{hamiltonian chords point of view} of Remark \ref{remark with three parts} and in subsequent sections. For now, however, we keep this transversality assumption. Letting $E^0 \to L^0$ and $E^1 \to L^1$ be local systems, we then make the following definition.
 \begin{definition}\label{definition of Floer cochains with local systems}
 The Floer cochain groups of $L^0$ and $L^1$ with coefficients in the local systems $E^0$ and $E^1$ are defined to be 
 $$CF^*((L^0, E^0),(L^1,E^1)) \coloneqq \bigoplus_{p \in L^0\cap L^1} \hom_{\FF_2}(E^0_p, E^1_p)$$
 \end{definition}
 Where no confusion can arise we will drop $L^0$ and $L^1$ from the notation and just write $CF^*(E^0, E^1)$.
  Let us now briefly recall the construction of the Floer differential on these groups. To define it one needs one more piece of extra data, namely a family of almost complex structures $J \in C^{\infty}([0, 1], \J(M, \omega))$.
  In this setting one considers \emph{strips}, i.e. smooth maps $u \colon \RR_{s} \times [0, 1]_{t} \to M$ (the subscripts denote coordinates on the respective domain components) which are $J$-holomorphic, i.e. they again satisfy the Cauchy-Riemann equation (rewritten here with respect to the global conformal coordinates $(s, t)$ on $\RR \times [0, 1]$):
\begin{equation} \label{Cauchy-Riemann for strips}
\cl{\del}_J(u) \coloneqq \partial_s u + J_t(u)\partial_t u = 0
\end{equation}
and are subject to the boundary constraints $u(s, j) \in L^j$ for $j \in \{0, 1\}$ and all $s \in \RR$.
The energy of such a map is defined to be 
$$E(u) \coloneqq \int_{\RR}\int_{0}^1 \vert\vert\del_s u \vert\vert^2_{g_J} d t d s.$$
Note in particular that $E(u) =0$ if and only if $u$ is a constant map. 
Floer showed in \cite{floer1988unregularized} that the condition $E(u) < \infty$ is equivalent to the existence of intersection points $p, q \in L^0 \cap L^1$ such that $\lim_{s \to -\infty}u(s, t) = p$ and $\lim_{s \to +\infty}u(s, t) = q$ for all $t \in[0, 1]$. Thus we have a partition of the set
$$\PM(L^0, L^1; J) \coloneqq \{u \colon \RR \times [0, 1] \to M \st \cl{\del}_J(u) =0, u(s, j) \in L^j\; \forall\; s \in \RR,\; j \in\{0, 1\}, \; E(u) < \infty\} $$
into the sets
\begin{align*}
\PM(p, q; J) \coloneqq \{u \colon \RR \times [0, 1] \to M \st & \cl{\del}_J(u)=0, u(s, j)  \in  L^j\; \forall\; s \in \RR,\; j \in\{0, 1\}, \\
& \;\lim_{s \to -\infty} u(s, t) = p\; \lim_{s \to +\infty}u(s, t) = q\}.
\end{align*}
In particular any $u \in \PM(L^0, L^1;J)$ has a unique continuous extension to the domain $[-\infty, +\infty] \times [0, 1]$ which defines a class $[u]$ in $\pi_2(M, L^0 \cup L^1)$. Thus we have a further partition of each set $\PM(p, q; J)$ into sets $\PM^A(p, q; J) = \{u \in \PM(p, q; J) \st [u] = A \in \pi_2(M, L^0 \cup L^1)\}$. A coarser partition is provided by the sets $\PM(p, q, k; J) \coloneqq \cup_{\mu(A)=k} \PM^A(p, q; J)$,
where $\mu$ here denotes the Maslov-Viterbo index of a strip (see \cite{viterbo1987intersection} or \cite[equation (2.6)]{floer1988relative} for the definition). This index depends only on the homotopy class of $u$ relative $\{p, q\}$. 

Note that since $E(u) = \int_{\RR \times [0, 1]} u^*\omega$, energy is constant on the sets $\PM^A(p, q; J)$ (although a priori not on $\PM(p, q, k; J)$). Note also that since equation (\ref{Cauchy-Riemann for strips}) is translation invariant in the variable $s$, there is a natural $\RR$-action on $\PM(L^0, L^1; J)$ preserving the sets $\PM^A(p, q; J)$. Dividing by this action, we set $\M(L^0, L^1; J)\coloneqq \PM(L^0, L^1; J)/\RR$,  $\M(p, q;J)\coloneqq \PM(p, q; J)/\RR$, $\M^A(p, q;J)\coloneqq \PM^A(p, q; J)/\RR$, $\M(p, q, k;J)\coloneqq \PM(p, q, k; J)/\RR$. 
 
One then has the following theorem of Floer :
\begin{theorem}(\cite{floer1988relative})
Let $L^0$, $L^1$ be two Lagrangian submanifolds, intersecting transversely at the points $p, q$. Then there exists a Baire subset $\J_{reg}^1(p, q) \subseteq C^{\infty}([0, 1], \J(M, \omega))$ such that for every $J \in \J_{reg}^1(p, q)$ the set $\PM(p, q; J)$ has locally the structure of a smooth manifold whose dimension near $u \in \PM(p, q; J)$ equals $\mu(u)$.
\end{theorem}
In particular then note that $\M(p, q, 1; J)$ is just a union of points. In order to make the theory work we now need to impose the monotonicity assumption. From now on, $(M, \omega)$ will be a closed monotone symplectic manifold and  $(L^0, L^1)$ \--- a monotone pair of Lagrangians. The next theorem of Oh relies heavily on these assumptions and makes the definition of the Floer differential possible. 

\begin{theorem}(\cite{oh1993floer})\label{finiteness of index 1 strips}
If $L^0$ and $L^1$ are two monotone Lagrangians intersecting transversely at the points $p, q \in L^0 \cap L^1$ then there exists a Baire subset $\J_{reg}^2(p, q) \subseteq \J_{reg}^1(p,q)$ such that for each $J \in \J_{reg}^2(p, q)$ the set $\M^A(p, q; J)$ is a finite set for every class $A \in \pi_2(M, L^0 \cup L^1)$ with $\mu(A) = 1$. Further, if the pair $(L^0, L^1)$ is monotone, then $\M(p, q, 1; J)$ is also a finite union of points, i.e. there are only finitely many classes $A \in \pi_2(M, L^0 \cup L^1)$ with $\mu(A) = 1$ and $\M^A(p, q; J) \neq \emptyset$.
\end{theorem}
We are now ready to define a candidate differential on our cochain groups $CF^*((L^0, E^0),(L^1,E^1))$. For every $u \in \M(p, q; J)$ and $j \in \{0, 1\}$ we write $\gamma^j_u \colon [-\infty, +\infty] \to L^j$ for the paths $\gamma^j_u(s) = u((-1)^js, j)$ with $\gamma^0_u(-\infty) = p = \gamma^1_u(+\infty)$ and $\gamma^0_u(+\infty) = q = \gamma^1_u(-\infty)$. 
\begin{definition}\label{Floer differential with local coefficients}
We define a map $$d^J \colon CF^*(E^0, E^1) \to CF^*(E^0, E^1)$$ as follows. For all intersection points $q \in L^0 \cap L^1$ and all linear maps $\alpha \in \hom_{\FF}(E^0_q, E^1_q)$
 $$ d^J \alpha \coloneqq \sum_{p \in L^0 \cap L^1} \sum_{u \in \M(p, q, 1; J)} P_{\gamma^1_u}\circ \alpha \circ P_{\gamma^0_u}.$$ 
\end{definition}
\begin{remark}\label{definition of d^J makes sense}
Note that in the above definition we are assuming that the time-dependent $\omega$-compatible almost complex structure $J$ is chosen generically enough so that the above sum is in fact finite. In light of Theorem \ref{finiteness of index 1 strips} this amounts to asking that $J \in \bigcap_{p, q \in L^0 \cap L^1}  \J_{reg}^2(p, q)$ which is again a Baire subset of $C^{\infty}([0, 1], \J(M, \omega))$ since $L^0$ and $L^1$ are assumed to intersect transversely and thus in a finite number of points. 
\end{remark}

We finally come to the definition of Floer cohomology for a monotone pair $((L^0, E^0), (L^1, E^1))$ of Lagrangians, equipped with $\FF_2$-local systems.

\begin{definition}\label{definition of floer cohomology with local coefficients}
Let $d^J$ be as above and assume $\left(d^J\right)^2=0$. Then the Floer cohomology of $L^0$ and $L^1$ with coefficients in the local systems $E^0$ and $E^1$ is defined to be 
$$HF^*((L^0,E^0),(L^1, E^1))\coloneqq H^*(CF^*(E^0, E^1); d^J).$$ 
\end{definition}
Again we shall use $HF^*(E^0,E^1)$ as a shorthand. The notation in Definition \ref{definition of floer cohomology with local coefficients} makes sense as long as these cohomology groups are invariant under changes of $J$. This is a well-known fact if the local systems are assumed trivial. We shall see that the same proofs apply to our case just as well. Essentially the only interesting phenomenon which enters the picture when one considers non-trivial local systems is condition \eqref{obstruction equation} for $(d^J)^2=0$, which involves the obstruction sections $m_0(E^0)$ and $m_0(E^1)$. To make these statements precise we package them in the following theorem, consisting mainly of well-known facts:

\begin{theorem}\label{main theorem of floer theory with local coefficients}
Let $(M, \omega)$ be a monotone symplectic manifold and let $(L^0, L^1)$ be a monotone pair 
of closed Lagrangian submanifolds with $N_{L^j} \ge 2$ for $j \in \{0, 1\}$, equipped with $\FF_2$-local systems $E^j \to L^j$. There exists a Baire subset $\J_{reg}(L^0, L^1) \subseteq C^{\infty}([0, 1], \J(M, \omega))$ of time-dependent $\omega$-compatible almost complex structures such that:
\begin{enumerate}[A)]
\item  For all $J \in \J_{reg}(L^0, L^1)$ \label{existence of J_reg(L^0, L^1)}
\begin{enumerate}[i)]
\item (well-defined) the map $d^J$ is well-defined; \label{d^J is well-defined}
\item (obstruction)\label{obstruction in strip model}
 the map $d^J$ satisfies $\left(d^J\right)^2 = 0$ if and only if for some (and hence every) intersection point $p \in L^0 \cap L^1$ and every linear map $\alpha \in \hom_{\FF_2}(E^0_p, E^1_p)$ one has
\begin{equation}
\alpha \circ m_0(E^0)(p) + m_0(E^1)(p) \circ \alpha = 0; \label{obstruction equation}
\end{equation}
\end{enumerate}
\item (invariance)\label{invariance in strip model} Let $H\colon [0, 1] \times M \to \RR$ be a (time-dependent) Hamiltonian and $\psi_t \colon M \to M$ be its corresponding flow.
 Suppose that $\psi_1(L^0) \pitchfork L^1$ and let $J \in \Jreg(L^0, L^1)$, $J' \in \Jreg(\psi_1(L^0), L^1)$. Then one has:
\begin{enumerate}[i)]
\item \label{Hamiltonian invariance well-defined}$(d^{J'})^2 = 0 \in \End(CF^*((\psi_1(L^0), (\psi_1)_*E^0), (L^1, E^1)))$ if and only if $(d^{J})^2 = 0 \in \End(CF^*((L^0, E^0), (L^1, E^1)))$. 
\item \label{Hamiltonian invariance isomorphic}in the case when $(d^{J})^2 = 0$, there exists an isomorphism 
$$\Psi^H \colon HF^*((\psi_1(L^0), (\psi_1)_*E^0), (L^1, E^1);d^{J'}) \to HF^*((L^0, E^0), (L^1, E^1);d^J).$$
In particular, the isomorphism type of $HF^*((L^0, E^0), (L^1, E^1))$ does not depend on the choice of $J \in \Jreg(L^0, L^1)$.
\end{enumerate}
\end{enumerate}
\end{theorem}
A few comments are now in order, which we list in the following remark.
\begin{remark}:\label{remark with three parts}

\begin{enumerate}
\item  Since the obstruction criterion in \ref{obstruction in strip model}) is independent of choice of almost complex structure $J$, we can and do extend Definition \ref{definition of floer cohomology with local coefficients} by setting $HF^*((L^0,E^0),(L^1, E^1)) = 0$ in the case when equation (\ref{obstruction equation}) is not satisfied. We will however still say that Floer cohomology is not well-defined in this case.
\item \label{non-vanishing floer implies non-disp} By part \ref{invariance in strip model}), it is clear that if $HF^*((L^0,E^0),(L^1, E^1)) \neq 0$ for some local systems $E^0$, $E^1$ then, for every Hamiltonian diffeomorphism $\psi$, one has $\psi(L^0) \cap L^1 \neq \emptyset$, i.e. $L^0$ and $L^1$ cannot be displaced by a Hamiltonian isotopy.
\item \label{hamiltonian chords point of view} To rephrase part \ref{invariance in strip model}), it is useful to adopt a slightly different point of view on the complex $(CF^*((L^0,E^0),(L^1, E^1)),d^J)$ (by abuse of terminology we shall call this a complex even though a priori we don't have $(d^{J})^2 = 0$). Let $(L^0, L^1)$ be a monotone pair of Lagragians (not necessarily intersecting transversely, in particular we allow $L^0 = L^1$).
A \emph{regular Floer datum} for $(L^0, L^1)$, as defined in \cite{seidel2008fukaya}, is a pair $(H, J)$, where $H \colon [0, 1] \times M \to \RR$ is a regular Hamiltonian for $(L^0, L^1)$ (recall that this just means that its Hamiltonian flow $\psi_t$ satisfies $\psi_1(L^0) \pitchfork L^1$) and $J$ is an element of $\Jreg(\psi_1(L^0), L^1)$, the Baire set of almost complex structures, whose existence is asserted by Theorem \ref{main theorem of floer theory with local coefficients}. Then to any regular Floer datum $(H, J)$ for $(L^0, L^1)$ one can associate the complex  
$$CF^*((L^0, E^0), (L^1, E^1); H, J) \coloneqq CF^*((\psi_1(L^0), (\psi_1)_*E^0), (L^1, E^1); d^J).$$
From this point of view, part \ref{invariance in strip model}) of Theorem \ref{main theorem of floer theory with local coefficients} states that the well-definedness and isomorphism type of the cohomology of 
$CF^*((L^0, E^0), (L^1, E^1); H, J)$
 is independent of the choice of regular Floer data. In fact (although this is not obvious from the statement of the above theorem), there is even a canonical choice of isomorphism 
$$H^*(CF^*((L^0, E^0), (L^1, E^1); H, J))\quad  \cong \quad H^*(CF^*((L^0, E^0), (L^1, E^1); H', J'))$$
for any two choices of regular Floer data $(H, J)$ and $(H', J')$ (see \cite[Proposition 11.2.8]{audin2014morse} for the analogous statement in Hamiltonian Floer homology).
In particular, the following definition makes sense:
\begin{specialDefinition}\label{self-Floer with local coefficients}
Let $L \subseteq (M, \omega)$ be a closed monotone Lagrangian submanifold with $N_L \ge 2$, equipped with a pair of $\FF_2-$local systems $E^0, E^1$. Assume that there exists $p \in L$ such that
\begin{equation}\label{obstruction for self Floer with local systems}
\alpha \circ m_0(E^0)(p) + m_0(E^1)(p) \circ \alpha = 0 \quad \forall \alpha \in \hom_{\FF_2}(E^0_p, E^1_p).	
\end{equation}
We then define
$$HF^*(E^0, E^1) \coloneqq H^*(CF^*((L, E^0), (L, E^1); H, J)),$$
for some regular Floer datum $(H, J)$.	
\end{specialDefinition}

One can view the complex $CF^*((L^0, E^0), (L^1, E^1); H, J)$ in yet another way. Note that each of its generators $q \in \psi_1(L^0) \cap L^1$ corresponds to a Hamiltonian chord 
$$x_q \colon [0, 1] \to M, \quad x_q(t) = \psi_t\left((\psi_1)^{-1}(q)\right)$$
with endpoints on $L^0$ and $L^1$. Let us write $\mathcal{X}_H(L^0, L^1)$ for the set of Hamiltonian chords as above. One can then rewrite the Floer chain complex as 
$$CF^*((L^0, E^0), (L^1, E^1);H,J) = \bigoplus_{x \in \mathcal{X}_H(L^0, L^1)} \hom_{\FF^2}(E^0_{x(0)}, E^1_{x(1)}).$$
For any two intersection points $p, q \in \psi_1(L^0) \cap L^1$ one can further replace the moduli spaces $\PM(p, q; J)$ by a moduli space of \emph{Floer trajectories} $v \colon \RR \times [0, 1] \to M$, satisfying
\begin{itemize}
\item $v(s,j) \in L^j$ for $s \in \RR$ and $j \in \{0, 1\}$,
\item $\lim_{s \to -\infty} v(s, t) = x_p(t)$ and $\lim_{s \to +\infty}v(s, t) = x_q(t)$ uniformly in $t$ 
\item the Floer equation
\begin{equation}\label{floer equation}\del_s v + \left((\psi_t)^*J_t\right)(\del_t v - X_t(v)) = 0,
\end{equation}
where $X_t$ is the Hamiltonian vector field of $H$.
\end{itemize}
A bijective correspondence between the two moduli spaces is given by associating to every $u \in \PM(p, q; J)$, the map $v(s, t) = \psi_t\left(\psi_1^{-1}(u(s, t)\right)$. Each Floer trajectory $v$ gives rise to paths 
$\gamma^0_v(s) = v(s, 0)$ and $\gamma^1_v(s) = v(-s, 1)$ and using parallel transport maps along these, one obtains an alternative description of the differential on the complex $CF^*((L^0, E^0), (L^1, E^1); H, J)$. We will adopt this point of view after the end of Section \ref{Definition, Obstruction and Invariance}, as it makes the constructions of 
the monotone Fukaya category easier to describe (see 
section \ref{The Monotone Fukaya Category} below). 

\item \label{important remark on self-floer cohomology}
Note further, that in the special case when also $E^0 = E^1$, equation (\ref{obstruction for self Floer with local systems}) says that the cohomology of $CF^*((L, E), (L, E))$ is well-defined if and only if $m_0(E)$ is a scalar operator. Since we are working over $\FF_2$, this means $m_0(E) \in \{0, \Id \}$. Observe that this condition is always satisfied when $E$ has rank 1; in that context the construction is very well known and widely used, especially in topics related to mirror symmetry.

 On the other hand, condition (\ref{obstruction for self Floer with local systems}) is not guaranteed to be satisfied when the local systems have higher rank or when $E^0 \neq E^1$. This is precisely the point exploited by Damian in \cite{damian2012floer} to obtain restrictions on monotone Lagrangian submanifolds of $\CC^n$.
\end{enumerate}
\end{remark}

For the remaining part of this section we will give sketch proofs of the different parts of Theorem \ref{main theorem of floer theory with local coefficients}. As mentioned above, for all statements apart from \ref{obstruction in strip model}) one only needs to translate classical results to our setting with local coefficients. We shall give the needed references and indicate how to insert local coefficients in the respective arguments.

Note first that according to Remark \ref{definition of d^J makes sense}, to prove part \ref{d^J is well-defined}) it suffices that we set $\J_{reg}(L^0, L^1)\coloneqq \bigcap_{p, q \in L^0 \cap L^1}  \J_{reg}^2(p, q)$. In order for \ref{obstruction in strip model}) to hold, however, we will need to possibly shrink $\J_{reg}(L^0, L^1)$ to a slightly smaller subset. Let us first introduce some more notation.

For two Lagrangians $L^0$ and $L^1$ which intersect transversely we set:
\begin{itemize}
\item for every pair of intersection points $r, q \in L^0 \cap L^1$ we set 
$$B(r, q;J) \coloneqq \bigcup_{p \in L^0 \cap L^1} \M(r, p, 1;J) \times \M(p, q, 1;J);$$
\item for every intersection point $q \in L^0 \cap L^1$ we set 
$$B(q;J) \coloneqq \M_{0,1}(q,2, L^0;J_0) \cup \M_{0,1}(q, 2,L^1;J_1) \cup B(q,q;J);$$
\item for any pair of \emph{distinct} intersection points $r, q \in L^0 \cap L^1$ we set 
$$\overline{\M(r, q, 2;J)} \coloneqq \M(r, q, 2;J) \cup B(r, q;J)$$
\item for any single intersection point $q \in L^0 \cap L^1$ and we set 
$$\overline{\M(q, q, 2;J)} \coloneqq \M(q, q, 2;J) \cup B(q;J).$$
\end{itemize}

With these notions in place, Gromov compactness and gluing for moduli spaces of strips yield the following:
\begin{theorem}(\cite{oh1993floer})\label{theorem on compactified moduli spaces of strips}
Let $(L^0, L^1)$ be a monotone pair of Lagrangians, which intersect transversely in $M$. Then for every pair of intersection points $r, q \in L^0 \cap L^1$ (not necessarily distinct) there exists a Baire subset $\J^3_{reg}(r, q) \subseteq \J^2_{reg}(r, q)$ such that for every $J \in \J^3_{reg}(r, q)$ one has $J_0 \in \J_{reg}(L^0 \vert \{r, q\})$, $J_1 \in \J_{reg}(L^1 \vert \{r, q\})$ and the set $\overline{\M(r, q, 2;J)}$ has the structure of a compact $1$-dimensional manifold with boundary. Further $\del \overline{\M(r, q, 2;J)} = B(r, q;J)$ when $r \neq q$ and $\del \overline{\M(q, q, 2;J)} = B(q;J)$.
\end{theorem}

We now set $\J_{reg}(L^0, L^1) \coloneqq \bigcap_{r, q \in L^0 \cap L^1}  \J_{reg}^3(r, q)$. The proof of part \ref{obstruction in strip model}) is then confined to the following proposition:

\begin{proposition}\label{proposition for d^2=0 in strip model}
Let $J \in \J_{reg}(L^0, L^1)$. Then $\left(d^J\right)^2=0$ if and only if for all intersection points $q \in L^0 \cap L^1$ and all maps $\alpha \in \hom_{\FF}(E^0_q, E^1_q)$ we have
\begin{equation}\label{equation for d^2=0 with local systems}
\alpha \circ m_0(q, E^0;J_0) + m_0(q, E^1;J_1) \circ \alpha =0.
\end{equation}
\end{proposition}
\begin{proof}
We have:
\begin{eqnarray*}
\left(d^J\right)^2 \alpha & = & \sum_{r \in L^0 \cap L^1}\left[\sum_{p \in L^0 \cap L^1} \sum_{\substack{u \in \M(r, p, 1;J) \\
v \in \M(p, q, 1;J)}} P_{\gamma^1_v \cdot \gamma^1_u}\circ \alpha \circ P_{\gamma^0_u \cdot \gamma^0_v}\right],\\
\end{eqnarray*}
where the dot denotes concatenation of paths. 
Thus, for every intersection point $r$ the corresponding element in $\hom_{\FF_2}(E^0_r, E^1_r)$ appearing in $\left(d^J\right)^2 \alpha$ can be rewritten as
\begin{equation}\label{matrix entry for d^2 for strips}
\langle \left(d^J\right)^2 \alpha, r \rangle \coloneqq \sum_{\cl{u} \in B(r, q;J)} P_{\gamma^1_{\cl{u}}} \circ \alpha \circ P_{\gamma^0_{\cl{u}}},
\end{equation}
where for $\cl{u} = (u,v) \in B(r,q;J)$ we define $\gamma^0_{\cl{u}} \coloneqq \gamma^0_u \cdot \gamma^0_v$ and $\gamma^1_{\cl{u}} \coloneqq \gamma^1_v \cdot \gamma^1_u$. One now observes that whenever $r \neq q$ we have that the elements in $B(r, q;J)$ are naturally paired-up as opposite ends of the closed intervals which are the connected components of the compactified $1$-dimensional moduli space $\overline{\M(r, q, 2;J)}$.
Let $\{\cl{u}, \cl{u}'\}\subseteq B(r, q;J)$ be such a pair. It follows (see e.g. \cite{damian2007constraints}, Lemma 3.16) that $\gamma^0_{\cl{u}} = \gamma^0_{\cl{u}'} \in \Pi_1L^0(r,q)$ and $\gamma^1_{\cl{u}} = \gamma^1_{\cl{u}'} \in \Pi_1L^1(q,r)$.  Thus we have the identity
$$P_{\gamma^1_{\cl{u}}} \circ\alpha\circ P_{\gamma^0_{\cl{u}}} = P_{\gamma^1_{\cl{u}'}} \circ\alpha\circ P_{\gamma^0_{\cl{u}'}}.$$
Since all isolated broken strips $(u, v)$ from $r$ to $q$ come in such pairs, every summand in the right-hand side of \eqref{matrix entry for d^2 for strips} appears twice, yielding $\langle \left(d^J\right)^2 \alpha, r \rangle = 0$.

We now consider the case when $r=q$. In that case the boundary of the Gromov compactification $\cl{\M(q, q;J)}$ is $B(q;J)$. For elements $\cl{u} \in B(q;J) \setminus B(q,q;J)$ we set $\gamma^0_{\cl{u}} = \del u$, $\gamma^1_{\cl{u}} \equiv q$, if $\cl{u} = u \in \M_{0,1}(q, 2, L^0;J_0)$ and $\gamma^0_{\cl{u}} \equiv q$, $\gamma^1_{\cl{u}} = \del u$, if $\cl{u} = u \in \M_{0,1}(q, 2, L^1;J_1)$. Again the elements of $B(q;J)$ are paired-up as end points of closed intervals and when $\{\cl{u}, \cl{u}'\}$ is such a pair, we have $\gamma^j_{\cl{u}} = \gamma^j_{\cl{u}'} \in \Pi_1 L^j(q,q)$, hence
$$P_{\gamma^1_{\cl{u}}} \circ\alpha\circ P_{\gamma^0_{\cl{u}}} = P_{\gamma^1_{\cl{u}'}} \circ\alpha\circ P_{\gamma^0_{\cl{u}'}}.$$
Thus
$\sum_{\cl{u} \in B(q; J)} P_{\gamma^1_{\cl{u}}} \circ\alpha\circ P_{\gamma^0_{\cl{u}}} = 0$, again since every summand appears twice. Expanding the left-hand side yields
$$\sum_{\cl{u} \in B(q,q;J)} P_{\gamma^1_{\cl{u}}} \circ\alpha\circ P_{\gamma^0_{\cl{u}}} \quad + \sum_{\substack{u \in \M_{0, 1}(q ,2, L^0;J_0)}} \alpha \circ P_{\del u}\quad + \sum_{\substack{u \in \M_{0, 1}(q,2, L^1;J_1)}}P_{\del u}\circ \alpha\quad =\quad 0.$$
This can be rewritten as
$$\langle \left(d^J\right)^2 \alpha, q \rangle + \alpha \circ m_0(q, E^0;J_0) + m_0(q, E^1;J_1) \circ \alpha =0,$$
which proves the proposition.
\end{proof}

We now move on to part \ref{invariance in strip model}) of Theorem \ref{main theorem of floer theory with local coefficients}. To show part \ref{Hamiltonian invariance well-defined}) let $p \in L^0 \cap L^1$, $q \in \psi_1(L^0) \cap L^1$. Further, let $\gamma \colon [0, 1] \to L^0$, $\delta \colon [0, 1] \to L^1$ be any paths with $\gamma(0) = \delta(0) = p$, $\gamma(1) = \psi_1^{-1}(q)$, $\delta(1) = q$. From Proposition \ref{invariance of m0} \ref{m0 is a parallel section}) we have
\begin{eqnarray*}
m_0(E^0)(p) & = & P_{\gamma}^{-1} \circ m_0(E^0)(\psi_1^{-1}(q)) \circ P_{\gamma}\\
m_0(E^1)(p) & = & P_{\delta}^{-1} \circ m_0(E^1)(q) \circ P_{\delta}.
\end{eqnarray*}
Then by Proposition \ref{invariance of m0} \ref{m0 is invariant under symplectos}) we have that for every $\beta \in \hom_{\FF_2}(((\psi_1)_*E^0)_q, E^1_q)$:
\begin{eqnarray*}
\beta \circ m_0((\psi_1)_*E^0)(q) + m_0(E^1)(q) \circ \beta & = &
\beta \circ m_0(E^0)(\psi_1^{-1}(q)) + m_0(E^1)(q) \circ \beta\\ & = &
P_{\delta} \circ \left(\alpha \circ m_0(E^0)(p) + m_0(E^1)(p) \circ \alpha\right)\circ P_{\gamma}^{-1},
\end{eqnarray*}
where $\alpha = P_{\delta}^{-1} \circ \beta \circ P_{\gamma} \in \hom_{\FF_2}(E^0_p, E^1_p)$. From this it immediately follows that the cohomology $HF\left(\left(\psi_1(L^0), (\psi_1)_*E^0\right), \left(L^1, E^1\right)\right)$ is well-defined if and only if $HF((L^0, E^0), (L^1, E^1))$ is well-defined.

The proof of part \ref{Hamiltonian invariance isomorphic}) is standard and is based on Floer's original idea of \emph{continuation maps}. It is best seen from the point of view of the complex $CF^*((L^0, E^0), (L^1, E^1); H, J)$, generated by linear maps between fibres of the local systems over start and end points of Hamiltonian chords. One studies strips which satisfy a version of the Floer equation \eqref{floer equation} which is not translation-invariant. The condition $N_{L^j}\ge 2$ is used here to establish compactness for moduli spaces of such maps of index 0 and 1. The boundaries of these strips can be used to define parallel transport maps. Using these, one constructs chain maps 
\begin{equation}\label{continuation map}
\Psi_{H', J'}^{H, J} \colon CF^*((L^0, E^0), (L^1, E^1); H', J') \longrightarrow CF^*((L^0, E^0), (L^1, E^1); H, J)
\end{equation}
 which are then shown to be homotopy equivalences. Since the proof does not depend in any way on the rank and/or triviality of the local systems we refer the reader to \cite[Theorem 5.1]{oh1993floer}.  (see also \cite[Chapter 11]{audin2014morse} for a detailed description of the same argument for Hamiltonian Floer homology). \qed

\subsection{The Pearl Complex and the Obstruction Revisited}\label{The Pearl Complex and the Obstruction Revisited}
We now recall an alternative approach to calculating self-Floer cohomology of a single monotone Lagrangian, namely Biran and Cornea's pearl complex (see \cite{biran2008lagrangian} for an extensive account of this theory or \cite{biran2007quantum} for the full details). This is precisely the machinery we shall use in section \ref{application to the Chiang Lagrangian} for computations related to the Chiang Lagrangian. In this section we explain how to adapt this theory to incorporate local coefficients. 
 Let $L \subseteq (M, \omega)$ be a closed monotone Lagrangian submanifold with minimal Maslov number $N_L \ge 2$. Further, let $L$ be equipped with a pair of $\FF_2$-local systems $E^0$, $E^1$. Choose a Morse function $f \colon L \to \RR$ and a Riemannian metric $g$, such that $\F = (f, g)$ is a Morse-Smale pair. We shall refer to $\F$ as a \emph{Morse datum}. The cochain groups in this case are given by:
$$C^*_{f}(E^0,E^1) = \bigoplus_{x \in Crit(f)} \hom_{\FF_2}(E^0_x,E^1_x).$$ 
To define the appropriate candidate differential one chooses a time-independent $\omega$-compatible almost complex structure $J \in \J(M, \omega)$. Then one considers the following moduli spaces of \emph{pearly trajectories}. 
\begin{definition}\label{pearls}
For any pair of critical points $y, x  \in Crit(f)$ a parametrised pearly trajectory from $y$ to $x$ is defined to be a configuration $\mathbf{u} = (u_1, \ldots, u_r)$ of $J$-holomorphic discs 
$$u_{\ell} \colon (D, \del D) \to (M, L), \quad du_{\ell} + J\circ du_{\ell} \circ i =0,$$
such that if $\phi \colon \RR \times L \to L$ denotes the negative gradient flow of $f$ with respect to the metric $g$, then
there exist elements $\{t_1, \ldots, t_{r-1}\} \subseteq (0, \infty)$ such that
\begin{enumerate}
\item $\lim_{t\to -\infty}\phi_{t}(u_1(-1)) = y$;
\item for all $1 \le \ell \le r-1$, $\phi_{t_{\ell}}(u_{\ell}(1)) = u_{\ell+1}(-1)$;
\item $\lim_{t \to +\infty}\phi_{t}(u_r(1)) = x$.
\end{enumerate}
\end{definition}

\noindent The relevant moduli spaces now are:
\begin{itemize}
\item For any vector $\mathbf{A} = (A_1, \ldots, A_r) \in \left(H_2(M, L) \setminus 0\right)^r$ we denote by $\widetilde{\P}(y, x, \mathbf{A};\F,J)$ the set of all parametrised pearly trajectories $\mathbf{u} = (u_1, \ldots, u_r)$ such that $[u_i] = A_i$ for all $1 \le i \le r$;
\item For any positive integer $k$ we define 
$$\widetilde{\P}(y, x, kN_L;\F,J) \coloneqq \bigcup_{\substack{\mathbf{A}\\ \mu(\mathbf{A})=kN_L}}\widetilde{\P}(y, x, \mathbf{A};\F,J),$$
where the length $r$ of the vector $\mathbf{A}$ is allowed to vary and $\mu(\mathbf{A})\coloneqq \sum_{i=1}^r \mu(A_i)$.
\item We impose the following equivalence relation on $r-$tuples of $J-$holomorphic discs (for varying $r$):
$\mathbf{u} =(u_1, u_2, \ldots, u_r) \sim \mathbf{u'}=(u'_1, u'_2, \ldots, u'_{r'})$ if and only if $r=r'$ and there exist elements $\sigma_{\ell} \in G_{-1, 1} \coloneqq \{g \in PSL(2, \RR) \st g(-1)=-1, \; g(1)=1\}$ such that $u_{\ell}\circ \sigma_{\ell} = u'_{\ell}$. We now set
\begin{eqnarray*}
\P(y, x, \mathbf{A};\F,J) & \coloneqq & \widetilde{\P}(y, x, \mathbf{A};\F,J)/\sim \\
\P(y, x, kN_L;\F,J) & \coloneqq & \widetilde{\P}(y, x, kN_L;\F,J)/\sim 
\end{eqnarray*}
These definitions extend naturally to the case when $\mathbf{A}$ is the empty vector, in which case one defines $\P(y, x, \mathbf{\emptyset};\F,J) = \P(y, x, 0;\F,J)$ to be the space of unparametrised negative gradient trajectories of $f$ connecting $y$ to $x$. 
\item We also declare the following to be standing notation: 
$$\delta(y, x, \mathbf{A}) \coloneqq \ind{y} - \ind{x} + \mu(\mathbf{A}) - 1,$$
$$\delta(y, x, kN_L) \coloneqq  \ind{y} - \ind{x} + kN_L - 1.$$
\end{itemize}

These moduli spaces of pearly trajectories have natural descriptions as pre-images of certain submanifolds of products of $L$ under suitable evaluation maps and are thus endowed with a topology. That is, given a vector $\mathbf{A} \neq \emptyset$ as above, one considers the map
$$\ev_{\mathbf{A}} \colon \M^{A_1}(L;J) \times \cdots \times \M^{A_r}(L;J) \to L^{2r},$$
$$\ev_{\mathbf{A}}(u_1, \ldots, u_r) \coloneqq (u_1(-1), u_1(1), u_2(-1), u_2(1), \ldots, u_r(-1), u_r(1)).$$
 Then, putting $Q \coloneqq \{(x, \phi_t(x)) \in L \times L \st t > 0, \; x \in L \setminus Crit(f)\}$, we have that
$$\widetilde{\P}(y, x, \mathbf{A};\F,J) = \ev_{\mathbf{A}}^{-1}(W^d(y) \times Q^{r-1} \times W^a(x)).$$
Note that from this and our discussion about dimensions of moduli spaces of discs in Section \ref{The Obstruction Section} it follows that the expected dimension of the space $\P(y, x, \mathbf{A};\F,J)$ is $\delta(y, x, \mathbf{A})$. 

Following \cite{biran2007quantum}, one can also use these descriptions to exhibit $\P(y, x, \mathbf{A};\F,J)$ as a topological subspace of the much larger space $\L$, defined as follows. Let $\P_L$ denote the space of continuous paths $\{\gamma \colon [0, b] \to L \st b \ge 0\}$ (with the compact-open topology) and let $\P_{\F} \subseteq \P_L$ denote the subspace consisting of paths which parametrise negative gradient flowlines of $f$ in the unique way such that $f(\gamma(t)) = f(\gamma(0))-t$.

 Then $\P(y, x, \mathbf{A};\F, J)$ embeds continuously into the space 
$$\L \coloneqq \P_{\F} \times \PM^{A_1}(L;J)/G_{-1, 1} \times \P_{\F} \times \PM^{A_2}(L;J)/G_{-1, 1} \times \cdots \times \PM^{A_r}(L;J)/G_{-1, 1} \times \P_{\F}.$$

Now let $\mathbf{u} = (u_1, u_2, \ldots, u_r) \in \widetilde{\P}(y, x, \mathbf{A};\F,J)$ be a parametrised pearly trajectory connecting $y$ to $x$ and let $(\tau_0, [u_1], \tau_1, [u_2], \ldots, \tau_{r-1}, [u_r], \tau_r)$ be the corresponding element of $\L$. For any $1 \le \ell \le r$ and $j \in \{0, 1\}$ define $\gamma^j_{u_{\ell}} \colon [0, 1] \to L$, $\gamma^j_{u_{\ell}}(t) = u_{\ell}(e^{i\pi(j+t+1)})$ (that is, $\gamma^0_{u_{\ell}}$ parametrises the image of the ``bottom'' half-circle, traversed counter clockwise, while $\gamma^1_{u_{\ell}}$ parametrises the ``top'' half-circle). 
We now define the following two paths:
\begin{eqnarray}\label{pearly paths for local systems}
\gamma^0_{\mathbf{u}} &\coloneqq & \tau_0 \cdot \gamma^0_{u_1}\cdot \tau_1 \cdots \gamma^0_{u_r} \cdot \tau_r \in \Pi_1L(y,x)  \nonumber \\
\gamma^1_{\mathbf{u}} &\coloneqq & \tau^{-1}_r \cdot \gamma^1_{u_r}\cdot \tau^{-1}_{r-1} \cdots \gamma^1_{u_1} \cdot \tau^{-1}_1 \in \Pi_1L(x,y).
\end{eqnarray}
We then get corresponding parallel transport maps $P_{j,\gamma^j_{\mathbf{u}}}\colon E^j_y \to E^j_x$ for $j \in \{0, 1\}$. Whenever we have $E^0 = E^1 = E$ we will just write $P_{\gamma^j_{\mathbf{u}}} = P_{j,\gamma^j_{\mathbf{u}}}$ as before.

 We wish to define a candidate differential on $C^*_f(E^0,E^1)$ by using parallel transport maps along the paths \eqref{pearly paths for local systems} corresponding to isolated pearly trajectories. The relevant theorem, guaranteeing that this is possible is the following.

\begin{theorem}(\cite{biran2007quantum}, Proposition 3.1.3)\label{biran-cornea}
For any Morse datum $\F$, there exists a Baire subset $\Jreg(\F) \subseteq \J(M, \omega)$ such that for every $J \in \Jreg(\F)$ and every pair of points $x, y \in Crit(f)$  the set $\P(y, x, kN_L;\F,J)$ has naturally the structure of a smooth manifold of dimension $\delta(y, x, kN_L)$, whenever $\delta(y, x, kN_L) \le 1$. Furthermore, when $\delta(y, x, kN_L) = 0$ the space $\P(y, x, kN_L;\F,J)$ is compact and hence consists of a finite number of points.
\end{theorem}

We can now define the candidate differential:

\begin{definition}\label{Pearly differential with local systems}
For a Morse datum $\F$ and an almost complex structure $J \in \Jreg(\F)$ we define a map:
$$d^{(\F,J)} \colon C^*_f(E^0,E^1) \to  C^*_f(E^0,E^1) $$
by setting for every $x \in Crit(f)$ and every $\alpha \in \hom_{\FF_2}(E^0_x, E^1_x)$, 
$$d^{(\F,J)}(\alpha) = \sum_{k\in \NN}\sum_{\substack{y \in Crit(f) \\ \delta(y, x, kN_L) = 0}}
                            \sum_{\mathbf{u}\in \P(y, x, kN_L;\F,J)} P_{1,\gamma^1_{\mathbf{u}}}\circ \alpha \circ P_{0,\gamma^0_{\mathbf{u}}}.$$
\end{definition}
Propositions 5.1.2 and 5.6.2 in \cite{biran2007quantum} then assert that (for a possibly smaller Baire subset of almost complex structures, still denoted $\Jreg(\F)$) the above map is a differential whenever the local systems $E^0$ and $E^1$ are assumed trivial of rank 1, and the resulting cohomology is canonically isomorphic to the Floer cohomology $HF^*(L, L)$. In the theorem below we state the modified versions of these facts when the non-trivial local systems are incorporated into the picture.

\begin{theorem}\label{main theorem on local pearly theory}
Let $(M, \omega)$ be a closed monotone symplectic manifold and let $L \subseteq M$ be a closed monotone Lagrangian submanifold with $N_L \ge 2$, equipped with a pair of $\,\FF_2-$local systems $E^0, E^1$ and a Morse datum $\F = (f, g)$. Then there exists a Baire subset $\Jreg(\F) \subseteq \J(M, \omega)$ such that for every $J \in \Jreg(\F)$:
\begin{enumerate}[A)]
\item
\begin{enumerate}[i)]
\item the map $d^{(\F, J)}$ is well-defined; \label{pearly theorem part: pearly differential well-defined}
\item \label{pearly theorem part: conditions to square to zero}$\left(d^{(\F, J)}\right)^2 = 0$ if and only if for some (and hence every) point $x \in Crit(f)$ one has 
\begin{equation}
\alpha \circ m_0(E^0)(x) + m_0(E^1)(x) \circ \alpha = 0 \label{obstruction equation for pearls}
\end{equation}
for every linear map $\alpha \in \hom_{\FF_2}(E^0_x, E^1_x)$.
\end{enumerate}
\item \label{pearly invariance} Let $\F = (f, g)$ and $\F'=(f', g')$ be two sets of Morse data for $L$ and $J \in \Jreg(\F)$, $J' \in \Jreg(\F')$ be regular almost complex structures. If equation \eqref{obstruction equation for pearls} holds, there exists a canonical isomorphism 
\begin{equation}\label{pearly continuation map}
\Psi_{\F, J}^{\F', J'} \colon H^*\left(C^*_f(E^0, E^1), d^{(\F, J)}\right) \to H^*\left(C^*_{f'}(E^0, E^1), d^{(\F', J')}\right)
\end{equation}
\item \label{pearly theorem part: PSS} When equation \eqref{obstruction equation for pearls} holds, there exists a canonical isomorphism 
\begin{equation} \label{pss isomorphism}
\Psi_{PSS} \colon H^*\left(C^*_f(E^0, E^1), d^{(\F, J)}\right) \to HF^*((L, E^0), (L,E^1)).
\end{equation}

\end{enumerate}
\end{theorem}

In the remaining part of this section we give a sketch proof of this theorem, emphasising part \ref{pearly theorem part: conditions to square to zero}) which is the only place where higher rank local systems make a difference.

\emph{Proof of part \ref{pearly theorem part: pearly differential well-defined}):}
This is an immediate consequence of Theorem \ref{biran-cornea} above.  \qed

\emph{Proof of part \ref{pearly theorem part: conditions to square to zero}):} The proof relies on analysing the natural Gromov compactifications of the spaces $\P(y, x, \mathbf{A}; \F, J)$ when $\delta(y, x, \mathbf{A}) = 1$. 
These compactifications are described in detail by Biran and Cornea in \cite[Lemma 5.1.3]{biran2007quantum}, where they also prove that $d^2 =0$ in the case of trivial rank 1 local systems (we have dropped the decoration $(\F, J)$ from the differential to alleviate notation). Generalising the same arguments to the case of arbitrary local systems yields that for any \emph{distinct} $x, y \in Crit(f)$ and each $\alpha \in \hom_{\FF_2}(E^0_x, E^1_x)$ one has $\langle d^2(\alpha), y \rangle = 0$.

However, some care needs to be taken when evaluating $\langle d^2(\alpha), x \rangle$. To that end we consider the space of twice marked discs $\PM^A(L; J)/G_{-1, 1}$ for $\mu(A) = 2$. 
Its Gromov compactification is obtained by adding stable maps with two components: one is a Maslov 2 disc while the other is a constant disc component and contains the two marked points. We distinguish these configurations into two types, depending on the cyclic order of the special points on the constant component.
That is, with the marked points at $-1$ and $1$, we have (up to equivalence of stable maps) two possibilities for the nodal point: we write 
$$\del \left(\cl{\PM^A(L; J)/G_{-1, 1}}\right) = \D^-(A) \cup \D^+(A),$$ 
where $\D^-(A)$ consists of equivalence classes with the nodal point of the constant component at $-i$, while $\D^+(A)$ consists of the ones with the nodal point at $i$. 
The extended evaluation map $\cl{\ev}_{(A)} \colon \cl{\PM^A(L; J)/G_{-1, 1}} \to L^2$ maps $\D^-(A) \cup \D^+(A)$ to $\operatorname{diag}(L)$.
 We shall write $\D^-(A, x) \coloneqq \D^-(A) \cap \cl{\ev}_{(A)}^{-1}(x, x)$, $\D^+(A, x) \coloneqq \D^+(A) \cap \cl{\ev}_{(A)}^{-1}(x, x)$ and $\D^{\mp}(A,x) \coloneqq \D^-(A, x) \cup \D^+(A, x)$ for any point $x \in L$.
Then, one has the following addendum to \cite[Lemma 5.1.3]{biran2007quantum}:

\begin{theorem}\label{pearly compactness}
There exists a Baire subset $\Jreg(\F) \subseteq \J(M, \omega)$ such that for every $J \in \Jreg(\F)$ one has that for each $x \in Crit(f)$ and $A \in \pi_2(M, L)$ with $\mu(A)=2$, the Gromov compactification $\cl{\P(x, x, (A);\F,J)}$  has naturally the structure of a compact 1-dimensional manifold with boundary. Furthermore, the boundary is given by 
\begin{eqnarray*}
\del \cl{\P(x, x, (A);\F,J)} & = & \left[\bigcup_{\substack{z \in Crit(f) \\ \delta(x,z,0)=0}}\P(x, z, \emptyset;\F,J) \times \P(z, x, (A);\F,J)\right] \cup \\
&&\left[\bigcup_{\substack{z \in Crit(f) \\ \delta(x,z,2)=0}}\P(x, z, (A);\F,J) \times \P(z, x, \emptyset;\F,J)\right] \cup \D^{\mp}(A, x).
\end{eqnarray*}
\end{theorem}
 \begin{figure}[h!]
\labellist
\pinlabel $\mathbf{2}$ at 112 300
\pinlabel $\mathbf{2}$ at 523 300
\pinlabel $\mathbf{2}$ at 324 233
\pinlabel $\mathbf{2}$ at 115 42
\pinlabel $\mathbf{2}$ at 498 102
\pinlabel $\mathbf{0}$ at 115 100
\pinlabel $\mathbf{0}$ at 498 45
\pinlabel $x$ at 6 342
\pinlabel $x$ at 190 298
\pinlabel $x$ at 445 298
\pinlabel $x$ at 630 252
\pinlabel $x$ at 210 232
\pinlabel $x$ at 432 232
\pinlabel $x$ at 8 100
\pinlabel $x$ at 228 100
\pinlabel $x$ at 388 45
\pinlabel $x$ at 606 45
\endlabellist
\begin{center}
\includegraphics[scale=0.5]{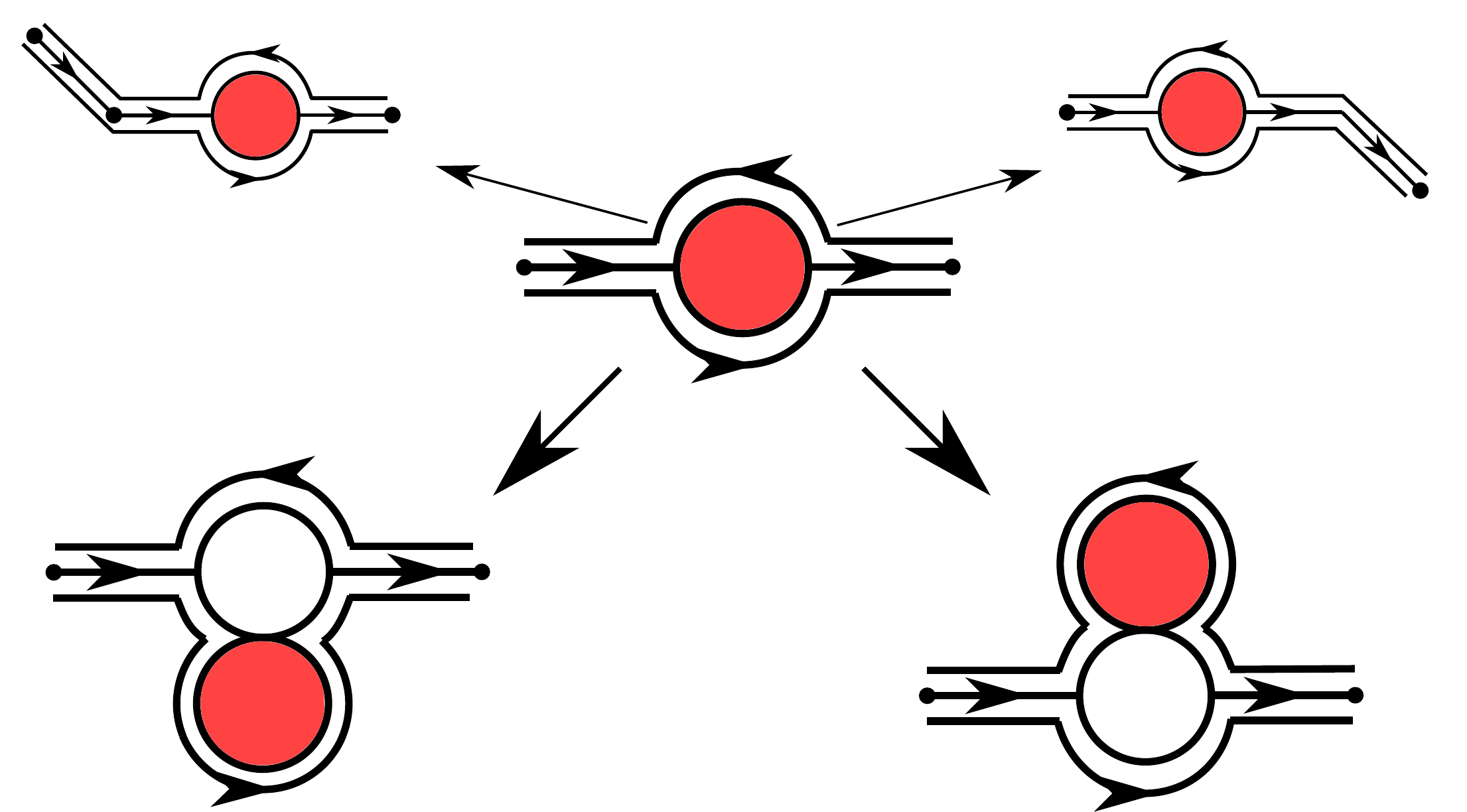}
\caption{The obstruction for the pearl complex.} \label{figure for obstruction for pearls}
\end{center}
\end{figure}
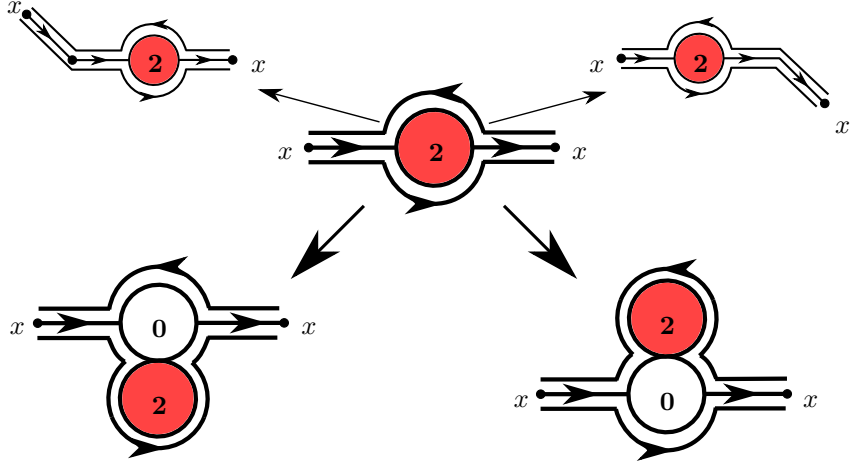
 The situation is illustrated in Figure \ref{figure for obstruction for pearls}. The above description of the boundary $\del\cl{\P(x, x, (A);\F,J)}$ is not explicitly mentioned in \cite{biran2007quantum} since there a natural bijection between $\D^{-}(A, x)$ and $\D^{+}(A, x)$ is implicitly used to glue the two spaces together and thus treat them as points in the interior of $\cl{\P(x, x, (A);\F,J)}$. An explicit description of this idea can be found in \cite{zapolsky2015lagrangian}, Section 6.2. 
 
 We claim that the above theorem suffices to prove part \ref{pearly theorem part: conditions to square to zero}) of Theorem \ref{main theorem on local pearly theory}. Indeed, consider 
 $$w = [(\tilde{u}_{\alpha}, \tilde{u}_{\beta}), \{(\alpha, -i) ,(\beta, z)\}, \{(\alpha, -1), (\alpha, 1)\}] \in D^-(A, x),$$ 
 the notation being $[(\text{maps}), \{\text{nodal points}\}, \{\text{marked points}\}]$; note in particular that $\tilde{u}_{\alpha}$ is constant.
 Write $\del \tilde{u}_{\beta} \in \Pi_1L(x,x)$ for the boundary of $\tilde{u}_{\beta}$ viewed as a loop based at $x$. We define $\gamma^0_w \coloneqq \del \tilde{u}_{\beta}$ and $\gamma^1_w$ to be the constant path at $x$. Similarly, if $w \in D^+(A, x)$ we define $\gamma^1_w \coloneqq \del \tilde{u}_{\beta}$ and $\gamma^0_w$ to be the constant path at $x$. Note that there are obvious diffeomorphisms $\D^{\mp}(A, x) \cong \M^A_{0, 1}(x, L;J)$, given by
 $$w = [(\tilde{u}_{\alpha}, \tilde{u}_{\beta}), \{(\alpha, \mp i) ,(\beta, z)\}, \{(\alpha, -1), (\alpha, 1)\}] \mapsto u_w = [\tilde{u}_{\beta}, z].$$
Clearly, if $w \in \D^-(A, x)$, then $\gamma^0_w = \del u_w$ and if $w \in \D^+(A, x)$, then $\gamma^1_w = \del u_w$. From this it is immediate (at least when $J \in \Jreg(L, x)$) that for every $\alpha \in \hom(E^0_x, E^1_x)$ we have

\begin{equation}\label{pearly obstruction appears for the first time}
m(E^1)(x) \circ \alpha + \alpha \circ m(E^0)(x)=\sum_{\substack{A, \\ \mu(A) = 2}} \sum_{\substack{w \in \D^{\mp}(A, x)}} P_{\gamma^1_w} \circ \alpha \circ P_{\gamma^0_w}.
\end{equation}

\noindent Note now that we also have
\begin{eqnarray*}
\langle d^2 \alpha, x \rangle & = & \sum_{\substack{z \in Crit(f) \\ \delta(x, z, 0) = 0}}\sum_{\substack{(\mathbf{u}, \mathbf{v}) \in  \\\P(x, z, 0;\F,J) \times \P(z, x, 2;\F,J)}} 
P_{1,\gamma^1_{\mathbf{v}}\cdot\gamma^1_{\mathbf{u}}} \circ \alpha \circ 
P_{0,\gamma^0_{\mathbf{u}}\cdot\gamma^0_{\mathbf{v}}} \quad +\\
& & \sum_{\substack{z \in Crit(f) \\ \delta(x, z, 2) = 0}}\sum_{\substack{(\mathbf{u}, \mathbf{v}) \in \\ \P(x, z, 2;\F,J) \times \P(z, x, 0;\F,J) }} P_{1,\gamma^1_{\mathbf{v}}\cdot\gamma^1_{\mathbf{u}}} \circ \alpha \circ 
P_{0,\gamma^0_{\mathbf{u}}\cdot\gamma^0_{\mathbf{v}}}.
\end{eqnarray*}
 Adding this to eqution \eqref{pearly obstruction appears for the first time}, we obtain what we were after: 
 $$\langle d^2 \alpha, x \rangle + m_0(E^1)(x) \circ \alpha + \alpha \circ m_0(E^0)(x) = 0,$$
 where the right-hand side vanishes since, by Theorem \ref{pearly compactness}, the sum runs over all boundary points of the compact 1-dimensional manifold $\cl{\P(x,x,2;\F, J)}$. This completes the proof of part \ref{pearly theorem part: conditions to square to zero}) of Theorem \ref{main theorem on local pearly theory}. \qed 
 
 \emph{Proofs of part \ref{pearly invariance}) and part \ref{pearly theorem part: PSS}):} These are proved for trivial rank 1 local systems in \cite[Section 5.1.2]{biran2007quantum}  and \cite[Proposition 5.6.2]{biran2007quantum}, respectively. Straightforward generalisations of these arguments to the case of higher rank local systems yield the results. \qed

\vspace{\baselineskip}
\noindent Observe that the map $d^J$ can be written as 
\begin{equation}\label{splitting of differential}
d^J = \del_0 + \del^J_1 + \cdots, 
\end{equation}
$$
\text{where}\quad \del^J_{k} \alpha \; = \sum_{\substack{y \in Crit(f) \\ \ind(y) = \ind(x) + 1 - kN_L}}
                            \sum_{\mathbf{u}\in \P(y, x, kN_L;\F,J)} P_{1,\gamma^1_{\mathbf{u}}} \circ \alpha \circ P_{0,\gamma^0_{\mathbf{u}}}.$$
In particular $\del_0$ arises only from counts of gradient flow lines. In the case of the trivial rank 1 local system, it is just the standard Morse differential. It is shown by Abouzaid in \cite[Appendix B]{abouzaid2012nearby} that the cohomology of  $(C^*_f(E^0, E^1), \del_0)$ (which is always well-defined) is isomorphic to $H^*(L; \hom(E^0, E^1))$, where the latter is the singular cohomology of $L$ with coefficients in the local system $\hom(E^0, E^1)$.

\subsection{Honest Subcomplexes}\label{honest subcomplexes}

Above we have seen that the complex $CF^*((L^0, E^0), (L^1, E^1);d^J) = \bigoplus_{q \in L^0 \cap L^1}\hom_{\FF_2}(E^0_q, E^1_q)$ can be obstructed by the presence of linear maps $\alpha \in \hom_{\FF_2}(E^0_q, E^1_q)$ for which 
$$\alpha \circ m_0(E^0)(q) \neq m_0(E^1)(q) \circ \alpha.$$
We also used on several occasions the fact that if $p, q \in L^0 \cap L^1$ and $\gamma \in \Pi_1L^0(p, q)$, $\delta \in \Pi_1L^1(q, p)$ then for every $\alpha \in \hom_{\FF_2}(E^0_q, E^1_q)$ one has that 
\begin{equation}\label{m0 parallel again}
\begin{array}{rcl}
(P_{\delta} \circ \alpha \circ P_{\gamma}) \circ m_0(E^0)(p)& + & m_0(E^1)(p) \circ (P_{\delta} \circ \alpha \circ P_{\gamma})  = \\
 & = & P_{\delta} \circ \bigg(\alpha\circ m_0(E^0)(q) \;\; + \;\; m_0(E^1)(q) \circ \alpha\bigg) \circ P_{\gamma}
\end{array}
\end{equation}
which follows from Proposition \ref{invariance of m0} \ref{m0 is a parallel section}).

Since the Floer differential is defined precisely by pre- and post-composing each homomorphism $\alpha$ by parallel transport along paths on the Lagrangians, the above two observations show that the subspace
$$\overline{CF}^*((L^0, E^0), (L^1, E^1)) \coloneqq \bigoplus_{q \in L^0 \cap L^1} \{\alpha \in \hom_{\FF_2}(E^0_q, E^1_q) \st \alpha \circ m_0(E^0)(q) + m_0(E^1)(q) \circ \alpha = 0 \}$$
is actually preserved by $d^J$ and is an unobstructed subcomplex of $CF^*(E^0, E^1)$. We call $\overline{CF}^*(E^0, E^1)$ the \emph{central subcomplex} of $CF^*(E^0, E^1)$. From the point of view of Hamiltonian chords the central subcomplex is 
\[\begin{array}{rcl}
\overline{CF}^*(E^0, E^1; H,J) & \coloneqq & \bigoplus_{x \in \X_H(L^0, L^1)}\{\alpha \in \hom_{\FF_2}(E^0_{x(0)}, E^1_{x(1)})
 \st \\
  && \qquad \qquad \qquad \qquad\alpha \circ m_0(E^0)(x(0)) + m_0(E^1)(x(1)) \circ \alpha = 0 \}.
\end{array}\]
When $L^0 = L^1$ one can also consider the central subcomplex of the pearl complex:
$$\overline{C}^*_f(E^0, E^1) \coloneqq \bigoplus_{x \in Crit(f)}\{\alpha \in \hom_{\FF_2}(E^0_{x}, E^1_{x})
 \st\alpha \circ m_0(E^0)(x) + m_0(E^1)(x) \circ \alpha = 0 \}.$$
 Again, $\overline{C}^*_f(E^0, E^1)$ is unobstructed by construction. Further, since all chain-level continuation maps 
 \begin{equation}\label{continuation map again}
 \Psi_{H,J}^{H', J'} \colon CF^*((L^0,E^0), (L^1,E^1); H,J) \to CF^*((L^0,E^0), (L^1,E^1); H',J'),
 \end{equation}
 \begin{equation}\label{pearly continuation map again}
 \Psi_{\F, J}^{\F', J'} \colon \left(C^*_f(E^0, E^1), d^{(\F, J)}\right) \to \left(C^*_{f'}(E^0, E^1), d^{(\F', J')}\right)
 \end{equation}
 and the PSS morphism 
 \begin{equation}\label{pss chain morphism}
 \Psi_{PSS} \colon \left(C^*_f(E^0, E^1), d^{(\F, J)}\right) \to CF^*((L, E^0), (L,E^1);H,J).
 \end{equation}
 are again defined using pre- and post-composition by parallel transport maps, we conclude by \eqref{m0 parallel again} that they all restrict to chain maps between the corresponding central subcomplexes. The proofs of Theorem \ref{main theorem of floer theory with local coefficients} \ref{invariance in strip model}) and Theorem \ref{main theorem on local pearly theory} \ref{pearly invariance}) \& \ref{pearly theorem part: PSS}) then apply to show that the restricted maps are in fact chain-homotopy equivalences. One can then make the following definition.
 \begin{definition}
 We define the central Floer cohomology of $E^0 \to L^0$ and $E^1 \to L^1$ to be
 $$\overline{HF}^*((L^0, E^0), (L^1, E^1)) \coloneqq H^*\left(\overline{CF}^*((L^0,E^0), (L^1,E^1); H,J), d^J\right)$$
 for some choice or regular Floer data $(H,J)$.
 When $L^0 = L^1 = L$ one has that $\overline{HF}^*((L, E^0), (L, E^1))$ is canonically isomorphic to 
 $H^*\left(\overline{C}^*_f(E^0, E^1), d^{(\F, J)}\right)$ via the PSS isomorphism.
 \end{definition}
 
 Particularly interesting is the case when $L^0 = L^1 = L$ and $E^0 = E^1 = E$. The Floer differential then preserves a further subcomplex whose definition is much more geometric. For each pair of points $x$ and $y$ on $L$ set
 $$\hom_{mon}(E_x, E_y) \coloneqq \spaN_{\FF_2}\{P_{\gamma} \colon E_{x} \to E_{y} \st \gamma \in \Pi_1L(x, y)\}.$$
 We then consider the following complex.
 \begin{definition}
 The monodromy Floer cochain complex of $E \to L$ is defined to be 
 \begin{equation}\label{monodromy floer complex}
CF^*_{mon}(E;H, J) \coloneqq \bigoplus_{x \in \mathcal{X_H}(L,L)} \hom_{mon}(E_{x(0)}, E_{x(1)}).
 \end{equation}
 \end{definition}
The observation that $CF^*_{mon}(E;H, J)\subseteq \overline{CF}^*(E,E;H,J)$ is equivalent to the fact that $m_0(E)$ is a parallel section, i.e. Proposition \ref{invariance of m0} \ref{m0 is a parallel section}).  
Since the Floer differential and continuation maps are defined using pre- and post-concatination by parallel transport maps and parallel transport maps form a set which is closed under composition, it is clear that $CF^*_{mon}(E;H, J)$ is in fact a subcomplex of $\overline{CF}^*(E,E;H,J)$ and that the maps \eqref{continuation map again} restrict to give chain-homotopy equivalences between monodromy cochain complexes. We then make the following definition.
\begin{definition}
The monodromy Floer cohomology of $E \to L$ is defined to be
$$HF^*_{mon}(E) \coloneqq H^*(CF^*_{mon}(E;H,J), d^J)$$
for some choice of regular Floer data $(H,J)$.
\end{definition}
Analogously, writing
$\End_{mon}(E_x) \coloneqq \hom_{mon}(E_x, E_x),$
one can consider the complex
\begin{equation}\label{monodromy pearl complex}
C^*_{f, mon}(E) \coloneqq \bigoplus_{x \in Crit(f)}\End_{mon}(E_x)
\end{equation}
and the same arguments as above show that $C^*_{f, mon}(E)$ is a subcomplex of the central complex $\overline{C}^*_f(E, E)$, that its homology is invariant under changes of Morse data and that the PSS morphism induces a canonical isomorphism 
$$H^*\left(C^*_{f, mon}(E), d^{(\F, J)}\right) \cong HF^*_{mon}(E).$$

\vspace{5mm}
Note that if one considers only the action of the Morse differential, one naturally obtains the \emph{central Morse complex} $\left(\overline{C}^*_f(E^0,E^1), \del_0\right)$ and the \emph{monodromy Morse complex} $\left(C^*_{f, mon}(E), \del_0\right)$. One can interpret these complexes using the following notions.
Observe that given local systems $E^0, E^1 \to L$, the fact that $m_0(E^0)$ and $m_0(E^1)$ are parallel sections implies that the assignment:
$$\forall x \in L \quad Z_{m_0}(E^0, E^1)_x \coloneqq \{\alpha \in \hom_{\FF_2}(E^0_{x}, E^1_{x})
 \st\alpha \circ m_0(E^0)(x) + m_0(E^1)(x) \circ \alpha = 0 \}$$
defines a local subsystem of $\hom(E^0, E^1)$.

 On the other hand, for any local system $E \to L$, the assignment
$$\forall x\in L \quad \End_{mon}(E)_x \coloneqq \End_{mon}(E_x)$$ 
defines a local subsystem of $\End(E)$ (which is the canonical \emph{local system of operator rings} in the terminology of \cite{steenrod1943homology}) and the fact that $m_0(E)$ is parallel implies the following inclusions of local systems on $L$:
$$\End_{mon}(E) \leq Z_{m_0}(E, E) \leq \End(E).$$ 
In the particular case when $E=E_{reg}$ is the local system induced by the right regular representation of $\pi_1(L)$ on $\FF_2[\pi_1(L)]$ it is not hard to check that 
\begin{equation}\label{endmonEreg is Econj}
\End_{mon}(E_{reg}) \cong E_{conj},
\end{equation}
where $E_{conj}$ is the local system induced by the conjugation action of $\pi_1(L)$ on $\FF_2[\pi_1(L)]$. 
  
  Appealing again to \cite[Appendix B]{abouzaid2012nearby} we conclude that: 
\begin{eqnarray*}
H^*\left(\overline{C}^*_f(E^0,E^1), \del_0\right) & \cong & H^*(L;Z_{m_0}(E^0, E^1))\\
H^*\left(C^*_{f, mon}(E), \del_0\right) & \cong & H^*(L; \End_{mon}(E)),
\end{eqnarray*}
where on the right-hand side we have the singular cohomologies of $L$ with coefficients in the local systems $Z_{m_0}(E^0, E^1)$ and $\End_{mon}(E)$, respectively.

\vspace{5mm}
Finally, let us examine the behaviour of these invariants under taking direct sums of local systems. Given local systems $E^{01}$, $E^{02}$ on $L^0$ and $E^{11}$, $E^{12}$ on $L^1$ we have for every $x_0 \in L^0$, $x_1 \in L^1$, the decomposition
\begin{eqnarray}\label{decomposition of homs}
\hom((E^{01} \oplus E^{02})_{x_0}, (E^{11} \oplus E^{12})_{x_1}) = \bigoplus_{\substack{i\in\{1,2\} \\ j \in \{1, 2\}}}\hom(E^{0j}_{x_0}, E^{1i}_{x_1}).
\end{eqnarray}
It is then convenient to represent an element $ \alpha \in \hom((E^{01} \oplus E^{02})_{x_0}, (E^{11} \oplus E^{12})_{x_1})$ as a matrix $\begin{pmatrix}
\alpha_{11}&\alpha_{12}\\
\alpha_{21}&\alpha_{22}
\end{pmatrix}$
with $\alpha_{ij} \in \hom(E^{0j}_{x_0}, E^{1i}_{x_1})$.
When similarly represented as matrices, the parallel transport maps for $E^{01} \oplus E^{02}$ and $E^{11} \oplus E^{12}$ have block-diagonal from. Since the Floer differential involves only pre- and post-composing elements $\alpha$ by such block-diagonal matrices, it follows that $d^J$ preserves the decomposition
$$CF^*((E^{01} \oplus E^{02}), (E^{11} \oplus E^{12})) = \bigoplus_{\substack{i\in\{1,2\} \\ j \in \{1, 2\}}}CF^*(E^{0j}, E^{1i}),$$
arising from \eqref{decomposition of homs}. Hence,  $CF^*((E^{01} \oplus E^{02}), (E^{11} \oplus E^{12}))$ is unobstructed if and only if $CF^*(E^{0j}, E^{1i})$ is unobstructed for all $i,j \in \{1,2\}$ and then
$$HF^*((E^{01} \oplus E^{02}), (E^{11} \oplus E^{12})) = \bigoplus_{\substack{i\in\{1,2\} \\ j \in \{1, 2\}}}HF^*(E^{0j}, E^{1i}).$$
Even when the full Floer complex is obstructed one still has the decomposition in central Floer cohomology, i.e.
\begin{equation}\label{decomposition for central complex}
\overline{HF}^*((E^{01} \oplus E^{02}), (E^{11} \oplus E^{12})) = \bigoplus_{\substack{i\in\{1,2\} \\ j \in \{1, 2\}}}\overline{HF}^*(E^{0j}, E^{1i}).
\end{equation}
This is because if $H$ is some regular Hamlitonian for $(L^0, L^1)$ and $x \in \X_{H}(L^0, L^1)$ with $x_0 =x(0)$ and $x_1=x(1)$, then $\alpha \in \hom((E^{01} \oplus E^{02})_{x_0}, (E^{11} \oplus E^{12})_{x_1})$ defines an element of $\overline{CF}^*((E^{01} \oplus E^{02}), (E^{11} \oplus E^{12}))$ if and only if 
$$\begin{pmatrix}
\alpha_{11}&\alpha_{12}\\
\alpha_{21}&\alpha_{22}
\end{pmatrix}
\begin{pmatrix}
m_0(E^{01})(x_0) & 0\\
0 & m_0(E^{02})(x_0)
\end{pmatrix}=
\begin{pmatrix}
m_0(E^{11})(x_1) & 0\\
0 & m_0(E^{12})(x_1)
\end{pmatrix}
\begin{pmatrix}
\alpha_{11}&\alpha_{12}\\
\alpha_{21}&\alpha_{22}
\end{pmatrix}.$$
This holds if and only if $\alpha_{ij} \in \overline{CF}^*(E^{0j}, E^{1i};H)$ for all $i,j \in \{1, 2\}$, which establishes the decomposition \eqref{decomposition for central complex}. 

On the other hand, the monodromy Morse and Floer complexes behave rather differently. The important observation here is that if $E$ and $W$ are local systems on $L$ and $m, n$ are two \emph{strictly} positive integers, then one has an isomorphism of complexes
\begin{equation}
C_{f, mon}(E^{\oplus m}\oplus W^{\oplus n}) \cong C_{f, mon}(E\oplus W), 
\end{equation}
i.e. the monodromy Floer complex does not see multiplicities. This is because for any $x, y \in L$ and $\epsilon, \delta \in \{0,1\}$ one has a surjective linear map
\begin{eqnarray}
\phi_{yx} \colon \hom_{mon}((E^{\oplus m}\oplus W^{\oplus n})_x, (E^{\oplus m}\oplus W^{\oplus n})_y)  & \longrightarrow & \hom_{mon}((E^{\oplus \epsilon}\oplus W^{\oplus \delta})_x, (E^{\oplus \epsilon}\oplus W^{\oplus \delta})_y) \nonumber \\
P_{\gamma}^{E^{\oplus m}\oplus W^{\oplus n}} & \longmapsto & P_{\gamma}^{E^{\oplus \epsilon}\oplus W^{\oplus \delta}}.\label{associating parallel transport maps}
\end{eqnarray}
Note that this is well-defined since we require that $m\ge 1$ and $n\ge 1$. It is also guaranteed to be an isomorphism whenever $\epsilon = \delta =1$ since then one has a well-defined inverse, given by reversing the arrow in \eqref{associating parallel transport maps}. Further, since $\phi_{yx}$ are defined just by matching the parallel transport maps to the underlying paths on $L$, they respect composition. That is, if $X \in \hom_{mon}((E^{\oplus m}\oplus W^{\oplus n})_x, (E^{\oplus m}\oplus W^{\oplus n})_y)$ and $Y \in \hom_{mon}((E^{\oplus m}\oplus W^{\oplus n})_y, (E^{\oplus m}\oplus W^{\oplus n})_z)$ then
\begin{equation}\label{phixy preserve composition}
\phi_{zx}(Y\circ X) = \phi_{zy}(Y) \circ \phi_{yx}(X).
\end{equation}
Now, specialising \eqref{associating parallel transport maps} to $x=y \in Crit(f)$ we have an isomorphism
\begin{equation}\label{phixx isomorphisms}
\xymatrix{
\bigoplus_{x \in Crit(f)}\phi_{xx}\; \colon C_{f, mon}(E^{\oplus m}\oplus W^{\oplus n}) \ar[r]^-{\cong}& C_{f, mon}(E\oplus W)}.
\end{equation}
Further, since $d^{(\F, J)}$ involves only pre- and post-composition by parallel transport maps, we see form \eqref{phixy preserve composition} that \eqref{phixx isomorphisms} commutes with the Floer differential and hence is an isomorphism of complexes.
\begin{remark}
The monodromy Floer cohomology seems to capture interesting information about the homotopy theory of loops on $L$. We defer more in-depth investigation of this invariant for future work. In section \ref{some additional calculations} below we provide some explicit calculations of monodromy Floer cohomology for the Chiang Lagrangian.
\end{remark}

\section{The Monotone Fukaya Category}\label{The Monotone Fukaya Category}
\subsection{Setup}\label{setup}
The next standard construction to which we seek to add local systems of arbitrary rank, is the monotone Fukaya category. That is, for a monotone symplectic manifold $(M, \omega)$, we would like to define a (non-curved) $\FF_2-$linear $A_{\infty}$ category whose objects are compact monotone Lagrangian submanifolds equipped with $\FF_2-$local systems, the morphism spaces are Floer cochain groups as in Definition \ref{definition of Floer cochains with local systems} and the first of the $A_{\infty}$ operations $\mu^1$ is a map as in Definition \ref{Floer differential with local coefficients} for appropriate choices of almost complex structures and Hamiltonian perturbations. Note that if $(L, E)$ is to be an object of such a category then we would need there to exist a Floer datum $(H, J)$ such that $CF^*((L, E), (L, E); H, J)$ is an honest complex (i.e. the associated map $d^J$ squares to zero). By point \ref{important remark on self-floer cohomology} in Remark \ref{remark with three parts}, this forces $m_0(E)$ to be a scalar operator and since we are working over $\FF_2$ this means $m_0(E) \in \{0, \Id\}$. 

Having made this preliminary observation, let us describe the constructions more precisely. We do so following closely the exposition in \cite{sheridan2013fukaya}, based in turn on \cite{seidel2008fukaya}. For each $w \in \{0, \Id\}$ we define an $\FF_2-$linear $A_{\infty}$ category $\F(M)_w$ whose objects are pairs $(L,E)$, where $L$ is a compact monotone Lagrangian with $N_L \ge 2$ and $E$ is a finite-rank $\FF_2-$local system with $m_0(E) = w$.  For technical reasons we also require that the image $\iota_*(\pi_1(L)) \subseteq \pi_1(M)$ under the map induced by inclusion is trivial (this is the analogue of the requirement to work with monotone \emph{pairs} of Lagrangians from Section \ref{floer cohomology and local systems}). For simplicity (and since this is what we need for applications) let us only construct a full subcategory of $\F(M)_w$ with a finite set of objects $\mathcal{L} = \{(L^i, E^i)\}$.

For every ordered pair $(L^i, L^j)$ ($i$ and $j$ not necessarily distinct) choose a regular Hamiltonian $H^{ij}\colon [0, 1] \times M \to \RR$ with corresponding flow $\psi^{ij}$ (so $\psi^{ij}_1(L^i) \pitchfork L^j$) and then for every $L^i$ choose 
$J_{L_i} \in \Jreg(L^i \vert \cup_j((L^i \cap (\psi^{ij}_1)^{-1}(L^j)) \cup (\psi^{ji}_1(L^j)\cap L^i))$ (recall that this notation means that evaluation maps from simple, $J_{L^i}-$holomorphic discs with one boundary marked point are transverse to all start and end points of Hamiltonian chords for the chosen $H^{ij}$). 
Now complete $H^{ij}$ to a regular Floer datum by choosing $J^{ij} \in \Jreg(\psi^{ij}_1(L^i), L^j)$ such that $J^{ij}_0 = (\psi^{ij}_1)_*J_{L^i}$ and $J^{ij}_1 = J_{L^j}$.
 We now define the hom--spaces in $\F(M)_w$ to be 
$$hom_{\F(M)_w}((L^i, E^i), (L^j, E^j)) \coloneqq CF^*((L^i, E^i), (L^j, E^j); H^{ij}, J^{ij})$$
and the first $A_{\infty}$ operation $\mu^1$ to consist of the differentials $d^{J^{ij}}$ on all these complexes. Having fixed all Floer data, we now drop it from the notation. We shall write $\X(L^i, L^j)$ for the set of Hamiltonian chords from $L^i$ to $L^j$ for the fixed regular Hamiltonian $H^{ij}$.

The construction of the higher $A_{\infty}$ operations is well-established, at least in the case of rank 1 local systems (see e.g. \cite[Section 2.3]{sheridan2013fukaya}, based on the constructions for exact manifolds from \cite{seidel2008fukaya}). In the wrapped setting, higher rank local systems have been described in detail by Abouzaid \cite{abouzaid2012nearby}. Thus, throughout this discussion we omit a lot of technical details (mostly from \cite[Section 9]{seidel2008fukaya}), in particular the fact that Floer and perturbation data can be chosen in such a way that all moduli spaces which appear are smooth manifolds of the correct dimensions and admitting the correct compactifications.
That this is possible (i.e. that modified proofs from \cite{seidel2008fukaya} apply) is an artefact of monotonicity.

Let us now give a brief description of how the constructions are modified to incorporate local coefficients. 
For every $d \ge 2$ and any $d$-tuple of objects $\{(L^j, E^j)\}_{0 \le j \le d}$ there is a linear map
$$\mu^d \colon CF^*(E^{d-1}, E^{d}) \otimes \cdots \otimes CF^*(E^{0}, E^{1}) \longrightarrow CF^*(E^0, E^d),$$
which is defined by counting isolated perturbed pseudoholomorphic polygons with boundary on the Lagrangians $L^0, L^1, \ldots, L^d$ and using their boundary components for parallel transport. More precisely, let $D$ denote the unit disc in $\CC$ and let $\{\zeta_0, \zeta_1, \ldots, \zeta_d\}$ be a counterclockwise cyclicly ordered set of points on $\del D$ which are labeled either positive (also called incoming) or negative (outgoing). We call each $\zeta_j$ a positive, respectively negative puncture. A choice of strip-like ends for $(D, \zeta_0, \ldots, \zeta_d)$ is a collection of pairwise disjoint open neighbourhoods $\zeta_j \in U_j \subseteq D$, together with biholomorphisms $\epsilon_j \colon \RR^{\pm} \times [0, 1] \to U_j$,
satisfying $\epsilon_j^{-1}(\del U_j) = \del (\RR^{\pm} \times [0, 1])$ and $\lim_{\lvert s \rvert \to +\infty} \epsilon_j(s, t) = \zeta_j$, where $\RR^+  = (0, +\infty)$, $\RR^- = (-\infty, 0)$ and the choice between the two domains is determined by whether the corresponding puncture is labeled positive or negative. Suppose one is given a set of objects $\{(L^j, E^j)\}_{0 \le j \le d}$ and Hamiltonian chords:
$x_0 \in \X(L^0, L^d)$ and $\{x_j\}_{1 \le j \le d}$, with $x_j \in \X(L^{j-1}, L^j)$. Let $(D, \zeta_0, \ldots, \zeta_d)$ be as above with $\zeta_0$ labeled negative and all other punctures labeled positive and assume one has made a choice of strip-like ends. Then any continuous map $u \colon D \setminus \{\zeta_0, \ldots, \zeta_d\} \to M$, mapping the boundary arc between $\zeta_{j}$ and $\zeta_{j+1}$ to $L^j$ (with $\zeta_{d+1} \coloneqq \zeta_0$) and satisfying $\lim_{\lvert s \rvert \to +\infty}u(\epsilon_j(s, t)) = x_j(t)$ uniformly in $t$, gives rise to a linear map
\begin{equation}\label{the map mu^u}
\mu_u \colon \hom_{\FF_2}(E^{d-1}_{x_d(0)}, E^d_{x_d(1)}) \otimes \cdots \otimes \hom_{\FF_2}(E^0_{x_1(0)}, E^1_{x_1(1)}) \longrightarrow \hom_{\FF_2}(E^0_{x_0(0)},E^d_{x_0(1)})
\end{equation}
$$\mu_u(\alpha_d \otimes \alpha_{d-1} \otimes \cdots \otimes \alpha_1) = P_{\gamma^d_u} \circ \alpha_d \circ P_{\gamma^{d-1}_u} \circ \alpha_{d-1} \circ \cdots \circ P_{\gamma^1_u} \circ \alpha_1 \circ P_{\gamma^0_u},$$
where $\gamma^0_u \in \Pi_1L^0(x_0(0), x_{1}(0))$, $\gamma^d_u \in \Pi_1L^d(x_d(1), x_0(1))$ and $\gamma^j_u \in \Pi_1L^j(x_j(1), x_{j+1}(0)), \; 1 \le j \le d-1$ are the compactified images under $u$ of the arcs between $\zeta_j$ and $\zeta_{j+1}$ (see Figure \ref{figure structure maps}).
\begin{figure}[h!]
\labellist
\pinlabel $\alpha_d$ at 496 639
\pinlabel $\alpha_1$ at 456 82
\pinlabel $\alpha_2$ at 797 49
\pinlabel $E^0_{x_1(0)}$ at 430 175
\pinlabel $E^1_{x_1(1)}$ at 570 120
\pinlabel $E^1_{x_2(0)}$ at 695 103
\pinlabel $E^2_{x_2(1)}$ at 819 153
\pinlabel $E^{d-1}_{x_d(0)}$ at 587 570
\pinlabel $E^d_{x_d(1)}$ at 451 526
\pinlabel $x_1$ at 502 151
\pinlabel $x_2$ at 758 126
\pinlabel $x_d$ at 523 569
\pinlabel $x_0$ at 203 375
\pinlabel $\gamma^0_u$ at 514 325
\pinlabel $\gamma^1_u$ at 629 254
\pinlabel $\gamma^2_u$ at 730 288
\pinlabel $\gamma^{d-1}_u$ at 656 439
\pinlabel $\gamma^d_u$ at 500 400
\endlabellist
\begin{center}
\includegraphics[scale=0.4]{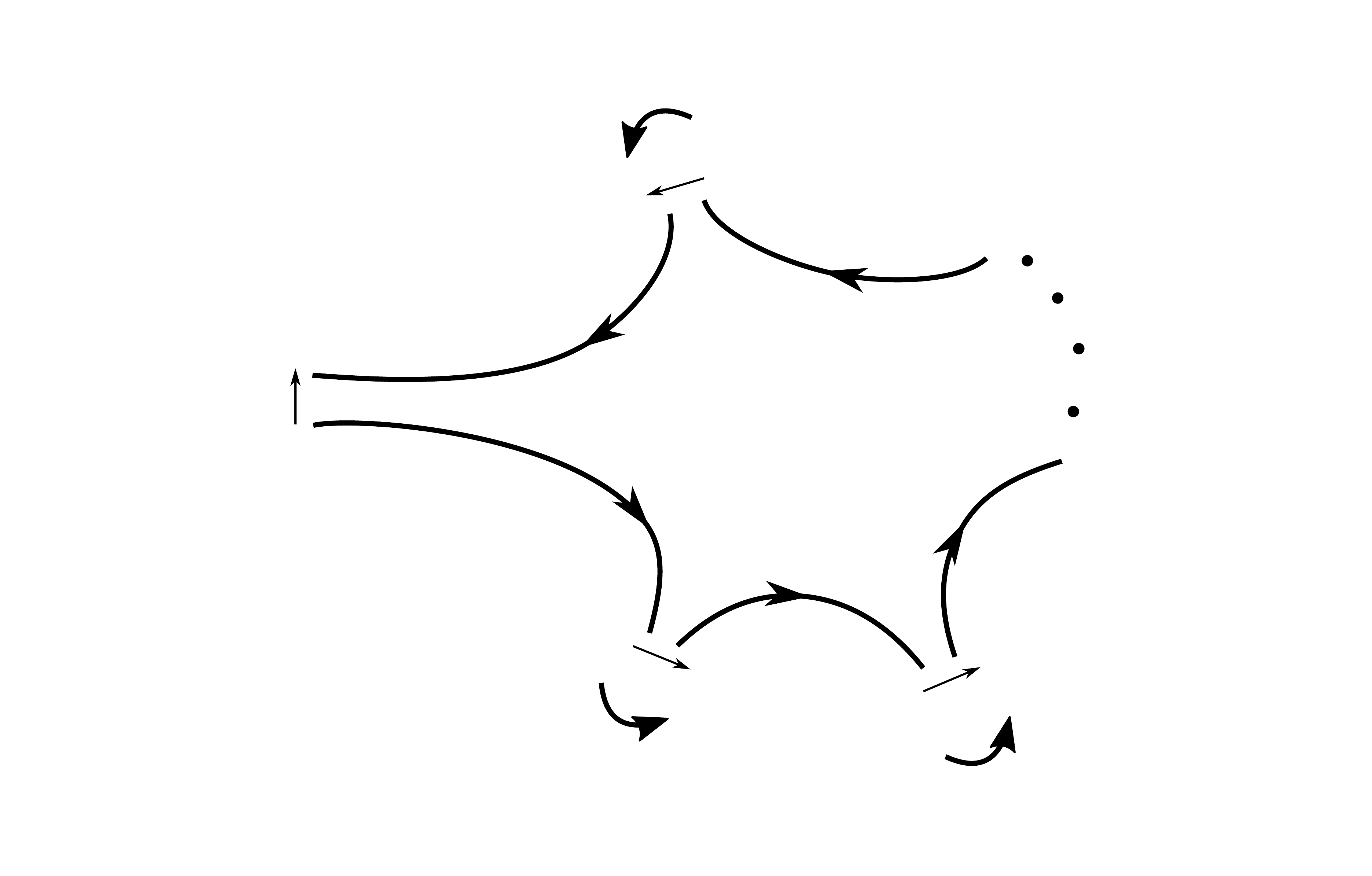}
\caption{The structure maps}\label{figure structure maps}
\end{center}
\end{figure}
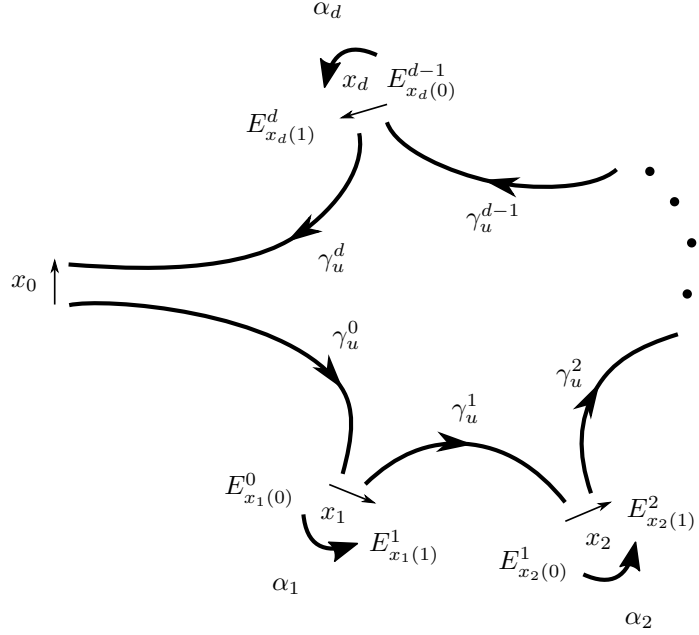
We denote the moduli space of maps as above, satisfying the appropriately perturbed Cauchy-Riemann equation by $\R_{1:d}(x_0:x_1, \ldots, x_d)$ and its $k$-dimensional component by $\R^k_{1:d}(x_0:x_1, \ldots, x_d)$ (for this to make sense one needs to first make consistent choices of strip-like ends for the universal families $\R_{1:d}$ of abstract holomorphic discs with $d$ positive punctures and one negative and then make choices of perturbation data for these families which is consistent with gluing and ensures transversality - see \cite[(9g),(9h),(9i)]{seidel2008fukaya}; similar procedures need to be applied to all moduli spaces we discuss below).
One then defines the $A_{\infty}$ operations by setting:
$$\mu^d \colon CF^*(E^{d-1}, E^{d}) \otimes \cdots \otimes CF^*(E^{0}, E^{1}) \longrightarrow CF^*(E^0, E^d),$$
$$\mu^d \coloneqq \sum_{\substack{x_0 \in \X(L^0,L^d)\\ (x_1, \ldots, x_d) \in \Pi_{j=1}^d\X(L^{j-1}, L^j)}}\sum_{u \in \R_{1:d}^0(x_0: x_1, \ldots, x_d)} \mu_u.$$ 
Note that $\mu^1$ is indeed just the collection of differentials $d^{J^{ij}}$. We call an object $(L,E)$ of $\F(M)_w$ \emph{essential} whenever $H^*(hom_{\F(M)_w}(E,E), \mu^1) = HF^*(E,E) \neq 0$. 
\vspace{1em}
\par \begin{remark}
Monotonicity, together with the requirement that the images $\iota^i_*(\pi_1(L^i)) \subseteq \pi_1(M)$ be trivial, ensures uniform energy bounds on pseudoholomorphic maps belonging to spaces of the same expected dimension, so that Gromov compactness applies. 
In particular zero-dimensional moduli spaces are compact, so that all sums ranging over such spaces are finite. Finally, disc and sphere bubbles do not appear in any of the constructions apart from $\mu^1$ which we discussed at length above.
This is because all other constructions involve only zero- and one-dimensional moduli spaces of maps which satisfy a perturbed Cauchy-Riemann equation that does not admit an $\RR-$action and so rely only on Fredholm problems of index 0 and 1. The conditions $N_M \ge 1$ and $N_{L_i} \ge 2$ imply that any sphere or disc bubble would reduce the sum of the Fredholm indices governing the remaining components by at least 2, making them all negative and thus contradicting transversality. 
\end{remark}\par
\noindent The $A_{\infty}$ associativity relations 
$$ \sum_{j=1}^d\sum_{i=0}^{d-j} \mu^{d-j+1}(\alpha_d, \ldots, \alpha_{i+j+1}, \mu^j(\alpha_{i+j}, \ldots, \alpha_{i+1}), \alpha_i \ldots, \alpha_1) = 0
$$
are shown to hold by considering the Gromov compactification $\cl{\R}_{1:d}^1(x_0: x_1, \ldots, x_d)$ of the one-dimensional component of such maps (see \cite[(9l)]{seidel2008fukaya}) and using the fact that the paths used for parallel transport, which are determined by configurations of broken curves appearing at opposite ends of an interval in $\cl{\R}_{1:d}^1(x_0: x_1, \ldots, x_d)$ are homotopic (for an example of a similar argument see Figure \ref{figure for OC} below). This finishes the setup of the extended monotone Fukaya category $\F(M)_{w}$, which is now allowed to contain any set of objects $\{(L^i, E^i)\}$ which satisfy $m_0(E^i) = w$. 
Note that for any object $(L, E)$ of this category, the structure maps $\mu^*$ make $CF^*(E,E)$ into an $A_{\infty}$ algebra. 
\vspace{1em}
\begin{remark}
The above construction depends heavily on choices of strip-like ends and regular Floer and perturbation data. It is a fact that different choices yield quasi-equivalent categories (see \cite{seidel2008fukaya}, (10a)). We will not require this fact here.
\end{remark}

\subsection{Split-Generation}\label{split-generation}
Our main purpose for discussing this extended Fukaya category is so that we can prove the non-displaceablity result Corollary \ref{chiangRp3nonDisp}. To that end we need a version of Abouzaid's split-generation criterion \cite{abouzaid2010geometric} to hold in this setting. Such an extension has already been proved in \cite{abouzaid2012nearby} for the wrapped Fukaya category and our situation is in fact a lot simpler since we won't have to deal with infinite-dimensional $\hom-\,$spaces. 
For the sake of completeness we include a discussion of the split-generation criterion and its proof below.
Recall first that if $\A$ is any cohomologically unital $A_{\infty}$ category then an object $L$ is said to split-generate an object $K$ if $K$ is quasi-isomorphic to an object in the smallest triangulated (in the $A_{\infty}$ sense) and idempotent closed subcategory of $\Pi(Tw \A)$ containing $L$, where $\Pi(Tw \A)$ denotes the split-closure of the category $Tw(\A)$ of twisted complexes over $\A$ (see \cite[(3l), (4c)]{seidel2008fukaya}). Split-generation is important for computations in Fukaya categories but in the present work we are interested only in the following well-known consequence.
 \begin{lemma}\label{split-generation implies non-vanishing of HF}
Suppose that $K$ is split-generated by $L$ and $H^*(hom_{\A}(K,K), \mu^1) \neq 0$. Then $H^*(hom_{\A}(L,K), \mu^1) \neq 0$.\qed
\end{lemma}

 Suppose now that $(L, E)$ and $(K, W)$ are objects of $\F(M)_{w}$ and $E$ and $W$ have finite ranks. There exist linear maps 
 $$\CO^* \colon QH^*(M) \to HH^*(CF^*((L,E),(L,E)))$$
 $$\OC_* \colon HH_*(CF^*((L, E), (L, E))) \to QH^*(M),$$
  the \emph{closed-open} and \emph{open-closed string maps}, relating the quantum cohomology of the ambient manifold to Hochschild invariants of the $A_{\infty}$ algebra $CF^*(E,E)$ (we review these concepts and the maps themselves below). The version of the split-generation criterion we need is the following:
 \begin{theorem}\label{split-generation criterion}
 Let $(L, E)$ be an object of $\F(M)_w$, where $E$ has finite rank. If the map $\CO^*$ is injective, then any other object $(K, W) \in \F(M)_w$ with $W$ of finite rank is split-generated by $(L,E)$.
\end{theorem}

This theorem is due to Abouzaid (\cite{abouzaid2010geometric}) in the case of exact Lagrangians in an exact symplectic manifold and when $E$ and $W$ are trivial of rank 1. The case of a general symplectic manifold is work in progress by Abouzaid-Fukaya-Oh-Ohta-Ono \cite{afoooinpreparation}. Still in the exact case, the paper \cite{abouzaid2012nearby} proves a version in which $W$ is allowed to be non-trivial and possibly of infinite rank.
This last requirement is the cause of several algebraic complications which we avoid here.
The proof for the monotone setting and with $E$ and $W$ of rank 1 (though possibly non-trivial) is treated in 
\cite[Section 2.11]{sheridan2013fukaya}. We include a sketch of that proof, modified to incorporate local systems of any finite rank. In our application to the Chiang Lagrangian we shall only use the split-generation criterion in the case when $E$ is trivial of rank 1 but for completeness we treat the slightly more general case here. We begin by reviewing the  closed-open and open-closed string maps.

\subsubsection{Hochschild cohomology and the closed-open string map}
Let us first describe the map
\begin{equation}\label{closed-open map}
\CO^* \colon QH^*(M) \to HH^*(CF^*((L,E),(L,E)).
\end{equation}
Its domain is (in our case) the ungraded small quantum cohomology ring of $M$, whose inderlying vector space is simply $H^*(M;\FF_2)$ but whose ring structure is deformed by ``quantum contributions'' arising from counts of pseudoholomorphic spheres (for a brief account see e.g. \cite[Section 2.2]{sheridan2013fukaya}; full details are given in \cite[Chapter 11]{mcduff2012j}). We denote this product by $\star$. It is a fact that $\star$ is associative, commutative (graded commutative when one works over characteristic different from 2 and $QH^*$ is graded) and together with the Poincar\'{e} pairing makes $QH^*(M)$ into a Frobenius algebra, i.e.
$$\langle a \star b, c \rangle = \langle a, b \star c \rangle.$$
Further, the usual unit $1 \in H^*(M;\FF_2)$ is also a unit for the $\star$ product.

 The target of $\CO^*$ is the Hochschild cohomology of the $A_{\infty}$ algebra $CF^*(E,E)$. The Hochschild cochain complex is 
$CC^*(CF^*(E,E)) \coloneqq \Pi_{d \ge 0}\; CC_c^*(E,E)^d$, where
$$CC_c^*(E,E)^d \coloneqq \hom_{\FF_2}(CF^*(E,E)^{\otimes d}, CF^*(E,E)),$$ equipped with the differential 
\begin{eqnarray*}
\delta((\phi^0, \phi^1, \ldots))^d(\alpha_d, \ldots, \alpha_1) & = & \sum_{j=0}^d\sum_{i=0}^{d-j} \mu^{d-j+1}(\alpha_d, \ldots, \alpha_{i+j+1}, \phi^j(\alpha_{i+j}, \ldots, \alpha_{i+1}), \alpha_i \ldots, \alpha_1) \\
&+& \sum_{j=1}^d\sum_{i=0}^{d-j} \phi^{d-j+1}(\alpha_d, \ldots, \alpha_{i+j+1}, \mu^j(\alpha_{i+j}, \ldots, \alpha_{i+1}), \alpha_i \ldots, \alpha_1).
\end{eqnarray*}
Given an element $\beta \in QH^*(M)$ and a pseudocycle $f \colon B \to M$ representing the Poincar\'{e} dual of $\beta$, one defines a corresponding Hochschild cochain $\CO^*(\beta;f) = (\CO^*(\beta;f)^d)_{d \ge 0} \in CC^*(CF^*(E,E))$ as follows. For every tuple of Hamiltonian chords $(x, \vec{x}) \coloneqq (x, x_1, \ldots, x_d)$ in $ \X(L,L)$ one considers the moduli space $\R_{1:d;1}(x:\vec{x};f)$ of perturbed pseudoholomorphic maps $u$ from a disc with $d$ positive boundary punctures, asymptotic to $\vec{x}$, one negative boundary puncture which is asymptotic to $x$ and an internal marked point which is mapped to $im(f)$. Every $u \in \R_{1:d;1}(x:\vec{x};f)$ defines a map $\mu_u$ as in equation \eqref{the map mu^u}. One then sets
$$\CO^*(\beta;f)^d \coloneqq \sum_{(x, \vec{x}) \in \X(L,L)^{d+1}} \sum _{u \in \R^0_{1:d;1}(x:\vec{x};f)} \mu_u.$$
 
The facts that the resulting element is $\delta-$closed and that its cohomology class is independent of the choice of pseudocycle $f$ are proved for rank 1 local systems in \cite[Section 2.5]{sheridan2013fukaya} and the proofs hold just as well in our case (we review a similar argument for the open-closed string map in more detail below). 
By inspecting the definition of the differential $\delta$ one sees that the length-zero projection  
$CC^*(CF^*(E,E)) \to CF^*(E,E)$, $(\phi^0, \phi^1, \ldots) \mapsto \phi^0$ is a chain map. Composing $\CO^*$ with this projection at the level of cohomology gives the map
$$\CO^0 \colon QH^*(M) \to HF^*(E,E).$$
It is an important result that when $HF^*(E,E)$ is equipped with the Floer product\footnote{That is, the product induced by the operation $\mu^2$; the $A_{\infty}$ relations guarantee that this product is well-defined and associative on cohomology.} the map $\CO^0$ is an algebra homomorphism and the element $e_E = e_{(L,E)} \coloneqq \CO^0(1)$ is a unit for $HF^*(E,E)$. Again the proof can be taken directly from \cite[Sections 2.4, 2.5 ]{sheridan2013fukaya}.
\begin{remark}
More generally, the Hochschild cohomology $HH^*(CF^*(E,E))$ itself is an algebra when equipped with the so-called Yoneda product (see e.g. \cite[equation (A.4.1)]{sheridan2013fukaya}); $e_E$ is a unit for this structure as well and the full map $\CO^*$ is also a unital algebra homomorphism. 
\end{remark}

\subsubsection{Hochschild homology and the open-closed string map}

We now describe the open-closed string map $\OC_*\colon HH_*(CF^*(E,E)) \to QH^*(M)$, paying attention to the fact that we allow $E$ to have finite rank higher than one. There is a Hochschild homology group $HH_*(CF^*(E,E), \N)$ for any $A_{\infty}$ bimodule over $CF^*(E,E)$. It is the homology of the complex
$$CC_*(CF^*(E,E), \N) \coloneqq \bigoplus_{d \ge 0} \N \otimes CF^*(E,E)^{\otimes d}$$ 
with respect to the $A_{\infty}$ cyclic bar differential 
\begin{eqnarray*}
b(\underline{n}, \alpha_d, \ldots, \alpha_1) &=& \sum_{\substack{r \ge 0, s \ge 0 \\ r+s \le d}} \mu_{\N}^{r\vert1\vert s} (\alpha_r, \ldots, \alpha_1, \underline{n}, \alpha_d, \ldots,\alpha_{d-s+1})\otimes \alpha_{d-s} \otimes \cdots \otimes \alpha_{r+1} \\
&+& \sum_{\substack{i \ge 0, j\ge 1 \\ i+j \le d}} \underline{n} \otimes \alpha_d \otimes \cdots \otimes \alpha_{i+j+1} \otimes \mu^{j}(\alpha_{i+j}, \ldots, \alpha_{i+1}) \otimes \alpha_i \otimes \cdots \otimes \alpha_1,
\end{eqnarray*}
where $\mu_{\N}^{\cdot\vert1\vert \cdot}$ denote the bimodule structure maps for $\N$. Substituting $\N = CF^*(E,E)$ one obtains the group $HH_*(CF^*(E,E))$, which is the source of $\OC_*$. Following \cite[Section 2.6]{sheridan2013fukaya}, we define the open-closed string map in terms of a pairing 
\begin{equation}\label{OC pairing}
( \OC_*(-),- ) \colon HH_*(CF^*(E,E)) \otimes H_*(M; \FF_2) \to \FF_2.
\end{equation}
Given a generator 
\begin{eqnarray*}
\underline{\alpha} \otimes \alpha_d \otimes \cdots \otimes \alpha_1 &\in& \hom_{\FF_2}(E_{\underline{x}(0)}, E_{\underline{x}(1)}) \otimes\hom_{\FF_2}(E_{x_d(0)}, E_{x_d(1)}) \otimes \cdots \otimes \hom_{\FF_2}(E_{x_1(0)}, E_{x_1(1)})\\
& \le &  CC_*(CF^*(E,E))
\end{eqnarray*}
and a pseudocycle $f$, representing a homology class $a$, we consider the moduli space $\R_{0: d+1;1}(\underline{x},\vec{x}; f)$, consisting of perturbed pseudoholomorphic discs 
asymptotic to $\underline{x}$ and $\vec{x} \coloneqq (x_1, \ldots, x_d)$ at the boundary punctures and mapping the boundary to $L$ and the internal marked point to $im(f)$. We define
$$( \OC_*(\underline{\alpha} \otimes \alpha_d \otimes \cdots \otimes \alpha_1), a ;f ) \coloneqq 
\sum_{u \in \R^0_{0, d+1;1}(\underline{x}, \vec{x}; f)} tr(P_{\gamma^d_u} \circ \alpha_d \circ P_{\gamma^{d-1}_u} \circ \alpha_{d-1} \circ \cdots \circ P_{\gamma^1_u} \circ \alpha_1 \circ P_{\gamma^0_u} \circ \underline{\alpha}),
$$
where on the right hand side one takes the trace of the element in brackets which is an endomorphism of $E_{\underline{x}(0)}$. For index reasons, the boundary of the Gromov compactification of the 1-dimensional component $\R^1_{0, d+1;1}(\underline{x},\vec{x}; f)$ consists only of strip breakings at the incoming punctures and configurations of pairs of discs, one of which carries the internal marked point and the other carries at least two punctures. These are precisely the moduli spaces contributing to the composition
\begin{equation}\label{OC composed with b}
\xymatrixcolsep{5pc}
\xymatrix{
CC_*(CF^*(E,E)) \ar[r]^b & CC_*(CF^*(E,E)) \ar[r]^-{( \OC_*(-), a; f )} & \FF_2
}.
\end{equation}
 We claim that this implies $\langle\OC_*(b(\underline{\alpha}\otimes \alpha_d \otimes \cdots \otimes \alpha_1)), a ;f \rangle =0$. Let us illustrate this by an example. Suppose that $d=4$ and the two broken configurations in Figure \ref{figure for OC} appear as opposite boundary points of a connected component of $\cl{\R}_{0: 5;1}^1(\underline{x}, x_1, x_2, x_3, x_4;f)$. 
 \begin{figure}
 \begin{center}
 \labellist
 \pinlabel $\underline{\alpha}$ at 975 744
 \pinlabel $\alpha_1$ at 681 1039
 \pinlabel $\alpha_2$ at 74 924
 \pinlabel $\alpha_3$ at 85 710
 \pinlabel $\alpha_4$ at 655 564
 \pinlabel {$(u, v)$} at 519 916
 \pinlabel $f$ [l] at 338 830
 \pinlabel $\gamma^0_u$ at 372 926
 \pinlabel $\gamma^1_u$ at 219 820
 \pinlabel $\gamma^2_u$ at 354 710
 \pinlabel $\gamma^0_v$ at 611 731
 \pinlabel $\gamma^1_v$ at 793 676
 \pinlabel $\gamma^2_v$ at 812 856
 \pinlabel $\gamma^3_v$ at 650 893
 \pinlabel $\underline{\alpha}$ at 328 522
 \pinlabel $\alpha_1$ at 31 240
 \pinlabel $\alpha_2$ at 337 30
 \pinlabel $\alpha_3$ at 929 184
 \pinlabel $\alpha_4$ at 935 403
 \pinlabel {$(u', v')$} at 479 390
 \pinlabel $f$ [l] at 284 269
 \pinlabel $\gamma^0_{u'}$ at 206 354
 \pinlabel $\gamma^1_{u'}$ at 197 150 
 \pinlabel $\gamma^2_{u'}$ at 390 200
 \pinlabel $\gamma^3_{u'}$ at 375 380
 \pinlabel $\gamma^0_{v'}$ at 634 190
 \pinlabel $\gamma^1_{v'}$ at 780 295
 \pinlabel $\gamma^2_{v'}$ at 642 430
 \endlabellist
 \includegraphics[scale=0.24]{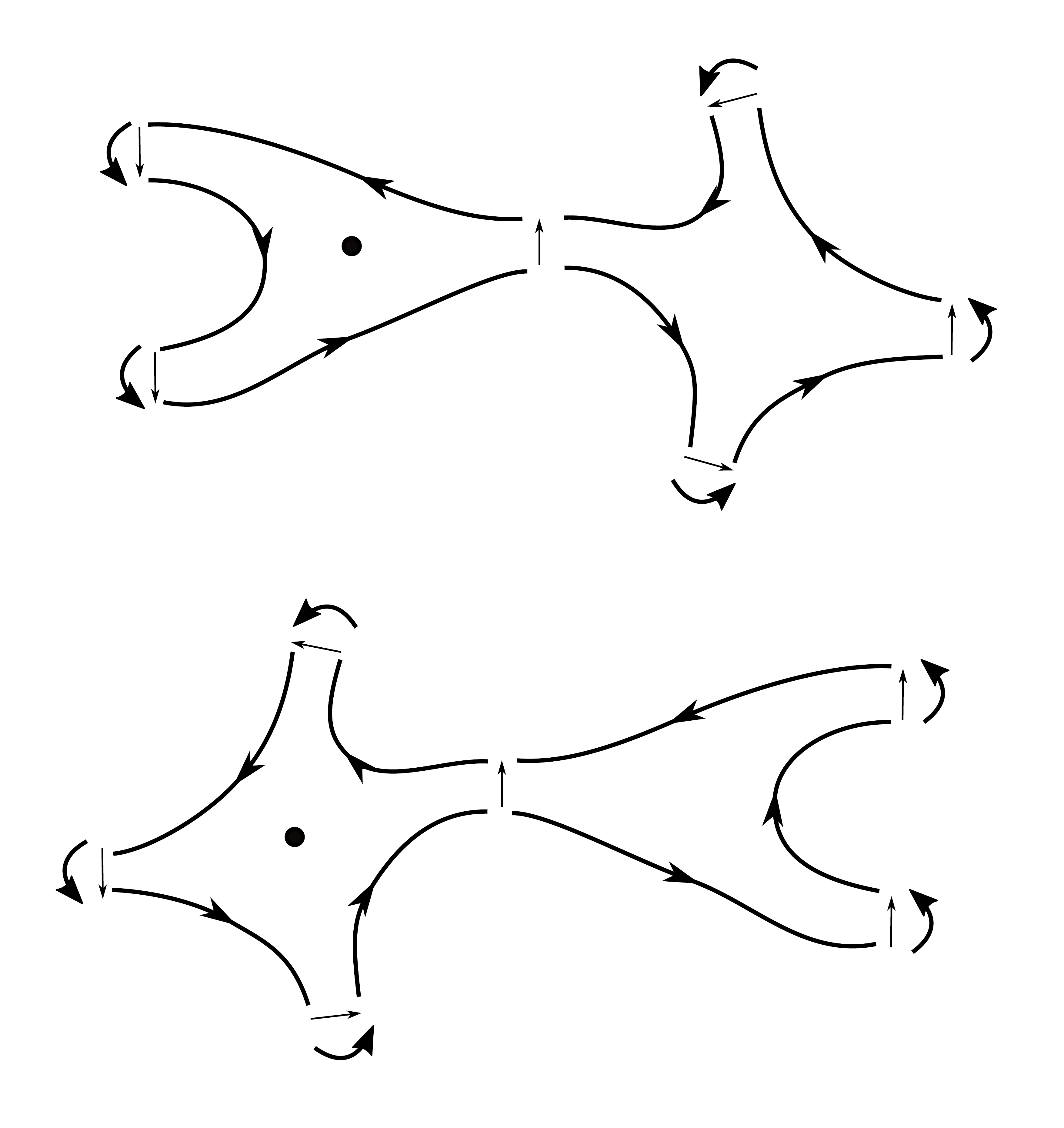}
 \caption{Evaluating $\OC_*$ on a Hochschild boundary}\label{figure for OC}
 \end{center}
 \end{figure}
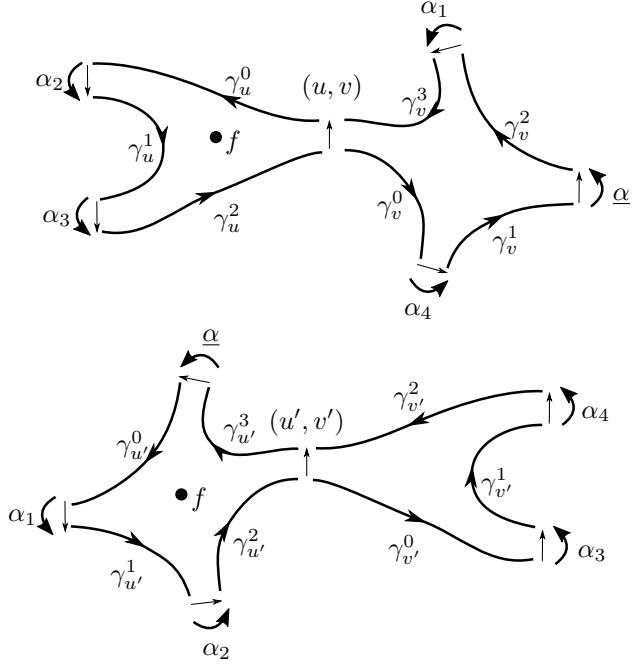
 
 Their contributions to the composition \eqref{OC composed with b} are given by:
 $$tr \left(P_{\gamma^2_u} \circ \alpha_3 \circ P_{\gamma^1_u} \circ \alpha_2 \circ P_{\gamma^0_u} \circ \left(P_{\gamma^3_v} \circ \alpha_1 \circ P_{\gamma^2_v} \circ\underline{\alpha} \circ P_{\gamma^1_v} \circ \alpha_4 \circ P_{\gamma^0_v}\right)\right)$$
 and
 $$tr \left(P_{\gamma^3_{u'}} \circ \left(P_{\gamma^2_{v'}} \circ \alpha_4 \circ P_{\gamma^1_{v'}} \circ \alpha_3 \circ P_{\gamma^0_{v'}}\right) \circ P_{\gamma^2_{u'}} \circ \alpha_2 \circ P_{\gamma^1_{u'}} \circ \alpha_1 P_{\gamma^0_{u'}} \circ \underline{\alpha}\right).$$
Since there is a 1-parameter family of glued curves interpolating between the two broken configurations, we conclude that for every $0 \le j \le 4$ the two paths connecting $x_j(1)$ to $x_{j+1}(0)$ (where $x_0 = x_5 = \underline{x}$) arising from $(u, v)$ and $(u', v')$ are homotopic. In particular $\gamma^3_v \cdot \gamma^0_u = \gamma^1_{u'} \in \Pi_1L(x_1(1), x_2(0))$, $\gamma^2_u \cdot \gamma^0_v = \gamma^1_{v'} \in \Pi_1L(x_3(1), x_4(0))$, $\gamma^1_{v}  = \gamma^2_{v'} \cdot \gamma^3_{u'}\in \Pi_1L(x_4(1), \underline{x}(0))$ and $\gamma^1_{u}  = \gamma^2_{u'} \cdot \gamma^0_{v'}\in \Pi_1L(x_2(1), x_3(0))$. Using this we see that the two expressions of which we are taking the trace are cyclic permutations of compositions of the same maps and hence the traces agree. Since all broken configurations contributing to \eqref{OC composed with b} come in such pairs, we conclude that the composition vanishes altogether.

 On the other hand, given a Hochschild chain $\varphi$ and two pseudocycles $f$, $g$ representing $a$, then by considering moduli spaces of discs with asymptotics determined by $\varphi$ and which map the internal marked point to a homology between $f$ and $g$ one can show (see \cite[Section 2.6]{sheridan2013fukaya}) that $( \OC_*(\varphi) , a ;f ) + ( \OC_*(\varphi) , a ;g )$ depends only on $b(\varphi)$ and so vanishes when $\varphi$ is a Hochschild cycle. One thus obtains a well defined pairing \eqref{OC pairing} which defines the map $\OC_*$.
 
\subsubsection{The bimodule $P_W(E)$ and the evaluation map $H(\mu)$}
Let us consider for a moment a purely algebraic setup. Let $\A$ be an $A_{\infty}$ category and let $E$ be an object of $\A$. Then for every object $W$ one can consider the space $P_W(E) \coloneqq hom_{\A}(E,W) \otimes hom_{\A}(W,E)$ which is an $A_{\infty}$ bimodule over $hom_{A}(E,E)$ with structure maps 
\begin{eqnarray*}
\mu^{r\vert 1 \vert 0} \colon hom_{\A}(E,E)^{\otimes r} \otimes P_W(E) & \to & P_W(E) \\
\mu^{r\vert 1 \vert 0}(\alpha_r, \ldots, \alpha_1, f \otimes g)& = & f \otimes \mu^{r+1}(\alpha_r, \ldots, \alpha_1, g),
\end{eqnarray*}
\begin{eqnarray*}
\mu^{0\vert 1 \vert s} \colon P_W(E) \otimes hom_{\A}(E,E)^{\otimes s} & \to & P_W(E) \\
\mu^{0\vert 1 \vert s}(f \otimes g, \alpha_{\vert 1}, \ldots, \alpha_{\vert s})& = & \mu^{s+1}(f, \alpha_{\vert 1}, \ldots, \alpha_{\vert s})\otimes g
\end{eqnarray*}
and $\mu^{r \vert 1 \vert s} = 0$ for $r \neq 0 \neq s$. Thus one has a Hochschild homology group $HH_*(hom_{\A}(E,E), P_W(E))$. There is a natural evaluation map:
$$H(\mu) \colon HH_*(hom_{\A}(E,E), P_W(E)) \to H^*(hom_{\A}(W,W), \mu^1),$$ induced on the chain level by the map:
\begin{eqnarray*}
C(\mu) \colon  CC_*(hom_{\A}(E,E), P_W(E)) & \to & hom_{\A}(W,W)\\
C(\mu) \colon  (f \otimes g) \otimes \alpha_d \otimes \cdots \otimes \alpha_1 & \mapsto & \mu^{d+2}(f, \alpha_d, \ldots, \alpha_1, g).
\end{eqnarray*}
In this setting one has the following lemma of Abouzaid:
\begin{lemma}(\cite[Lemma 1.4]{abouzaid2010geometric})\label{algebraic split-generation}
Let $\A$ be a cohomologically unital $A_{\infty}$ category and $E$, $W$ be objects in $\A$. If the unit $e_W \in H^*(hom_{\A}(W,W), \mu^1)$ lies in the image of the evaluation map $H(\mu)$,
then $W$ is split-generated by $E$.
\end{lemma}
Let us now specialise to the case where $\A$ is the category $\F(M)_w$ from section \ref{setup} above. The  bimodule is then $P_W(E) = P_{(K,W)}(L,E) = CF^*(E, W) \otimes CF^*(W, E)$. 
Note that this can be rewritten as 
\begin{eqnarray*}
CF^*(E, W) \otimes CF^*(W, E) &= &\left(\bigoplus_{y \in \X(L, K)}\hom(E_{y(0)}, W_{y(1)})\right)  \otimes \left(\bigoplus_{z \in \X(K, L)}\hom(W_{z(0)}, E_{z(1))})\right)\\
&=&\bigoplus_{\substack{y \in \X(L, K)\\z \in \X(K, L)}}\hom(E_{y(0)}, W_{y(1)})  \otimes  \hom(W_{z(0)}, E_{z(1))})\\
&=&\bigoplus_{\substack{y \in \X(L, K)\\z \in \X(K, L)}} E^{\vee}_{y(0)} \otimes W_{y(1)}  \otimes  W^{\vee}_{z(0)} \otimes E_{z(1)}\\
&=&\bigoplus_{\substack{y \in \X(L, K)\\z \in \X(K, L)}} \hom(W_{z(0)}, W_{y(1)}) \otimes \hom(E_{y(0)}, E_{z(1)}),
\end{eqnarray*}
where we have crucially used the fact that $E$ and $W$ have finite rank.
From now on we shall refer to the elements of the components of $P_W(E)$ in one of the following two ways
\begin{itemize}
\item $\hat{f}_y \otimes \hat{g}_z \in \hom(E_{y(0)}, W_{y(1)}) \otimes \hom(W_{z(0)}, E_{z(1))})$
\item $f_{zy} \otimes g_{yz} \in \hom(W_{z(0)}, W_{y(1)}) \otimes \hom(E_{y(0)}, E_{z(1)})$.
\end{itemize}
We will find the second description more useful. We then need an expression for the output of the evaluation map $C(\mu)$, when it is applied to elements of the form $f_{zy} \otimes g_{yz}$.

\begin{lemma}\label{evaluation map via traces}
For elements $f_{zy} \otimes g_{yz} \in \hom(W_{z(0)}, W_{y(1)}) \otimes \hom(E_{y(0)}, E_{z(1)}) \le P_W(E)$ and $\alpha_d \otimes \cdots \otimes \alpha_1 \in \hom(E_{x_d(0)},E_{x_d(1)}) \otimes \cdots \otimes \hom(E_{x_1(0)},E_{x_1(1)}) \le CF^*(E,E)^{\otimes d}$, the evaluation map $C(\mu)$ is given by
\begin{equation}\label{compostition is taking trace}
C(\mu) ((f_{zy} \otimes g_{yz}) \otimes \alpha_d \otimes \cdots \otimes \alpha_1) = 
\end{equation}
$$\sum_{w \in \X(K, K)} \sum_{u \in \R^0_{1: d+2}(w: z, x_1, \ldots, x_d, y)} \mathrm{tr}(P_{\gamma^{d+1}_u} \circ \alpha_d \cdots \circ P_{\gamma^2_u} \circ \alpha_1 \circ P_{\gamma^1_u}\circ g_{yz}) P_{\gamma^{d+2}_u} \circ f_{zy}
\circ P_{\gamma^0_u},
$$
where one takes the trace of the element in brackets which is an endomorphism of $E_{y(0)}$. 
\end{lemma}
\begin{proof}
Note that the contribution of every disc $u \in \R^0_{1:d+2}(w:z, x_1, \ldots, x_d, y)$ to $$C(\mu)((\hat{f}_y \otimes \hat{g}_z) \otimes \alpha_d \otimes \cdots \otimes \alpha_1) = \mu^{d+2}(\hat{f}_y, \alpha_d, \ldots, \alpha_1, \hat{g}_z)$$ is obtained by applying the composition map:
\begin{equation}\label{composition map}
\xymatrix{
\hom(E_{y(0)}, W_{y(1)}) \otimes \hom(E_{z(1)}, E_{y(0)}) \otimes \hom(W_{z(0)}, E_{z(1)}) \ar[r]^-{-\circ - \circ -} & \hom(W_{z(0)}, W_{y(1)})}
\end{equation}
to the element $\hat{f}_y \otimes T \otimes \hat{g}_z$, where $T = P_{\gamma_{u}^{d+1}} \circ \alpha_d \circ \cdots \circ \alpha_1 \circ P_{\gamma_u^1}$. Using again that our local systems have finite ranks, we have
$$
\begin{array}{rll}
&\hom(E_{y(0)}, W_{y(1)}) \otimes \hom(E_{z(1)}, E_{y(0)}) \otimes \hom(W_{z(0)}, E_{z(1)}) & \\
=&E_{y(0)}^{\vee}\otimes W_{y(1)} \otimes E_{z(1)}^{\vee}\otimes E_{y(0)}\otimes W_{z(0)}^{\vee} \otimes E_{z(1)}&\\
=&\hom(E_{z(1)}, E_{y(0)}) \otimes \hom(E_{y(0)},E_{z(1)}) \otimes \hom(W_{z(0)}, W_{y(1)}).&
\end{array}
$$
We then see that the composition map \eqref{composition map} coincides with the map 
\begin{eqnarray}
\hom(E_{z(1)}, E_{y(0)}) \otimes \hom(E_{y(0)},E_{z(1)}) \otimes \hom(W_{z(0)}, W_{y(1)}) & \to & \hom(W_{z(0)}, W_{y(1)}) \nonumber 	\\
T \otimes g_{yz} \otimes f_{zy} & \mapsto & \text{tr}(T \circ g_{yz})f_{zy},\label{trace instead of composition}
\end{eqnarray}
as both are given by performing all possible contractions of dual tensor factors in $E_{y(0)}^{\vee}\otimes W_{y(1)} \otimes E_{z(1)}^{\vee}\otimes E_{y(0)}\otimes W_{z(0)}^{\vee} \otimes E_{z(1)}$. 
\end{proof}
\subsubsection{The coproduct map $\Delta$}
Following \cite[Section 3.3 and 4.2]{abouzaid2010geometric}, \cite[Section 5.1]{abouzaid2012nearby}, \cite[Section 2.11]{sheridan2013fukaya}, we relate $CF^*(E,E)$ to the bimodule $P_W(E)$ via an $A_{\infty}$ bimodule homomorphism 
obtained from counts of pseudoholomorphic discs with two outgoing boundary punctures. More precisely,
one defines a coproduct map
$$\Delta \colon CF^*(E, E) \to P_W(E)$$
as follows.
Consider holomorphic discs with two negative boundary punctures $\zeta_{01}$, $\zeta_{02}$ and one positive $\zeta_1$, appearing in this cyclic order counterclockwise around the boundary of the disc. 
For every choice of Hamiltonian chords $x \in \X(L, L)$, $y \in \X(L, K)$ and $z \in \X(K, L)$, one has the 
moduli space $\R_{2:1}(z,y:x)$ of perturbed pseudoholomorphic discs $u$ 
which are 
asymptotic at $\zeta_{01}$, $\zeta_{02}$ and $\zeta_1$ to $z$, $y$ and $x$, respectively, and which map the 
boundary arc between $\zeta_{01}$ and $\zeta_{02}$ to $K$ and the remaining two arcs to $L$. Every map $u \in \R_{2:1}(z, y: x)$ defines paths $\gamma^0_u \in \Pi_1K(z(0), y(1))$, $\gamma^1_u \in \Pi_1L(y(0), x(0))$, $\gamma^1_u \in \Pi_1L(x(1), z(1))$ which are the images of the boundary arcs connecting $\zeta_{01}$ to $\zeta_{02}$, $\zeta_{02}$ to $\zeta_1$ and $\zeta_1$ to $\zeta_{01}$, respectively. The map $\Delta$ is then defined by setting for every $\alpha \in \hom(E_{x(0)}, E_{x(1)})$
$$\Delta(\alpha) = \sum_{\substack{y \in \X(L, K)\\z \in \X(K, L)}} \sum_{u \in \R_{2:1}^0(z, y: x)} P_{\gamma^0_u} \otimes (P_{\gamma^2_u} \circ \alpha \circ P_{\gamma^1_u})$$
with $P_{\gamma^0_u} \otimes (P_{\gamma^2_u} \circ \alpha \circ P_{\gamma^1_u}) \in \hom(W_{z(0)}, W_{y(1)}) \otimes \hom(E_{y(0)}, E_{z(1)}) \le P_W(E)$.
 
One can now extend the map $\Delta$ to a homomorphism of $A_{\infty}$ bimodules. That is, for every $r\ge 0$, $s \ge 0$ one defines an operation
$$\Delta^{r\vert 1 \vert s} \colon CF^*(E,E)^{\otimes r} \otimes CF^*(E,E) \otimes  CF^*(E,E)^{\otimes s} \to P_{W}(E)$$
by considering discs with two negative punctures and $r+1+s$ positive ones. Given chords $\vec{x} = (x_1, \ldots , x_r)$, $\underline{x}$, $\vec{x}_{\vert} = (x_{\vert s}, \ldots, x_{\vert 1})$, all connecting $L$ to $L$, and elements $\alpha_i \in \hom(E_{x_i(0)}, E_{x_i(1)})$, $\underline{\alpha} \in \hom(E_{\underline{x}(0)}, E_{\underline{x}(1)})$, $\alpha_{\vert i} \in \hom(E_{x_{\vert i}(0)}, E_{x_{\vert i}(1)})$ one sets
$$\Delta^{r\vert 1 \vert s} (\alpha_r, \ldots, \alpha_1, \underline{\alpha}, \alpha_{\vert 1}, \ldots, \alpha_{\vert s}) =$$
$$\sum_{\substack{y \in \X(L, K)\\z \in \X(K, L)}} \sum_{u \in \R^0_{2: r+1+s}(z, y: \vec{x}_{\vert}, \underline{x}, \vec{x})} P_{\gamma^0_u}
\otimes 
(P_{\gamma^{r+s+2}_u} \circ \alpha_r \circ P_{\gamma^{r+s+1}_u} \circ
\cdots 
\circ P_{\gamma^{s+2}_u} \circ \underline{\alpha} \circ P_{\gamma^{s+1}_u} \circ 
\cdots 
\circ P_{\gamma^2_u} \circ \alpha_{\vert s} \circ P_{\gamma^1_u}),$$
where $\gamma^0_u$ is again the image of the arc between the two negative punctures, which is mapped to $K$ and the other arcs are ordered counterclockwise around the boundary of the disc.
Note that $\Delta^{0 \vert 1 \vert 0}$ is the initially defined coproduct map. The fact that $\Delta$ is indeed an $A_{\infty}$ bimodule homomorphism (i.e. satisfies \cite[Equation (4.13)]{abouzaid2010geometric}) is verified again by considering the Gromov compactification of the one-dimensional component $\R^1_{2: r+1+s}(z, y: \vec{x}_{\vert}, \underline{x}, \vec{x})$.
It follows that $\Delta$ induces a map $HH_*(\Delta)$ in Hochschild homology. It is defined on the chain level by using all cyclic shifts of arguments of $\Delta^{\cdot \vert 1 \vert \cdot}$. That is, given an element $\underline{\alpha} \otimes \alpha_d \otimes \cdots \otimes \alpha_1 \in CC_*(CF^*(E,E))$, one has 
$$CC_*(\Delta)(\underline{\alpha} \otimes \alpha_d \otimes \cdots \otimes \alpha_1) = \sum_{r+s \le d} \Delta^{r \vert 1 \vert s} (\alpha_r \otimes \cdots \otimes \alpha_1 \otimes \underline{\alpha}\otimes \alpha_d \otimes \cdots \otimes \alpha_{d-s+1}) \otimes \alpha_{d-s} \otimes \cdots \otimes \alpha_{r+1}.$$

\subsubsection{Proof of the split-generation criterion}
We are now in a position to give a sketch proof of Theorem \ref{split-generation criterion}. It follows from the following two facts:
\begin{proposition}\label{pairing between hochshild cohomology and homology}
(compare \cite[Corollary 2.5, Proposition 2.6]{sheridan2013fukaya})
There exists a perfect pairing 
\begin{equation}\label{perfect pairing HH* and HH_*}
HH^*(CF^*(E,E)) \otimes HH_*(CF^*(E,E)) \to \FF_2.
\end{equation}
Further, the diagram 
\begin{equation}\label{diagram for duality of CO and OC}
\xymatrix{
QH^*(M) \ar[d]^{\CO^*} \ar[r]^{\cong} & QH^*(M)^{\vee} \ar[d]^{\OC_*^{\vee}} \\
HH^*(CF^*(E,E)) \ar[r]^-{\cong} & HH_*(CF^*(E,E))^{\vee} 
}
\end{equation}
commutes, where the top isomorphism is given by the Poincar\'{e} pairing and the bottom one comes from \eqref{perfect pairing HH* and HH_*}.
\end{proposition}

\begin{proposition}\label{commutativity of split-generation diagram} (compare \cite[Proposition 4.1]{abouzaid2012nearby}, \cite[Lemma 2.15]{sheridan2013fukaya}) The following diagram commutes:
 \begin{equation}\label{generation diagram}
\xymatrixcolsep{5pc}
\xymatrix{
HH_*(CF^*(E, E)) \ar[d]^{\OC_*} \ar[r]^-{HH_*(\Delta)} & HH_*(CF^*(E, E), P_W(E))\ar[d]^{H(\mu)}\\ 
QH^*(M) \ar[r]^{\CO^0} & HF^*(W, W).}
\end{equation}
\end{proposition}

Assuming these facts we have:

\noindent\textit{Proof of Theorem \ref{split-generation criterion}:}
If $\CO^*$ is injective, then Proposition \ref{pairing between hochshild cohomology and homology} implies  that $\OC_*$ is surjective and in particular $1 \in QH^*(M)$ lies in the image of $\OC_*$. Since $\CO^0$ is a unital algebra homomorphism it follows that $e_W$ lies in the image of $\CO^0 \circ \OC_*$. By Proposition \ref{commutativity of split-generation diagram} we then have that $e_W$ lies in the image of $H(\mu)$ and applying Lemma \ref{algebraic split-generation} yields that $(K, W)$ is split-generated by $(L,E)$. \qed

\noindent\textit{Proof of Proposition \ref{pairing between hochshild cohomology and homology}:}
The construction of the pairing \eqref{perfect pairing HH* and HH_*} and the proof that it is perfect can be taken directly from \cite[Lemma 2.4 \& Corollary 2.5]{sheridan2013fukaya}. The only extra input needed to deal with local systems of higher finite rank is a linear algebra argument, analogous to Lemma \ref{evaluation map via traces} above (the proof of Lemma 2.4 in \cite{sheridan2013fukaya} uses the coproduct map $\Delta$; as seen above, the output of $\Delta$ lies in a slightly awkward tensor product of spaces of linear maps; one needs to rearrange the tensor factors to make this output more manageable). We omit the details of this proof here.

The fact that diagram \eqref{diagram for duality of CO and OC} commutes is proved in \cite[Proposition 2.6]{sheridan2013fukaya}.
\qed

We now give a sketch proof of Proposition \ref{commutativity of split-generation diagram}, following \cite[Section 2.11]{sheridan2013fukaya}. 

\noindent\textit{Proof of Proposition \ref{commutativity of split-generation diagram}:}
Given Hamiltonian chords $\{\underline{x}, x_1, \ldots, x_d\} \in \X(L, L)$ and $w \in \X(K,K)$, consider the moduli space
$$\D(w: \underline{x}, x_1, \ldots, x_d) \coloneqq \left\{(u, v) \in \R_{1: 0;1}(w;M) \times \R_{0: d+1;1}(\underline{x}, x_1, \ldots, x_d; M) \st \ev(u)=\ev(v)\right\},$$
which consists of pairs of discs, connected at an internal node and asymptotic to the prescribed chords at their boundary punctures.
One can use the zero-dimensional component $\D^0(w: \underline{x}, x_1, \ldots, x_d)$ to define a map:
$$
\chi \colon CC_*(CF^*(E,E)) \to CF^*(W,W)$$
\begin{equation}
\chi(\underline{\alpha}\otimes \alpha_{d}\otimes \ldots \otimes \alpha_{1}) = \sum_{w \in \X(K, K)}\sum_{\substack{(u', v')\in \\\D^0(w: \underline{x}, x_1, \ldots, x_d)}} \mathrm{tr}\left(P_{\gamma^d_{v'}} \circ \alpha_d  \circ \cdots \circ P_{\gamma^0_{v'}} \circ \underline{\alpha}\right)P_{\del u'}.
\end{equation}
By considering the boundary of the Gromov compactification of the one-dimensional component $\D^1(w: \underline{x}, x_1, \ldots, x_d)$, one shows that $\chi$ is a chain map. As a preparatory step for proving Proposition \ref{commutativity of split-generation diagram} one needs the following lemma:
\begin{lemma}
Let $H(\chi) \colon HH_*(CF^*(E,E)) \to HF^*(W,W)$ denote the induced map on homology. Then $H(\chi) = \CO^0 \circ \OC_*$.
\end{lemma}
\begin{proof}
Let $\{e_1, \ldots, e_m\}$ be a basis for $H_*(M;\FF_2)$ elements of pure degree and let $\{e^1, \ldots, e^m\} \subseteq H^*(M;\FF_2)$ denote its dual basis. Further set $\epsilon_i = PD(e^i)$ and $\epsilon^i = PD(e_i)$. Choose pseudocycles $f_i$, $g_i$ representing $e_i$ and $\epsilon_i$ respectively. Then, given a Hochschild cycle $\varphi = \sum_j \lambda_j\; \underline{\alpha}_j\otimes \alpha_{jd}\otimes \ldots \otimes \alpha_{j1}$, one has 
$$\CO^0(\OC_*(\varphi)) = \left[\sum_j \lambda_j\sigma\left(\underline{\alpha}_j\otimes \alpha_{jd}\otimes \ldots \otimes \alpha_{j1};\{f_i\}, \{g_i\}\right)\right],$$ where the square brackets denote the cohomology class in $HF^*(W,W)$ and
\begin{eqnarray*}
\sigma\left(\underline{\alpha}\otimes \alpha_{d}\otimes \ldots \otimes \alpha_{1};\{f_i\}, \{g_i\}\right) &\coloneqq &\sum_{i=1}^m\left\langle \OC_*(\underline{\alpha}\otimes \alpha_{d}\otimes \ldots \otimes \alpha_{1}), e_i; f_i \right\rangle \,\CO^0(e^i;g_i) \nonumber \\
\end{eqnarray*}
$$
=\sum_{w \in \X(K, K)}\left(\sum_{\substack{(u, v) \in \\ \coprod_{i=1}^m \R^0_{1: 0;1}(w;g_i) \times \R^0_{0: d+1;1}(\underline{x}, x_1, \ldots, x_d; f_i)}} \mathrm{tr}(P_{\gamma^d_v} \circ \alpha_d  \circ \cdots \circ P_{\gamma^0_v} \circ \underline{\alpha})P_{\del_u}\right). \nonumber$$
Now, given Hamiltonian chords $\{\underline{x}, x_1, \ldots, x_d\} \in \X(L, L)$, $w \in \X(K,K)$ and a bordism $h \colon B \to M \times M$, realising a homology between $\sum_{i=1}^l e_i \times \epsilon_i$ and the diagonal, consider the moduli space
$$\H(w:\underline{x}, x_1, \ldots, x_d ; h) \coloneqq \left\{(u, v) \in \R_{1: 0;1}(w;M) \times \R_{0: d+1;1}(\underline{x}, x_1, \ldots, x_d; M) \st (\ev(u),\ev(v)) \in im(h)\right\}.$$
Then $\coprod_{i=1}^m \R^0_{1: 0;1}(w;g_i) \times \R^0_{0: d+1;1}(\underline{x}, x_1, \ldots, x_d; f_i)$ and the zero-dimensional component of discs connected at a node 
$\D^0(w: \underline{x}, x_1, \ldots, x_d)$ form part of the boundary of the Gromov compactification of the 1-dimensional component $\H^1(w: \underline{x}, x_1, \ldots, x_d; h)$. By analysing the remaining boundary components of this compactification and using again that the homotopy classes of the paths involved in parallel transport remain invariant in 1-parameter families, one finds that the sum 
$$\left(\sum_{\substack{(u, v) \in \\ \coprod_{i=1}^m \R^0_{1: 0;1}(w;g_i) \times \R^0_{0: d+1;1}(\underline{x}, x_1, \ldots, x_d; f_i)}} \mathrm{tr}(P_{\gamma^d_v} \circ \alpha_d  \circ \cdots \circ P_{\gamma^0_v} \circ \underline{\alpha})P_{\del_u}\right) + 
$$
$$+\left(\sum_{\substack{(u', v')\in \\\D^0(w: \underline{x}, x_1, \ldots, x_d)}} \mathrm{tr}(P_{\gamma^d_{v'}} \circ \alpha_d  \circ \cdots \circ P_{\gamma^0_{v'}} \circ \underline{\alpha})P_{\del u'}\right) \in \hom_{\FF_2}(W_{w(0)}, W_{w(1)})$$
depends linearly on $b(\underline{\alpha}, \alpha_d, \ldots, \alpha_1)$ up to a term which is the $\hom_{\FF_2}(W_{w(0)}, W_{w(1)})$-component of a $\mu^1$-exact element.
\end{proof}

To prove Proposition \ref{commutativity of split-generation diagram}, it remains to be shown that $H(\chi) = H(\mu) \circ HH_*(\Delta)$. This is implied by the following lemma.

\begin{lemma}
The maps $\chi$ and $C(\mu) \circ CC_*(\Delta)$ are chain-homotopic.
\end{lemma}
\begin{proof}
Following 
 \cite[Section 5.3]{abouzaid2012nearby}, we construct such a homotopy by considering a moduli space of perturbed pseudoholomorphic maps, whose domain is an annulus $A_r = \{z \in \CC \st 1 \le \vert z \vert \le r\}$ (for some $r$) with $d+1$ positive punctures $\{\underline{\zeta} = r, \zeta_1, \ldots, \zeta_d\}$ on the outer circle and one negative puncture on the inner circle, constrained to lie at $-1$. Given chords $\{\underline{x}, x_1, \ldots, x_d\} \in \X(L,L)$ and $w \in \X(K,K)$, we denote by $\C^-_{1:d+1}(w: \underline{x}, x_1, \ldots, x_d)$ the moduli space of maps as above, which are furthermore required to map the boundary component $\{z \in A_r \st \vert z \vert = 1\}$ to $K$, the remaining boundary components $\{z \in A_r \st \vert z \vert = r\}$ to $L$ and which are asymptotic to $w$ at $-1$ and to $\{\underline{x}, x_1, \ldots, x_d\}$ at $\{\underline{\zeta} = r, \zeta_1, \ldots, \zeta_d\}$. The boundary of the Gromov compactification of the one-dimensional component $\C^{-, 1}_{1:d+1}(w: \underline{x}, x_1, \ldots, x_d)$ consist of the following four types of configurations (see \cite[Equations (5.18), (5.19), (5.20)]{abouzaid2012nearby}):
\begin{enumerate}
\item  \label{contributing to mu1 circ h} a strip braking at the outgoing puncture; connected components of this stratum are given by products 
$$\R^0_{1:1}(w:w') \times \C^{-, 0}_{1:d+1}(w':\underline{x}, x_1, \ldots, x_d)$$ for some $w' \in \X(K,K)$.
\item \label{contributing to h circ b} a strip or a stable disc component (i.e. a disc carrying at least two punctures) breaking off at a positive puncture; connected components of this stratum are given by products 
$$\C^{-, 0}_{1:d-s-r+1}(w: \underline{x}', x_{r+1}, \ldots, x_{d-s}) \times \R^0_{1:r+s+1}(\underline{x}': x_{d-s+1}, \ldots, x_d, \underline{x}, x_1, \ldots, x_r)$$
for some $\underline{x}' \in \X(L,L)$ and
$$\C^{-, 0}_{1:d-j+2}(w: \underline{x}, x_1, \ldots, x_i, x', x_{i+j+1}, \ldots, x_d) \times \R^0_{1:j}(x': x_{i+1}, \ldots, x_{i+j})$$
for some $x' \in \X(L,L)$.
\item \label{contributing to c(mu) circ cc(delta)} a degeneration of the conformal modulus of the annulus as $r \to 1$; components of the boundary at $r=1$ are given by products
$$\R^0_{1:d-r-s+2}(w: z, x_{r+1}, \ldots, x_{d-s}, y) \times \R^0_{2:r+s+1}(z, y: x_{d-s+1}, \ldots, x_d, \underline{x}, x_1, \ldots, x_r)$$
for some $y \in \X(L, K)$, $z \in \X(K,L)$. 
\item \label{contributing to chi} a degeneration of the conformal modulus of the annulus as $r \to +\infty$; the boundary at $r=+\infty$ is the moduli space $\D^0(w: \underline{x}, x_1, \ldots, x_d)$ of pairs of discs, connected at a node. 
\end{enumerate}
Observe that the degenerations of types \ref{contributing to c(mu) circ cc(delta)} and \ref{contributing to chi} are precisely the ones which account for the $\hom_{\FF_2}(W_{w(0)}, W_{w(1)})$-component of $C(\mu) \circ CC_*(\Delta)(\underline{\alpha}\otimes \alpha_d\otimes \cdots \otimes \alpha_1)$ and $\chi(\underline{\alpha}\otimes \alpha_d\otimes \cdots \otimes \alpha_1)$, respectively. Further, from the description of $C(\mu)$ in Lemma \ref{evaluation map via traces} one can see that both $\chi(\underline{\alpha}\otimes \alpha_{k}\otimes \ldots \otimes \alpha_{1})$ and $C(\mu) \circ CC_*(\Delta)(\underline{\alpha}\otimes \alpha_{k}\otimes \ldots \otimes \alpha_{1})$ weight the parallel transport map along the boundary component mapping to $K$ by 
the trace of the the loop of linear maps, obtained by composing the elements $\alpha_i$ with the parallel transport along the boundary components mapped to $L$. On the other hand, each $a \in \C^-_{d+1}(w: \underline{x}, x_1, \ldots, x_d)$ defines paths $\gamma^j_a \in \Pi_1L(x_j(1), x_{j+1}(0))$, $0 \le j \le d$, which are the images of the boundary arcs connecting $\zeta_j$ to $\zeta_{j+1}$ (again the notation means $\zeta_0 = \zeta_{d+1} = \underline{\zeta}$ and $x_0 = x_{d+1} = \underline{x}$) and $\gamma_a \in \Pi_1K(w(0), w(1))$, which is the image of the inner boundary circle, oriented clockwise. We then define a map 
$$h \colon CC_*(CF^*(E,E)) \to CF^*(W,W)$$
$$h(\underline{\alpha}\otimes \alpha_d\otimes \cdots \otimes \alpha_1) = \sum_{w \in \X(K,K)}\sum_{a \in \C^{-, 0}_{d+1}(w: \underline{x}, x_1, \ldots, x_d)} tr(P_{\gamma^d_a} \circ \alpha_d \circ \cdots \circ P_{\gamma^0_a} \circ \underline{\alpha})P_{\gamma_a}.$$
This is analogous to \cite[Equation (5.22)]{abouzaid2012nearby}, except that we weight the parallel transport on $K$ by the trace of the loop on $L$.
Looking at the remaining types of boundary components of the compactification of $\C^{-, 1}_{1:d+1}(w: \underline{x}, x_1, \ldots, x_d)$, we see that
the degenerations of types \ref{contributing to mu1 circ h} and \ref{contributing to h circ b} account for the $\hom(W_{w(0)}, W_{w(1)})$-component of $\mu^1(h(\underline{\alpha}\otimes \alpha_d\otimes \cdots \otimes \alpha_1))$ and $h(b(\underline{\alpha}\otimes \alpha_d\otimes \cdots \otimes \alpha_1))$, respectively. Using again that all these terms are paired-off as boundary points of closed intervals we conclude that $C(\mu)\circ CC_*(\Delta)+ \chi + \mu^1 \circ h + h\circ b = 0$, i.e. $h$ is a chain-homotopy between $\chi$ and $C(\mu)\circ CC_*(\Delta)$.
\end{proof}

\section{Application to the Chiang Lagrangian}\label{application to the Chiang Lagrangian}
We shall now illustrate all of the above constructions by applying them to the particular case of the Chiang Lagrangian $\chiang \subseteq (\CP^3, \omega_{FS})$. This Lagrangian was discovered by River Chiang in \cite{chiang2004new} and its Floer theory was studied extensively by Evans and Lekili in \cite{evans2014floer}. Our goal is to extend their calculations to the case of coefficients in a local system of rank higher than 1 in order to study the relation between $\chiang$ and $\RP^3$. To motivate the calculations that follow and the necessity for local coefficients, a few general comments are in order.

Note first that $(\CP^3, \omega_{FS})$ is a monotone symplectic manifold with minimal Chern number $N_{\CP^3}=4$ since $[\omega_{FS}] = c_1(T\CP^3)/4 \in H^2(\CP^3; \ZZ)$. Further, since $\CP^3$ is simply-connected, 
we can include any monotone Lagrangian as an object of the Fukaya category. 

Recall also that $\RP^3$ is a Lagrangian submanifold of $\CP^3$ since it is the fixed-point set of complex conjugation, which is an antisymplectic involution. The fact that it is monotone can also be seen using this involution or by observing that its fundamental group is finite and appealing to Remark 
 \ref{finite pi1 => monotone} above. Another important fact is that $m_0(E) = 0$ for any local system $E \to \RP^3$ (and over any characteristic) since in fact $\RP^3$ can bound no $J$-holomorphic Maslov 2 discs at all for any $\omega_{FS}$-compatible $J$. To see this one again uses the antisymplectic involution: a Maslov 2 disc with boundary in $\RP^3$ could be completed via complex conjugation to a sphere with symplectic area $1/2$ which is impossible since $[\omega_{FS}]$ is an integral class.
 
Let us now briefly recall the definition of the Chiang Lagrangian, using notation from \cite{evans2014floer}. To this end, we view $\CP^3 \cong \operatorname{Sym}^3(\CP^1)$ as configurations of triples of points on $\CP^1$. The action of $SL(2, \CC)$ on $\CP^1$ by M{\"o}bius transformations then defines an action on $\CP^3$ whose restriction to the compact form $SU(2) \subseteq SL(2,\CC)$ is Hamiltonian. 
 Setting $\Delta \coloneqq \{[1:1], [\omega^2:1], [\omega^4 : 1]\}$, where $\omega = e^{i\pi/3}$, we then have a decomposition $\CP^3 = W_{\Delta} \cup Y_{\Delta}$, where $W_{\Delta} = SL(2, \CC)\cdot \Delta$ is the orbit consisting of all triples of pairwise distinct points and $Y_{\Delta}$ is a compactifying divisor consisting of triples with at least two coinciding points (note then that $Y_{\Delta}$ is cut out by the discriminant of a cubic, which is a section of $\O_{\CP^3}(4)$; that is, $Y_{\Delta}$ is an anticanonical hypersurface).
  From this point of view, $\CP^3$ is a special case of an $SL(2, \CC)$-quasihomogeneous $3$-fold $X_C$, obtained by compatifying an $SL(2, \CC)$-orbit $W_C$ of a configuration $C \in \CP^n = \operatorname{Sym}^n(\CP^1)$ of $n$ distinct points in $\CP^1$. It is known since the work of Aluffi and Faber in \cite{aluffi1993linear} that $X_{\Delta}=\CP^3$ is the first of only four cases in which such a compactifiaction is smooth, the other three being when $C$ can be chosen to consist of the vertices of a regular tetrahedron, octahedron and icosahedron. It is a fact that in all 4 cases, when $C$ is chosen to be such a regular configuration, its orbit $L_C$ under the action of the compact real form $SU(2) \subseteq SL(2, \CC)$ is a Lagrangian submanifold of $X_C$ with respect to the restriction of the Fubini-Study symplectic form on $\CP^n$. We then have:
  \begin{definition}
  The Chiang Lagrangian is the orbit $\chiang = SU(2)\cdot\Delta$ in $X_{\Delta} = \CP^3$.
  \end{definition}
In \cite{evans2014floer} Evans and Lekili compute the Floer cohomology of $\chiang$ with itself and show that it is non-zero only over a field of characteristic 5, when it becomes additively isomorphic to the singular cohomology. Further, they prove that $\chiang$ generates its summand in the monotone Fukaya category of $\CP^3$ over characteristic 5. Their technique is to exploit the many symmetries in order to prove regularity for holomorphic discs with boundary on $\chiang$, classify the ones of small Maslov index and then use the pearl complex machinery. This approach has then been taken-up and generalised by Jack Smith in the paper \cite{smith2015floer}, culminating in the calculation of Floer cohomology for the three Lagrangians arising from the Platonic solids.

Let us address now our main question, namely whether $\chiang$ and $\RP^3$ can be displaced by a Hamiltonian isotopy. A natural attempt would be to compute $HF^*(\chiang, \RP^3)$ and hope that the result is non-zero.
  However, it is shown in \cite{evans2014floer} that, when $\chiang$ is equipped with a spin structure, so that the moduli space $\M_{0, 1}(2, \chiang;J)$ can be oriented and the count of Maslov 2 discs can be done over $\ZZ$, one has $m_0(\chiang) = \pm 3$.
  Therefore, over any characteristic different from 3, one has $m_0(\chiang) \neq m_0(\RP^3)$ and so $HF^*(\chiang, \RP^3)$ is not well-defined. On the other hand, $HF^*(\chiang, \RP^3)$ must vanish over characteristic 3, since it is a left module over the unital algebra $HF^*(\RP^3, \RP^3)$ which itself vanishes over characteristic different from 2. This is a consequence of the Auroux-Kontsevich-Seidel criterion (see e.g. \cite[Lemma 2.7]{sheridan2013fukaya}): $m_0(\RP^3)=0$ is an eigenvalue of quantum multiplication by $c_1(\CP^3)$ only in characteristic 2.
  Thus, if one wants to use Floer cohomology as a meaningful obstruction to displacing $\chiang$ and $\RP^3$, one needs to equip $\chiang$ with an appropriate local system $W$ so that $m_0(\chiang, W) = m_0(\RP^3) = 0$ over characteristic 2. One cannot use rank 1 local systems for this purpose, see Remark \ref{rank 1 doesnt work} below. In Section \ref{proof of Theorem} we show that there is a particular local system of rank 2 on $\chiang$ which satisfies this and which has non-zero Floer cohomology with $\RP^3$. The proof of this relies on an explicit calculation of parallel transport maps along the same pearly trajectories, which Evans and Lekili use in \cite{evans2014floer}. In order to describe these trajectories, we begin with a detailed account of the topology of $\chiang$. 

\subsection{Topology of $\chiang$}\label{Topology of chiang}
 The Lie algebra $\mathfrak{su}(2)$ is the real-linear span of the Pauli matrices:
  $$\sigma_1 = \begin{pmatrix}
  i & 0 \\
  0 & -i
  \end{pmatrix},\; 
  \sigma_2 = \begin{pmatrix}
  0 & 1 \\
  -1& 0
  \end{pmatrix} \text{\; and \;}
  \sigma_3 = \begin{pmatrix}
  0 & i\\
  i & 0
  \end{pmatrix}.$$
 From now on, $S^2$ will only be used to denote the unit sphere in $\RR^3$. All occurrences of ``$\exp$'' refer to the exponential map in $SU(2)$. For a unit vector $V=(v_1, v_2, v_3) \in S^2$ and $t \in \RR$ we will write $\exp(tV)$ to mean $\exp(t(v_1\sigma_1 + v_2\sigma_2 + v_3\sigma_3))$.
 The action of $SU(2)$ on $\CP^1$ by projective transformations can be identified with the action of $SU(2)$ on $S^2$ by quaternionic rotations, as long as we adopt the following conventions (for any other choice the two actions would of course be conjugate):
 \begin{itemize}
 \item for any unit vector $V \in S^2$, $\exp(\theta V)$ acts on $S^2$ by a right-hand rotation by $2\theta$ in the axis $V$; this is the adjoint action of $SU(2)$ on $S^2 \subseteq \mathfrak{su}(2)$, where we identify $\mathfrak{su}(2)$ with $\RR^3$ via the basis $\{\sigma_1, \sigma_2, \sigma_3\}$;  
 \item we identify $\CC \cup \{\infty\} \cong \CP^1$ via $z \mapsto [z:1]$, $\infty \mapsto [1:0]$;
 \item we identify $\CC \cup \{\infty\} \cong S^2$ via 
 $z \mapsto \left(\frac{\lvert z \rvert^2 - 1}{\lvert z \rvert^2 + 1}, \frac{2\Re(iz)}{\lvert z \rvert^2 + 1}, \frac{2\Im(iz)}{\lvert z \rvert^2 + 1}\right)$, $\infty \mapsto (1,0,0)$, i.e. via stereographic projection from $(1,0,0)$ followed by multiplication by $-i$.
\end{itemize}   
 The last two identifications combine to give the diffeomorphism
 $\Phi \colon \PP^1 \to S^2$, $\Phi([x:y]) = \left( \frac{\lvert x \rvert^2 - \lvert y \rvert^2}{\lvert x \rvert^2 + \lvert y \rvert^2}, \frac{2\Re(i x \cl{y})}{\lvert x \rvert^2 + \lvert y \rvert^2}, \frac{2\Im(i x \cl{y})}{\lvert x \rvert^2 + \lvert y \rvert^2} \right)$.
 In this way $\Delta = \{[1:1], [\omega^2:1], [\omega^4 : 1]\}$ corresponds to the equilateral triangle with vertices $V_1'\coloneqq(0, 0, 1)$, $V_3'\coloneqq(0, -\sqrt{3}/2, -1/2)$, and $V_2'\coloneqq(0, \sqrt{3}/2, -1/2)$ (our choice of names for the vertices will become apparent when we discuss a particular Morse function on $\chiang$ below). 
 
 Recall that $\chiang = SU(2)\cdot\Delta \subseteq \operatorname{Sym}^3(\CP^1) = \CP^3$. The stabiliser of $\Delta$ is easily seen to be the binary dihedral group of order 12, given explicitly by
 $$\bindih = \left\{ 
 \begin{pmatrix}
 \omega^k & 0 \\
 0 & \cl{\omega}^k
\end{pmatrix} \st k \in \{0, 1, \ldots 5\}
\right\} 
\cup \left\{
\begin{pmatrix}
0 & i\omega^k \\
i\cl{\omega}^k & 0
\end{pmatrix} \st  k \in \{0, 1, \ldots 5\}
\right\} \subseteq SU(2). $$
Abstractly, we view this group by the presentation 
$$\bindih = \langle a, b \;\vert\; a^6 = 1, b^2 = a^3, ab  = ba^{-1}\rangle, $$
the above complex representation being given by $a \mapsto \begin{pmatrix}
 \omega & 0 \\
 0 & \cl{\omega}
\end{pmatrix} $
and 
$b \mapsto \begin{pmatrix}
 0 & i \\
 i & 0
\end{pmatrix}$.
So we have $\chiang \cong SU(2) / \bindih$ and $SU(2)$ is tiled by 12 fundamental domains for the action of $\bindih$. Further, the quotient map $q \colon SU(2) \to \chiang$ induces a natural isomorphism 
\begin{eqnarray}
\bindih &\to & \pi_1(\chiang, q(\Id))^{Opp} \label{isomorphism with pi_1}\\
x & \mapsto & [q \circ \ell_x] \nonumber,
\end{eqnarray}
where $\ell_x \colon [0, 1] \to SU(2)$ is any path with $\ell(0) = \Id$ and $\ell(1) = x$. In particular $\chiang$ is monotone by Remark \ref{finite pi1 => monotone}. Further, since it is orientable, the Chiang Lagrangian must also satisfy $N_{\chiang}\in 2\ZZ$. We will see (\ref{maslov 2 discs on chiang} below) that in fact $N_{\chiang}=2$. In figure \ref{fundamental domain} below we give a schematic description of a fundamental domain for the right action of $\bindih$ on $SU(2)$. The picture is essentially borrowed from \cite{evans2014floer} with the difference that the fundamental domain given there is (erroneously) for a left $\bindih$-action. A detailed derivation of the domain can be found in \cite[Section 5]{smith2015floer}.
\begin{figure}[h!]
\captionsetup{justification=centering,margin=2cm}
\begin{center}
\tdplotsetmaincoords{80}{185}
\begin{tikzpicture}[tdplot_main_coords, scale=1.5]
\draw[thick,->] (0,0,0.75) -- (0.4,0,0.75) node[anchor=north east]{$\sigma_2$};
\draw[thick,->] (0,0,0.75) -- (0,1,0.75) node[anchor=north west]{$\sigma_3$};
\draw[thick,->] (0,0,0.75) -- (0,0,1) node[anchor=south]{$\sigma_1$};
\def\RI{2}
\def\RII{2}

\node at (0:\RI) (x1bottomleft){$\cdot$};
\node [below, black] at (x1bottomleft) {$x_1$};
\draw[thick] [black, -triangle 45] (x1bottomleft.center) -- (60:\RI) node at (60:\RI) (x2bottomfront){$\cdot$}
node [below, black] at (x2bottomfront){$x_2$};
\draw[thick] [yellow, -triangle 45] (x2bottomfront.center) -- (120:\RI) node at (120:\RI) (x3bottomfront){$\cdot$} 
node [below, black] at (x3bottomfront) {$x_3$};
\draw[thick] [orange, -triangle 45] (x3bottomfront.center) --  (180:\RI) node at (180:\RI) (x1bottomright){$\cdot$} 
node [below, black] at (x1bottomright) {$x_1$};
\draw[dashed, thick, -triangle 45] [cyan] (x1bottomright.center) --  (240:\RI) node at (240:\RI) (x2bottomback){$\cdot$}
node [below, black] at (x2bottomback) {$x_2$};
\draw[dashed, thick, -triangle 45] [red] (x2bottomback.center) --  (300:\RI) node at (300:\RI) (x3bottomback){$\cdot$}
node [below, black] at (x3bottomback) {$x_3$};
\draw[dashed, thick, -triangle 45] [magenta] (x3bottomback.center) --  (x1bottomleft.center);
\path[fill=gray, fill opacity=0.2](0:\RI)
   \foreach \x in {0,60,120,180,240,300,360} { -- (\x:\RI)};

\begin{scope}[yshift=1.7cm]

\node at (0:\RII) (x2topleft){$\cdot$}
node [above, black] at (x2topleft) {$x_2$};
\draw[thick] [yellow, -triangle 45] (0:\RII) -- (60:\RII) node at (60:\RII) (x3topfront){$\cdot$}
node [above, black] at (x3topfront) {$x_3$};
\draw[thick] [orange, -triangle 45] (60:\RII) -- (120:\RII) node at (120:\RII) (x1topfront){$\cdot$}
node [above, black] at (x1topfront) {$x_1$};
\draw[thick] [cyan, -triangle 45] (120:\RII) --  (180:\RII) node at (180:\RII) (x2topright){$\cdot$}
node [above, black] at (x2topright) {$x_2$};
\draw[thick] [red, -triangle 45] (180:\RII) --  (240:\RII) node at (240:\RII) (x3topback){$\cdot$}
 node [above, black] at (x3topback) {$x_3$};
\draw[thick] [magenta, -triangle 45] (240:\RII) --  (300:\RII) node at (300:\RII) (x1topback){$\cdot$}
node [above, black] at (x1topback) {$x_1$};
\draw[thick] [black, -triangle 45] (300:\RII) --  (x2topleft.center);

\path[thick,fill=gray!30,opacity=0.3] (\RII,0)
  \foreach \x in {0,60,120,180,240,300,360}
    { --  (\x:\RII) node at (\x:\RII) (R2-\x) {}};

\end{scope}
\draw [thick, cyan, -triangle 45] (x1bottomleft.center) -- (x2topleft.center);
\draw [thick, red, -triangle 45] (x2bottomfront.center) -- (x3topfront.center);
\draw [thick, magenta, -triangle 45] (x3bottomfront.center) -- (x1topfront.center);
\draw [thick, black, -triangle 45] (x1bottomright.center) -- (x2topright.center);
\draw [thick, yellow, dashed, -triangle 45] (x2bottomback.center) -- (x3topback.center);
\draw [thick, orange, dashed, -triangle 45] (x3bottomback.center) -- (x1topback.center);

\end{tikzpicture}
\end{center}
\caption{
The fundamental domain for $\chiang$. Opposite quadrilateral faces are identified by a $90^{\circ}$ rotation and the two hexagonal faces are identified by a $60^{\circ}$ rotation so that colours of edges match. The fundamental domain is viewed as sitting in $SU(2)$ with $\Id$ at the center of the prism and the matrices $\{\sigma_1, \sigma_2, \sigma_3\} \subseteq T_{\Id}SU(2)$ are given for orientation.
}\label{fundamental domain}
\end{figure}
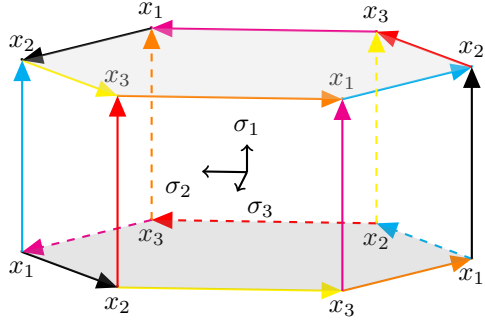

Evans and Lekili also describe a Morse function on $\chiang$ by specifying its critical points and some of its flowlines. We shall use essentially the same Morse function (depicted in figure \ref{prism and morse flowlines} below) to compute Floer cohomology but since we want to work with local coefficients, we are particularly concerned with where exactly its index 1 downward gradient flowlines (with respect to the round metric on $SU(2)$) pass. This is what we shall now spell out. Throughout this discussion it is useful to keep in mind the picture of rotating equilateral triangles, inscribed equatorially in the unit sphere in $\RR^3$. For example, for any unit vectors $V, W \in \RR^3$ we think of the point $q(\exp(sV)\exp(tW)) \in \chiang$ as the triangle, obtained from $\Delta$ by first applying a right-hand rotation by $2t$ in the axis $W$ and then a right-hand rotation by $2s$ in the axis $V$. 

Recall that we defined  $V_1'\coloneqq(0, 0, 1)$, $V_2'\coloneqq(0, \sqrt{3}/2, -1/2)$, $V_3'\coloneqq(0, \sqrt{3}/2, -1/2)$. We now further set $V_1 \coloneqq (0, -\sqrt{3}/2, 1/2)$, $V_2 \coloneqq (0, \sqrt{3}/2, 1/2)$, $V_3 \coloneqq (0, 0, -1)$ and $h \coloneqq \exp\left(\frac{\pi}{6}\sigma_1\right)\in SU(2)$. We then define the Morse function $f \colon \chiang \to \RR$ to have:
\begin{itemize}\label{morse function}
\item one minimum at $m' \coloneqq q(\Id)$; 
\item three critical points of index 1: $x_1' \coloneqq q\left(\exp\left(\frac{\pi}{4}V_1'\right)\right)$, $x_2' \coloneqq q\left(\exp\left(\frac{\pi}{4}V_2'\right)\right)$ and $x_3' \coloneqq q\left(\exp\left(\frac{\pi}{4}V_3'\right)\right)$. They are connected to the minimum $m'$ via 6 flowlines whose compactified images can be parametrised for $t \in [0,1]$ by $\gamma_i'(t) = q\left(\exp\left(\left(1- t\right)\frac{\pi}{4}V_i'\right)\right)$ and $\tilde{\gamma}_i'(t) =  q\left(\exp\left(-\left(1-t\right)\frac{\pi}{4}V_i'\right)\right)$ for $i \in \{1, 2, 3\}$;
\item three critical points of index 2:  $x_1 \coloneqq q\left(h \exp\left(\frac{\pi}{4}V_1\right)\right)$, $x_2 \coloneqq q\left(h \exp\left(\frac{\pi}{4}V_2\right)\right)$ and $x_3 \coloneqq q\left(h \exp\left(\frac{\pi}{4}V_3\right)\right)$;
\item one maximum at $m \coloneqq q(h)$. It connects to the index 2 critical points via 6 flowlines whose images are similarly given by $\gamma_i(t) = q\left(h\exp\left(t\frac{\pi}{4}V_i \right)\right)$ and $\tilde{\gamma}_i(t) = q\left(h\exp\left(-t\frac{\pi}{4}V_i \right)\right)$ for $t \in [0, 1]$ and $i \in \{1, 2, 3\}$;
\item there are 12 other index 1 flowlines, connecting critical points of index 2 to ones of index 1. For our purposes we do not need a similarly precise description of their images and the schematic description from figure \ref{prism and morse flowlines} will do.
\end{itemize}

%
%
%

\begin{figure}[h]
\begin{minipage}{0.33\textwidth}
\tdplotsetmaincoords{65}{185}
\begin{tikzpicture}[tdplot_main_coords, scale = 1]
\pgfmathsetmacro\x{cos(60)}
\pgfmathsetmacro\y{sin(60)}

\def\RI{2}

\node [below, black] at (\RI,0,-\RI/2) (x1.bottomleft) {$x_1$};
\draw[thick] [black] (\RI, 0, -\RI/2) -- (\RI*\x, \RI*\y, -\RI/2) node [below, black] at (\RI*\x, \RI*\y, -\RI/2) {$x_2$};
\draw[thick] [yellow] (\RI*\x, \RI*\y, -\RI/2) -- (-\RI*\x, \RI*\y, -\RI/2) node [below, black] at (-\RI*\x, \RI*\y, -\RI/2) {$x_3$};
\draw[thick] [orange] (-\RI*\x, \RI*\y, -\RI/2) --  (-\RI, 0, -\RI/2) node [below, black] at (-\RI, 0, -\RI/2){$x_1$};
\draw[dashed, thick] [cyan] (-\RI, 0, -\RI/2) --  (-\RI*\x, -\RI*\y, -\RI/2) node [below, black] at (-\RI*\x, -\RI*\y, -\RI/2) {$x_2$};
\draw[dashed, thick] [red] (-\RI*\x, -\RI*\y, -\RI/2) --  (\RI*\x, -\RI*\y, -\RI/2) node [below, black] at (\RI*\x, -\RI*\y, -\RI/2) {$x_3$};
\draw[dashed, thick] [magenta] (\RI*\x, -\RI*\y, -\RI/2) --  (\RI, 0, -\RI/2);

\def\RII{2}
\def\height{1}

\draw node [above, black] at (\RII, 0, \height) {$x_2$};
\draw[thick] [yellow] (\RII, 0, \height) -- (\RII*\x, \RII*\y, \height) node [above, black] at (\RII*\x, \RII*\y, \height) {$x_3$};
\draw[thick] [orange] (\RII*\x, \RII*\y, \height) -- (-\RII*\x, \RII*\y, \height) node [above, black] at (-\RII*\x, \RII*\y, \height) {$x_1$};
\draw[thick] [cyan] (-\RII*\x, \RII*\y, \height) --  (-\RII, 0, \height) node [above, black] at (-\RII, 0, \height) {$x_2$};
\draw[thick] [red] (-\RII, 0, \height) --  (-\RII*\x, -\RII*\y, \height) node [above, black] at (-\RII*\x, -\RII*\y, \height) {$x_3$};
\draw[thick] [magenta] (-\RII*\x, -\RII*\y, \height) --  (\RII*\x, -\RII*\y, \height) node [above, black] at (\RII*\x, -\RII*\y, \height) {$x_1$};
\draw[thick] [black] (\RII*\x, -\RII*\y, \height) --  (\RII, 0, \height);

\draw [thick, cyan] (\RI, 0, -\RI/2) -- (\RII, 0, \height);
\draw [thick, red] (\RI*\x, \RI*\y, -\RI/2) -- (\RII*\x, \RII*\y, \height);
\draw [thick, magenta] (-\RI*\x, \RI*\y, -\RI/2) -- (-\RII*\x, \RII*\y, \height);
\draw [thick, black] (-\RI, 0, -\RI/2) -- (-\RII, 0, \height);
\draw [thick, dashed, yellow, opacity=0.5] (-\RI*\x, -\RI*\y, -\RI/2) -- (-\RII*\x, -\RII*\y, \height);
\draw [thick, dashed, orange, opacity=0.5] (\RI*\x, -\RI*\y, -\RI/2) -- (\RII*\x, -\RII*\y, \height);

\draw  node [above] at (0, 0, \height) {$m$};
\draw [thick, -latex', blue] (0, 0, \height) -- (\RII, 0, \height) node [above] at (17/30*\RII, 0, \height) {$\tilde{\gamma}_{2}$};
\draw [thick, -latex', blue] (0, 0, \height) -- (\RII*\x, \RII*\y, \height) node [above] at (22/30*\RII*\x, 22/30*\RII*\y, \height){$\tilde{\gamma}_{3}$};
\draw [thick, -latex', green] (0, 0, \height) -- (-\RII*\x, \RII*\y, \height) node [left] at (17/30*-\RII*\x,17/30*\RII*\y, \height){$\gamma_1$};
\draw [thick, -latex', green] (0, 0, \height) -- (-\RII, 0, \height) node [below] at (20/30*-\RII, 0, \height){$\gamma_{2}$};
\draw [thick, -latex', green] (0, 0, \height) -- (-\RII*\x, -\RII*\y, \height) node [right] at (12/30*-\RII*\x, 12/30*-\RII*\y, \height){\hspace{1mm}$\gamma_{3}$};
\draw [thick, -latex', blue] (0, 0, \height) -- (\RII*\x, -\RII*\y, \height) node [above] at (10/30*\RII*\x, 10/30*-\RII*\y, \height){$\tilde{\gamma}_{1}$};
\end{tikzpicture}
\begin{center}
\vspace{1.3mm}a)
\end{center}
\end{minipage}%
%
\begin{minipage}{0.30\textwidth}
\tdplotsetmaincoords{53}{180}
\begin{tikzpicture}[tdplot_main_coords, scale = 1.2]
\pgfmathsetmacro\x{cos(60)}
\pgfmathsetmacro\y{sin(60)}
\def\RI{2}
\def\RII{2}
\def\height{\RI/2}
\draw[dashed, thick] [cyan] (-\RI, 0, -\RI/2) --  (-\RI*\x, -\RI*\y, -\RI/2);
\draw[dashed, thick, opacity=1] [red] (-\RI*\x, -\RI*\y, -\RI/2) --  (\RI*\x, -\RI*\y, -\RI/2);
\draw[dashed, thick] [magenta] (\RI*\x, -\RI*\y, -\RI/2) --  (\RI, 0, -\RI/2);
\draw [thick, dashed, yellow, opacity=1] (-\RI*\x, -\RI*\y, -\RI/2) -- (-\RII*\x, -\RII*\y, \height);
\draw [thick, dashed, orange, opacity=1] (\RI*\x, -\RI*\y, -\RI/2) -- (\RII*\x, -\RII*\y, \height);
\begin{scope}[help lines, dashed]
\draw (\RI, 0, -\RI/2) -- (\x*\RI, \y*\RI, \height);
\draw (\RI, 0, \height) -- (\x*\RI, \y*\RI, -\height);
\draw (\x*\RI,\y*\RI, -\RI/2) -- (-\x*\RI, \y*\RI, \height);
\draw (\x*\RI,\y*\RI, \height) -- (-\x*\RI, \y*\RI, -\height);
\draw (-\x*\RI, \y*\RI, -\height) -- (-\RI, 0, \height);
\draw (-\RI, 0, -\height) -- (-\x*\RI, \y*\RI, \height);
\draw (-\RI, 0, -\height) -- (-\x*\RI, -\y*\RI, \height);
\draw (-\RI, 0, \height) -- (-\x*\RI, -\y*\RI, -\height);
\draw (-\x*\RI, -\y*\RI, -\height) -- (\x*\RI, -\y*\RI, \height);
\draw (-\x*\RI, -\y*\RI, \height) -- (\x*\RI, -\y*\RI, -\height);
\draw (\x*\RI, -\y*\RI, -\height) -- (\RI, 0, \height);
\draw (\x*\RI, -\y*\RI, \height) -- (\RI, 0, -\height);
\end{scope}

\begin{scope}[latex'-]
\draw  node [above right, thick, black] at (0, 0, 0) {$m'$};
\draw [thick, blue] (0, 0, 0) -- (\y*\y*\RI, \y*\x*\RI, 0) node [above] at (17/30*\y*\y*\RI, 17/30*\y*\x*\RI, 0) {$\tilde{\gamma}'_{3}$} 
node [above, black] at (\y*\y*\RI, \y*\x*\RI, 0){$x'_3$};
\draw [thick, green] (0, 0, 0) -- (0, \y*\RI, 0) node [left] at (0, 17/30*\y*\RI,0){$\gamma'_{1}$} 
node [below, black] at (0, \y*\RI, 0){$x'_1$};
\draw [thick, green] (0, 0, 0) -- (-\y*\RI*\y, \y*\RI*\x, 0) node [below] at (20/30*-\y*\RI*\y,17/30*\y*\RI*\x, 0){$\gamma'_{2}$} 
node [below, black] at (-\y*\RI*\y, \y*\RI*\x, 0){$\;x'_2$};
\draw [thick, green] (0, 0, 0) -- (-\y*\RI*\y,-\y*\RI*\x, 0) node [below] at (15/30*-\y*\RI*\y, 20/30*-\y*\RI*\x, 0){$\gamma'_{3}$} 
node [above, black] at (-\y*\RI*\y, -\y*\RI*\x, 0){$x'_3$};
\draw [thick, blue] (0, 0, 0) -- (0, -\y*\RI,0) node [right] at (5/30, 20/30*-\y*\RI,0){\hspace{1mm}$\tilde{\gamma}'_{1}$}
 node [above, black] at (0, -\y*\RI,0){$x'_1$};
\draw [thick, blue] (0, 0, 0) -- (\y*\RI*\y, -\y*\RI*\x, 0) node [above] at (20/30*\y*\RI*\y, 20/30*-\y*\RI*\x, 0){$\tilde{\gamma}'_{2}$}
 node [above, black] at (\y*\RI*\y, -\y*\RI*\x, 0){$x'_2$};
\end{scope}


\draw[thick] [black] (\RI, 0, -\RI/2) -- (\RI*\x, \RI*\y, -\RI/2);
\draw[thick] [yellow] (\RI*\x, \RI*\y, -\RI/2) -- (-\RI*\x, \RI*\y, -\RI/2);
\draw[thick] [orange] (-\RI*\x, \RI*\y, -\RI/2) --  (-\RI, 0, -\RI/2);


\draw[thick] [yellow] (\RII, 0, \height) -- (\RII*\x, \RII*\y, \height);
\draw[thick] [orange] (\RII*\x, \RII*\y, \height) -- (-\RII*\x, \RII*\y, \height);
\draw[thick] [cyan] (-\RII*\x, \RII*\y, \height) --  (-\RII, 0, \height);
\draw[thick] [red] (-\RII, 0, \height) --  (-\RII*\x, -\RII*\y, \height);
\draw[thick] [magenta] (-\RII*\x, -\RII*\y, \height) --  (\RII*\x, -\RII*\y, \height);
\draw[thick] [black] (\RII*\x, -\RII*\y, \height) --  (\RII, 0, \height);

\draw [thick, cyan] (\RI, 0, -\RI/2) -- (\RII, 0, \height);
\draw [thick, red] (\RI*\x, \RI*\y, -\RI/2) -- (\RII*\x, \RII*\y, \height);
\draw [thick, magenta, opacity=1] (-\RI*\x, \RI*\y, -\RI/2) -- (-\RII*\x, \RII*\y, \height);
\draw [thick, black] (-\RI, 0, -\RI/2) -- (-\RII, 0, \height);

\end{tikzpicture}
\begin{center}
b)
\end{center}
\end{minipage}%
%
\begin{minipage}{0.33\textwidth}
\tdplotsetmaincoords{70}{185}
\hspace{6mm}
\begin{tikzpicture}[tdplot_main_coords, scale = 1.1]
\pgfmathsetmacro\x{cos(60)}
\pgfmathsetmacro\y{sin(60)}
\draw[thick, latex'-, green] (0,0,0) -- (0,0,1) node[right] at (0,0,1/2){\textcolor{green}{$\sigma$}} node [left, black] at (0,0,0){$m'$} node [above, black] at (0, 0, 1){$m$};
\draw[thick, latex'-, blue] (0,0,0) -- (0,0,-1) node[right] at (0,0,-1.2/2){\textcolor{blue}{$\tilde{\sigma}$}};
\def\RI{2}
\def\RII{2}
\def\height{\RI/2}
\draw [thick, dashed, yellow, opacity=1] (-\RI*\x, -\RI*\y, -\RI/2) -- (-\RII*\x, -\RII*\y, \height);
\draw [thick, dashed, orange, opacity=1] (\RI*\x, -\RI*\y, -\RI/2) -- (\RII*\x, -\RII*\y, \height);

\draw[thick] [black] (\RI, 0, -\RI/2) -- (\RI*\x, \RI*\y, -\RI/2);
\draw[thick] [yellow] (\RI*\x, \RI*\y, -\RI/2) -- (-\RI*\x, \RI*\y, -\RI/2);
\draw[thick] [orange] (-\RI*\x, \RI*\y, -\RI/2) --  (-\RI, 0, -\RI/2);
\draw[dashed, thick] [cyan] (-\RI, 0, -\RI/2) --  (-\RI*\x, -\RI*\y, -\RI/2);
\draw[dashed, thick] [red, opacity=0.5] (-\RI*\x, -\RI*\y, -\RI/2) --  (\RI*\x, -\RI*\y, -\RI/2);
\draw[dashed, thick] [magenta, opacity=0.5] (\RI*\x, -\RI*\y, -\RI/2) --  (\RI, 0, -\RI/2);

\begin{scope}[help lines, dashed]
\draw (\RII*\x, \RII*\y, \height) -- (-\RII*\x, -\RII*\y, \height);
\draw (-\RII*\x, \RII*\y, \height) -- (\RII*\x, -\RII*\y, \height);

\draw (\RII*\x, \RII*\y, -\height) -- (-\RII*\x, -\RII*\y, -\height);
\draw (-\RII*\x, \RII*\y, -\height) -- (\RII*\x, -\RII*\y, -\height);

\draw (\RII*\x, \RII*\y, \height) -- (-\RII*\x, -\RII*\y, -\height);
\draw (\RII*\x, \RII*\y, -\height) -- (-\RII*\x, -\RII*\y, \height);
\end{scope}

\draw[thick] [yellow] (\RII, 0, \height) -- (\RII*\x, \RII*\y, \height);
\draw[thick] [orange] (\RII*\x, \RII*\y, \height) -- (-\RII*\x, \RII*\y, \height);
\draw[thick] [cyan] (-\RII*\x, \RII*\y, \height) --  (-\RII, 0, \height);
\draw[thick] [red] (-\RII, 0, \height) --  (-\RII*\x, -\RII*\y, \height);
\draw[thick] [magenta] (-\RII*\x, -\RII*\y, \height) --  (\RII*\x, -\RII*\y, \height);
\draw[thick] [black] (\RII*\x, -\RII*\y, \height) --  (\RII, 0, \height);
\draw [thick, cyan, opacity=1] (\RI, 0, -\RI/2) -- (\RII, 0, \height);
\draw [thick, red, opacity=1] (\RI*\x, \RI*\y, -\RI/2) -- (\RII*\x, \RII*\y, \height);
\draw [thick, magenta, opacity=1] (-\RI*\x, \RI*\y, -\RI/2) -- (-\RII*\x, \RII*\y, \height);
\draw [thick, black, opacity=1] (-\RI, 0, -\RI/2) -- (-\RII, 0, \height);

\begin{scope}[thick, -latex', blue]
\draw (\RI, 0, -\height) -- (\y*\y*\RI, \y*\x*\RI, 0)
node [above right] at (\y*\y*\RI, \y*\x*\RI, 0){$\;\;\delta_{33}$}
node [above left] at (\y*\y*\RI, \y*\x*\RI, 0){$\tilde{\delta}_{23}\;$}
node [below left] at (\y*\y*\RI, \y*\x*\RI, 0){$\delta_{13}\;$}
node [below right] at (\y*\y*\RI, \y*\x*\RI, 0){$\;\delta_{23}$};
\draw (\x*\RI, \y*\RI, \height) -- (\y*\y*\RI, \y*\x*\RI, 0);
\draw (\RI, 0, \height) -- (\y*\y*\RI, \y*\x*\RI, 0);
\draw (\x*\RI, \y*\RI, -\height)--(\y*\y*\RI, \y*\x*\RI, 0);
\draw (\x*\RI,\y*\RI, -\height) -- (0, \y*\RI, 0) 
node [above right] at(0, \y*\RI, 0) {$\;\;\;\;\delta_{11}$}
node [above left] at (0, \y*\RI, 0){$\delta_{31}\;\;$}
node [below left] at (0, \y*\RI, 0){$\delta_{21}\;\;\;\;\;$}
node [below right] at (0, \y*\RI, 0) {$\;\;\;\tilde{\delta}_{31}$};
\draw (-\x*\RI, \y*\RI, \height)--(0, \y*\RI, 0);
\draw (\x*\RI,\y*\RI, \height) --(0, \y*\RI, 0); 
\draw (-\x*\RI, \y*\RI, -\height)--(0, \y*\RI, 0);
\draw (-\x*\RI, \y*\RI, -\height) -- (-\y*\RI*\y, \y*\RI*\x, 0) 
node [above right] at (-\y*\RI*\y, \y*\RI*\x, 0){$\;\delta_{22}$}
node [above left] at (-\y*\RI*\y, \y*\RI*\x, 0){$\delta_{12}\;\;$}
node [below left] at (-\y*\RI*\y, \y*\RI*\x, 0){$\delta_{32}\;\;$}
node [below right] at (-\y*\RI*\y, \y*\RI*\x, 0){$\;\;\;\,\tilde{\delta}_{12}$};
\draw (-\RI, 0, \height)--(-\y*\RI*\y, \y*\RI*\x, 0);
\draw (-\RI, 0, -\height) -- (-\y*\RI*\y, \y*\RI*\x, 0);
\draw (-\x*\RI, \y*\RI, \height)--(-\y*\RI*\y, \y*\RI*\x, 0);
\end{scope}


\end{tikzpicture}
\begin{center}
\vspace{9.3mm} \hspace{8mm} c)
\end{center}
\end{minipage}

\caption{
A Morse function $f \colon \chiang \to \RR$. All flowlines of index 1 are depicted with arrows pointing in the direction of downward gradient flow. Note that in diagram c) the flowlines $\delta_{ij}$ and $\tilde{\delta}_{ij}$ always go from $x_i$ to $x_j'$. The index 3 flowlines $\sigma$ and $\tilde{\sigma}$ and the different colouring (green and blue) of the flowlines will be used below for the calculation of parallel transport maps.}\label{prism and morse flowlines}
\end{figure}
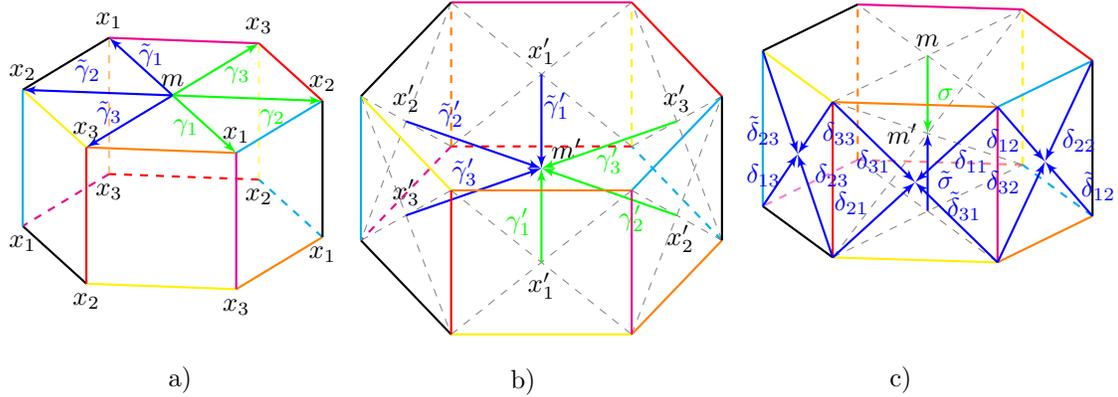

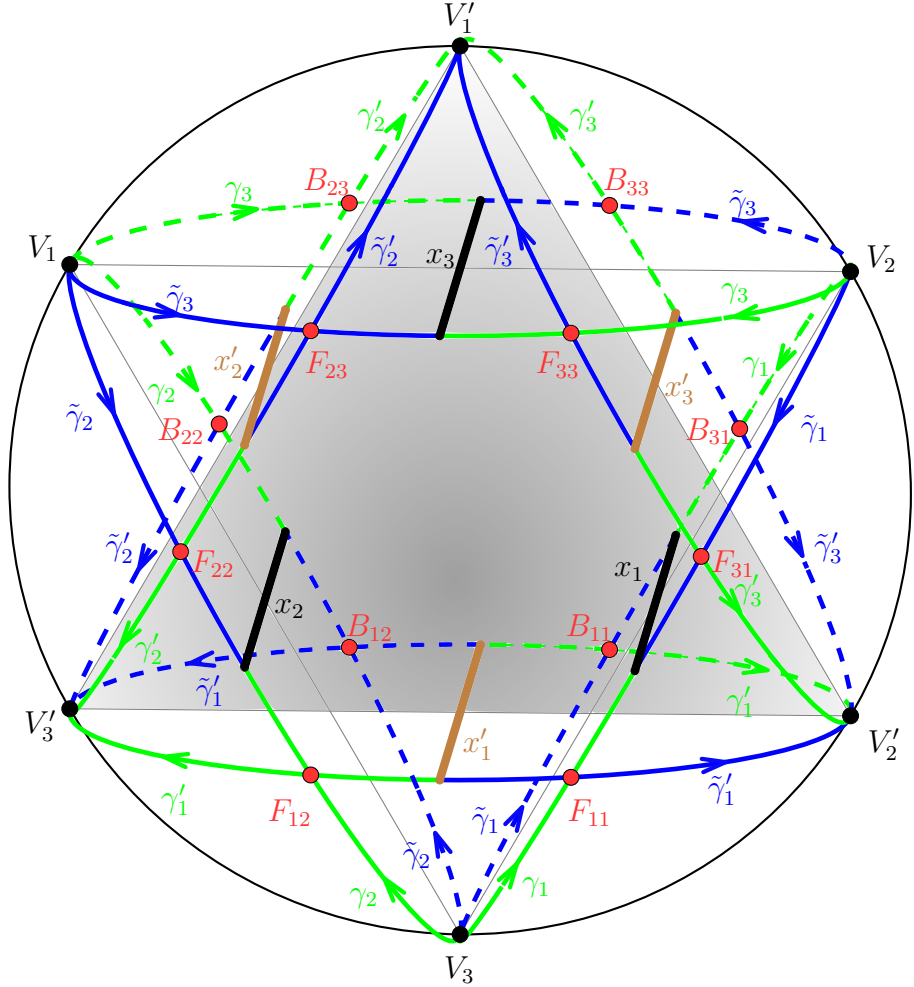
\begin{figure}[h!]
\labellist
\large
\pinlabel $V'_1$ at 178 368
\pinlabel $V'_2$ at 338 94
\pinlabel $V'_3$ at 20 101
\pinlabel $V_1$ at 21 281
\pinlabel $V_2$ at 337 277
\pinlabel $V_3$ at 178 9
\pinlabel $x_1$ at 242 159
\pinlabel $x_2$ at 114 145
\pinlabel $x_3$ at 171 277
\pinlabel {\textcolor{brown}{$x'_1$}} at 185 94
\pinlabel {\textcolor{brown}{$x'_2$}} at 91 238
\pinlabel {\textcolor{brown}{$x'_3$}} at 261 228
\pinlabel {\textcolor{green}{$\gamma'_1$}} at 71 73
\pinlabel {\textcolor{green}{$\gamma'_1$}} at 284 110
\pinlabel {\textcolor{green}{$\gamma'_2$}} at 61 130
\pinlabel {\textcolor{green}{$\gamma'_2$}} at 145 331
\pinlabel {\textcolor{green}{$\gamma'_3$}} at 287 152
\pinlabel {\textcolor{green}{$\gamma'_3$}} at 225 331
\pinlabel {\textcolor{green}{$\gamma_1$}} at 207 40
\pinlabel {\textcolor{green}{$\gamma_1$}} at 290 236
\pinlabel {\textcolor{green}{$\gamma_2$}} at 142 36
\pinlabel {\textcolor{green}{$\gamma_2$}} at 66 226
\pinlabel {\textcolor{green}{$\gamma_3$}} at 282 264
\pinlabel {\textcolor{green}{$\gamma_3$}} at 95 302
\pinlabel {\textcolor{blue}{$\tilde{\gamma}'_1$}} at 277 77
\pinlabel {\textcolor{blue}{$\tilde{\gamma}'_1$}} at 84 114
\pinlabel {\textcolor{blue}{$\tilde{\gamma}'_2$}} at 150 280
\pinlabel {\textcolor{blue}{$\tilde{\gamma}'_2$}} at 50 168
\pinlabel {\textcolor{blue}{$\tilde{\gamma}'_3$}} at 193 280
\pinlabel {\textcolor{blue}{$\tilde{\gamma}'_3$}} at 317 168
\pinlabel {\textcolor{blue}{$\tilde{\gamma}_1$}} at 312 213
\pinlabel {\textcolor{blue}{$\tilde{\gamma}_1$}} at 188 65
\pinlabel {\textcolor{blue}{$\tilde{\gamma}_2$}} at 35 217
\pinlabel {\textcolor{blue}{$\tilde{\gamma}_2$}} at 162 52
\pinlabel {\textcolor{blue}{$\tilde{\gamma}_3$}} at 73 262
\pinlabel {\textcolor{blue}{$\tilde{\gamma}_3$}} at 285 298
\pinlabel {\textcolor{red!80}{$F_{11}$}} at 227 68
\pinlabel {\textcolor{red!80}{$B_{11}$}} at 227 138
\pinlabel {\textcolor{red!80}{$F_{12}$}} at 114 68
\pinlabel {\textcolor{red!80}{$B_{12}$}} at 144 138
\pinlabel {\textcolor{red!80}{$F_{22}$}} at 85 162
\pinlabel {\textcolor{red!80}{$B_{22}$}} at 73 212
\pinlabel {\textcolor{red!80}{$F_{23}$}} at 128 237
\pinlabel {\textcolor{red!80}{$B_{23}$}} at 128 306
\pinlabel {\textcolor{red!80}{$F_{33}$}} at 215 237
\pinlabel {\textcolor{red!80}{$B_{33}$}} at 241 306
\pinlabel {\textcolor{red!80}{$F_{31}$}} at 281 162
\pinlabel {\textcolor{red!80}{$B_{31}$}} at 272 212
\endlabellist
\begin{center}
\includegraphics[scale=1]{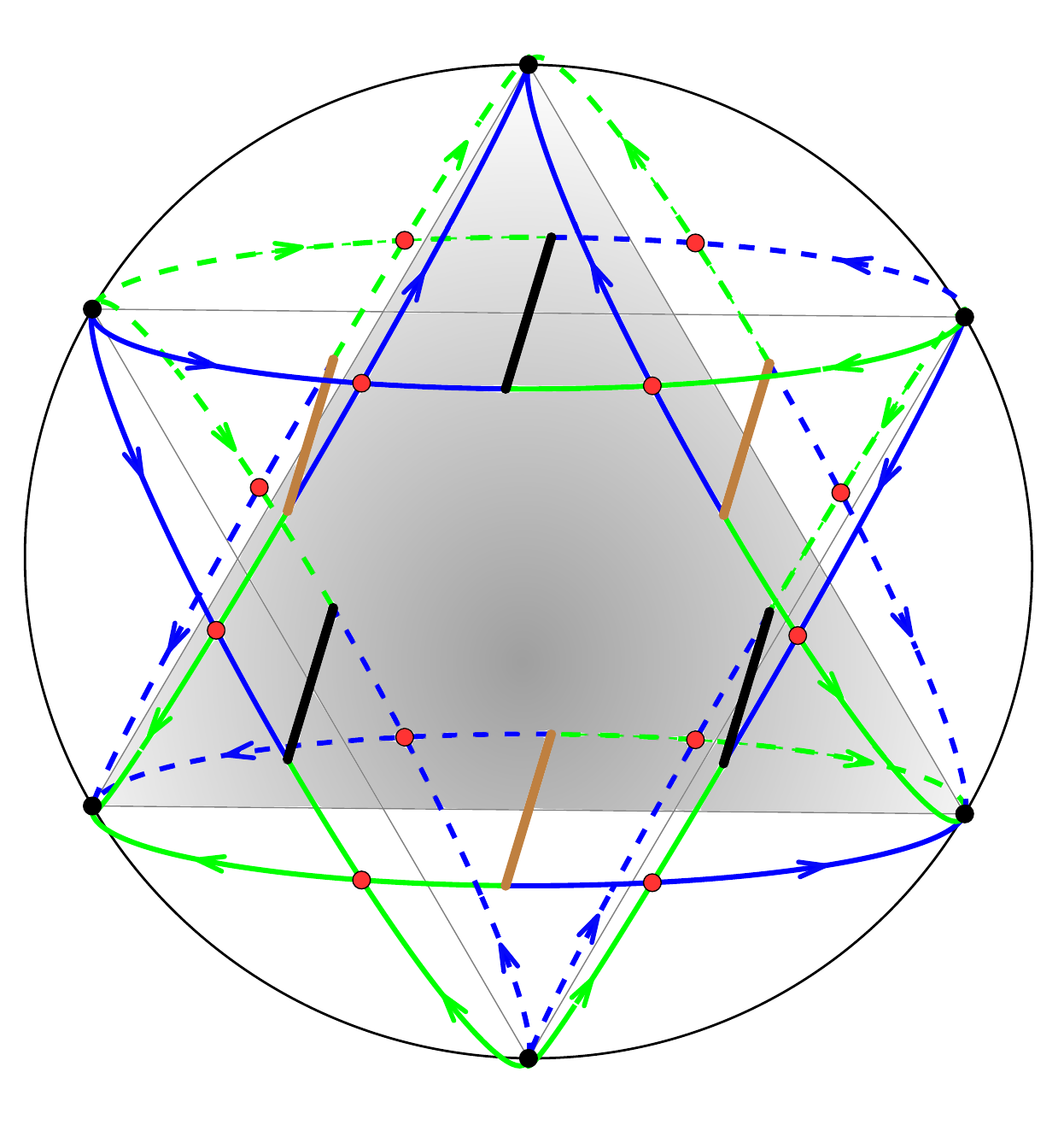}
\end{center}
\caption{Another representation of the Morse function $f$: The minimum $m'$ corresponds to the triangle $\triangle V'_1V'_2V'_3$ and the maximum $m$ is $\triangle V_1V_2V_3$. The critical points of index one $\{x'_i\}_{1 \le i \le 3}$ and index two $\{x_i\}_{1 \le i \le 3}$ correspond to triangles with one side along the segment with the respective label. The flowlines of index 1 through the minimum and maximum are also illustrated by the pairs of circular arcs with matching labels. Each downward flowline consists of triangles rotating around a fixed vertex, with their other two vertices tracing out the two arcs in the indicated directions. The labels of these arcs match the ones on the flowlines in Figure \ref{prism and morse flowlines}.}\label{the one with all the circles}
\end{figure}
For the sake of completeness, let us now give a formula for such a function. To describe it we will use coordinates on $\chiang$ coming from the Hopf coordinates on $S^3$. Consider the following ``Euler angles map'':
$$G \colon (\RR/2\pi\ZZ)^3 \to SU(2), \quad G(\theta, \varphi, \psi) \coloneqq \exp\left(\frac{\varphi+\psi}{2}\sigma_1\right) \cdot \exp\left(\theta \sigma_3\right) \cdot \exp\left(\frac{\varphi-\psi}{2}\sigma_1\right).$$ 
The map $G$ is a degree 4 ramified covering whose singular values (i.e. where ``gimbal lock'' occurs) form the standard Hopf link 
$\lbrace \exp(t \sigma_1) \st t \in [0, 2\pi] \rbrace \cup \lbrace \exp(t \sigma_1) \cdot \exp((\pi/2)\sigma_3) \st t \in [0, 2\pi] \rbrace$. Now our Morse function $f \colon \chiang \to \RR$ (or rather, its pull-back under $q \circ G$) is given by 
\begin{equation}\label{formula for morse function}
f(\theta, \varphi, \psi) = -\cos^4(\theta)\cos(6\varphi) - \sin^4(\theta)\cos(6\psi)
\end{equation}
and one can easily check that in these coordinates its critical points are indeed $m' =  (0, 0, 0)$, $x'_1 =  (\pi/4, 0, 0)$, $x'_2 =  (\pi/4, 2\pi/3, 2\pi/3)$, $x'_3 =  (\pi/4, \pi/3, \pi/3)$, $x_1 =  (\pi/4, \pi/6, \pi/6)$, $x_2 =  (\pi/4, 5\pi/6, 5\pi/6)$, $x_3 =  (\pi/4, 3\pi/6, 3\pi/6)$ and $m = (0, \pi/6, \pi/6)$.

\subsection{Computation of Floer Cohomology with Local Coefficients}\label{Computation of Floer Cohomology with Local Coefficients}
\subsubsection{Morse Differential}\label{Morse Differential}
We now move on to calculating $\del_0$ explicitly in the case when $L$ is the Chiang Lagrangian $\chiang$. More precisely, let $V$ be any vector space over $\FF_2$ and let $\rho \colon \bindih \to \Aut(V)$ be a representation. Since $\bindih\cong\pi_1(\chiang, m')^{Opp}$, $\rho$ determines a \emph{right} action of $\pi_1(\chiang)$ on V and so we obtain a local system $W \to \chiang$ with fibre isomorphic to $V$ by the recipe from \ref{local systems}. As Morse data for the pearl complex we shall use the pair $\F = (f, g)$, where $f$ is the Morse function (\ref{formula for morse function}) and $g$ is the round metric on $SU(2)$. In the Appendix we explain how to perturb $\F$ slightly, so that all transversality conditions are satisfied and the results of our calculations are not altered. Our goal for now is to explicitly compute the Morse differential $\del_0$ on the complex $C^*_{\F}(W, W)$. 

We thus need to calculate parallel transport maps on $W$ along the index 1 flow lines of $f$. To that end we first fix an identification $W_{m'} \cong V$. Next, we also identify with $V$ the fibres of $W$ which lie over other critical points. We do so in the unique way so that parallel transport maps along $(\gamma'_1)^{-1}$, $(\gamma'_2)^{-1}$, $(\gamma'_3)^{-1}$, $\sigma^{-1}$, $(\sigma^{-1}\cdot\gamma_1)$, $(\sigma^{-1}\cdot\gamma_2)$ and $(\sigma^{-1}\cdot\gamma_3)$  are represented by the identity map $V \to V$. From now on we refer to the paths in this list as \emph{identification paths} and we draw them in green on all diagrams (see also figure \ref{prism and morse flowlines} above). 

Suppose now that $\ell$ is a path from $s(\ell)\in Crit(f)$ to $t(\ell)\in Crit(f)$. By pre-concatenating $\ell$ with the identification path to $s(\ell)$ and post-concatenating it with the inverse of the identification path to $t(\ell)$ we obtain the corresponding loop $\hat{\ell}$, based at $m'$. We identify this loop with an element $\left[\hat{\ell}\right] \in \bindih$ via the isomorphism (\ref{isomorphism with pi_1}). Then, using the identifications above we have 
$$\begin{array}{rclccl}
&&&\rho\left(\left[\hat{\ell}\right]\right)&&\\
P_{\ell} & \colon & V                       & \xrightarrow{\hspace*{1.5cm}} && V \\
          &        &$\rotatebox{90}{$\cong$}$  &     &&$\rotatebox{90}{$\cong$}$\\
          &        &W_{s(\ell)}                &    && W_{t(\ell)}
\end{array}$$
We now use this set-up, together with the universal cover $SU(2) \to \chiang$ to calculate $\del_0$. Note that the fundamental domain whose centre lies at $\Id \in SU(2)$ borders 8 other fundamental domains with centres at $a = \exp(\sigma_1 \pi/3)$, $a^5 = \exp(-\sigma_1 \pi/3)$, 
$b = \exp(\sigma_3\pi/2)$, $ab$, $a^2b$, $a^3b$, $a^4b$ and $a^5b$. The first, fourth and eighth of these are schematically depicted (after stereographic projection from $-Id \in SU(2)$) in figures \ref{main and top}, \ref{main and front right} and \ref{main and front left}  respectively.\footnote{See Figure 3 in \cite{evans2014floer} or Figures \ref{actual plot b11aller}, \ref{actual plot b11retour} below for accurate pictures of the stereographically projected fundamental domains.}

Let us now compute $P_{\tilde{\gamma}'_2} \colon W_{x'_2} \to W_{m'}$ (identifications with $V$ are implicit here and in what follows). The corresponding loop is $(\gamma'_2)^{-1}\cdot\tilde{\gamma}'_2$. A lift of this loop at $\Id \in SU(2)$ is shown in Figure \ref{main and front right}.
%
%
\begin{figure}[h!]\label{parallel transport along gamma2primetilde}
\begin{center}
\tdplotsetmaincoords{80}{185}
\begin{tikzpicture}[tdplot_main_coords, scale = 1.2]
\pgfmathsetmacro\x{cos(60)}
\pgfmathsetmacro\y{sin(60)}

\draw node at (0,0,0) {$\cdot$};
\draw node [above, left] at (0,0,0) {$\operatorname{Id}$};

\def\RI{2}

\node [below, black] at (\RI,0,-\RI/2) {$x_1$};
\draw[thick] [black] (\RI, 0, -\RI/2) -- (\RI*\x, \RI*\y, -\RI/2) node [below, black] at (\RI*\x, \RI*\y, -\RI/2) {$x_2$};
\draw[thick] [yellow] (\RI*\x, \RI*\y, -\RI/2) -- (-\RI*\x, \RI*\y, -\RI/2) node [below, black] at (-\RI*\x, \RI*\y, -\RI/2) {$x_3$};
\draw[thick] [orange, dashed] (-\RI*\x, \RI*\y, -\RI/2) --  (-\RI, 0, -\RI/2) node [below, black] at (-\RI, 0, -\RI/2){$x_1$};
\draw[dashed, thick] [cyan] (-\RI, 0, -\RI/2) --  (-\RI*\x, -\RI*\y, -\RI/2) node [below, black] at (-\RI*\x, -\RI*\y, -\RI/2) {$x_2$};
\draw[dashed, thick] [red] (-\RI*\x, -\RI*\y, -\RI/2) --  (\RI*\x, -\RI*\y, -\RI/2) node [below, black] at (\RI*\x, -\RI*\y, -\RI/2) {$x_3$};
\draw[dashed, thick] [magenta] (\RI*\x, -\RI*\y, -\RI/2) --  (\RI, 0, -\RI/2);

\def\RII{2}
\def\height{1}

\draw node [above, black] at (\RII, 0, \height) {$x_2$};
\draw[thick] [yellow] (\RII, 0, \height) -- (\RII*\x, \RII*\y, \height) node [above, black] at (\RII*\x, \RII*\y, \height) {$x_3$};
\draw[thick] [orange] (\RII*\x, \RII*\y, \height) -- (-\RII*\x, \RII*\y, \height) node [above, black] at (-\RII*\x, \RII*\y, \height) {$x_1$};
\draw[thick] [cyan, dashed] (-\RII*\x, \RII*\y, \height) --  (-\RII, 0, \height) node [above, black] at (-\RII, 0, \height) {$x_2$};
\draw[thick] [red] (-\RII, 0, \height) --  (-\RII*\x, -\RII*\y, \height) node [above, black] at (-\RII*\x, -\RII*\y, \height) {$x_3$};
\draw[thick] [magenta] (-\RII*\x, -\RII*\y, \height) --  (\RII*\x, -\RII*\y, \height) node [above, black] at (\RII*\x, -\RII*\y, \height) {$x_1$};
\draw[thick] [black] (\RII*\x, -\RII*\y, \height) --  (\RII, 0, \height);

\draw [thick, cyan] (\RI, 0, -\RI/2) -- (\RII, 0, \height);
\draw [thick, red] (\RI*\x, \RI*\y, -\RI/2) -- (\RII*\x, \RII*\y, \height);
\draw [thick, magenta] (-\RI*\x, \RI*\y, -\RI/2) -- (-\RII*\x, \RII*\y, \height);
\draw [thick, black, dashed, opacity=0.5] (-\RI, 0, -\RI/2) -- (-\RII, 0, \height);
\draw [thick, yellow, dashed, opacity=0.5] (-\RI*\x, -\RI*\y, -\RI/2) -- (-\RII*\x, -\RII*\y, \height);
\draw [thick, orange, dashed] (\RI*\x, -\RI*\y, -\RI/2) -- (\RII*\x, -\RII*\y, \height);

\tdplotsetrotatedcoords{60}{270}{300}
\coordinate (Shift) at (-2*\RI*\y*\y,2*\RI*\y*\x,0);
\tdplotsetrotatedcoordsorigin{(Shift)}
\begin{scope}[tdplot_rotated_coords]
\draw  node at (0,0,0) {$\cdot$};
\draw  node [above, right] at (0,0,0) {$ab$};
\def\RI{2}

\draw[thick] [black] (\RI, 0, -\RI/2) -- (\RI*\x, \RI*\y, -\RI/2) node [above, black] at (\RI*\x, \RI*\y, -\RI/2) {$x_2$};
\draw[thick] [yellow] (\RI*\x, \RI*\y, -\RI/2) -- (-\RI*\x, \RI*\y, -\RI/2) node [above, black] at (-\RI*\x, \RI*\y, -\RI/2) {$x_3$};
\draw[thick] [orange] (-\RI*\x, \RI*\y, -\RI/2) --  (-\RI, 0, -\RI/2) node [above right, black] at (-\RI, 0, -\RI/2){$x_1$};
\draw[thick] [cyan] (-\RI, 0, -\RI/2) --  (-\RI*\x, -\RI*\y, -\RI/2) node [below, black] at (-\RI*\x, -\RI*\y, -\RI/2) {$x_2$};
\draw[thick] [red] (-\RI*\x, -\RI*\y, -\RI/2) --  (\RI*\x, -\RI*\y, -\RI/2);

\def\RII{2}
\def\height{1}

\def\RII{2}
\def\height{1}

\draw[thick] [yellow, dashed] (\RII, 0, \height) -- (\RII*\x, \RII*\y, \height) node [above, black] at (\RII*\x, \RII*\y, \height) {$x_3$};
\draw[thick] [orange] (\RII*\x, \RII*\y, \height) -- (-\RII*\x, \RII*\y, \height) node [right, black] at (-\RII*\x, \RII*\y, \height) {$x_1$};
\draw[thick] [cyan] (-\RII*\x, \RII*\y, \height) --  (-\RII, 0, \height) node [right, black] at (-\RII, 0, \height) {$x_2$};
\draw[thick] [red] (-\RII, 0, \height) --  (-\RII*\x, -\RII*\y, \height) node [below, black] at (-\RII*\x, -\RII*\y, \height) {$x_3$};
\draw[thick] [magenta, dashed] (-\RII*\x, -\RII*\y, \height) --  (\RII*\x, -\RII*\y, \height);

\draw [thick, red] (\RI*\x, \RI*\y, -\RI/2) -- (\RII*\x, \RII*\y, \height);
\draw [thick, magenta] (-\RI*\x, \RI*\y, -\RI/2) -- (-\RII*\x, \RII*\y, \height);
\draw [thick, black] (-\RI, 0, -\RI/2) -- (-\RII, 0, \height);
\draw [thick, yellow] (-\RI*\x, -\RI*\y, -\RI/2) -- (-\RII*\x, -\RII*\y, \height);
\end{scope}
\tdplotresetrotatedcoordsorigin;

\draw [thick, green, opacity=0.7, -latex'] (0, 0, 0) -- (-\y*\RI*\y, \y*\RI*\x, 0) node [below] at (20/30*-\y*\RI*\y,17/30*\y*\RI*\x, 0){$(\gamma'_{2})^{-1}\;\;\;\;\;\;\;\;\;$};
\draw [thick, opacity=0.7, -latex', blue](-\y*\RI*\y, \y*\RI*\x, 0) -- (-2*\RI*\y*\y,2*\RI*\y*\x,0) node [below] at
(25/30*-2*\RI*\y*\y,25/30*2*\RI*\y*\x,0) {$\tilde{\gamma}'_2$};

\end{tikzpicture}
\end{center}
\caption{Parallel transport along $\tilde{\gamma}'_2$.}\label{main and front right}
\end{figure}
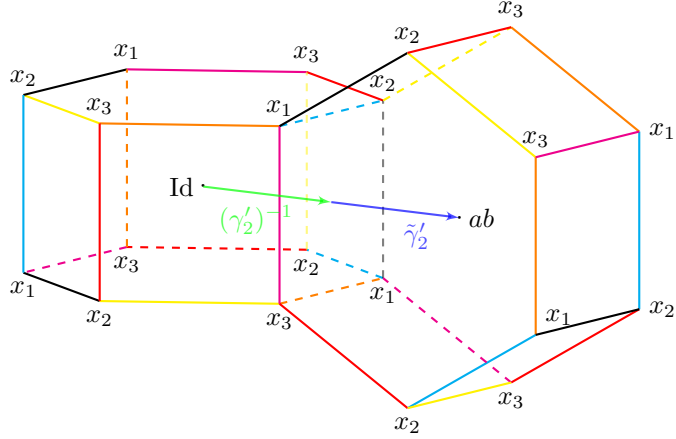
From there we read off:
$$\begin{array}{rcl}
&AB&\\
P_{\tilde{\gamma}'_2}  \colon  W_{x'_2}  & \xrightarrow{\hspace*{1.5cm}} & W_{m'},
 \end{array}$$
 where we write $A\coloneqq\rho(a)$ and $B\coloneqq \rho(b)$.
\noindent Similarly we can compute $P_{\tilde{\delta}_{23}} \colon W_{x_2} \to W_{x'_3}$ by lifting the loop $(\sigma)^{-1}\cdot\gamma_2\cdot\tilde{\delta}_{23}\cdot(\gamma'_3)^{-1}$ to the universal cover (see Figure \ref{main and top}). 
%
%
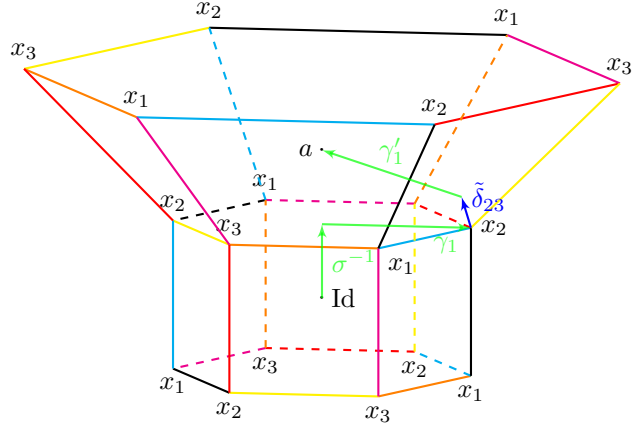
\begin{figure}[h!]\label{parallel transport along delta23tilde}
\begin{center}
\tdplotsetmaincoords{80}{188}
\begin{tikzpicture}[tdplot_main_coords, scale = 1]
\pgfmathsetmacro\x{cos(60)}
\pgfmathsetmacro\y{sin(60)}
\pgfmathsetmacro\hypoup{sqrt(7)/6}
\pgfmathsetmacro\yyup{(sin(45)*\hypoup}
\pgfmathsetmacro\yy{\y+\yyup}
\pgfmathsetmacro\zz{-1/2 + \yyup}

\draw node at (0,0,0) {$\cdot$};
\draw node [above, right] at (0,0,0) {$\operatorname{Id}$};

\def\RI{2}

\node [below, black] at (\RI,0,-\RI/2) {$x_1$};
\draw[thick] [black] (\RI, 0, -\RI/2) -- (\RI*\x, \RI*\y, -\RI/2) node [below, black] at (\RI*\x, \RI*\y, -\RI/2) {$x_2$};
\draw[thick] [yellow] (\RI*\x, \RI*\y, -\RI/2) -- (-\RI*\x, \RI*\y, -\RI/2) node [below, black] at (-\RI*\x, \RI*\y, -\RI/2) {$x_3$};
\draw[thick] [orange] (-\RI*\x, \RI*\y, -\RI/2) --  (-\RI, 0, -\RI/2) node [below, black] at (-\RI, 0, -\RI/2){$x_1$};
\draw[dashed, thick] [cyan] (-\RI, 0, -\RI/2) --  (-\RI*\x, -\RI*\y, -\RI/2) node [below, black] at (-\RI*\x, -\RI*\y, -\RI/2) {$x_2$};
\draw[dashed, thick] [red] (-\RI*\x, -\RI*\y, -\RI/2) --  (\RI*\x, -\RI*\y, -\RI/2) node [below, black] at (\RI*\x, -\RI*\y, -\RI/2) {$x_3$};
\draw[dashed, thick] [magenta] (\RI*\x, -\RI*\y, -\RI/2) --  (\RI, 0, -\RI/2);

\def\RII{2}
\def\height{1}

\draw node [above, black] at (\RII, 0, \height) {$x_2$};
\draw[thick] [yellow] (\RII, 0, \height) -- (\RII*\x, \RII*\y, \height) node [above, black] at (\RII*\x, \RII*\y, \height) {$x_3$};
\draw[thick] [orange] (\RII*\x, \RII*\y, \height) -- (-\RII*\x, \RII*\y, \height) node [below right, black] at (-\RII*\x, \RII*\y, \height) {$x_1$};
\draw[thick] [cyan] (-\RII*\x, \RII*\y, \height) --  (-\RII, 0, \height) node [right, black] at (-\RII, 0, \height) {$x_2$};
\draw[thick] [red, dashed] (-\RII, 0, \height) --  (-\RII*\x, -\RII*\y, \height);
\draw[thick] [magenta, dashed] (-\RII*\x, -\RII*\y, \height) --  (\RII*\x, -\RII*\y, \height) node [above, black] at (\RII*\x, -\RII*\y, \height) {$x_1$};
\draw[thick] [black, dashed] (\RII*\x, -\RII*\y, \height) --  (\RII, 0, \height);

\draw [thick, cyan] (\RI, 0, -\RI/2) -- (\RII, 0, \height);
\draw [thick, red] (\RI*\x, \RI*\y, -\RI/2) -- (\RII*\x, \RII*\y, \height);
\draw [thick, magenta] (-\RI*\x, \RI*\y, -\RI/2) -- (-\RII*\x, \RII*\y, \height);
\draw [thick, black] (-\RI, 0, -\RI/2) -- (-\RII, 0, \height);
\draw [thick, yellow, dashed] (-\RI*\x, -\RI*\y, -\RI/2) -- (-\RII*\x, -\RII*\y, \height);
\draw [thick, orange, dashed] (\RI*\x, -\RI*\y, -\RI/2) -- (\RII*\x, -\RII*\y, \height);


\tdplotsetrotatedcoords{0}{0}{300}
\coordinate (Shift) at (0,0,2);
\tdplotsetrotatedcoordsorigin{(Shift)}
\begin{scope}[tdplot_rotated_coords]
\draw node at (0,0,0) {$\cdot$};
\draw node [above, left] at (0,0,0) {$a$};
\def\RI{2}


\def\RII{4}
\def\height{1}

\draw  node [above, black] at (\RII, 0, \height) {$x_2$};
\draw[thick] [yellow] (\RII, 0, \height) -- (\RII*\x, \RII*\y, \height) node [above, black] at (\RII*\x, \RII*\y, \height) {$x_3$};
\draw[thick] [orange] (\RII*\x, \RII*\y, \height) -- (-\RII*\x, \RII*\y, \height) node [above, black] at (-\RII*\x, \RII*\y, \height) {$x_1$};
\draw[thick] [cyan] (-\RII*\x, \RII*\y, \height) --  (-\RII, 0, \height) node [above, black] at (-\RII, 0, \height) {$x_2$};
\draw[thick] [red] (-\RII, 0, \height) --  (-\RII*\x, -\RII*\y, \height) node [above, black] at (-\RII*\x, -\RII*\y, \height) {$x_3$};
\draw[thick] [magenta] (-\RII*\x, -\RII*\y, \height) --  (\RII*\x, -\RII*\y, \height) node [above, black] at (\RII*\x, -\RII*\y, \height) {$x_1$};
\draw[thick] [black] (\RII*\x, -\RII*\y, \height) --  (\RII, 0, \height);

\draw [thick, cyan, dashed] (\RI, 0, -\RI/2) -- (\RII, 0, \height);
\draw [thick, red] (\RI*\x, \RI*\y, -\RI/2) -- (\RII*\x, \RII*\y, \height);
\draw [thick, magenta] (-\RI*\x, \RI*\y, -\RI/2) -- (-\RII*\x, \RII*\y, \height);
\draw [thick, black] (-\RI, 0, -\RI/2) -- (-\RII, 0, \height);
\draw [thick, yellow] (-\RI*\x, -\RI*\y, -\RI/2) -- (-\RII*\x, -\RII*\y, \height);
\draw [thick, orange, dashed] (\RI*\x, -\RI*\y, -\RI/2) -- (\RII*\x, -\RII*\y, \height);
%
%

\draw [thick, opacity=1, -latex', blue] (-\x*\RI, -\y*\RI, -\height) -- (-\y*\yy*\RI, -\x*\yy*\RI, \zz*\RI) node[right] at (-\y*\yy*\RI, -\x*\yy*\RI, \zz*\RI) {$\tilde{\delta}_{23}$};
\draw [thick, green, opacity=0.7, -latex'] (-\y*\yy*\RI, -\x*\yy*\RI, \zz*\RI) -- (0, 0, 0) node [above] at (-1/2*\y*\yy*\RI, -1/2*\x*\yy*\RI, 1/2*\zz*\RI) {$\gamma'_1$};
\end{scope}
\draw [thick, green, opacity=0.7, -latex'] (0, 0, 0) -- (0,0, 1) node [right] at (0,0, 1/2){$\sigma^{-1}$};
\draw [thick, green, opacity=0.7, -latex'](0,0, 1) -- (-\RI,0,1) node [below] at
(25/30*-\RI, 0,1) {$\gamma_1$};

\end{tikzpicture}
\end{center}
\caption{Parallel transport along $\tilde{\delta}_{23}$}\label{main and top}
\end{figure}
\noindent Thus we obtain: 
$$\begin{array}{rcl}
&A&\\
P_{\tilde{\delta}_{23}}  \colon  W_{x_2}  & \xrightarrow{\hspace*{1.5cm}} & W_{x'_3}.
 \end{array}$$
\hspace{1mm} \\
\noindent We compute one more example, namely $P_{\tilde{\gamma}_3} \colon W_{m} \to W_{x_3}$. The loop $(\sigma)^{-1}\cdot\tilde{\gamma}_3\cdot(\gamma_{3})^{-1}\cdot\sigma$ lifts as shown in Figure \ref{main and front left}.
%
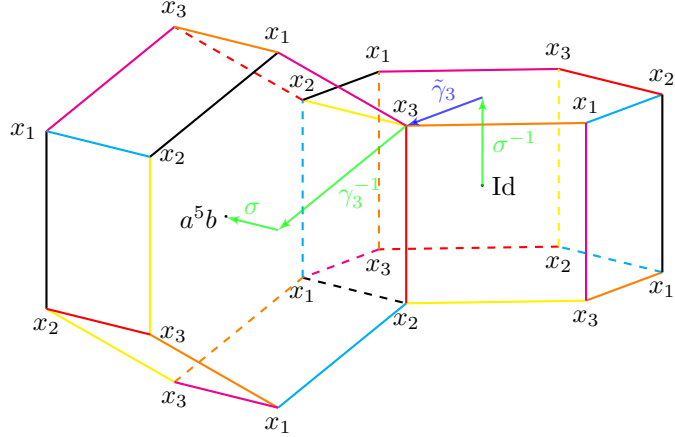
\begin{figure}[h!]\label{parallel transport along gamma3tilde}
\begin{center}
\tdplotsetmaincoords{80}{175}
\begin{tikzpicture}[tdplot_main_coords, scale = 1.2]
\pgfmathsetmacro\x{cos(60)}
\pgfmathsetmacro\y{sin(60)}

\draw node at (0,0,0) {$\cdot$};
\draw node [above, right] at (0,0,0) {$\operatorname{Id}$};

\def\RI{2}	

\node [below, black] at (\RI,0,-\RI/2) {$x_1$};
\draw[thick] [black, dashed] (\RI, 0, -\RI/2) -- (\RI*\x, \RI*\y, -\RI/2) node [below, black] at (\RI*\x, \RI*\y, -\RI/2) {$x_2$};
\draw[thick] [yellow] (\RI*\x, \RI*\y, -\RI/2) -- (-\RI*\x, \RI*\y, -\RI/2) node [below, black] at (-\RI*\x, \RI*\y, -\RI/2) {$x_3$};
\draw[thick] [orange] (-\RI*\x, \RI*\y, -\RI/2) --  (-\RI, 0, -\RI/2) node [below, black] at (-\RI, 0, -\RI/2){$x_1$};
\draw[dashed, thick] [cyan] (-\RI, 0, -\RI/2) --  (-\RI*\x, -\RI*\y, -\RI/2) node [below, black] at (-\RI*\x, -\RI*\y, -\RI/2) {$x_2$};
\draw[dashed, thick] [red] (-\RI*\x, -\RI*\y, -\RI/2) --  (\RI*\x, -\RI*\y, -\RI/2) node [below, black] at (\RI*\x, -\RI*\y, -\RI/2) {$x_3$};
\draw[dashed, thick] [magenta] (\RI*\x, -\RI*\y, -\RI/2) --  (\RI, 0, -\RI/2);

\def\RII{2}
\def\height{1}

\draw node [above, black] at (\RII, 0, \height) {$x_2$};
\draw[thick] [yellow] (\RII, 0, \height) -- (\RII*\x, \RII*\y, \height) node [above, black] at (\RII*\x, \RII*\y, \height) {$x_3$};
\draw[thick] [orange] (\RII*\x, \RII*\y, \height) -- (-\RII*\x, \RII*\y, \height) node [above, black] at (-\RII*\x, \RII*\y, \height) {$x_1$};
\draw[thick] [cyan] (-\RII*\x, \RII*\y, \height) --  (-\RII, 0, \height) node [above, black] at (-\RII, 0, \height) {$x_2$};
\draw[thick] [red] (-\RII, 0, \height) --  (-\RII*\x, -\RII*\y, \height) node [above, black] at (-\RII*\x, -\RII*\y, \height) {$x_3$};
\draw[thick] [magenta] (-\RII*\x, -\RII*\y, \height) --  (\RII*\x, -\RII*\y, \height) node [above, black] at (\RII*\x, -\RII*\y, \height) {$x_1$};
\draw[thick] [black] (\RII*\x, -\RII*\y, \height) --  (\RII, 0, \height);

\draw [thick, cyan, dashed] (\RI, 0, -\RI/2) -- (\RII, 0, \height);
\draw [thick, red] (\RI*\x, \RI*\y, -\RI/2) -- (\RII*\x, \RII*\y, \height);
\draw [thick, magenta] (-\RI*\x, \RI*\y, -\RI/2) -- (-\RII*\x, \RII*\y, \height);
\draw [thick, black] (-\RI, 0, -\RI/2) -- (-\RII, 0, \height);
\draw [thick, yellow, dashed] (-\RI*\x, -\RI*\y, -\RI/2) -- (-\RII*\x, -\RII*\y, \height);
\draw [thick, orange, dashed] (\RI*\x, -\RI*\y, -\RI/2) -- (\RII*\x, -\RII*\y, \height);

\draw [thick, green, opacity=0.7, -latex'] (0, 0, 0) -- (0,0, 1) node [right] at (0,0, 1/2){$\sigma^{-1}$};
\draw [thick, blue, opacity=0.7, -latex'] (0, 0, \height) -- (\x*\RII, \y*\RII, \height) node [above] at (15/30*\x*\RII, 15/30*\y*\RII, \height) {$\tilde{\gamma}_{3}$};


\tdplotsetrotatedcoords{300}{270}{60}
\coordinate (Shift) at (2*\RI*\y*\y,2*\RI*\y*\x,0);
\tdplotsetrotatedcoordsorigin{(Shift)}
\begin{scope}[tdplot_rotated_coords]
\draw node at (0,0,0) {$\cdot$};
\draw node [above, left] at (0,0,0) {$a^5b$};
\def\RI{2}

\node [left, black] at (\RI,0,-\RI/2) {$x_1$};
\draw[thick] [black] (\RI, 0, -\RI/2) -- (\RI*\x, \RI*\y, -\RI/2) node [below, black] at (\RI*\x, \RI*\y, -\RI/2) {$x_2$};
\draw[thick] [yellow] (\RI*\x, \RI*\y, -\RI/2) -- (-\RI*\x, \RI*\y, -\RI/2) node [below, black] at (-\RI*\x, \RI*\y, -\RI/2) {$x_3$};
\draw[thick, dashed] [orange] (-\RI*\x, \RI*\y, -\RI/2) --  (-\RI, 0, -\RI/2); 
\draw[thick, dashed] [red] (-\RI*\x, -\RI*\y, -\RI/2) --  (\RI*\x, -\RI*\y, -\RI/2) node [above, black] at (\RI*\x, -\RI*\y, -\RI/2) {$x_3$};
\draw[thick] [magenta] (\RI*\x, -\RI*\y, -\RI/2) --  (\RI, 0, -\RI/2);

\def\RII{2}
\def\height{1}

\def\RII{2}
\def\height{1}

\draw node [right, black] at (\RII, 0, \height) {$x_2$};
\draw[thick] [yellow] (\RII, 0, \height) -- (\RII*\x, \RII*\y, \height) node [right, black] at (\RII*\x, \RII*\y, \height) {$x_3$};
\draw[thick] [orange] (\RII*\x, \RII*\y, \height) -- (-\RII*\x, \RII*\y, \height) node [below, black] at (-\RII*\x, \RII*\y, \height) {$x_1$};
\draw[thick] [cyan] (-\RII*\x, \RII*\y, \height) --  (-\RII, 0, \height);
\draw[thick] [magenta] (-\RII*\x, -\RII*\y, \height) --  (\RII*\x, -\RII*\y, \height) node [above, black] at (\RII*\x, -\RII*\y, \height) {$x_1$};
\draw[thick] [black] (\RII*\x, -\RII*\y, \height) --  (\RII, 0, \height);

\draw [thick, cyan] (\RI, 0, -\RI/2) -- (\RII, 0, \height);
\draw [thick, red] (\RI*\x, \RI*\y, -\RI/2) -- (\RII*\x, \RII*\y, \height);
\draw [thick, magenta] (-\RI*\x, \RI*\y, -\RI/2) -- (-\RII*\x, \RII*\y, \height);
\draw [thick, orange] (\RI*\x, -\RI*\y, -\RI/2) -- (\RII*\x, -\RII*\y, \height);
\draw [thick, green, latex'-, opacity=0.7] (0, 0, \height) -- (-\x*\RII, -\y*\RII, \height) node [below] at (20/30*-\x*\RII, 20/30-\y*\RII, \height){$\gamma_{3}^{-1}$};
\draw [thick, green, opacity=0.7, latex'-] (0, 0, 0) -- (0,0, 1) node [above] at (0,0, 1/2){$\sigma$};
\end{scope}
\tdplotresetrotatedcoordsorigin;

\end{tikzpicture}
\end{center}
\caption{Parallel transport along $\tilde{\gamma}_3$}\label{main and front left}
\end{figure}
This yields:
$$\begin{array}{rcl}
&A^5B&\\
P_{\tilde{\gamma}_3}  \colon  W_{m}  & \xrightarrow{\hspace*{1.5cm}} & W_{x_3}.
 \end{array}$$
Continuing this way we obtain all the needed maps for the calculation of $\del_0$, which we summarise in the following table:

\begin{minipage}{0.19\textwidth}
\begin{eqnarray*}
P_{\tilde{\gamma}'_1} &= &B,\\
P_{\tilde{\gamma}'_2} &= &AB,\\
P_{\tilde{\gamma}'_3} &=& A^2B,\\ 
\end{eqnarray*}
\end{minipage}
\begin{minipage}{0.19\textwidth}
\begin{eqnarray*}
P_{\delta_{11}} &=&\Id,\\ 
P_{\delta_{21}} &=& A^2B,\\
P_{\delta_{31}} &=& A^2B,\\
P_{\tilde{\delta}_{31}}& =&A^3B,\\
\end{eqnarray*}
\end{minipage} 
\begin{minipage}{0.19\textwidth}
\begin{eqnarray*}
P_{\delta_{12}} &= &\Id, \\
P_{\tilde{\delta}_{12}} &=& A,\\
P_{\delta_{22}}& = &\Id, \\
P_{\delta_{32}}& = &A^3B, \\
\end{eqnarray*}
\end{minipage}
\begin{minipage}{0.19\textwidth}
\begin{eqnarray*}
P_{\delta_{13}} &=& A, \\
P_{\delta_{23}}& =& \Id,\\
P_{\tilde{\delta}_{23}}& =& A,\\ 
P_{\delta_{33}} &=& \Id,\\ 
\end{eqnarray*}
\end{minipage}
\begin{minipage}{0.19\textwidth}
\begin{eqnarray*}
P_{\tilde{\gamma}_1}& = &A^3B,\\ 
P_{\tilde{\gamma}_2} &= &A^4B,\\ 
P_{\tilde{\gamma}_3} &= &A^5B.\\
\end{eqnarray*}
\end{minipage}
Now every flowline $\gamma$ connecting $y$ to $x$ gives us a map $\End(W_x) \to \End(W_y)$ by conjugation $\alpha \mapsto P_{\gamma}^{-1}\circ \alpha \circ P_{\gamma}$. Using the above expressions we obtain:

\begin{itemize}
\item for every $\alpha' \in \End(W_{m'})$:
\begin{eqnarray}
\del_0 (\alpha') & = &\left(\alpha' + \bin\alpha'\b,\, \alpha' + \abin\alpha'\ab,\, \alpha' + \aabin\alpha'\aab\right)\nonumber\\
 &\in & \End(W_{x'_1})\oplus\End(W_{x'_2})\oplus\End(W_{x'_3}) \label{del 0 first part}
\end{eqnarray}
\item for every $(\alpha_1', \alpha_2', \alpha_3') \in \End(W_{x'_1})\oplus\End(W_{x'_2})\oplus\End(W_{x'_3})$:
\begin{eqnarray}
\del_0(\alpha_1', \alpha_2', \alpha_3')& = &\left(
\alpha_1' + \alpha_2' + \ain\alpha_2'\a + \ain\alpha_3'\a\, ,\right.\nonumber\\
&&\aabin\alpha_1'\aab + \alpha_2' + \ain\alpha_3'\a + \alpha_3'\, ,\nonumber\\
&&\left.\aabin\alpha_1'\aab + \aaabin\alpha_1'\aaab + \aaabin\alpha_2'\aaab + \alpha_3'\right)\nonumber\\
&\in & \End(W_{x_1})\oplus\End(W_{x_2})\oplus\End(W_{x_3})\label{del 0 second part}
\end{eqnarray}
\item and for every $(\alpha_1, \alpha_2, \alpha_3) \in \End(W_{x_1})\oplus\End(W_{x_2})\oplus\End(W_{x_3})$:
\begin{eqnarray}
\del_0 (\alpha_1, \alpha_2, \alpha_3) &=& \alpha_1+ \aaabin\alpha_1\aaab + \alpha_2 + \aaaabin\alpha_2\aaaab\nonumber\\
&&+ \alpha_3  + \aaaaabin\alpha_3\aaaaab \nonumber\\
&\in & \End(W_m) \label{del 0 third part}
\end{eqnarray}
\end{itemize}

\subsubsection{Contributions from holomorphic discs}\label{Contributions from holomorphic discs}
The classification of holomorphic discs of Maslov indices $2$ and $4$ with boundary on $\chiang$ (which are precisely the ones appearing in the pearly differential) has been carried out in \cite{evans2014floer} (see also \cite{smith2015floer}). Our main goal for this section will be to trace where their boundaries pass and to determine the parallel transport maps $\gamma^j_{\mathbf{u}}$ for all relevant pearly trajectories $\mathbf{u}$. Let us first recall the main notions and results on holomorphic discs from \cite{evans2014floer} which give us total control over the positions of the Maslov 2 discs.
\begin{definition}
Let $L$ be a manifold and let $K$ be a Lie group acting on $L$. Denote the stabiliser of a point $x \in L$ by $K_x$. An $x$-admissible homomorphism is defined to be any homomorphism $R \colon \RR \to K$ such that $R(2 \pi) \in K_x$. Such a homomorphism is called primitive if $R(\theta) \notin K_x,\; \forall \theta \in (0, 2\pi)$. 
\end{definition}
\begin{definition}
Let $X$ be a manifold and let $K$ be a Lie group acting on $X$. Suppose further that $L \subseteq X$ is a submanifold preserved by the action. A disc $u \colon (D, \del D) \to (X, L)$ is called axial if (after possibly reparametrising $u$) there exists a $u(1)$-admissible homomorphism $R$ such that $u(re^{i \theta}) = R(\theta)\cdot u(r)$ for every $r \in [0, 1)$, $\theta \in \RR$. For the particular case when $K=SU(2)$ we also define the axis of a non-constant axial disc $u$ to be the normalised infinitesimal generator $R'(0)/\lvert \lvert R'(0) \rvert \rvert \in S^2$, where again we identify $\mathfrak{su}(2)$ with $\RR^3$ via the basis $\{\sigma_1, \sigma_2, \sigma_3\}$.
\end{definition}

Let $J_0$ denote the standard (integrable) almost complex structure on $\CP^3$. Using the above notions, one can summarise Evans and Lekili's classification results for $J_0$-holomorphic discs $u \colon (D, \del D) \to (\CP^3, \chiang)$ in the following three theorems.
\begin{theorem}(\cite{evans2014floer}, lemma 3.3.1)\label{regularity}
All $J_0$-holomorphic discs $u \colon (D, \del D) \to (\CP^3, \chiang)$ are \emph{regular}.\footnote{In particular $J_0 \in \mathcal{J}_{reg}(\chiang)$.} 
\end{theorem}
\begin{theorem}(\cite{evans2014floer}, sections 3.5 and 6.1) \label{maslov 2 discs on chiang}
All $J_0$-holomorphic discs of  Maslov index 2 with boundary on $\chiang$ are axial. Through every point on $\chiang$ there pass exactly three such discs, namely the appropriate $SU(2)$-translates of the discs $\{u'_1, u'_2, u'_3\}$ with $u'_i(1) = m'$ for $i \in \{1, 2, 3\}$ and axes $V'_1$, $V'_2$ and $V'_3$ respectively.
\end{theorem}
\begin{theorem}(\cite{evans2014floer}, sections 3.6 and 6.2, example 6.1.3 and lemma 7.2.2 )\label{maslov 4 discs through m and m'}
There are precisely two $J_0$-holomorphic discs $w_1, w_{-1}\colon (D, \del D) \to (\CP^3, \chiang)$ of Maslov index 4 and passing through $m$ and $m'$. They are both simple and axial with axes $(1, 0, 0)$ and $(-1, 0, 0)$ respectively.
\end{theorem}
Note first that Theorem \ref{maslov 2 discs on chiang} immediately allows us to compute the value of the obstruction section $m_0(E)$ at the point $m'$. Indeed, the boundaries of the three Maslov 2 discs passing through $m'$ are given by
\begin{equation}\label{boundaries of maslov 2 discs through m'}
 \del u'_1 = (\gamma'_1)^{-1}\cdot\tilde{\gamma}'_1,\quad \del u'_2 = (\tilde{\gamma}'_2)^{-1} \cdot \gamma'_2 \quad \text{and} \quad \del u'_3 = (\gamma'_3)^{-1} \cdot \tilde{\gamma}'_3.
 \end{equation}
  Referring to our calculations in Section \ref{Morse Differential}, we have $P_{\del u'_1} = \b$, $P_{\del u'_2} = \abin = A^4B$ and $P_{\del u'_3} = A^2B$. This gives:
\begin{equation}\label{m_0 for chiang}
m_0(E)(m') = (\Id + \aa + \aaaa)\b.
\end{equation}
Note that by point \ref{important remark on self-floer cohomology} in Remark \ref{remark with three parts}, the cohomology $HF^*(W,W)$ is well-defined if and only if $m_0(W)(m') \in \{0, \Id\}$. Using (\ref{m_0 for chiang}) and the identities $(A^2 - \Id)(A^4 + A^2 + \Id) = 0$ and $A^3 = B^2$, it is easy to see that $m_0(W) = \Id$ only when $W$ is trivial and $m_0(W) = 0$ precisely when $A^2$ has no non-zero fixed vector.
\begin{remark}\label{rank 1 doesnt work}
Recall that in order to have well-defined cohomology $HF((\chiang, W), \RP^3)$ we need $m_0(W)=0$. But any $1-$dimensional representation of $\bindih$ must satisfy $A^2 = \Id$ and so the resulting local system has non-vanishing obstruction section.
\end{remark}



The three theorems above actually allow us to completely determine all isolated pearly trajectories which a candidate differential $d^{(\mathcal{F}, J_0)} \colon C^*_{f}(W,W) \to  C^*_{f}(W,W)$ would count. Note that while Theorems \ref{regularity}, \ref{maslov 2 discs on chiang} and \ref{maslov 4 discs through m and m'} give us a strong control over the moduli spaces of discs involved in $d^{(\mathcal{F}, J_0)}$, we cannot a priori be sure that $J_0 \in \Jreg(\F)$. This potential problem has been dealt with already in \cite{evans2014floer} and further elaborated on in \cite{smith2015floer} and the solution is to perturb the Morse data $\F$ by pushing it forward through a diffeomorphism of $\chiang$. In \ref{perturbation of morse data} below we explain why implementing this perturbation does not affect any of our calculations, so for now let us work directly with the complex $C^*_{f}(W,W)$ and determine the candidate differential $d^{(\F, J_0)}$. To alleviate notation we shall also temporarily drop the decorations $(\F, J_0)$.

From equation (\ref{splitting of differential}) we know that the maps which we need to figure out are:

$$
\begin{array}{rrcl}
\del_1 \colon & \End(W_{x'_1})\oplus\End(W_{x'_2})\oplus\End(W_{x'_3}) & \to &\End(W_{m'}),\\
\del_1 \colon & \End(W_{x_1})\oplus\End(W_{x_2})\oplus\End(W_{x_3}) &\to& \End(W_{x'_1})\oplus\End(W_{x'_2})\oplus\End(W_{x'_3}),\\
\del_1 \colon &\End(W_{m}) &\to& \End(W_{x_1})\oplus\End(W_{x_2})\oplus\End(W_{x_3}),\\
\del_2 \colon &\End(W_{m}) &\to &\End(W_{m'}).
\end{array}
$$

To determine the first one of these, we are interested in pearly configurations, consisting of a single Maslov 2 disc $u$ such that $u(-1) \in W^d(m')$ and $u(1) \in W^a(x'_i)$. Since $W^d(m') = \{m'\}$ such a disc must be one of $\{u'_1, u'_2, u'_3\}$. From \eqref{boundaries of maslov 2 discs through m'} we see that the corresponding parallel transport maps are 
$$
\begin{array}{lll}
P_{\gamma^0_{u'_1}} = \Id, \; & P_{\gamma^0_{u'_2}} = \abin = A^4B, \; & P_{\gamma^0_{u'_3}} = \Id \\
P_{\gamma^1_{u'_1}} = B, \; & P_{\gamma^1_{u'_2}} = \Id,\; & P_{\gamma^1_{u'_3}} = A^2B.
\end{array}
$$
Thus for every $(\alpha'_1, \alpha'_2, \alpha'_3) \in \End(W_{x'_1})\oplus\End(W_{x'_2})\oplus\End(W_{x'_3})$ we have
\begin{equation}\label{del 1 from index 1 to index 0}
\del_1 (\alpha'_1, \alpha'_2, \alpha'_3) = \b \alpha'_1 + \alpha'_2 \aaaab + \aab \alpha'_3 \quad \in \quad \End(W_{m'}).
\end{equation}

Similarly, to determine $\del_1 \colon \End(W_{m}) \to \End(W_{x_1})\oplus\End(W_{x_2})\oplus\End(W_{x_3})$ we look for pearly configurations containing one Maslov 2 disc $u$ such that $u(-1) \in W^d(x_i)$ and $u(1) \in W^a(m) = \{m\}$. From Theorem \ref{maslov 2 discs on chiang} we know that these discs must be axial. Consulting Figure \ref{the one with all the circles} we see that their axes are $\{V_1, V_2, V_3\}$ and, denoting by $u_i$ the disc with axis $V_i$, we have: $\gamma^0_{u_1} = (\tilde{\gamma}_1)^{-1}$, $\gamma^1_{u_1} = \gamma_1$, $\gamma^0_{u_2} = \gamma_2^{-1}$, $\gamma^1_{u_2} = \tilde{\gamma}_2$, $\gamma^0_{u_3} = (\tilde{\gamma}_3)^{-1}$, $\gamma^1_{u_3} = \gamma_3$. It follows from our calculations in Section \ref{Morse Differential} that:
$$
\begin{array}{lll}
P_{\gamma^0_{u_1}} = \aaabin = B, \; & P_{\gamma^0_{u_2}} = \Id, \; & P_{\gamma^0_{u_3}} = \aaaaabin = A^2B\ \\
P_{\gamma^1_{u_1}} = \Id, \; & P_{\gamma^1_{u_2}} = A^4B,\; & P_{\gamma^1_{u_3}} = \Id.
\end{array}
$$
Thus for every $\alpha \in \End(W_m)$ we have:
\begin{equation}\label{del 1 from index 3 to index 2}
\del_1 (\alpha) = (\alpha \b, \aaaab \alpha, \alpha \aab) \quad \in \quad \End(W_{x_1})\oplus\End(W_{x_2})\oplus\End(W_{x_3}).
\end{equation}

Determining $\del_1 \colon \End(W_{x_1})\oplus\End(W_{x_2})\oplus\End(W_{x_3}) \to \End(W_{x'_1})\oplus\End(W_{x'_2})\oplus\End(W_{x'_3})$ requires a bit more work. In the case of trivial local systems \cite{evans2014floer} are able to deduce that this part of the differential must be zero by a simple algebraic argument using the fact that the whole pearly differential has to square to zero. We cannot appeal to such an argument in our case (indeed, for specific choices of $W$ this part of the differential is non-zero, see Section \ref{proof of Theorem} below) and so we must analyse all relevant pearly trajectories.
That is, we are looking at pearly trajectories consisting of one Maslov 2 disc $u$, satisfying $u(-1) \in W^d(x'_i)$ and $u(1) \in W^a(x_j)$, as depicted in Figure \ref{just one pearly trajectory from index 1 to index 2}. To find all such trajectories we consider again Figure \ref{the one with all the circles} and argue in terms of triangles inscribed in the unit sphere in $\RR^3$.

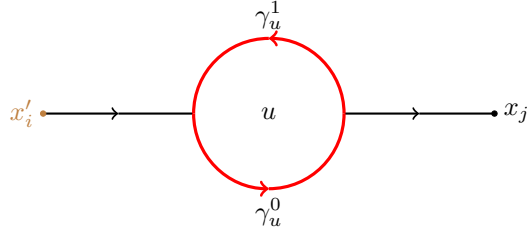
\begin{figure}[h!]
\begin{center}
\begin{tikzpicture}
\def\centrearc[#1](#2)(#3:#4:#5)[#6]\{#7\}
{
\draw[#1] ($(#2)+({#5*cos(#3)},{#5*sin(#3)})$) arc (#3:#4:#5);
\node[#6] at ($(#2)+({#5*cos(#4)},{#5*sin(#4)})$) {#7};
}

\def\howred{0} 

\draw [thick, ->] (-3, 0) -- (-2, 0);
\draw [thick] (-2, 0) -- (-1, 0);
\draw [thick, ->] (1, 0) -- (2, 0);
\draw [thick] (2, 0) -- (3, 0);
\filldraw [brown] (-3, 0) circle (1pt) node [left, brown] at (-3, 0) {$x'_i$};
\filldraw (3, 0) circle (1pt) node [right] at (3, 0) {$x_j$};

\draw [red!\howred] (-1, 0) -- (1, 0);
\filldraw [red!\howred] (0, 0) circle (1);
\centrearc[very thick, color=red, ->](0, 0)(-180:-90:1)[below]\{$\gamma^0_u$\}
\centrearc[very thick, color=red](0, 0)(-90:0:1)[]\{\}
\centrearc[very thick, red, ->](0, 0)(0:90:1)[above]\{$\gamma^1_u$\}
\centrearc[very thick, red](0, 0)(90:180:1)[]\{\}

\node at (0, 0) {$u$};
\end{tikzpicture}

\caption{A pearly trajectory $\mathbf{u} = (u)$ connecting $x'_i$ to $x_j$.}\label{just one pearly trajectory from index 1 to index 2}
\end{center}

\end{figure}

Theorem \ref{maslov 2 discs on chiang} and our choice of Morse data give us that each of the following three sets
\begin{itemize}
\item the descending manifold of $x'_i$
\item the boundary of the Maslov 2 disc $u$
\item the ascending manifold of $x_j$
\end{itemize}
consists of triangles, obtained from a single equilateral triangle by applying a rotation which keeps one of its vertices fixed. In fact we know that for the descending manifold of $x'_i$ the fixed vertex is $V'_i$ and for the ascending manifold of $x_j$, it is $V_j$. For any unit vector $p \in S^2$, let $S^1_p$ denote the circle, obtained by intersecting $S^2$ with the plane $< p > ^{\perp} -\; \frac{1}{2}p$, where the angular brackets denote linear span. Then, since $u(-1)$ lies on the descending manifold of $x'_i$, we have that one of the vertices of $u(-1)$ is $V'_i$ and the other two lie on $S^1_{V'_i}$. Similarly one of the vertices of $u(1)$ is $V_j$ and the other two lie on $S^1_{V_j}$. Let us temporarily denote the axis of $u$ by $A \in S^2$ (note then that $A$ is a vertex which all triangles in $u(\del D)$ share). Since $V'_i \neq A \neq V_j$ we see from the above that we must have $A \in S^1_{V'_i}\cap S^1_{V_j}$. That is, we must have $j \equiv i \text{ or } i+1 \pmod 3$  and $A = F_{ij}$ or $A = B_{ij}$ (see figure \ref{the one with all the circles}). Let us denote by $u^{F_{ij}}$ and $u^{B_{ij}}$ the Maslov 2 axial discs with axes $F_{ij}$ and $B_{ij}$ respectively. 

\begin{proposition}\label{euclidean proof}
For every $i, j \in \{1, 2, 3\}$ with $j \equiv i \text{ or } i+1 \pmod 3$, there are precisely two pearly trajectories $\mathbf{u}^{F_{ij}}$ and $\mathbf{u}^{B_{ij}}$ connecting $x'_i$ to $x_j$. They are given by $\mathbf{u}^{F_{ij}} = (u^{F_{ij}})$ and $\mathbf{u}^{B_{ij}} = (u^{B_{ij}})$. If $j \equiv i+2 \pmod 3$, there are no pearly trajectories connecting $x'_i$ to $x_j$.
\end{proposition}
\begin{wrapfigure}[]{R}{0.53\textwidth}
\begin{center}
\tdplotsetmaincoords{80}{90}
\begin{tikzpicture}[tdplot_main_coords, scale = 0.7]
\def\centrearc[#1](#2)(#3:#4:#5)
{ \draw[#1] ($(#2)+({#5*cos(#3)},{#5*sin(#3)})$) arc (#3:#4:#5); }
\def\centrearcendpoint(#1)(#2:#3:#4)\{#5\}
{ \coordinate (#5) at ($(#1)+({#4*cos(#3)},{#4*sin(#3)},0)$); }
\def\labelcentrearc[#1](#2)(#3:#4:#5)\{#6\}
{\centrearcendpoint(#2)(#3:#4:#5)\{tempname\};
\node[#1] at (tempname) {#6};
}

\pgfmathsetmacro\x{cos(60)}
\pgfmathsetmacro\y{sin(60)}
\pgfmathsetmacro\tzero{acos(sqrt(2)/sqrt(3))}
\pgfmathsetmacro\tone{acos(sqrt(1)/sqrt(3))}
\pgfmathsetmacro\sqrtsix{sqrt(6)}
\pgfmathsetmacro\sqrtthree{2*\y}

\def\r{4}

\coordinate (v'1) at (0, 0, \r);
\coordinate (v'2) at (0, \y*\r, -\x*\r);
\coordinate (v'3) at (0, -\y*\r, -\x*\r);
\coordinate (v1) at (0, -\y*\r, \x*\r);
\coordinate (v2) at (0, \y*\r, \x*\r);
\coordinate (v3) at (0, 0, -\r);
\coordinate (f11) at (\r*\sqrtsix/3, \r*\sqrtthree/6, -\r/2);
\coordinate (b11) at (-\r*\sqrtsix/3, \r*\sqrtthree/6, -\r/2);
\coordinate (f12) at (\r*\sqrtsix/3, -\r*\sqrtthree/6, -\r/2);
\coordinate (b12) at (-\r*\sqrtsix/3, -\r*\sqrtthree/6, -\r/2);
\coordinate (f22) at (\r*\sqrtsix/3, -\r/\sqrtthree, 0);
\coordinate (b22) at (-\r*\sqrtsix/3, -\r/\sqrtthree, 0);
\coordinate (f23) at (\r*\sqrtsix/3, -\r/2*1/\sqrtthree, 1/2);
\coordinate (b23) at (-\r*\sqrtsix/3, -\r/2*1/\sqrtthree, 1/2);
\coordinate (f33) at (\r*\sqrtsix/3, \r/2*1/\sqrtthree, 1/2);
\coordinate (b33) at (-\r*\sqrtsix/3, \r/2*1/\sqrtthree, 1/2);
\coordinate (f31) at (\r*\sqrtsix/3, \r/\sqrtthree, 0);
\coordinate (b31) at (-\r*\sqrtsix/3, \r/\sqrtthree, 0);

\node [above, black] at (v'1) {$V'_1$};
\node [right, black] at (v'2) {$V'_2$};
\node [left, black] at (v'3) {$V'_3$};
\node [left, black] at (v1) {$V_1$};

\node [below, black] at (f12) {$F_{12}(t_0)$};
\node [above right, black] at (b11) {$B_{11}(t_0)$};

\draw [thick, black] (v'1) -- (v'2);
\draw [thick, black] (v'2) -- (v'3) node [midway, above] {$E$};
\draw [thick, black] (v'3) -- (v'1);
\draw [thick, black] (v'1) -- (f12) node [midway, left] {$H$};
\draw [thick, black, dashed] (v'1) -- (b11);
\draw [thick, black, dashed] (f12) -- (b11);
\draw [thick, black] (f12) -- (v1);
\draw [thick, black] (v'1) -- (v1);

\tdplotsetrotatedcoords{0}{90}{0};
\draw[tdplot_rotated_coords][thick](0,0, 0) circle (\r);
\centrearc[thin, dashed](0, 0, -\r/2)(90:270:\y*\r);
\centrearc[thin](0, 0, -\r/2)(270:450:\y*\r);
\centrearc[thin](0, 0, -\r/2)((270:270+2*\tzero:1/3*\y*\r);
\centrearc[thin](0, 0, -\r/2)((90:90+2*\tzero:1/3*\y*\r);
\labelcentrearc[]((0, 0, -\r/2)(0:105 + \tzero: 1.8/3*\y*\r)\{$2t_0$\}
\end{tikzpicture}
\end{center}
\caption{}\label{figure for euclidean proof}
\end{wrapfigure}
\emph{Proof.} Our discussion above already proves the uniqueness and non-existence parts of the proposition. 
We only need to show that $u^{F_{ij}}$ and $u^{B_{ij}}$ do indeed give rise to pearly trajectories when $j \equiv i \text{ or } i+1 \pmod 3$. 
In other words, we need to show that both $u^{F_{ij}}(\del D)$ and $u^{B_{ij}}(\del D)$ intersect the ascending manifold of $x_j$. We prove this only for $u^{B_{11}}$ since all other proofs follow by the symmetries of Figure \ref{the one with all the circles}. 

For all $t \in (0, \pi/4]$ one of the vertices of $\exp(tV'_1)\cdot\Delta$ is $V'_1$ and the other two lie on $S^1_{V'_1}$. Let us denote these two vertices by $F_{12}(t)$ and $B_{11}(t)$, where $F_{12}(t)$ is the one with positive $x$-coordinate (the letters F and B are to be read as ``front'' and ``back''; see figure \ref{the one with all the circles}). Define $c(t) \coloneqq \cos(\angle V'_1V_1F_{12}(t))$. Let $E$ denote the midpoint of the line segment $V'_3V'_2$, i.e. the centre of $S^1_{V'_1}$. Then $\angle B_{11}(t)EV'_2 = 2t$ and so $\angle V'_2EF_{12}(t) = \pi - 2t$. By the cosine rule for $\vartriangle V'_2EF_{12}(t)$ we have:
\begin{eqnarray*}
\lvert F_{12}(t)V'_2 \rvert ^2 &=& \lvert EV'_2 \rvert ^2 + \lvert EF_{12}(t) \rvert^2 - 2 \lvert EV'_2 \rvert \lvert EF_{12}(t) \rvert \cos(\pi - 2t)\\
&=& \frac{3}{2}(1 + \cos(2t)).
\end{eqnarray*}
Then from Pythagoras's theorem for $\vartriangle V_1F_{12}(t)V'_2$ we get that $\lvert V_1F_{12}(t) \rvert^2 = 4 - 3\cos^2(t)$. Substituting this and $\lvert V'_1F_{12}(t) \rvert = \sqrt{3}$ into the equation
$$\lvert V'_1F_{12}(t) \rvert^2 = \lvert V'_1V_1 \rvert^2 + \lvert V_1F_{12}(t) \rvert^2 - 2 \lvert V'_1V_1 \rvert \lvert V_1F_{12}(t) \rvert c(t),$$
which we have from the cosine rule for $\vartriangle V'_1V_1F_{12}(t)$, we get 
$$c(t) = \frac{2 - 3 \cos^2(t)}{2\sqrt{2} \sin(t)}.$$
Set $t_0 \coloneqq \arccos\left(\sqrt{2}/\sqrt{3}\right)$ and let $S^2_{V'_1F_{12}(t_0)}$ denote the sphere in $\RR^3$ whose diameter is the line segment $V'_1F_{12}(t_0)$. Since $c(t_0) =0$ we have that 
\begin{equation} \label{lying on the correct circle}
F_{12}(t_0) \in S^2_{V'_1F_{12}(t_0)} \cap S^2 = S^1_{B_{11}(t_0)}.
\end{equation}
Put $t_1 \coloneqq \frac{\pi}{2} - t_0$. Note that $\cos(\angle V_1V'_1F_{12}(t_0)) = \lvert V_1V'_1\rvert/\lvert V'_1F_{12}(t_0) \rvert = 1/\sqrt{3} =  \cos(t_1)$. Thus if $H$ is the midpoint of $V'_1F_{12}(t_0)$, i.e. the centre of $S^1_{B_{11}(t_0)}$, then $\angle F_{12}(t_0)HV_1 = 2t_1$. From this and (\ref{lying on the correct circle}) we deduce that a right-hand rotation by $2t_1$ in the axis $B_{11}(t_0)$ sends the point $F_{12}(t_0)$ to $V_1$. In other words, acting by $\exp(t_1B_{11}(t_0))$ on the triangle $\vartriangle F_{12}(t_0)V'_1B_{11}(t_0)  = \exp(t_0V'_1)\cdot\Delta$ gives a triangle, one of whose vertices is $V_1$ and hence the other two (among which is $B_{11}(t_0)$) lie on $S^1_{V_1}$. This shows first that $B_{11}(t_0)=B_{11}$ and (by symmetry) $F_{12}(t_0) = F_{12}$  and second, that the axial Maslov 2 disc with axis $B_{11}$ 
intersects the descending manifold of $x'_1$ in $y'_1 \coloneqq q(\exp(t_0V'_1))$ and the ascending manifold of $x_1$ at $y_1 \coloneqq q(\exp(t_1B_{11})\cdot\exp(t_0V'_1))$.
\qed
\\

Observe that the above proof gives us ways of parametrising the paths $\gamma^{\ast}_{u^{\Box}}$ for $* = 0\;\text{or}\; 1$ and $\Box = F_{ij} \; \text{or} \; B_{ij}$. For example for $\Box = F_{12} \; \text{or} \; B_{11}$ we get
\begin{eqnarray*}
\gamma^0_{u^{\Box}} & = & q\left(\exp\left(t\,\Box\right)\cdot\exp\left(t_0V'_1\right)\right), \quad t \in [0, t_1] \\
\gamma^1_{u^{\Box}} & = & q\left(\exp\left(t\,\Box\right)\cdot\exp\left(t_0V'_1\right)\right), \quad t \in [t_1, \pi/2]
\end{eqnarray*}
and for $\Box = F_{11} \; \text{or} \; B_{12}$ we get
\begin{eqnarray*}
\gamma^0_{u^{\Box}} & = & q\left(\exp\left(t\,\Box\right)\cdot\exp\left(t_1V'_1\right)\right), \quad t \in [0, t_0] \\
\gamma^1_{u^{\Box}} & = & q\left(\exp\left(t\,\Box\right)\cdot\exp\left(t_1V'_1\right)\right), \quad t \in [t_0, \pi/2].
\end{eqnarray*}
Using these and the parametrisations for the index 1 flowlines of $f$, described in Section \ref{morse function}, one can plot lifts of the paths $\gamma^*_{\mathbf{u}}$, associated with the pearly trajectories $\mathbf{u}=(u^{F_{ij}})$ or $\mathbf{u}=(u^{B_{ij}})$ (see Figures \ref{actual plot b11aller}, \ref{actual plot b11retour}; the Mathematica programme used for these plots can be found on the author's home page). From these plots the parallel transport maps are immediate to read off.
Applying this procedure to all 12 pearly trajectories $\lbrace \mathbf{u}^{F_{ij}},\, \mathbf{u}^{B_{ij}} \st i \equiv j, j+1 \pmod 3 \rbrace$ we obtain the results summarised in Table \ref{table}. We have thus computed that for every $(\alpha_1, \alpha_2, \alpha_3) \in \End(W_{x_1})\oplus\End(W_{x_2})\oplus\End(W_{x_3})$ we have
\begin{eqnarray}\label{del 1 from index 2 to index 1}
\del_1(\alpha_1, \alpha_2, \alpha_3) & = & (\aab\alpha_1 + \alpha_1\aaaab + \aab\alpha_2 + \aaaaa\alpha_2\aaaaab\,, \nonumber \\
&&\aab\alpha_2 + \alpha_2\b + \aaab\alpha_3\a + \alpha_3\b\,, \nonumber \\
&& \aaaab\alpha_3 + \alpha_3\b + \a\alpha_1\aaab + \aaaaab\alpha_1\aaaaa) \nonumber \\
&\in & \End(W_{x'_1})\oplus\End(W_{x'_2})\oplus\End(W_{x'_3}).
\end{eqnarray}
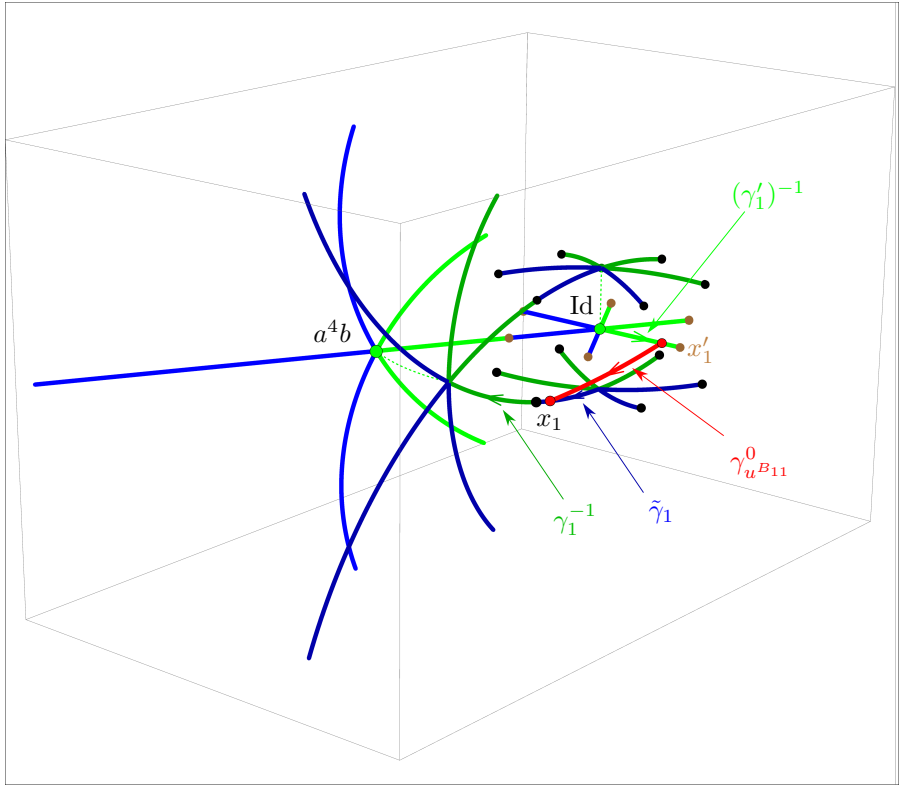
\begin{figure}[h!]
\labellist
\pinlabel $\Id$ at 974 812
\pinlabel $a^4b$ at 552 768
\pinlabel {\textcolor{brown}{$x'_1$}} at 1175 736
\pinlabel $x_1$ at 920 613
\pinlabel {\textcolor{red}{$\gamma^0_{u^{B_{11}}}$}} at 1277 545
\pinlabel {\textcolor{green!75!black}{$\gamma_1^{-1}$}} at 963 449
\pinlabel {\textcolor{green}{$(\gamma'_1)^{-1}$}} at 1289 998
\pinlabel {\textcolor{blue}{$\tilde{\gamma}_1$}} at 1109 462
\endlabellist
\centering
\includegraphics[height= 0.5\textheight]{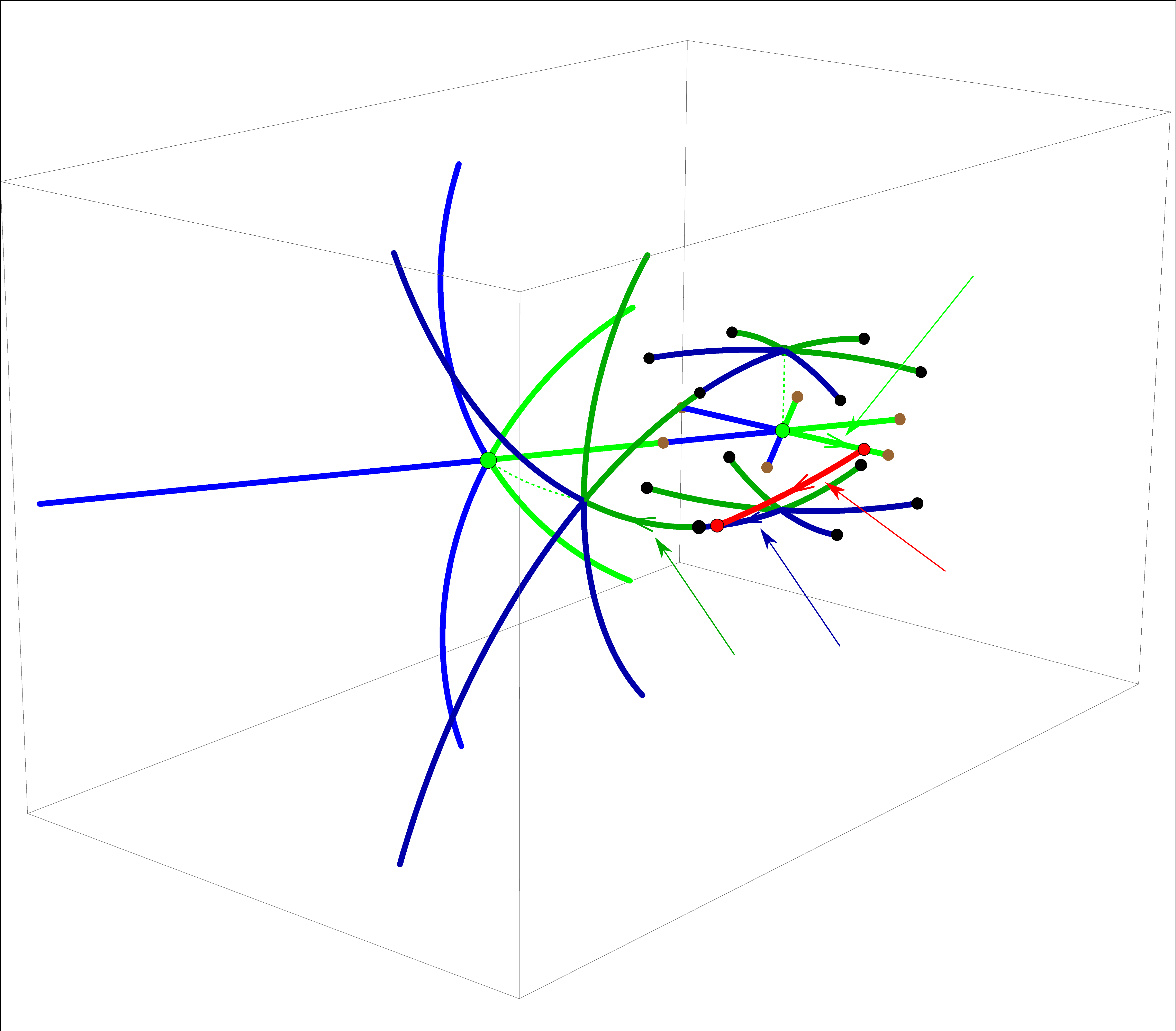}
\begin{center}
\end{center}
\caption{A lift of the path $\gamma^0_{u^{B_{11}}}$	 together with the descending manifolds of the index 1 critical points and the ascending manifolds for index 2 critical points for the fundamental domains centred  at $\Id$ and $a^4b$. Lifts of the index 3 flowline $\sigma$ are represented by the dashed line segments. From this plot one reads off that $P_{\gamma^0_{\mathbf{u}^{B_{11}}}}  = A^4B$.}\label{actual plot b11aller}
\end{figure}
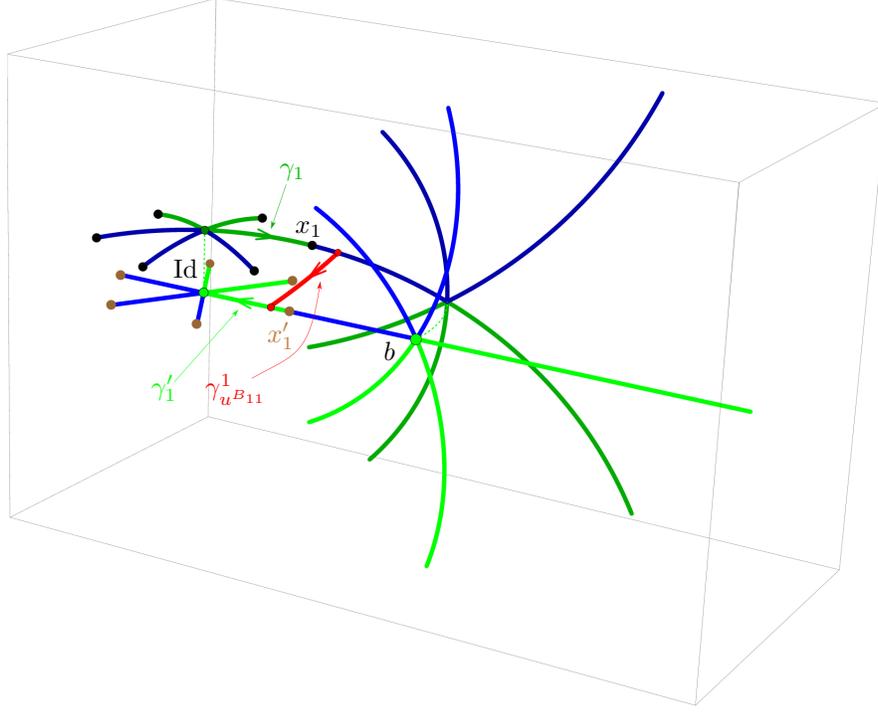
\begin{figure}[h!]
\labellist
\pinlabel $\Id$ at 305 815
\pinlabel $x_1$ at 516 883
\pinlabel {\textcolor{brown}{$x'_1$}} at 468 700
\pinlabel {\textcolor{red}{$\gamma^1_{u^{B_{11}}}$}} at 392 608
\pinlabel $b$ at 654 674
\pinlabel {\textcolor{green!75!black}{$\gamma_1$}} at 488 980
\pinlabel {\textcolor{green}{$\gamma'_1$}} at 268 610
\endlabellist
\begin{center}
\includegraphics[height= 0.5\textheight]{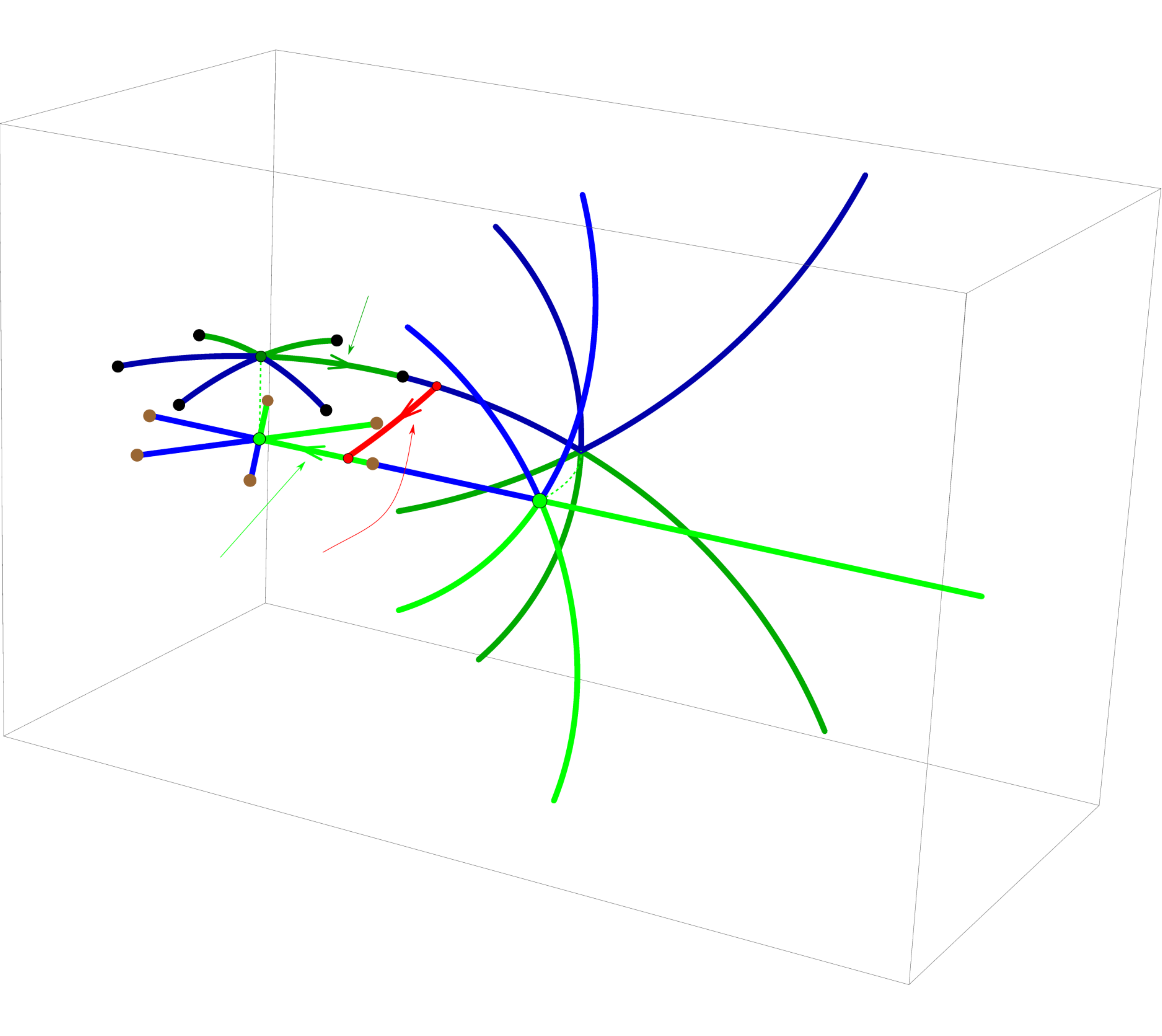}
\end{center}
\caption{A lift of the path $\gamma^1_{u^{B_{11}}}$ with the fundamental domains centred at $\Id$ and $b$.  From this plot one reads off that $P_{\gamma^1_{\mathbf{u}^{B_{11}}}}  = \Id$.}\label{actual plot b11retour}
\end{figure}

\begin{center}
\begin{table}[h!]
\setlength\tabcolsep{2pt}
\begin{tabular}[]{cccc}\label{table for parallel transport maps}
\begin{tikzpicture}[scale=0.4]
\pgfmathsetmacro\tzero{acos(sqrt(2)/sqrt(3)};
\pgfmathsetmacro\tone {acos(1/sqrt(3))};
\def\tuse{\tzero}
\def\radius{1}

\pgfmathsetmacro\ycoord{cos(2*\tuse)*\radius};
\pgfmathsetmacro\xcoord{sin(2*\tuse)*\radius};

\def\centrearc[#1](#2)(#3:#4:#5)[#6]\{#7\}
{
\draw[#1] ($(#2)+({#5*cos(#3)},{#5*sin(#3)})$) arc (#3:#4:#5);
\node[#6] at ($(#2)+({#5*cos(#4)},{#5*sin(#4)})$) {#7};
}

\draw [thick, blue, ->] (-3, 0) -- (-2*\xcoord, 0);
\draw [thick, blue] (-2*\xcoord, 0) -- (-\xcoord, 0);
\draw [thick, green, ->] (\xcoord, 0) -- (2*\xcoord, 0);
\draw [thick, green] (2*\xcoord, 0) -- (3, 0);
\filldraw [brown] (-3, 0) circle (1pt) node [left] at (-3, 0) {$x'_1$};
\filldraw [black] (3, 0) circle (1pt) node [right] at (3, 0) {$x_1$};

\def\howred{0}

\draw [red!\howred] (-\xcoord, 0) -- (\xcoord, 0);
\centrearc[very thick, color=red, fill=red!\howred](0, \ycoord)(-90-2*\tuse:-90+2*\tuse:\radius)[]\{\}
\centrearc[very thick, color=red, fill=red!\howred](0, \ycoord)(-90+2*\tuse:270-2*\tuse:\radius)[]\{\}
\node at (0, \ycoord) {$u^{F_{11}}$};

\def\rescaledradius{1.5*\radius}
\def\downshift{-0.5}
\pgfmathsetmacro\tuseshifted{1/2*acos((cos(2*\tuse) - \downshift)/\rescaledradius)}
\pgfmathsetmacro\xcoordshifted{sin(2*\tuseshifted)*\rescaledradius}
\draw [	] (-3, \downshift) -- (-\xcoordshifted, \downshift);
\centrearc[->](0, \ycoord)(-90-2*\tuseshifted:-90:\rescaledradius)[below]\{$\Id$\}
\centrearc[](0, \ycoord)(-90:-90+2*\tuseshifted:\rescaledradius)[]\{\}
\draw [	] (\xcoordshifted, \downshift) -- (3, \downshift);

\def\upshift{0.5}
\pgfmathsetmacro\tuseshifted{1/2*acos((cos(2*\tuse) - \upshift)/\rescaledradius)}
\pgfmathsetmacro\xcoordshifted{sin(2*\tuseshifted)*\rescaledradius}
\draw [	] (-3, \upshift) -- (-\xcoordshifted, \upshift);
\centrearc[->](0, \ycoord)(-90+2*\tuseshifted:90:\rescaledradius)[above]\{$A^2B$\}
\centrearc[](0, \ycoord)(90:270 - 2*\tuseshifted:\rescaledradius)[]\{\}
\draw [	] (\xcoordshifted, \upshift) -- (3, \upshift);
\end{tikzpicture}
&
\begin{tikzpicture}[scale=0.4]
\pgfmathsetmacro\tzero{acos(sqrt(2)/sqrt(3)};
\pgfmathsetmacro\tone {acos(1/sqrt(3))};
\def\tuse{\tone}
\def\radius{1}

\pgfmathsetmacro\ycoord{cos(2*\tuse)*\radius};
\pgfmathsetmacro\xcoord{sin(2*\tuse)*\radius};

\def\centrearc[#1](#2)(#3:#4:#5)[#6]\{#7\}
{
\draw[#1] ($(#2)+({#5*cos(#3)},{#5*sin(#3)})$) arc (#3:#4:#5);
\node[#6] at ($(#2)+({#5*cos(#4)},{#5*sin(#4)})$) {#7};
}

\draw [thick, green, ->] (-3, 0) -- (-2*\xcoord, 0);
\draw [thick, green] (-2*\xcoord, 0) -- (-\xcoord, 0);
\draw [thick, blue, ->] (\xcoord, 0) -- (2*\xcoord, 0);
\draw [thick, blue] (2*\xcoord, 0) -- (3, 0);
\filldraw [brown] (-3, 0) circle (1pt) node [left] at (-3, 0) {$x'_1$};
\filldraw [black] (3, 0) circle (1pt) node [right] at (3, 0) {$x_1$};

\def\howred{0}

\draw [red!\howred] (-\xcoord, 0) -- (\xcoord, 0);
\centrearc[very thick, color=red, fill=red!\howred](0, \ycoord)(-90-2*\tuse:-90+2*\tuse:\radius)[]\{\}
\centrearc[very thick, color=red, fill=red!\howred](0, \ycoord)(-90+2*\tuse:270-2*\tuse:\radius)[]\{\}
\node at (0, \ycoord) {$u^{B_{11}}$};

\def\rescaledradius{1.5*\radius}
\def\downshift{-0.5}
\pgfmathsetmacro\tuseshifted{1/2*acos((cos(2*\tuse) - \downshift)/\rescaledradius)}
\pgfmathsetmacro\xcoordshifted{sin(2*\tuseshifted)*\rescaledradius}
\draw [	] (-3, \downshift) -- (-\xcoordshifted, \downshift);
\centrearc[->](0, \ycoord)(-90-2*\tuseshifted:-90:\rescaledradius)[below]\{$A^4B$\}
\centrearc[](0, \ycoord)(-90:-90+2*\tuseshifted:\rescaledradius)[]\{\}
\draw [	] (\xcoordshifted, \downshift) -- (3, \downshift);

\def\upshift{0.5}
\pgfmathsetmacro\tuseshifted{1/2*acos((cos(2*\tuse) - \upshift)/\rescaledradius)}
\pgfmathsetmacro\xcoordshifted{sin(2*\tuseshifted)*\rescaledradius}
\draw [	] (-3, \upshift) -- (-\xcoordshifted, \upshift);
\centrearc[->](0, \ycoord)(-90+2*\tuseshifted:90:\rescaledradius)[above]\{$\Id$\}
\centrearc[](0, \ycoord)(90:270 - 2*\tuseshifted:\rescaledradius)[]\{\}
\draw [	] (\xcoordshifted, \upshift) -- (3, \upshift);
\end{tikzpicture}
&
\begin{tikzpicture}[scale=0.4]
\pgfmathsetmacro\tzero{acos(sqrt(2)/sqrt(3)};
\pgfmathsetmacro\tone {acos(1/sqrt(3))};
\def\tuse{\tone}
\def\radius{1}

\pgfmathsetmacro\ycoord{cos(2*\tuse)*\radius};
\pgfmathsetmacro\xcoord{sin(2*\tuse)*\radius};

\def\centrearc[#1](#2)(#3:#4:#5)[#6]\{#7\}
{
\draw[#1] ($(#2)+({#5*cos(#3)},{#5*sin(#3)})$) arc (#3:#4:#5);
\node[#6] at ($(#2)+({#5*cos(#4)},{#5*sin(#4)})$) {#7};
}

\draw [thick, green, ->] (-3, 0) -- (-2*\xcoord, 0);
\draw [thick, green] (-2*\xcoord, 0) -- (-\xcoord, 0);
\draw [thick, green, ->] (\xcoord, 0) -- (2*\xcoord, 0);
\draw [thick, green] (2*\xcoord, 0) -- (3, 0);
\filldraw [brown] (-3, 0) circle (1pt) node [left] at (-3, 0) {$x'_1$};
\filldraw [black] (3, 0) circle (1pt) node [right] at (3, 0) {$x_2$};

\def\howred{0}

\draw [red!\howred] (-\xcoord, 0) -- (\xcoord, 0);
\centrearc[very thick, color=red, fill=red!\howred](0, \ycoord)(-90-2*\tuse:-90+2*\tuse:\radius)[]\{\}
\centrearc[very thick, color=red, fill=red!\howred](0, \ycoord)(-90+2*\tuse:270-2*\tuse:\radius)[]\{\}
\node at (0, \ycoord) {$u^{F_{12}}$};

\def\rescaledradius{1.5*\radius}
\def\downshift{-0.5}
\pgfmathsetmacro\tuseshifted{1/2*acos((cos(2*\tuse) - \downshift)/\rescaledradius)}
\pgfmathsetmacro\xcoordshifted{sin(2*\tuseshifted)*\rescaledradius}
\draw [	] (-3, \downshift) -- (-\xcoordshifted, \downshift);
\centrearc[->](0, \ycoord)(-90-2*\tuseshifted:-90:\rescaledradius)[below]\{$\Id$\}
\centrearc[](0, \ycoord)(-90:-90+2*\tuseshifted:\rescaledradius)[]\{\}
\draw [	] (\xcoordshifted, \downshift) -- (3, \downshift);

\def\upshift{0.5}
\pgfmathsetmacro\tuseshifted{1/2*acos((cos(2*\tuse) - \upshift)/\rescaledradius)}
\pgfmathsetmacro\xcoordshifted{sin(2*\tuseshifted)*\rescaledradius}
\draw [	] (-3, \upshift) -- (-\xcoordshifted, \upshift);
\centrearc[->](0, \ycoord)(-90+2*\tuseshifted:90:\rescaledradius)[above]\{$A^2B$\}
\centrearc[](0, \ycoord)(90:270 - 2*\tuseshifted:\rescaledradius)[]\{\}
\draw [	] (\xcoordshifted, \upshift) -- (3, \upshift);
\end{tikzpicture}
&
\begin{tikzpicture}[scale=0.4]
\pgfmathsetmacro\tzero{acos(sqrt(2)/sqrt(3)};
\pgfmathsetmacro\tone {acos(1/sqrt(3))};
\def\tuse{\tzero}
\def\radius{1}

\pgfmathsetmacro\ycoord{cos(2*\tuse)*\radius};
\pgfmathsetmacro\xcoord{sin(2*\tuse)*\radius};

\def\centrearc[#1](#2)(#3:#4:#5)[#6]\{#7\}
{
\draw[#1] ($(#2)+({#5*cos(#3)},{#5*sin(#3)})$) arc (#3:#4:#5);
\node[#6] at ($(#2)+({#5*cos(#4)},{#5*sin(#4)})$) {#7};
}

\draw [thick, blue, ->] (-3, 0) -- (-2*\xcoord, 0);
\draw [thick, blue] (-2*\xcoord, 0) -- (-\xcoord, 0);
\draw [thick, blue, ->] (\xcoord, 0) -- (2*\xcoord, 0);
\draw [thick, blue] (2*\xcoord, 0) -- (3, 0);
\filldraw [brown] (-3, 0) circle (1pt) node [left] at (-3, 0) {$x'_1$};
\filldraw [black] (3, 0) circle (1pt) node [right] at (3, 0) {$x_2$};

\def\howred{0}

\draw [red!\howred] (-\xcoord, 0) -- (\xcoord, 0);
\centrearc[very thick, color=red, fill=red!\howred](0, \ycoord)(-90-2*\tuse:-90+2*\tuse:\radius)[]\{\}
\centrearc[very thick, color=red, fill=red!\howred](0, \ycoord)(-90+2*\tuse:270-2*\tuse:\radius)[]\{\}
\node at (0, \ycoord) {$u^{B_{12}}$};

\def\rescaledradius{1.5*\radius}
\def\downshift{-0.5}
\pgfmathsetmacro\tuseshifted{1/2*acos((cos(2*\tuse) - \downshift)/\rescaledradius)}
\pgfmathsetmacro\xcoordshifted{sin(2*\tuseshifted)*\rescaledradius}
\draw [	] (-3, \downshift) -- (-\xcoordshifted, \downshift);
\centrearc[->](0, \ycoord)(-90-2*\tuseshifted:-90:\rescaledradius)[below]\{$A^5B$\}
\centrearc[](0, \ycoord)(-90:-90+2*\tuseshifted:\rescaledradius)[]\{\}
\draw [	] (\xcoordshifted, \downshift) -- (3, \downshift);

\def\upshift{0.5}
\pgfmathsetmacro\tuseshifted{1/2*acos((cos(2*\tuse) - \upshift)/\rescaledradius)}
\pgfmathsetmacro\xcoordshifted{sin(2*\tuseshifted)*\rescaledradius}
\draw [	] (-3, \upshift) -- (-\xcoordshifted, \upshift);
\centrearc[->](0, \ycoord)(-90+2*\tuseshifted:90:\rescaledradius)[above]\{$A^5$\}
\centrearc[](0, \ycoord)(90:270 - 2*\tuseshifted:\rescaledradius)[]\{\}
\draw [	] (\xcoordshifted, \upshift) -- (3, \upshift);

\end{tikzpicture}
\\
\begin{tikzpicture}[scale=0.4]
\pgfmathsetmacro\tzero{acos(sqrt(2)/sqrt(3)};
\pgfmathsetmacro\tone {acos(1/sqrt(3))};
\def\tuse{\tzero}
\def\radius{1}

\pgfmathsetmacro\ycoord{cos(2*\tuse)*\radius};
\pgfmathsetmacro\xcoord{sin(2*\tuse)*\radius};

\def\centrearc[#1](#2)(#3:#4:#5)[#6]\{#7\}
{
\draw[#1] ($(#2)+({#5*cos(#3)},{#5*sin(#3)})$) arc (#3:#4:#5);
\node[#6] at ($(#2)+({#5*cos(#4)},{#5*sin(#4)})$) {#7};
}

\draw [thick, green, ->] (-3, 0) -- (-2*\xcoord, 0);
\draw [thick, green] (-2*\xcoord, 0) -- (-\xcoord, 0);
\draw [thick, blue, ->] (\xcoord, 0) -- (2*\xcoord, 0);
\draw [thick, blue] (2*\xcoord, 0) -- (3, 0);
\filldraw [brown] (-3, 0) circle (1pt) node [left] at (-3, 0) {$x'_2$};
\filldraw [black] (3, 0) circle (1pt) node [right] at (3, 0) {$x_2$};

\def\howred{0}

\draw [red!\howred] (-\xcoord, 0) -- (\xcoord, 0);
\centrearc[very thick, color=red, fill=red!\howred](0, \ycoord)(-90-2*\tuse:-90+2*\tuse:\radius)[]\{\}
\centrearc[very thick, color=red, fill=red!\howred](0, \ycoord)(-90+2*\tuse:270-2*\tuse:\radius)[]\{\}
\node at (0, \ycoord) {$u^{F_{22}}$};

\def\rescaledradius{1.5*\radius}
\def\downshift{-0.5}
\pgfmathsetmacro\tuseshifted{1/2*acos((cos(2*\tuse) - \downshift)/\rescaledradius)}
\pgfmathsetmacro\xcoordshifted{sin(2*\tuseshifted)*\rescaledradius}
\draw [	] (-3, \downshift) -- (-\xcoordshifted, \downshift);
\centrearc[->](0, \ycoord)(-90-2*\tuseshifted:-90:\rescaledradius)[below]\{$\Id$\}
\centrearc[](0, \ycoord)(-90:-90+2*\tuseshifted:\rescaledradius)[]\{\}
\draw [	] (\xcoordshifted, \downshift) -- (3, \downshift);

\def\upshift{0.5}
\pgfmathsetmacro\tuseshifted{1/2*acos((cos(2*\tuse) - \upshift)/\rescaledradius)}
\pgfmathsetmacro\xcoordshifted{sin(2*\tuseshifted)*\rescaledradius}
\draw [	] (-3, \upshift) -- (-\xcoordshifted, \upshift);
\centrearc[->](0, \ycoord)(-90+2*\tuseshifted:90:\rescaledradius)[above]\{$A^2B$\}
\centrearc[](0, \ycoord)(90:270 - 2*\tuseshifted:\rescaledradius)[]\{\}
\draw [	] (\xcoordshifted, \upshift) -- (3, \upshift);
\end{tikzpicture}
&
\begin{tikzpicture}[scale=0.4]
\pgfmathsetmacro\tzero{acos(sqrt(2)/sqrt(3)};
\pgfmathsetmacro\tone {acos(1/sqrt(3))};
\def\tuse{\tone}
\def\radius{1}

\pgfmathsetmacro\ycoord{cos(2*\tuse)*\radius};
\pgfmathsetmacro\xcoord{sin(2*\tuse)*\radius};

\def\centrearc[#1](#2)(#3:#4:#5)[#6]\{#7\}
{
\draw[#1] ($(#2)+({#5*cos(#3)},{#5*sin(#3)})$) arc (#3:#4:#5);
\node[#6] at ($(#2)+({#5*cos(#4)},{#5*sin(#4)})$) {#7};
}

\draw [thick, blue, ->] (-3, 0) -- (-2*\xcoord, 0);
\draw [thick, blue] (-2*\xcoord, 0) -- (-\xcoord, 0);
\draw [thick, green, ->] (\xcoord, 0) -- (2*\xcoord, 0);
\draw [thick, green] (2*\xcoord, 0) -- (3, 0);
\filldraw [brown] (-3, 0) circle (1pt) node [left] at (-3, 0) {$x'_2$};
\filldraw [black] (3, 0) circle (1pt) node [right] at (3, 0) {$x_2$};

\def\howred{0}

\draw [red!\howred] (-\xcoord, 0) -- (\xcoord, 0);
\centrearc[very thick, color=red, fill=red!\howred](0, \ycoord)(-90-2*\tuse:-90+2*\tuse:\radius)[]\{\}
\centrearc[very thick, color=red, fill=red!\howred](0, \ycoord)(-90+2*\tuse:270-2*\tuse:\radius)[]\{\}
\node at (0, \ycoord) {$u^{B_{22}}$};

\def\rescaledradius{1.5*\radius}
\def\downshift{-0.5}
\pgfmathsetmacro\tuseshifted{1/2*acos((cos(2*\tuse) - \downshift)/\rescaledradius)}
\pgfmathsetmacro\xcoordshifted{sin(2*\tuseshifted)*\rescaledradius}
\draw [	] (-3, \downshift) -- (-\xcoordshifted, \downshift);
\centrearc[->](0, \ycoord)(-90-2*\tuseshifted:-90:\rescaledradius)[below]\{$B$\}
\centrearc[](0, \ycoord)(-90:-90+2*\tuseshifted:\rescaledradius)[]\{\}
\draw [	] (\xcoordshifted, \downshift) -- (3, \downshift);

\def\upshift{0.5}
\pgfmathsetmacro\tuseshifted{1/2*acos((cos(2*\tuse) - \upshift)/\rescaledradius)}
\pgfmathsetmacro\xcoordshifted{sin(2*\tuseshifted)*\rescaledradius}
\draw [	] (-3, \upshift) -- (-\xcoordshifted, \upshift);
\centrearc[->](0, \ycoord)(-90+2*\tuseshifted:90:\rescaledradius)[above]\{$\Id$\}
\centrearc[](0, \ycoord)(90:270 - 2*\tuseshifted:\rescaledradius)[]\{\}
\draw [	] (\xcoordshifted, \upshift) -- (3, \upshift);
\end{tikzpicture}
&
\begin{tikzpicture}[scale=0.4]
\pgfmathsetmacro\tzero{acos(sqrt(2)/sqrt(3)};
\pgfmathsetmacro\tone {acos(1/sqrt(3))};
\def\tuse{\tone}
\def\radius{1}

\pgfmathsetmacro\ycoord{cos(2*\tuse)*\radius};
\pgfmathsetmacro\xcoord{sin(2*\tuse)*\radius};

\def\centrearc[#1](#2)(#3:#4:#5)[#6]\{#7\}
{
\draw[#1] ($(#2)+({#5*cos(#3)},{#5*sin(#3)})$) arc (#3:#4:#5);
\node[#6] at ($(#2)+({#5*cos(#4)},{#5*sin(#4)})$) {#7};
}

\draw [thick, blue, ->] (-3, 0) -- (-2*\xcoord, 0);
\draw [thick, blue] (-2*\xcoord, 0) -- (-\xcoord, 0);
\draw [thick, blue, ->] (\xcoord, 0) -- (2*\xcoord, 0);
\draw [thick, blue] (2*\xcoord, 0) -- (3, 0);
\filldraw [brown] (-3, 0) circle (1pt) node [left] at (-3, 0) {$x'_2$};
\filldraw [black] (3, 0) circle (1pt) node [right] at (3, 0) {$x_3$};

\def\howred{0}

\draw [red!\howred] (-\xcoord, 0) -- (\xcoord, 0);
\centrearc[very thick, color=red, fill=red!\howred](0, \ycoord)(-90-2*\tuse:-90+2*\tuse:\radius)[]\{\}
\centrearc[very thick, color=red, fill=red!\howred](0, \ycoord)(-90+2*\tuse:270-2*\tuse:\radius)[]\{\}
\node at (0, \ycoord) {$u^{F_{23}}$};

\def\rescaledradius{1.5*\radius}
\def\downshift{-0.5}
\pgfmathsetmacro\tuseshifted{1/2*acos((cos(2*\tuse) - \downshift)/\rescaledradius)}
\pgfmathsetmacro\xcoordshifted{sin(2*\tuseshifted)*\rescaledradius}
\draw [	] (-3, \downshift) -- (-\xcoordshifted, \downshift);
\centrearc[->](0, \ycoord)(-90-2*\tuseshifted:-90:\rescaledradius)[below]\{$A$\}
\centrearc[](0, \ycoord)(-90:-90+2*\tuseshifted:\rescaledradius)[]\{\}
\draw [	] (\xcoordshifted, \downshift) -- (3, \downshift);

\def\upshift{0.5}
\pgfmathsetmacro\tuseshifted{1/2*acos((cos(2*\tuse) - \upshift)/\rescaledradius)}
\pgfmathsetmacro\xcoordshifted{sin(2*\tuseshifted)*\rescaledradius}
\draw [	] (-3, \upshift) -- (-\xcoordshifted, \upshift);
\centrearc[->](0, \ycoord)(-90+2*\tuseshifted:90:\rescaledradius)[above]\{$A^3B$\}
\centrearc[](0, \ycoord)(90:270 - 2*\tuseshifted:\rescaledradius)[]\{\}
\draw [	] (\xcoordshifted, \upshift) -- (3, \upshift);
\end{tikzpicture}
&
\begin{tikzpicture}[scale=0.4]
\pgfmathsetmacro\tzero{acos(sqrt(2)/sqrt(3)};
\pgfmathsetmacro\tone {acos(1/sqrt(3))};
\def\tuse{\tzero}
\def\radius{1}

\pgfmathsetmacro\ycoord{cos(2*\tuse)*\radius};
\pgfmathsetmacro\xcoord{sin(2*\tuse)*\radius};

\def\centrearc[#1](#2)(#3:#4:#5)[#6]\{#7\}
{
\draw[#1] ($(#2)+({#5*cos(#3)},{#5*sin(#3)})$) arc (#3:#4:#5);
\node[#6] at ($(#2)+({#5*cos(#4)},{#5*sin(#4)})$) {#7};
}

\draw [thick, green, ->] (-3, 0) -- (-2*\xcoord, 0);
\draw [thick, green] (-2*\xcoord, 0) -- (-\xcoord, 0);
\draw [thick, green, ->] (\xcoord, 0) -- (2*\xcoord, 0);
\draw [thick, green] (2*\xcoord, 0) -- (3, 0);
\filldraw [brown] (-3, 0) circle (1pt) node [left] at (-3, 0) {$x'_2$};
\filldraw [black] (3, 0) circle (1pt) node [right] at (3, 0) {$x_3$};

\def\howred{0}

\draw [red!\howred] (-\xcoord, 0) -- (\xcoord, 0);
\centrearc[very thick, color=red, fill=red!\howred](0, \ycoord)(-90-2*\tuse:-90+2*\tuse:\radius)[]\{\}
\centrearc[very thick, color=red, fill=red!\howred](0, \ycoord)(-90+2*\tuse:270-2*\tuse:\radius)[]\{\}
\node at (0, \ycoord) {$u^{B_{23}}$};

\def\rescaledradius{1.5*\radius}
\def\downshift{-0.5}
\pgfmathsetmacro\tuseshifted{1/2*acos((cos(2*\tuse) - \downshift)/\rescaledradius)}
\pgfmathsetmacro\xcoordshifted{sin(2*\tuseshifted)*\rescaledradius}
\draw [	] (-3, \downshift) -- (-\xcoordshifted, \downshift);
\centrearc[->](0, \ycoord)(-90-2*\tuseshifted:-90:\rescaledradius)[below]\{$B$\}
\centrearc[](0, \ycoord)(-90:-90+2*\tuseshifted:\rescaledradius)[]\{\}
\draw [	] (\xcoordshifted, \downshift) -- (3, \downshift);

\def\upshift{0.5}
\pgfmathsetmacro\tuseshifted{1/2*acos((cos(2*\tuse) - \upshift)/\rescaledradius)}
\pgfmathsetmacro\xcoordshifted{sin(2*\tuseshifted)*\rescaledradius}
\draw [	] (-3, \upshift) -- (-\xcoordshifted, \upshift);
\centrearc[->](0, \ycoord)(-90+2*\tuseshifted:90:\rescaledradius)[above]\{$\Id$\}
\centrearc[](0, \ycoord)(90:270 - 2*\tuseshifted:\rescaledradius)[]\{\}
\draw [	] (\xcoordshifted, \upshift) -- (3, \upshift);
\end{tikzpicture}
\\
\begin{tikzpicture}[scale=0.4]
\pgfmathsetmacro\tzero{acos(sqrt(2)/sqrt(3)};
\pgfmathsetmacro\tone {acos(1/sqrt(3))};
\def\tuse{\tzero}
\def\radius{1}

\pgfmathsetmacro\ycoord{cos(2*\tuse)*\radius};
\pgfmathsetmacro\xcoord{sin(2*\tuse)*\radius};

\def\centrearc[#1](#2)(#3:#4:#5)[#6]\{#7\}
{
\draw[#1] ($(#2)+({#5*cos(#3)},{#5*sin(#3)})$) arc (#3:#4:#5);
\node[#6] at ($(#2)+({#5*cos(#4)},{#5*sin(#4)})$) {#7};
}

\draw [thick, blue, ->] (-3, 0) -- (-2*\xcoord, 0);
\draw [thick, blue] (-2*\xcoord, 0) -- (-\xcoord, 0);
\draw [thick, green, ->] (\xcoord, 0) -- (2*\xcoord, 0);
\draw [thick, green] (2*\xcoord, 0) -- (3, 0);
\filldraw [brown] (-3, 0) circle (1pt) node [left] at (-3, 0) {$x'_3$};
\filldraw [black] (3, 0) circle (1pt) node [right] at (3, 0) {$x_3$};

\def\howred{0}

\draw [red!\howred] (-\xcoord, 0) -- (\xcoord, 0);
\centrearc[very thick, color=red, fill=red!\howred](0, \ycoord)(-90-2*\tuse:-90+2*\tuse:\radius)[]\{\}
\centrearc[very thick, color=red, fill=red!\howred](0, \ycoord)(-90+2*\tuse:270-2*\tuse:\radius)[]\{\}
\node at (0, \ycoord) {$u^{F_{33}}$};

\def\rescaledradius{1.5*\radius}
\def\downshift{-0.5}
\pgfmathsetmacro\tuseshifted{1/2*acos((cos(2*\tuse) - \downshift)/\rescaledradius)}
\pgfmathsetmacro\xcoordshifted{sin(2*\tuseshifted)*\rescaledradius}
\draw [	] (-3, \downshift) -- (-\xcoordshifted, \downshift);
\centrearc[->](0, \ycoord)(-90-2*\tuseshifted:-90:\rescaledradius)[below]\{$\Id$\}
\centrearc[](0, \ycoord)(-90:-90+2*\tuseshifted:\rescaledradius)[]\{\}
\draw [	] (\xcoordshifted, \downshift) -- (3, \downshift);

\def\upshift{0.5}
\pgfmathsetmacro\tuseshifted{1/2*acos((cos(2*\tuse) - \upshift)/\rescaledradius)}
\pgfmathsetmacro\xcoordshifted{sin(2*\tuseshifted)*\rescaledradius}
\draw [	] (-3, \upshift) -- (-\xcoordshifted, \upshift);
\centrearc[->](0, \ycoord)(-90+2*\tuseshifted:90:\rescaledradius)[above]\{$A^4B$\}
\centrearc[](0, \ycoord)(90:270 - 2*\tuseshifted:\rescaledradius)[]\{\}
\draw [	] (\xcoordshifted, \upshift) -- (3, \upshift);
\end{tikzpicture}
&
\begin{tikzpicture}[scale=0.4]
\pgfmathsetmacro\tzero{acos(sqrt(2)/sqrt(3)};
\pgfmathsetmacro\tone {acos(1/sqrt(3))};
\def\tuse{\tone}
\def\radius{1}

\pgfmathsetmacro\ycoord{cos(2*\tuse)*\radius};
\pgfmathsetmacro\xcoord{sin(2*\tuse)*\radius};

\def\centrearc[#1](#2)(#3:#4:#5)[#6]\{#7\}
{
\draw[#1] ($(#2)+({#5*cos(#3)},{#5*sin(#3)})$) arc (#3:#4:#5);
\node[#6] at ($(#2)+({#5*cos(#4)},{#5*sin(#4)})$) {#7};
}

\draw [thick, green, ->] (-3, 0) -- (-2*\xcoord, 0);
\draw [thick, green] (-2*\xcoord, 0) -- (-\xcoord, 0);
\draw [thick, blue, ->] (\xcoord, 0) -- (2*\xcoord, 0);
\draw [thick, blue] (2*\xcoord, 0) -- (3, 0);
\filldraw [brown] (-3, 0) circle (1pt) node [left] at (-3, 0) {$x'_3$};
\filldraw [black] (3, 0) circle (1pt) node [right] at (3, 0) {$x_3$};

\def\howred{0}

\draw [red!\howred] (-\xcoord, 0) -- (\xcoord, 0);
\centrearc[very thick, color=red, fill=red!\howred](0, \ycoord)(-90-2*\tuse:-90+2*\tuse:\radius)[]\{\}
\centrearc[very thick, color=red, fill=red!\howred](0, \ycoord)(-90+2*\tuse:270-2*\tuse:\radius)[]\{\}
\node at (0, \ycoord) {$u^{B_{33}}$};

\def\rescaledradius{1.5*\radius}
\def\downshift{-0.5}
\pgfmathsetmacro\tuseshifted{1/2*acos((cos(2*\tuse) - \downshift)/\rescaledradius)}
\pgfmathsetmacro\xcoordshifted{sin(2*\tuseshifted)*\rescaledradius}
\draw [	] (-3, \downshift) -- (-\xcoordshifted, \downshift);
\centrearc[->](0, \ycoord)(-90-2*\tuseshifted:-90:\rescaledradius)[below]\{$B$\}
\centrearc[](0, \ycoord)(-90:-90+2*\tuseshifted:\rescaledradius)[]\{\}
\draw [	] (\xcoordshifted, \downshift) -- (3, \downshift);

\def\upshift{0.5}
\pgfmathsetmacro\tuseshifted{1/2*acos((cos(2*\tuse) - \upshift)/\rescaledradius)}
\pgfmathsetmacro\xcoordshifted{sin(2*\tuseshifted)*\rescaledradius}
\draw [	] (-3, \upshift) -- (-\xcoordshifted, \upshift);
\centrearc[->](0, \ycoord)(-90+2*\tuseshifted:90:\rescaledradius)[above]\{$\Id$\}
\centrearc[](0, \ycoord)(90:270 - 2*\tuseshifted:\rescaledradius)[]\{\}
\draw [	] (\xcoordshifted, \upshift) -- (3, \upshift);
\end{tikzpicture}
&
\begin{tikzpicture}[scale=0.4]
\pgfmathsetmacro\tzero{acos(sqrt(2)/sqrt(3)};
\pgfmathsetmacro\tone {acos(1/sqrt(3))};
\def\tuse{\tone}
\def\radius{1}

\pgfmathsetmacro\ycoord{cos(2*\tuse)*\radius};
\pgfmathsetmacro\xcoord{sin(2*\tuse)*\radius};

\def\centrearc[#1](#2)(#3:#4:#5)[#6]\{#7\}
{
\draw[#1] ($(#2)+({#5*cos(#3)},{#5*sin(#3)})$) arc (#3:#4:#5);
\node[#6] at ($(#2)+({#5*cos(#4)},{#5*sin(#4)})$) {#7};
}

\draw [thick, green, ->] (-3, 0) -- (-2*\xcoord, 0);
\draw [thick, green] (-2*\xcoord, 0) -- (-\xcoord, 0);
\draw [thick, blue, ->] (\xcoord, 0) -- (2*\xcoord, 0);
\draw [thick, blue] (2*\xcoord, 0) -- (3, 0);
\filldraw [brown] (-3, 0) circle (1pt) node [left] at (-3, 0) {$x'_3$};
\filldraw [black] (3, 0) circle (1pt) node [right] at (3, 0) {$x_1$};

\def\howred{0}

\draw [red!\howred] (-\xcoord, 0) -- (\xcoord, 0);
\centrearc[very thick, color=red, fill=red!\howred](0, \ycoord)(-90-2*\tuse:-90+2*\tuse:\radius)[]\{\}
\centrearc[very thick, color=red, fill=red!\howred](0, \ycoord)(-90+2*\tuse:270-2*\tuse:\radius)[]\{\}
\node at (0, \ycoord) {$u^{F_{31}}$};

\def\rescaledradius{1.5*\radius}
\def\downshift{-0.5}
\pgfmathsetmacro\tuseshifted{1/2*acos((cos(2*\tuse) - \downshift)/\rescaledradius)}
\pgfmathsetmacro\xcoordshifted{sin(2*\tuseshifted)*\rescaledradius}
\draw [	] (-3, \downshift) -- (-\xcoordshifted, \downshift);
\centrearc[->](0, \ycoord)(-90-2*\tuseshifted:-90:\rescaledradius)[below]\{$A^3B$\}
\centrearc[](0, \ycoord)(-90:-90+2*\tuseshifted:\rescaledradius)[]\{\}
\draw [	] (\xcoordshifted, \downshift) -- (3, \downshift);

\def\upshift{0.5}
\pgfmathsetmacro\tuseshifted{1/2*acos((cos(2*\tuse) - \upshift)/\rescaledradius)}
\pgfmathsetmacro\xcoordshifted{sin(2*\tuseshifted)*\rescaledradius}
\draw [	] (-3, \upshift) -- (-\xcoordshifted, \upshift);
\centrearc[->](0, \ycoord)(-90+2*\tuseshifted:90:\rescaledradius)[above]\{$A$\}
\centrearc[](0, \ycoord)(90:270 - 2*\tuseshifted:\rescaledradius)[]\{\}
\draw [	] (\xcoordshifted, \upshift) -- (3, \upshift);
\end{tikzpicture}
&
\begin{tikzpicture}[scale=0.4]
\pgfmathsetmacro\tzero{acos(sqrt(2)/sqrt(3)};
\pgfmathsetmacro\tone {acos(1/sqrt(3))};
\def\tuse{\tzero}
\def\radius{1}

\pgfmathsetmacro\ycoord{cos(2*\tuse)*\radius};
\pgfmathsetmacro\xcoord{sin(2*\tuse)*\radius};

\def\centrearc[#1](#2)(#3:#4:#5)[#6]\{#7\}
{
\draw[#1] ($(#2)+({#5*cos(#3)},{#5*sin(#3)})$) arc (#3:#4:#5);
\node[#6] at ($(#2)+({#5*cos(#4)},{#5*sin(#4)})$) {#7};
}

\draw [thick, blue, ->] (-3, 0) -- (-2*\xcoord, 0);
\draw [thick, blue] (-2*\xcoord, 0) -- (-\xcoord, 0);
\draw [thick, green, ->] (\xcoord, 0) -- (2*\xcoord, 0);
\draw [thick, green] (2*\xcoord, 0) -- (3, 0);
\filldraw [brown] (-3, 0) circle (1pt) node [left] at (-3, 0) {$x'_3$};
\filldraw [black] (3, 0) circle (1pt) node [right] at (3, 0) {$x_1$};

\def\howred{0}

\draw [red!\howred] (-\xcoord, 0) -- (\xcoord, 0);
\centrearc[very thick, color=red, fill=red!\howred](0, \ycoord)(-90-2*\tuse:-90+2*\tuse:\radius)[]\{\}
\centrearc[very thick, color=red, fill=red!\howred](0, \ycoord)(-90+2*\tuse:270-2*\tuse:\radius)[]\{\}
\node at (0, \ycoord) {$u^{B_{31}}$};

\def\rescaledradius{1.5*\radius}
\def\downshift{-0.5}
\pgfmathsetmacro\tuseshifted{1/2*acos((cos(2*\tuse) - \downshift)/\rescaledradius)}
\pgfmathsetmacro\xcoordshifted{sin(2*\tuseshifted)*\rescaledradius}
\draw [	] (-3, \downshift) -- (-\xcoordshifted, \downshift);
\centrearc[->](0, \ycoord)(-90-2*\tuseshifted:-90:\rescaledradius)[below]\{$A^5$\}
\centrearc[](0, \ycoord)(-90:-90+2*\tuseshifted:\rescaledradius)[]\{\}
\draw [	] (\xcoordshifted, \downshift) -- (3, \downshift);

\def\upshift{0.5}
\pgfmathsetmacro\tuseshifted{1/2*acos((cos(2*\tuse) - \upshift)/\rescaledradius)}
\pgfmathsetmacro\xcoordshifted{sin(2*\tuseshifted)*\rescaledradius}
\draw [	] (-3, \upshift) -- (-\xcoordshifted, \upshift);
\centrearc[->](0, \ycoord)(-90+2*\tuseshifted:90:\rescaledradius)[above]\{$A^5B$\}
\centrearc[](0, \ycoord)(90:270 - 2*\tuseshifted:\rescaledradius)[]\{\}
\draw [	] (\xcoordshifted, \upshift) -- (3, \upshift);

\end{tikzpicture}
\\

\end{tabular}
\caption{Parallel transport maps for the pearly trajectories $\lbrace \mathbf{u}^{F_{ij}},\, \mathbf{u}^{B_{ij}} \st i \equiv j, j+1 \pmod 3 \rbrace$.}\label{table}
\end{table}
\end{center}

Finally, let us determine $\del_2 \colon \End(W_m) \to \End(W_{m'})$. A pearly trajectory, connecting $m'$ to $m$ must have total Maslov index 4 and since we know all Maslov 2 discs through $\{m'\}$ or $\{m\}$, it is easy to see that it must in fact consist of a single Maslov 4 disc $u$, satisfying $u(-1) \in W^d(m') = \{m'\}$ and $u(1) \in W^a(m) = \{m\}$. Thus $u$ must be one of the discs $\{w_1, w_{-1}\}$ described in Theorem \ref{maslov 4 discs through m and m'}. From Figure \ref{the one with all the circles} again we see that $\gamma^0_{w_1} = (\sigma)^{-1}$, $\gamma^1_{w_1} = \tilde{\sigma}$, $\gamma^0_{w_{-1}} = (\tilde{\sigma})^{-1}$ and $\gamma^1_{w_{-1}} = \sigma$. It is clear then that
$$P_{\gamma^0_{w_{1}}} = \Id, \quad P_{\gamma^1_{w_1}} = A, \quad P_{\gamma^0_{w_{-1}}} = A^5, \quad P_{\gamma^1_{w_{-1}}} = \Id$$ 
and hence for every $\alpha \in \End(W_m)$ we have:
\begin{equation}\label{del 2}
\del_2 (\alpha) = A \alpha + \alpha A^5 \quad \in \quad \End(W_{m'}).
\end{equation}

We end this section by writing down the complete candidate Floer differential $d$. Note that $d$ preserves the parity of indices of critical points and so we can split it into the two following maps:
\begin{eqnarray*}
d^0 \colon \bigoplus_{\Box = m', x_1, x_2, x_3} \End(W_{\Box}) & \to & \bigoplus_{\Box = x'_1, x'_2, x'_3, m} \End(W_{\Box}) \\
d^1 \colon \bigoplus_{\Box= x'_1, x'_2, x'_3, m} \End(W_{\Box}) & \to & \bigoplus_{\Box = m', x_1, x_2, x_3} \End(W_{\Box})
\end{eqnarray*} 
Now from equations \eqref{del 0 first part}, \eqref{del 0 second part}, \eqref{del 0 third part}, \eqref{del 1 from index 1 to index 0}, \eqref{del 1 from index 3 to index 2}, \eqref{del 1 from index 2 to index 1} and \eqref{del 2} we have:

$\bullet$ for every $(\alpha', \alpha_1, \alpha_2, \alpha_3) \in \End(W_{m'})\oplus \End(W_{x_1})\oplus\End(W_{x_2})\oplus\End(W_{x_3})$

\begin{equation}\label{d0}
\begin{array}{rcl}
d^0(\alpha', \alpha_1, \alpha_2, \alpha_3) & = & \bigg(\del_0 (\alpha') + \del_1(\alpha_1, \alpha_2, \alpha_3), \, \del_0(\alpha_1, \alpha_2, \alpha_3) \bigg) \\
&=&\bigg( \big[\alpha' + \bin\alpha'\b \big ] + \aab\alpha_1 + \alpha_1\aaaab + \aab\alpha_2 + \aaaaa\alpha_2\aaaaab\, ,\\
&&\big[\alpha' + \abin\alpha'\ab \big] + \aab\alpha_2 + \alpha_2\b + \aaab\alpha_3\a + \alpha_3\b\,, \\
&& \big[\alpha' + \aabin\alpha'\ab \big] + \aaaab\alpha_3 + \alpha_3\b + \a\alpha_1\aaab + \aaaaab\alpha_1\aaaaa\, , \\
&& \big[\alpha_1 + \aaabin\alpha_1\aaab + \alpha_2 + \aaaabin\alpha_2\aaaab + \\
&& \qquad \qquad \qquad \qquad  \qquad  \qquad  \qquad  \qquad  + \alpha_3 + \aaaaabin\alpha_3\aaaaab \big]\bigg) \\
& \in & \End(W_{x'_1})\oplus\End(W_{x'_2})\oplus\End(W_{x'_3})\oplus\End(E_m). 
\end{array}
\end{equation}
$\bullet$ for every $(\alpha'_1, \alpha'_2, \alpha'_3, \alpha) \in \End(W_{x'_1})\oplus\End(W_{x'_2})\oplus\End(W_{x'_3})\oplus\End(E_m)$
\begin{equation}\label{d1}
\begin{array}{rcl}
d^1(\alpha'_1, \alpha'_2, \alpha'_3, \alpha) & = & \bigg ( \del_1(\alpha'_1, \alpha'_2, \alpha'_3) + \del_2(\alpha)\,,\, \del_0(\alpha'_1, \alpha'_2, \alpha'_3) + \del_1 (\alpha)\bigg ) \\
&=&\bigg (  B\alpha'_1 + \alpha'_2\aaaab + \aab\alpha'_3 + \big [A\alpha + \alpha A^5\big] \,, \\
&&\big [ \alpha_1' + \alpha_2' + \ain\alpha_2'\a + \ain\alpha_3'\a \big] + \alpha \b \, ,\\
&&\big [\aabin\alpha_1'\aab + \alpha_2' + \ain\alpha_3'\a + \alpha_3'\big] + \aaaab \alpha\, ,\\
&&\big [\aabin\alpha_1'\aab + \aaabin\alpha_1'\aaab +\\
&&  \qquad  \qquad  \qquad  \qquad  \qquad  \qquad + \aaabin\alpha_2'\aaab + \alpha_3'\big] + \alpha \aab\bigg)\\
& \in & \End(W_{m'})\oplus \End(W_{x_1})\oplus\End(W_{x_2})\oplus\End(W_{x_3}). 
\end{array}
\end{equation}
\vspace{5cm}

\subsection{Proof of Theorem \ref{mytheorem}}\label{proof of Theorem}
We are now in a position to prove Theorem \ref{mytheorem}. In Appendix \ref{appendix:classification of reps} we give an explicit description of all indecomposable representations of $\bindih$ over $\FF_2$. We find that there is a unique non-trivial irreducible representation which we denote by $D$. It is a two-dimensional faithful representation of the dihedral group of order 6 and its pullback to $\bindih$ is given explicitly by
\begin{eqnarray*}
\rho_D \colon \bindih & \to & GL(2, \FF_2)\\
\rho_D(a) = \begin{pmatrix}
0 & 1\\
1 & 1
\end{pmatrix},&&
\rho_D(b) = \begin{pmatrix}
0 & 1 \\
1 & 0
\end{pmatrix}.
\end{eqnarray*}
Let $W^D$ denote the resulting local system on $\chiang$.
It is immediate to check from \eqref{m_0 for chiang} that $m_0(W^D) = 0$ and thus $\left(C^*_{\F}(W^D,W^D), d^{(\F, J_0)}\right)$ is indeed a $\ZZ / 2-$graded complex which computes $HF^*(W^D,W^D)$. Further, since $\rho_D$ is surjective, we have
$$CF^*(W^D,W^D) = \overline{CF}^*(W^D,W^D) = CF^*_{mon}(W^D).$$
 We can now compute $HF^*(W^D,W^D)$ explicitly. As before, we identify both $\bigoplus_{\Box = m', x_1, x_2, x_3} \End(W^D_{\Box})$ and $\bigoplus_{\Box = x'_1, x'_2, x'_3, m} \End(W^D_{\Box})$ with $\End((\FF_2)^2)^4 = (\operatorname{Mat}_{2\times2}(\FF_2))^4$. For $(\operatorname{Mat}_{2\times2}(\FF_2))^4$ we choose a basis as follows:
\begin{itemize}
\item set 
$e_1 = \begin{pmatrix}
1 & 0\\
0 & 0
\end{pmatrix}$, 
$e_2 = \begin{pmatrix}
0 & 1\\
0 & 0
\end{pmatrix}$, 
$e_3 = \begin{pmatrix}
0 & 0\\
1 & 0
\end{pmatrix}$, 
$e_4 = \begin{pmatrix}
0 & 0\\
0 & 1
\end{pmatrix}$
\item for each $l \in \mathbb{N}$ write $l = 4(l) + \langle l \rangle$ for the division with remainder of $l$ by 4
\item for each  $1 \le l \le 16$ and $1 \le k \le 4$ define the matrix $\mathcal{E}_{lk} \in \operatorname{Mat}_{2\times2}(\FF_2)$ by
$$\mathcal{E}_{lk} \coloneqq
 \begin{cases} 
 e_{\langle l \rangle},& \text{when } \langle l \rangle \neq 0 \text{ and } k = (l)+1 \\
 e_4,& \text{when } \langle l \rangle = 0 \text{ and } k = (l) \\
 0& \text{otherwise.}
 \end{cases}$$
\item we now define a basis $\mathcal{E} \coloneqq \{\mathcal{E}_l \st 1 \le l \le 16\}$ for $(\operatorname{Mat}_{2\times2}(\FF_2))^4$ by setting
$$\mathcal{E}_l = (\mathcal{E}_{l1}, \mathcal{E}_{l2}, \mathcal{E}_{l3}, \mathcal{E}_{l4}).$$
\end{itemize}
For example $$\mathcal{E}_1 = \left(\begin{pmatrix}
1 & 0\\
0 & 0
\end{pmatrix},
\begin{pmatrix}
0 & 0\\
0 & 0
\end{pmatrix},
\begin{pmatrix}
0 & 0\\
0 & 0
\end{pmatrix},
\begin{pmatrix}
0 & 0\\
0 & 0
\end{pmatrix}\right), \;
\mathcal{E}_6 = \left(\begin{pmatrix}
0 & 0\\
0 & 0
\end{pmatrix},
\begin{pmatrix}
0 & 1\\
0 & 0
\end{pmatrix},
\begin{pmatrix}
0 & 0\\
0 & 0
\end{pmatrix},
\begin{pmatrix}
0 & 0\\
0 & 0
\end{pmatrix}\right).$$

Plugging \eqref{d0} and \eqref{d1} into a computer programme, one finds that with respect to the basis $\mathcal{E}$ the operators $d^0$ and $d^1$ are given respectively by the matrices
$$D^0 = \left(
\begin{array}{cccccccccccccccc}
 1 & 0 & 0 & 1 & 0 & 1 & 1 & 0 & 0 & 0 & 0 & 0 & 0 & 0 & 0 & 0 \\
 0 & 1 & 1 & 0 & 0 & 0 & 0 & 1 & 1 & 0 & 1 & 0 & 0 & 0 & 0 & 0 \\
 0 & 1 & 1 & 0 & 0 & 0 & 0 & 1 & 1 & 0 & 1 & 0 & 0 & 0 & 0 & 0 \\
 1 & 0 & 0 & 1 & 0 & 0 & 0 & 0 & 1 & 1 & 0 & 1 & 0 & 0 & 0 & 0 \\
 0 & 1 & 0 & 0 & 0 & 0 & 0 & 0 & 1 & 1 & 1 & 0 & 0 & 1 & 0 & 1 \\
 0 & 0 & 0 & 0 & 0 & 0 & 0 & 0 & 1 & 1 & 0 & 1 & 1 & 0 & 1 & 1 \\
 1 & 1 & 0 & 1 & 0 & 0 & 0 & 0 & 0 & 0 & 1 & 1 & 0 & 1 & 0 & 1 \\
 0 & 1 & 0 & 0 & 0 & 0 & 0 & 0 & 0 & 0 & 1 & 1 & 1 & 1 & 1 & 0 \\
 0 & 0 & 1 & 0 & 1 & 1 & 1 & 0 & 0 & 0 & 0 & 0 & 1 & 1 & 0 & 0 \\
 1 & 0 & 1 & 1 & 1 & 0 & 0 & 0 & 0 & 0 & 0 & 0 & 1 & 1 & 0 & 0 \\
 0 & 0 & 0 & 0 & 0 & 1 & 1 & 0 & 0 & 0 & 0 & 0 & 1 & 0 & 1 & 1 \\
 0 & 0 & 1 & 0 & 1 & 0 & 0 & 0 & 0 & 0 & 0 & 0 & 0 & 1 & 1 & 1 \\
 0 & 0 & 0 & 0 & 1 & 0 & 0 & 1 & 0 & 1 & 0 & 0 & 0 & 0 & 1 & 0 \\
 0 & 0 & 0 & 0 & 0 & 1 & 1 & 0 & 0 & 0 & 0 & 0 & 1 & 0 & 1 & 1 \\
 0 & 0 & 0 & 0 & 0 & 1 & 1 & 0 & 1 & 1 & 0 & 1 & 0 & 0 & 0 & 0 \\
 0 & 0 & 0 & 0 & 1 & 0 & 0 & 1 & 0 & 1 & 0 & 0 & 0 & 0 & 1 & 0 \\
\end{array}
\right)$$
and
$$
D^1=\left(
\begin{array}{cccccccccccccccc}
 0 & 0 & 1 & 0 & 1 & 1 & 0 & 0 & 1 & 0 & 1 & 0 & 1 & 1 & 1 & 0 \\
 0 & 0 & 0 & 1 & 0 & 1 & 0 & 0 & 0 & 1 & 0 & 1 & 1 & 0 & 0 & 1 \\
 1 & 0 & 0 & 0 & 0 & 0 & 1 & 1 & 0 & 0 & 1 & 0 & 1 & 0 & 0 & 1 \\
 0 & 1 & 0 & 0 & 0 & 0 & 0 & 1 & 0 & 0 & 0 & 1 & 0 & 1 & 1 & 1 \\
 1 & 0 & 0 & 0 & 1 & 1 & 0 & 1 & 0 & 1 & 0 & 1 & 0 & 1 & 0 & 0 \\
 0 & 1 & 0 & 0 & 1 & 0 & 1 & 1 & 1 & 1 & 1 & 1 & 1 & 0 & 0 & 0 \\
 0 & 0 & 1 & 0 & 0 & 1 & 1 & 0 & 0 & 1 & 0 & 0 & 0 & 0 & 0 & 1 \\
 0 & 0 & 0 & 1 & 1 & 1 & 0 & 1 & 1 & 1 & 0 & 0 & 0 & 0 & 1 & 0 \\
 1 & 0 & 1 & 0 & 1 & 0 & 0 & 0 & 1 & 1 & 0 & 1 & 1 & 0 & 0 & 0 \\
 1 & 1 & 1 & 1 & 0 & 1 & 0 & 0 & 1 & 0 & 1 & 1 & 0 & 1 & 0 & 0 \\
 0 & 0 & 1 & 0 & 0 & 0 & 1 & 0 & 0 & 1 & 1 & 0 & 1 & 0 & 1 & 0 \\
 0 & 0 & 1 & 1 & 0 & 0 & 0 & 1 & 1 & 1 & 0 & 1 & 0 & 1 & 0 & 1 \\
 1 & 0 & 1 & 1 & 0 & 0 & 0 & 1 & 1 & 0 & 0 & 0 & 1 & 0 & 0 & 0 \\
 1 & 1 & 0 & 1 & 0 & 0 & 1 & 0 & 0 & 1 & 0 & 0 & 1 & 1 & 0 & 0 \\
 0 & 1 & 1 & 0 & 0 & 1 & 0 & 0 & 0 & 0 & 1 & 0 & 0 & 0 & 1 & 0 \\
 1 & 0 & 1 & 1 & 1 & 0 & 0 & 0 & 0 & 0 & 0 & 1 & 0 & 0 & 1 & 1 \\
\end{array}
\right)
$$

One then computes that $\operatorname{rank}(D^0) = 6$, $\operatorname{rank}(D^1) = 8$ and thus
$$HF^0((\chiang,W^D), (\chiang,W^D)) \cong HF^1((\chiang,W^D),(\chiang,W^D)) \cong (\FF_2)^2.$$
Putting $W=W^D$ concludes the proof of Theorem \ref{mytheorem} as stated in the Introduction. \qed
\begin{remark}\label{minimality of W^D}
It is a fact (see section \ref{some additional calculations} below) that $W^D$ is the minimal $\FF_2-$local system on $\chiang$ for which the central Floer cohomology is non-vanishing. By this we mean that any other finite rank local system $W \to \chiang$ with $\overline{HF}^*(W,W) \neq 0$ must have $W^D$ as a direct summand. 
\end{remark}
\textit{Proof of Corollary \ref{chiangRp3nonDisp}:} In the setup of Theorem \ref{split-generation criterion} we set $L=\RP^3$ and $E$ to be the trivial local system of rank 1. As already noted $m_0(\RP^3) =0$. Proposition 1.1 in \cite{tonkonog2015closed} states that the map $\CO^* \colon QH^*(\CP^3) \to HH^*(CF^*(\RP^3, \RP^3))$ is injective and so $\RP^3$ split generates $\F(\CP^3)_0$. From Theorem \ref{mytheorem} we have that $(\chiang, W^D)$ is an essential object in $\F(\CP^3)_0$. By Lemma \ref{split-generation implies non-vanishing of HF} it then follows that $HF^*(\RP^3, (\chiang, W^D)) \neq 0$ and thus by point \ref{non-vanishing floer implies non-disp} of Remark \ref{remark with three parts}, $\chiang$ and $\RP^3$ cannot be displaced by a Hamiltonian isotopy. \qed

\subsection{Some additional calculations}\label{some additional calculations}
In this section we describe some results on the computation of the central and monodromy Floer cohomologies for some other local systems on $\chiang$. In Appendix \ref{appendix:classification of reps} we show that there are 6 isomorphism classes of indecomposable representations of $\bindih$ over $\FF_2$. These come into the following 3 groups whose central and monodromy Floer complexes exhibit quite different behaviour:
\begin{enumerate}
\item the representations $V_1, V_2, V_3, V_4$ (with $\dim V_i = i$) which are the indecomposable representations of the cyclic group $C_4$ over $\FF_2$ with $V_1$ being the trivial one and $V_4$ being the regular representation. Since $C_4$ is the quotient of $\bindih$ by $C_3 = \{1, a^2, a^4\}$, these are also representations of $\bindih$
\item the irreducible representation $D$ used in the proof of Theorem \ref{mytheorem}.
\item a faithful $\bindih-$representation $U_4$ of dimension $4$.
\end{enumerate}

Let $\rho \colon \bindih \to \operatorname{Aut}(V)$ be any finite-dimension representation of $\bindih$ over $\FF_2$. We know by the Krull-Schmidt theorem that there exist unique non-negative integers $k_1, k_2, k_3, k_4, m, n$ such that
\begin{equation}\label{Krull-Schmidt decomposition}
V \cong V_1^{\oplus k_1}\oplus V_2^{\oplus k_2} \oplus V_3^{\oplus k_3} \oplus V_4^{\oplus k_4} \oplus D^{\oplus m} \oplus U_4^{\oplus n}.
\end{equation}
Note that the relation $a^6 = 1$ in $\bindih$ yields
$$(\rho(a^2) - 1)\cdot(\rho(a^4) + \rho(a^2) + \Id) = 0.$$
Using \eqref{m_0 for chiang} we see that if $W^V \to \chiang$ is the corresponding local system, we have
\begin{equation}\label{relation for m0 of chiang}
(\rho(a^2) - 1)\cdot\big(m_0(W^V)(m')\cdot\rho(b)^{-1}\big)=0
\end{equation}
It is now not hard to see that $m_0(W^V) \neq 0$ precisely when $k_i \neq 0$ for some $i \in \{1,2,3,4\}$ and $m_0(W^V)=\Id$ if and only if $V \cong V_1^{\oplus k_1}$ is a trivial representation. It follows that $CF^*(W^V, W^V)$ is obstructed precisely when $k_i \neq 0$ for some $i \in \{2,3,4\}$. 

Using our general expressions for the Floer differential \eqref{d0} and \eqref{d1}, together with the observations \eqref{decomposition for central complex} and \eqref{phixx isomorphisms} one can do some explicit computer calculations of central and monodromy Floer cohomology for different local systems on $\chiang$. To alleviate notation we write $K_V = V_1^{\oplus k_1}\oplus V_2^{\oplus k_2} \oplus V_3^{\oplus k_3} \oplus V_4^{\oplus k_4}$, $M_V = D^{\oplus m} \oplus U_4^{\oplus n}$ and drop the notational distinction between a representation of $\bindih$ and a local system on $\chiang$. One then finds that:
\begin{enumerate}[1)]
\item $\forall i \in \{0,1,2,3\} \quad H^i(\chiang; \End(V)) \cong H^i(\chiang ; End(K_V)) \oplus H^i(\chiang; \End(M_V))$ 
\item \label{kv and mv are hm-bar-orthogonal} $\forall i \in \{0,1,2,3\} \quad \overline{HM}^i(V,V) \cong \overline{HM}^i(K_V, K_V) \oplus \overline{HM}^i(M_V, M_V)$
\item \label{kv and mv are hm-mon orthogonal}$\forall i \in \{0, 1, 2, 3\} \quad HM^i_{mon}(V) \cong HM^i_{mon}(K_V) \oplus HM^i_{mon}(M_V)$
\item \label{hm-mon for kv}$\forall i \in \{0, 1, 2, 3\} \quad HM^i_{mon}(K_V) \cong (\FF_2)^{\max\{i \st k_i \neq 0\}}$
\item \label{hm-mon for mv}$\forall i \in \{0, 1, 2, 3\} \quad HM^i_{mon}(M_V) \cong
\begin{cases}
0, \quad M_V = 0 \\
\FF_2, \quad m \neq 0, n=0\\
(\FF_2)^2, \quad n \neq 0
\end{cases}$
\item $\forall i \in \{0, 1\} \quad \overline{HF}^i(K_V,K_V) \cong 0$
\item \label{central floer depends only on the copies of D} $\forall i \in \{0, 1\} \quad \overline{HF}^i(V,V) \cong  \overline{HF}^i(M_V, M_V) \cong (\FF_2)^{2m^2}$
\item $\forall i \in \{0, 1\} \quad HF^i_{mon}(K_V) \cong 0$
\item \label{hf-mon for mv}$\forall i \in \{0, 1\} \quad HF^i_{mon}(V) \cong  HF^i_{mon}(M_V) \cong
\begin{cases}
0, \quad M_V = 0 \\
(\FF_2)^2, \quad m \neq 0, n=0\\
(\FF_2)^4, \quad n \neq 0
\end{cases}$
\end{enumerate}
Note first from \ref{central floer depends only on the copies of D}) that $HF^*(V,V) \neq 0$ if and only if $m \neq 0$. In other words any local system on $\chiang$ with non-vanishing (central) Floer cohomology must contain a copy of $D$ as a direct summand.

It is also worth noting (by combining \ref{kv and mv are hm-bar-orthogonal}) and \ref{central floer depends only on the copies of D})) that whenever $K_V \neq 0$ and $m \neq 0$ we have
$$0 \neq \overline{HF}^i(V,V) \neq \bigoplus_{j \in \{0, 1\}} \overline{HM}^{2j+i}(V,V) \quad \forall i \in \{0, 1\}.$$
In other words, the corrections to the Morse differential on $\overline{C}_f(V,V)$ coming from holomorphic discs are non-trivial but they also do not kill the cohomology entirely. 

It follows from \ref{kv and mv are hm-mon orthogonal}), \ref{hm-mon for kv}), \ref{hm-mon for mv}) and \ref{hf-mon for mv}) that the same is true for the monodromy Floer complex, i.e. whenever $K_V \neq 0$ and $M_V \neq 0$ one has
$$0 \neq HF^i_{mon}(V) \neq \bigoplus_{j \in \{0, 1\}}HM^{2j+i}_{mon}(V)\quad \forall i \in \{0, 1\}.$$
This in particular applies to the regular representation $V_{reg}$ of $\bindih$ since it is not hard to check that $V_{reg} \cong V_4 \oplus U_4 \oplus U_4$. Thus \ref{kv and mv are hm-mon orthogonal}), \ref{hm-mon for kv}) and  \ref{hm-mon for mv}) yield 
$$H^i(\chiang, W_{conj}) = HM^i_{mon}(W_{reg}) \cong (\FF_2)^6 \quad \forall i \in \{0, 1, 2, 3\},$$
while from \ref{hf-mon for mv}) we have 
$$HF^i_{mon}(W_{reg}) \cong (\FF_2)^4 \quad \forall i \in \{0, 1\}.$$
This is observation \ref{narrow-wide breaks}.

Finally, note that in many cases the central Floer cohomology vanishes, while $HF^*_{mon}$ does not. For example, from \ref{central floer depends only on the copies of D}) and \ref{hf-mon for mv}) this is always the case whenever $m=0$ but $n \neq 0$. Thus, sometimes even if the Floer complex is unobstructed, it could be that one obtains a non-zero invariant only after passing to the monodromy subcomplex.
\appendix

\refstepcounter{section}
\section*{Appendix A: Perturbation of Morse Data}\label{perturbation of morse data}\label{appendix A}
\addcontentsline{toc}{section}{Appendix A: Perturbation of Morse Data}
Many statements in the foundational theory of the pearl complex (\cite{biran2007quantum}) rely on a generic choice of almost complex structure, i.e. on showing the existence of the non-empty (and in fact Baire) subset $\Jreg(\F) \subseteq \J(M, \omega)$ and choosing $J \in \Jreg(\F)$. An obvious necessity for such a genericity assumption is ensuring that all moduli spaces of simple holomorphic discs are regular, i.e. that $J \in \Jreg(L)$. Pertaining more to the particular algebraic structure under construction is the requirement to shrink $\Jreg(L)$ to the smaller set $\Jreg(\F)$ in order to guarantee that the spaces of pearly trajectories of the form $\P(y, x, \mathbf{A};  \F, J)$ and $\P(y, x, (\mathbf{A}', \mathbf{A}'');  \F, J)$
are cut out transversely by their respective evaluation maps, at least when their expected dimensions are at most 1. 

All of our calculations for $\chiang$ above required that we work with the standard complex structure $J_0$ on $\CP^3$. There is no problem with the regularity of the moduli spaces of discs, due to Theorem \ref{regularity}. However the transversality for evaluation maps could a-priori not be satisfied for $J_0$ with respect to our very specific choice of Morse data \eqref{morse function}. Exactly this potential problem is addressed in \cite[Appendix A]{smith2015floer} and a direct application of Proposition A.2 from there is the following
\begin{proposition}(\cite[Proposition A.2]{smith2015floer})\label{transversality from Jack}
For any $C^{\infty}-$open neighbourhood $U$ of the identity in $C^{\infty}(\chiang, \chiang)$, there exists a diffeomorphism $\varphi \in U$ such that $J_0 \in \Jreg(\varphi_*\F)$.
\end{proposition}
In light of this, we see that to make sure our calculations are rigorous, we need only check that they are ``robust'' under small perturbations of Morse data. More precisely, we now prove the following

\begin{proposition}\label{robustness under perturbations}
There exists a smooth path of diffeomorphisms $\{\varphi_{\tau}\}_{\tau \in [0, 1]} \subseteq C^{\infty}(\chiang, \chiang)$ such that $\varphi_0 = \operatorname{id}_{\chiang}$, $J_0 \in \Jreg((\varphi_1)_*\F)$ and for every $x, y \in Crit(f)$ and $k \in \NN$ satisfying $\delta(y, x, 2k; \F) = 0$, there is a bijection
$$ S \colon \P(y, x, 2k; \F, J_0) \to \P(\varphi_1(y), \varphi_1(x), 2k; (\varphi_1)_*\F, J_0)$$
such that for every $\mathbf{u} \in \P(y, x, 2k; \F, J_0)$ we have that the paths $\gamma_{\mathbf{u}}^0$ and $\{\varphi_{\tau}(y)\} \cdot \gamma_{S(\mathbf{u})}^0 \cdot \{\varphi_{1-\tau}(x)\}$ define the same element of  $\Pi_1\chiang(y,x)$ and similarly $\gamma_{\mathbf{u}}^1$ and  $\{\varphi_{\tau}(x)\} \cdot \gamma_{S(\mathbf{u})}^1 \cdot \{\varphi_{1-\tau}(y)\}$ define the same element of $\Pi_1\chiang(x,y)$ (here the notation $\{\varphi_{\tau}(y)\}$ refers to the path $[0, 1] \to \chiang$, $\tau \to \varphi_{\tau}(y)$ and similarly for the other such paths).
\end{proposition}

\begin{proof}
The proof of this fact proceeds as follows. First we need to show that all isolated pearly trajectories which we counted to construct $d^{(\F, J_0)}$ have been cut-out transversely by their respective evaluation maps. More precisely, let $u$ be any of the discs from the set $P \coloneqq \{u'_i, u_i, u^{F_{ij}}, u^{B_{ij}}, w_1, w_{-1} \st i, j \in \{1,2,3\}, \; j\equiv i, i+1 \mod 3\}$ and write $y_u, x_u \in Crit(f)$ for the source and target of the corresponding pearly trajectory. For brevity we write $X_u \coloneqq \PM^{[u]}(\chiang; J_0)/(G_{-1, 1})$. Then we need to check that
 the map
$$\ev_u \coloneqq (\ev_{-1}, \ev_1) \colon X_u \longrightarrow \chiang \times \chiang$$
is transverse to $N_u \coloneqq W^d(y_u) \times W^a(x_u)$. Further, if $\cl{X}_u$ denotes the Gromov compactification of $X_u$ and $\cl{\ev}_u$ is the unique continuous extension of $\ev_u$ to $\cl{X}_u$, then we need to verify that 
\begin{eqnarray}
\cl{\ev}_u(\cl{X}_u) \cap (\cl{N}_u \setminus N_u)& = &\emptyset \label{broken pearly trajectories and shrinking of flowlines}\\
\cl{\ev}_u(\cl{X}_u \setminus X_u) \cap \cl{N}_u & = & \emptyset. \label{bubbled pearly trajectories}
\end{eqnarray}
Let us postpone these checks for a moment. The above facts now allow us to apply the mantra ``transverse intersections persist'' in the precise form that it appears in \cite[Appendix A]{smith2015floer}.
First, we invoke \cite[Lemma A.1]{smith2015floer} to assert the existence of a $C^1-$open neighbourhood $U$ of $\operatorname{id}_{\chiang}$ in $C^{\infty}(\chiang, \chiang)$ such that for every $\varphi \in U$ and every $u \in P$ the map $\ev_u$ is transverse to $\varphi(N_u)$ (note that $\varphi(N_u) = W^d(\varphi(x_u)) \times W^a(\varphi(y_u))$, where descending and ascending manifolds are now taken with respect to the perturbed Morse datum $\varphi_*\F$). Now, applying Proposition \ref{transversality from Jack} we can find $\varphi_1 \in U$ such that $J_0 \in \Jreg((\varphi_1)_*\F)$. Further, we may assume $\varphi_1$ to be connected to $\varphi_0 = \operatorname{id}_{\chiang}$ via a path $\{\varphi_{\tau}\}_{\tau \in [0, 1]} \subseteq U$. It follows then from general differential topology that for every $u \in P$ the space
\begin{eqnarray*}
\P(y_u, x_u, [u]; \F, J_0, \{\varphi_{\tau}\}) & \coloneqq &\{(\tau, \hat{u}) \in [0, 1] \times X_u \st \hat{u} \in \ev_u^{-1}(N_u)\} \\
&=&\{(\tau, \hat{u}) \in [0, 1] \times X_u
 \st (\hat{u}) \in
\P(\varphi_{\tau}(y_u), \varphi_{\tau}(x_u), [u]; (\varphi_{\tau})_*\F, J_0)\}
\end{eqnarray*} 
is a trivial cobordism between $\P(y_u, x_u, [u]; \F, J_0)$ and $\P(\varphi_{1}(y_u), \varphi_{1}(x_u), [u]; (\varphi_{1})_*\F, J_0)$. In fact, the projection 
$$\pi \colon \P(y_u, x_u, [u]; \F, J_0, \{\varphi_{\tau}\}) \to [0, 1]$$ 
is a diffeomorphism and so if we write $\hat{u}_{\tau_0} = \pi^{-1}(\tau_0)$, we can define
\begin{eqnarray*}
S \colon \P(y, x, 2k; \F, J_0) & \to & \P(\varphi_1(y), \varphi_1(x), 2k; (\varphi_1)_*\F, J_0)\\
\hat{u}_0 & \mapsto & \hat{u}_1.
\end{eqnarray*}
Further, the collection of paths $\{\phi_{\tau}(y_u)\}_{\tau \in [0, \tau_0]} \cdot \gamma^0_{(\hat{u}_{\tau_0})} \cdot \{\phi_{\tau_0-\tau}(x_u)\}_{\tau \in [0, \tau_0]}$ for varying $\tau_0 \in [0, 1]$ provide the needed homotopy between $\gamma_{\mathbf{u}}^0$ and $\{\varphi_{\tau}(y)\} \cdot \gamma_{S(\mathbf{u})}^0 \cdot \{\varphi_{1-\tau}(x)\}$ and similarly the paths $\{\phi_{\tau}(x_u)\}_{\tau \in [0, \tau_0]} \cdot \gamma^1_{(\hat{u}_{\tau_0})} \cdot \{\phi_{\tau_0-\tau}(y_u)\}_{\tau \in [0, \tau_0]}$ give a homotopy between $\gamma_{\mathbf{u}}^1$ and  $\{\varphi_{\tau}(x)\} \cdot \gamma_{S(\mathbf{u})}^1 \cdot \{\varphi_{1-\tau}(y)\}$.


Let us now make the necessary checks. First, we verify that all isolated pearly trajectories have been transversely cut-out by their respective evaluation maps. To that end, let us introduce the following notation. Given a point $y \in \chiang$ we will write $\psi_y \colon \mathfrak{su}(2) \to T_{y}\chiang$ for the infinitesimal action. Note that $\psi_y$ is always an isomorphism. Similarly, for a parametrised holomorphic disc $\tilde{u}$ in some moduli space of parametrised discs $\PM^{[\tilde{u}]}(\chiang, J_0)$ we write $\psi_{\tilde{u}} \colon \mathfrak{su}(2) \to T_{\tilde{u}}\PM^{[\tilde{u}]}(\chiang, J_0)$ for the infinitesimal action of $SU(2)$ and $\phi_{\tilde{u}} \colon \mathfrak{g} \to T_{\tilde{u}}\PM^{[\tilde{u}]}(\chiang, J_0)$ for the infinitesimal reparametrisation action (recall that we denote by $G$ the biholomorphism group of the unit disc in $\CC$ and $\mathfrak{g}$ is its Lie algebra).

In order to check that the trajectory $(u'_1)$ is transversely cut-out, we need to show that the map
$$(\ev_{-1}, \ev_{1}) \colon \PM^{[u'_1]}(\chiang; J_0) \to \chiang \times \chiang$$ 
is transverse to $W^d(m') \times W^a(x'_1) = \{m'\} \times W^a(x'_1)$ at any holomorphic parametrisation $\tilde{u}'_1$ of $u'_1$, satisfying $\tilde{u}'_1(-1) = m'$, $\tilde{u}'_1(1) = x'_1$. From general linear algebra considerations, it follows that this transversality holds if and only if the map
\begin{eqnarray*}
d \ev_{-1} ({\tilde{u}'_1}) \colon T_{\tilde{u}'_1}\PM^{[u'_1]}(\chiang; J_0) & \to & T_{m'}\chiang \\
\xi &\mapsto & \xi(-1)
\end{eqnarray*}
is surjective and there exists $\xi_0 \in T_{\tilde{u}'_1}\PM^{[u'_1]}(\chiang; J_0)$ such that $\xi_0(-1)=0 \in T_{m'}\chiang$ and $\xi_0(1) \in T_{x'_1}\chiang \setminus T_{x'_1}W^a(x'_1)$. The first assertion is immediate from the fact that $(d \ev_{-1}({\tilde{u}'_1})) \circ \psi_{\tilde{u}'_1} = \psi_{m'}$ and the latter is an isomorphism. To find a suitable $\xi_0$ on the other hand, let $\{\varphi_{\theta}\}$ be any path in $G$ such that $\varphi_{\theta}(-1) =  -1$, $\varphi_{\theta}(1) = e^{i\theta}$ for all $\theta \in \RR$ and $\varphi_0$ is the identity\footnote{e.g. $\varphi_{\theta}(z) = \frac{(1+3e^{i\theta})z - (1-e^{i\theta})}{(1-e^{i\theta})z + (3 + e^{i\theta})}$.}. Let $\nu \in \mathfrak{g}$ be its derivative at 0. Set $\xi_0 = \phi_{\tilde{u}'_1}(\nu)$. Then  $\xi_0(-1) = (\phi_{\tilde{u}'_1}(\nu))(-1) =0 \in T_{m'}\chiang$ and
$$\xi_0(1) = (\phi_{\tilde{u}'_1}(\nu))(1) = \at{\frac{d}{d\theta}}{\theta=0} \tilde{u}'_1(e^{i\theta}) = d_1 \tilde{u}'_1 \cdot \del_{\theta}\quad \in \quad T_{x'_1}W^d(x'_1) \setminus \{0\}$$
since the boundary of $u'_1$ coincides with $W^d(x'_1)$. Since the function $f$ is Morse we then have $\xi_0(1) \notin T_{x'_1}W^a(x'_1)$. Thus the pearly trajectory $(u'_1)$ is cut-out transversely. It follows by symmetry that transversality also holds for the trajectories $(u'_2)$, $(u'_3)$, $(u_1)$, $(u_2)$, $(u_3)$. 

Now consider a trajectory connecting an index 1 critical point to an index 2 critical point, e.g. $\mathbf{u}^{B_{11}} = (u^{B_{11}})$. We need to check that 
$$(\ev_{-1}, \ev_{1}) \colon \PM^{[u^{B_{11}}]}(\chiang; J_0) \to \chiang \times \chiang$$ is transverse to $W^d(x'_1) \times W^a(x_1)$ at $\tilde{u}^{B_{11}}$, where again $\tilde{u}^{B_{11}}$ is a holomorphic parametrisation of $u^{B_{11}}$ with $\tilde{u}^{F_{11}}(-1) = y'_1 \in W^d(x'_1)$ and $\tilde{u}^{F_{11}}(1)  = y_1 \in W^a(x_1)$. Define the projections $\pi_{y'_1} \colon T_{y'_1}\chiang \to T_{y'_1}\chiang / T_{y'_1}W^d(x'_1)$ and $\pi_{y_1} \colon T_{y_1}\chiang \to T_{y_1}\chiang / T_{y_1}W^a(x_1)$. Then the desired transversality holds if and only if the maps
\begin{equation}\label{tangent evaluation at -1}
\pi_{y'_1} \circ (d \ev_{-1}({\tilde{u}^{B_{11}}})) \quad \colon \quad T_{\tilde{u}^{B_{11}}}\PM^{[u^{B_{11}}]}(\chiang; J_0) \quad \longrightarrow \quad T_{y'_1}\chiang / T_{y'_1}W^d(x'_1) 
\end{equation}
and
\begin{equation}\label{tangent evaluation at 1}
\pi_{y_1} \circ (d \ev_1(\tilde{u}^{B_{11}})) \quad \colon \quad \ker (\pi_{y'_1} \circ (d \ev_{-1}({\tilde{u}^{B_{11}}})))\quad  \longrightarrow \quad T_{y_1}\chiang / T_{y_1}W^a(x_1)
\end{equation}
are surjective. The first one of these is clearly surjective again since $(d \ev_{-1}\left({\tilde{u}^{B_{11}}}\right)) \circ \psi_{\tilde{u}^{B_{11}}} = \psi_{y'_1}$ is an isomorphism. Again, by differentiating a suitable path of reparametrisations of $\tilde{u}^{B_{11}}$, we can find $\xi_0 \in T_{\tilde{u}^{B_{11}}}\PM^{[u^{B_{11}}]}(\chiang; J_0)$ such that $\xi_0(-1)=0$ and $\xi_0(1) = d \tilde{u}^{B_{11}}(1) \cdot \del_{\theta} \in T_{y_1}\chiang \setminus T_{y_1}W^a(x_1)$. Now, rewriting 
$$\ker (\pi_{y'_1} \circ (d \ev_{-1}({\tilde{u}^{B_{11}}}))) = \{\xi \in T_{\tilde{u}^{B_{11}}}\PM^{[u^{B_{11}}]}(\chiang; J_0) \st \xi(-1) \in T_{y'_1}W^d(x'_1)\},$$
it is clear that to show that \eqref{tangent evaluation at 1} is surjective, it suffices to find $\xi_1 \in T_{\tilde{u}^{B_{11}}}\PM^{[u^{B_{11}}]}(\chiang; J_0)$ such that $\xi_1(-1) \in T_{y'_1}W^d(x'_1)$ and $\xi_1(1) \notin \operatorname{Span}\{\xi_0(1)\} \oplus T_{y_1}W^a(x_1)$. We claim that it suffices to set $\xi_1 \coloneqq \psi_{\tilde{u}^{B_{11}}}(V'_1)$ \footnote{Recall that we identify $\RR^3$ with $\mathfrak{su}(2)$ via the basis of Pauli matrices $\{\sigma_1, \sigma_2, \sigma_3\}$.}. Indeed 
$$\xi_1(-1) = \psi_{y'_1}(V'_1) = \at{\frac{d}{dt}}{t=0} \exp(tV'_1)\cdot y'_1 = \at{\frac{d}{dt}}{t=0} \exp(tV'_1)\cdot \exp(t_0V'_1) \cdot \Delta$$ which clearly lies in $W^d(x'_1)$. To see that $\xi_1(1) \notin \operatorname{Span}\{\xi_0(1)\} \oplus T_{y_1}W^a(x_1)$ observe that 
\begin{eqnarray*}
\psi_{y_1}^{-1}(\xi_1(1)) & = & V'_1\\
\psi_{y_1}^{-1} \left(\operatorname{Span}\lbrace\xi_0(1)\rbrace \right) = \psi_{y_1}^{-1} \left(\operatorname{Span}\{d \tilde{u}^{B_{11}}(1)\cdot \del_{\theta}\} \right) & = & \operatorname{Span}\{B_{11}\} \\
\psi_{y_1}^{-1} \left(T_{y_1}W^a(x_1)\right) & = & \operatorname{Span}\{V_1\}
\end{eqnarray*}
and $\{V'_1, B_{11}, V_1\}$ are linearly independent. This shows that the pearly trajectory $\mathbf{u}^{B_{11}}$ is cut-out transversely. By symmetry, the same is clearly true for all other trajectories connecting a critical point of index 1 to a critical point of index 2.

It remains for us to check that the trajectories $(w_1)$ and $(w_{-1})$ are transversely cut-out. This cannot be verified as easily as the cases above and one needs to carefully analyse the appropriate dimensions for spaces of holomorphic sections of the Riemann-Hilbert pairs $\left(w_j^* T\CP^3, (\at{w_j}{\del D})^*T\chiang \right)$. A more general analysis is carried out in \cite[Section 4.6]{smith2015floer} and the precise transversality result we need is contained in \cite[Corollary 4.6.4]{smith2015floer}.  


Let us now verify conditions \eqref{broken pearly trajectories and shrinking of flowlines} and \eqref{bubbled pearly trajectories}. For all trajectories of total Maslov index 2 these are immediate. For example, note that in these cases the elements of $\cl{X}_{u} \setminus X_{u}$ are configurations consisting of a Maslov two disc and a constant bubble, containing the two marked points. Since the map $\cl{\ev}_u$ restricts to $\cl{X}_{u} \setminus X_{u}$ as the evaluation on the constant bubble to both $\chiang$ factors, its image is contained in $(\cl{W^d(y_u)} \times \cl{W^a(x_u)}) \cap \operatorname{diag}(\chiang) = \emptyset$. This verifies equation \eqref{bubbled pearly trajectories}. Equation \eqref{broken pearly trajectories and shrinking of flowlines} is equally easy to check in this case. 

Consider now the trajectories containing a single Maslov 4 disc. In this case for $j \in \{-1, 1\}$ we have that $N_{w_j} = \cl{N_{w_j}} = \{m'\} \times \{m\}$ and the elements of $\cl{X}_{w_j} \setminus X_{w_j}$ are configurations of one of the following types:
\begin{itemize}
\item a Maslov 4 disc and a constant bubble containing the marked points;
\item two Maslov 2 discs with both marked points on the same component;
\item two Maslov 2 discs with each containing one marked point.
\end{itemize}
The image under $\cl{\ev}_{w_j}$ of the strata consisting of the configurations of the first two types, is then clearly disjoint from $\{m'\} \times \{m\}$. To see that the same holds for the third type just note that the boundaries of the Maslov 2 discs through $m'$ are disjoint from the boundaries of the Maslov 2 discs through $m$.

Finally, putting $X \coloneqq  \PM(2, \chiang; J_0)/(G_{-1, 1}) \times \PM(2, \chiang; J_0)/(G_{-1, 1})$ we must consider the map: 	
\begin{equation}
\ev_X \coloneqq (\ev_{-1}, \ev_1, \ev_{-1}, \ev_1) \colon X \longrightarrow \chiang \times \chiang \times \chiang \times \chiang.
\end{equation}
It is vacuously transverse to $N_X \coloneqq {m'} \times Q_{\chiang} \times {m}$. To construct the Gromov compactification $\cl{X}$ one needs to add only pairs of Maslov 2 discs where one or more of them have a ghost component containing the marked points. Therefore, if $\hat{u} \in \cl{X}$ and $\cl{\ev}_X(\hat{u}) = (p_1, p_2, p_3, p_4) \in \cl{N}_X$, we must have that $p_1 = m'$, $p_4 = m$ and $p_2$ lies on a Maslov 2 disc passing through $m'$, while $p_3$ lies on a Maslov 2 disc passing through $m$. Observe however that $\cl{Q}_{\chiang} \subseteq \{(p, q) \in \chiang^2 \st f(p) \ge f(q)\}$ and $(\del u'_i) \times (\del u_j)$ is disjoint from the latter set for any $i, j \in \{1,2,3\}$. Therefore, $\cl{\ev}_X(\cl{X})$ is disjoint from $\cl{N}_X$ and so conditions \eqref{broken pearly trajectories and shrinking of flowlines} and \eqref{bubbled pearly trajectories} are satisfied. The proof of Proposition \ref{robustness under perturbations} is now complete. 
\end{proof}
\refstepcounter{section}
\section*{Appendix B: Representations of the binary dihedral group of order 12 over $\FF_2$ }\label{appendix:classification of reps}
\addcontentsline{toc}{section}{Appendix B: Representations of the binary dihedral group of order 12 over $\FF_2$}
In this appendix we describe all indecomposable $\FF_2-$representations of the binary dihedral group of order twelve. Such a classification is, of course, not new and much more general results are proved for example in \cite{janusz1969indecomposable}. Here we give a rather direct and pedestrian argument for the classification in order to make the arguments in Section \ref{some additional calculations} complete and the paper more self-contained.

We start by making the following observations. First note that the binary dihedral group
 $$\bindih = \left\langle a,b \;\vert\; a^6 = 1, a^2 = b^3, ab = ba^5\right\rangle$$
 can be viewed as the semi-direct product 
 $$C_3 \rtimes C_4 = \left\langle c, b \;\vert\; c^3 = 1, b^4 = 1, bc = c^2b \right\rangle$$
 by simply setting $c \coloneqq a^2$. This point of view will be particularly convenient for us since we will classify representations of $\bindih$ by viewing them simultaneously as $C_3-$representations and $C_4-$representations. To that end, let us introduce some notation. We put
 \begin{eqnarray*}
 R_3 & \coloneqq & \FF_2[C_3] = \FF_2[c]/(c^3 -1) \\
 R_4 & \coloneqq & \FF_2[C_4] = \FF_2[b]/(b^4 -1) = \FF_2[b]/(b + 1)^4.
 \end{eqnarray*}
If $V$ is a $\bindih-$representation, we shall write $O_{C_3}(V)$ for the set of orbits of the $C_3-$action on $V\setminus \{0\}$. Note that since $C_3$ is a normal subgroup of $\bindih$, we have a $C_4-$action on $O_{C_3}(V)$. We denote the set of orbits of this action by $O_{C_4}(O_{C_3}(V))$. For an element $\A \in O_{C_4}(O_{C_3}(V))$ we shall write $\spaN \A \coloneqq \sum_{A \in \A}\spaN A \le V$. Note that $\spaN \A$ is always a $\bindih-$subrepresentation of $V$.
Further, given a $\bindih-$representation $V$ and a $C_i-$representation $W$ for some $i \in \{3, 4\}$, we will write $V \cong_i W$ to mean that $V$ and $W$ are isomorphic as representations of $C_i$.

Note now that the ring $R_3$ is semisimple with 
$$R_3 \cong \FF_2 \oplus \frac{\FF_2[c]}{1+c+c^2}$$
and hence any finite-dimensional $R_3-$module $\widehat{V}$ can be written as 
$$\widehat{V} \cong \widehat{V}_1^{\oplus k_1} \oplus \widehat{D}^{\oplus k_2},$$
where $\widehat{V}_1$ is the trivial one-dimensional $R_3-$module and $\widehat{D} \coloneqq \FF_2[c]/(1+c+c^2)$.

On the other hand, the ring $R_4$ is not semisimple but from the structure theorem for finitely-generated modules over a PID, we know that the only indecomposable finite-dimensional $R_4-$modules are
$$\overline{V}_1 \coloneqq \frac{\FF_2[b]}{b+1}, \; \overline{V}_2 \coloneqq \frac{\FF_2[b]}{(b+1)^2}, \;\overline{V}_3 \coloneqq \frac{\FF_2[b]}{(b+1)^3}, \; \overline{V}_4 \coloneqq R_4.$$
Observe that since we have the short exact sequence $1 \to C_3 \to \bindih \to C_4 \to 1$, the above vector spaces are also indecomposable $\bindih-$representations with trivial $C_3-$action. When we view them as such, we will lose the bar on top and denote them as $V_1, V_2, V_3, V_4$. 

Further, since we have the short exact sequence 
$$1 \longrightarrow C_2=\{1, b^2\} \longrightarrow \bindih \longrightarrow D_6 = \left\langle c, \hat{b}\;\vert\; c^3=1, \hat{b}^2 = 1, c\hat{b} = \hat{b}c^2\right\rangle \longrightarrow 1$$
and $D_6$ acts naturally on $\widehat{D}= \FF_2[c]/(1+c+c^2)$ by $\hat{b}\cdot 1 = 1$, $\hat{b}\cdot c = c^2$, $\hat{b}\cdot c^2 = c$, we see that $\widehat{D}$ has naturally the structure of a non-faithful indecomposable $\bindih-$representation. We denote this representation by $D$.

Finally, we define the following faithful representation of $\bindih$. 
Let
$$U_4 \coloneqq \S \oplus \S x$$
and set $b \cdot x = 1+ cx$. Using linearity and the relation $bc=c^2b$ this extends uniquely to an action of $C_4$ on $U_4$, thus making $U_4$ into a well-defined $\bindih-$representation. It is important to note that $U_4 \cong_4 \overline{V}_4$, for example via the map 
\begin{eqnarray*}
U_4 &\longrightarrow& \FF_2[b]/(b+1)^4\\
1 &\mapsto & 1+b+b^2+b^3\\
c &\mapsto & 1+b^2\\
x &\mapsto & 1 \\
cx & \mapsto & 1 +b^2+b^3.
\end{eqnarray*}
We are now ready to state the classification.
\begin{proposition}\label{classification of representations}
The only finite-dimensional indecomposable representations of $\bindih$ over $\FF_2$ are $V_1, V_2, V_3, V_4, D$ and $U_4$.
\end{proposition}

We will prove this statement in several steps and
in the course of the proof it will become apparent that all these representations are indeed indecomposable. Note that $V_1$ and $D$ are the only irreducible representations since $U_4$ contains a copy of $D$ and $V_i \le V_j$ whenever $i\le j$. 

It will be useful for us to also consider the following representation. Let 
$$U_8 \coloneqq \S \oplus \S x \oplus \S x^2 \oplus \S x^3$$
and let $b \in C_4$ act as the cyclic permutation $(1, x, x^2, x^3)$. Again, the relation $bc=c^2b$ allows us to extend this action making $U_8$ into a $\bindih-$representation. In fact,  we have $U_8 \cong U_4 \oplus U_4$ via the map
\begin{eqnarray*}
U_8 & \longrightarrow & U_4 \oplus U_4 \\
1 & \mapsto & (x, cx).
\end{eqnarray*} 

To begin the classification, we first observe that we can restrict attention only to representations on which $C_3$ acts non-trivially. 
\begin{lemma}
Let $V$ be a $\bindih-$representation over $\FF_2$. Define 
\begin{eqnarray*}
V^{C_3} & \coloneqq & \{v \in V \st c\cdot v = v\}\\
W & \coloneqq & \{v \in V \st v + c \cdot v + c^2 \cdot v = 0\}.
\end{eqnarray*}
Then $V^{C_3}$ and $W$ are $\bindih-$representations and we have a decomposition $V = V^{C_3}\oplus W$.
\end{lemma}
\begin{proof}
The fact that $V \cong_3 V^{C_3} \oplus W$ is just a restatement of the fact that $R_3$ is semisimple. To see that $V^{C_3}$ and $W$ are preserved by the $C_4-$action note that if $v \in V^{C_3}$ then $c \cdot (b \cdot v) = b \cdot (c^2 \cdot v) = b \cdot v$, i.e. $b \cdot v \in V^{C_3}$ and if $v \in W$ then $(1 + c + c^2) \cdot (b \cdot v) = b \cdot ( (1 + c + c^2) \cdot v) = 0$, i.e. $b \cdot v \in W$.
\end{proof}
We thus have that $V \cong V_1^{\oplus k_1} \oplus V_2^{\oplus k_2} \oplus V_3^{\oplus k_3} \oplus V_4^{\oplus k_4} \oplus W$, where $W^{C_3} =0$. To prove Proposition \ref{classification of representations} it then suffices to show that the only indecomposable representations $V$ with $V^{C_3} = 0$ are $D$ and $U_4$. We do this in two steps: first, we show that these are the only indecomposable $\bindih-$representations of dimension at most $8$ and then we prove that any $\bindih-$representation $V$ with $V^{C_3} = 0$ and $\dim V > 8$ cannot be indecomposable. 

Classifying the two-dimensional representations is easy. Indeed, if $V$ is such, then we have $V \cong_3 \widehat{D}=\s$ and $V$ contains exactly $3$ non-zero vectors $\{1, c, c^2\}$. Since $C_4$ acts on this set, we must have that either this action is trivial, or that $b$ fixes one of the three vectors and swaps the other two. But $C_4$ cannot act trivially since then we would have 
$$c^2 = c^2 \cdot(b\cdot 1) = b \cdot (c \cdot 1) = c,$$
a contradiction. Hence $b$ fixes exactly one non-zero vector and without loss of generality $b \cdot 1 = 1$ and $b \cdot c = c^2$, $b \cdot c^2 = c$. Thus $V \cong D$ as $\bindih-$representations.

In fact, the only indecomposable representations of the dihedral group $D_6$ over $\FF_2$ are the trivial representation, the regular representation of $C_2$ and $D$. This is an easy special case of \cite{bondarenko1975representations} and can also be proved directly by writing $D_6 = C_3 \rtimes C_2$ and using the same methods we employ here (see Remark \ref{remark on how to classify dihedral reps} below). Thus, we can restrict ourselves to finding the faithful indecomposable  $\bindih-$representations. 

So, let $V$ be a faithful $\bindih-$representation with $V^{C_3}=0$ and $\dim V = 4$. Then we have
$$V \cong_3 \S \oplus \S x.$$
and
\begin{eqnarray*}
O_{C_3}(V) & = & \{\{1, c, c^2\}, \{x, cx, c^2 x\}, \{1+x, c + cx, c^2 + c^2 x\},\\
&& \qquad \qquad \{1+cx, c+c^2 x, c^2 + x\}, \{1+ c^2x, c+x, c^2 + cx\}\}.
\end{eqnarray*}
Since the size of each orbit of the $C_4-$action on $O_{C_3}(V)$ action must divide $\vert C_4 \vert = 4$ and $\vert O_{C_3}(V) \vert  = 5$ we see that $C_4$ must preserve at least one $C_3-$orbit. Up to a $C_3-$equivariant change of basis for $V$, we may assume that $b \cdot \{1, c, c^2\}  = \{1, c, c^2\}$ and further $b \cdot 1 = 1$, $b \cdot c = c^2$, $b \cdot c^2 = c$. Since we are assuming that $V$ is a faithful representation, the action of $C_4$ on $O_{C_3}(V)$ must also be faithful (otherwise $b^2$ must fix all elements of $O_{C_3}(V)$ and it is not hard to see that it will then have to act trivially on $V$). Hence
$$\A \coloneqq \{\{x, cx, c^2 x\}, \{1+x, c + cx, c^2 + c^2 x\}, \{1+cx, c+c^2 x, c^2 + x\}, \{1+ c^2x, c+x, c^2 + cx\}\}$$ forms a single orbit of the $C_4-$action on $O_{C_3}(V)$. By linearity and the relation $bc = c^2b$, the action of $b$ on $\A$ is uniquely determined by which element $b \cdot x$ is. Note that an element of $V \setminus \{0, 1, c, c^2\}$ can be written as a linear polynomial $f(x)$ with coefficients in $\{1, c, c^2\}$. Call $f(x)$ \emph{primitive} if the coefficient of its lowest degree term is $1$. Hence, primitive elements of $V \setminus \{0, 1, c, c^2\}$ are in one-to-one correspondence with elements of $\A$. We now have the following cases:
\begin{enumerate}
\item\label{primitive} Suppose that $b \cdot x = f(x)$ is primitive. 
\begin{enumerate}
\item if $b \cdot x = x$ then $V = \spaN\{1, c\} \oplus \spaN\{x, cx\} \cong D \oplus D$ which contradicts faithfulness.
\item if $b \cdot x = 1+x$ then $b \cdot (c^2 + x) = c + 1 + x = c^2 + x$ and so
$$V = \spaN\{1, c\}\oplus\spaN\{c^2 + x, 1+cx\} \cong D\oplus D$$ which again contradicts faithfulness.
\item if $b \cdot x = 1+cx$ we obtain precisely the representation $U_4$. It is clearly indecomposable since it is not isomorphic to $D\oplus D$.
\item if $b \cdot x = 1+c^2x$ then consider the $C_3-$equivariant change of basis for $V$ given by the substitution $\tilde{x}=1+c^2x$. Then we have $x=c\cdot(1+\tilde{x})$ and $b \cdot \tilde{x} = b \cdot (1 +c^2x) = 1+c\cdot(b\cdot x) = 1+c \cdot(1+c^2x) = 1+c+x = c^2 + x = c^2 + c + c\tilde{x} = 1+c\tilde{x}$ and thus $V \cong U_4$.
\end{enumerate}
\item  Suppose that $b \cdot x$ is not primitive. 
\begin{enumerate}
\item if $b \cdot x = c\cdot(1+\alpha x)$ for some $\alpha \in \{1, c, c^2\}$ then consider the substitution 
$\tilde{x} = cx$. Then $x = c^2\tilde{x}$ and $b\cdot\tilde{x} = b\cdot cx = c^2b\cdot x = c^2c(1 + \alpha x) = 1 + \alpha x = 1+\alpha c^2 \tilde{x}$ is primitive and we are back in case \ref{primitive}.
\item if $b \cdot x = c^2\cdot(1+\alpha x)$ then, putting $\tilde{x}= c^2x$ we obtain $b \cdot \tilde{x} = 1+\alpha c \tilde{x}$ is primitive and again we can apply case \ref{primitive}.
\end{enumerate}
\end{enumerate}
We have thus seen that indeed the only faithful indecomposable $\bindih-$representation of dimension 4 is $U_4$. 

Recall that $U_4 \cong_4 \overline{V}_4$.
In order to extend the classification to higher-dimensional representations we will repeatedly use this fact, together with the following lemma.
\begin{lemma}\label{vFourBarsSplitAsGammaReps}
Let $V$ be a finite-dimensional representation of $\bindih$ over $\FF_2$ and let $U \le V$ be a subrepresentation. Suppose that $U \cong_4 \overline{V}_4^{\oplus k}$ for some $k \ge 1$. Then there exists a subrepresentation $W \le V$ such that $V = U \oplus W$ as representations of $\bindih$.
\end{lemma} 
\begin{remark}\label{remark on how to classify dihedral reps}
We note here that a similar statement holds also for $\FF_2-$representations of the dihedral group $D_6 = C_3 \rtimes C_2$. That is, if $U \le V$ is a pair of representations of $D_6$ and $U$ is isomorphic to a direct sum of copies of the regular representation of $C_2$, then $U$ is actually a direct summand of $V$. The proof is an easier version of the proof we present below.
\end{remark}
The proof of Lemma \ref{vFourBarsSplitAsGammaReps} requires a short detour. We first note the following standard fact whose proof is straightforward.
\begin{lemma}\label{summand replacement lemma}
Let $R$ be a ring (not necessarily commutative) and let $X$ be an $R-$module. Let $M, N \le X$ be submodules such that $X = M \oplus N$ and let $\pi \colon X \to M$ be the projection along $N$. Let $M' \le X$ be another $R-$submodule. Then $X=M' \oplus N$ if and only if $\at{\pi}{M'}\colon M'\to M$ is an isomorphism.\qed
\end{lemma} 
Using this fact, we can now make a step towards Lemma \ref{vFourBarsSplitAsGammaReps} by first showing that copies of $\overline{V}_4$ are always direct summands of $C_4-$representations. 
\begin{lemma}\label{vFourBarsSplitAsCfourReps}
Let $\overline{V}$ be an $R_4-$module which is finite-dimensional over $\FF_2$. Suppose $\overline{V}_4 \le \overline{V}$. Then there exists an $R_4-$submodule $\overline{W}\le \overline{V}$ such that $\overline{V}=\overline{V}_4 \oplus \overline{W}$.
\end{lemma}
\begin{proof}
Since $V$ is an $R_4-$module, there exist non-negative integers $n_1, n_2, n_3, n_4$ such that
\begin{equation}\label{splitting of a C4 rep}
\overline{V}\cong \overline{V}_1^{\oplus n_1} \oplus \overline{V}_2^{\oplus n_2} \oplus \overline{V}_3^{\oplus n_3} \oplus \overline{V}_4^{\oplus n_4}.
\end{equation}
Let \begin{eqnarray*}
\phi \colon R_4 &\longrightarrow & \overline{V}_1^{\oplus n_1} \oplus \overline{V}_2^{\oplus n_2} \oplus \overline{V}_3^{\oplus n_3} \oplus \overline{V}_4^{\oplus n_4}\\
1 &\mapsto & \vec{v}_1 + \vec{v}_2 + \vec{v}_3 + \vec{v}_4  
\end{eqnarray*}
denote the inclusion of $R_4-$modules obtained by restricting the isomorphism \eqref{splitting of a C4 rep} to $R_4 \cong \overline{V}_4 \le \overline{V}$. For any element $v$ of an $R_4-$module let us write $\ord(v) \in \{1, 2,4\}$ for the size of the orbit of $v$ under the $C_4-$action. We know that $\ord(\vec{v}_1) = 1$ and $\ord(\vec{v}_2) \le 2$. Since $\phi$ is a $C_4-$equivariant embedding, we must have 
\begin{eqnarray*}
\ord(\phi(1))& = & \ord(1) = 4 \\
\ord(\phi(1+b)) &=& \ord(1+b) = \vert \{ 1+b, b+b^2, b^2 + b^3, b^3 + 1\} \rvert = 4
\end{eqnarray*}
We claim that this implies that $\ord(\vec{v}_4)=4$. Suppose not, i.e. that $\ord(\vec{v}_4)\le 2$. Then, since $4=\ord(\phi(1)) = \max \{\ord(\vec{v}_1), \ord(\vec{v}_2), \ord(\vec{v}_3), \ord(\vec{v}_4)\}$, we must have $\ord(\vec{v}_3) = 4$. 
The elements of $\overline{V}_3 \cong \FF_2[b]/(b+1)^3 = \FF_2[b]/(1+b+b^2+b^3)$ are grouped into orbits of the $C_4-$action as follows
$$\overline{V}_3 = \{0\} \cup \{1+b^2\} \cup \{1+b, b+b^2\} \cup \{1, b, b^2, 1+b+b^2\}.$$
It is clear then that for all $x \in \overline{V}_3$ one has $\ord(x + b\cdot x) \le 2$. Hence $\ord(\vec{v}_3 + b \cdot \vec{v}_3) \le 2$ and thus $4 =  \ord(\phi(1+b)) = \max\{\ord(\vec{v}_1 + b\cdot \vec{v}_1), \ord(\vec{v}_2 + b\cdot \vec{v}_2), \ord(\vec{v}_3 + b\cdot \vec{v}_3), \ord(\vec{v}_4 + b\cdot \vec{v}_4)\} \le 2$, a contradiction.
So we must have $\ord(\vec{v}_4)= 4$. Writing $\vec{v}_4 = (v_{41}, v_{42}, \ldots, v_{4n_4}) \in \overline{V}_4^{\oplus n_4}$, there must exist $1 \le j \le n_4$ such that $\ord(v_{4j})=4$. Let $\pi \colon \overline{V} \to \overline{V}_4$ denote the projection to the $j^{\text{th}}$ $\overline{V}_4-$factor along the other factors in the decomposition \eqref{splitting of a C4 rep}. We claim that $\pi \circ \phi \colon R_4 \to \overline{V}_4$ is an isomorphism. Indeed, $\ord(\pi(\phi(1))) = \ord(v_{4j}) =4$ and it is easy to see that the $\FF_2-$span of any $C_4-$orbit of length $4$ in $\overline{V}_4$ has dimension $4$. Comparing dimensions, we see that $\pi \circ \phi$ is an isomorphism. 

The existence of a complement for $\overline{V}_4 \le \overline{V}$ now follows from Lemma \ref{summand replacement lemma}.
\end{proof}

To finish the proof of Lemma \ref{vFourBarsSplitAsGammaReps} we also need the following general lemma. 

\begin{lemma}\label{splitting reps of semidirect products}
Let $G$ be a group with subgroups $H \unlhd G$, $K \le G$ such that $G = H \rtimes K$. Suppose that $H$ is finite and that $\FF$ is a field with $\operatorname{char}(\FF) \nmid \vert H \vert$. Let $V$ be a representation of $G$ over $\FF$ and assume that there is a splitting $V = U \oplus \overline{W}$ as $K-$representations. Further, suppose that $U$ is actually a $G-$subrepresentation of $V$. Then there exists a $G-$representation $W \le V$ such that $V=U \oplus W$ as $G-$representations.
\end{lemma}
\begin{proof}
The proof is based on the standard technique of ``averaging the projection''. Namely, let $\overline{\pi} \colon V \to V$ denote the projection to $U$ along $\overline{W}$, followed by the inclusion $\iota \colon U \hookrightarrow V$. Since $\operatorname{char}(\FF) \nmid \vert H \vert$ we can define
\begin{eqnarray}
\pi \colon V & \longrightarrow & V \nonumber\\
v & \longmapsto & \frac{1}{\vert H \vert}\sum_{h\in H} (h^{-1}, 1) \cdot \overline{\pi} ( (h, 1) \cdot v).\label{definition of averaged projection}
\end{eqnarray} 
We claim that $\pi$ is $G-$equivariant and $\at{\pi}{U} = \iota$. For the second claim note that since $H$ preserves $U$ and $\at{\overline{\pi}}{U} = \iota$ then for all $h \in H$ and $u \in U$ we have $\overline{\pi}((h,1)\cdot u) = (h, 1)\cdot u$. Plugging this into \eqref{definition of averaged projection}, we see that 
$$\pi(u) = \frac{1}{\vert H \vert}\vert H\vert u = u \quad \forall u \in U.$$
Now, to see that that $\pi$ is $G-$equivariant we let $v\in V$, $(h_0, k_0) \in H \rtimes K = G$ and compute

\begin{eqnarray*}
\pi((h_0, k_0)\cdot v) &=& \frac{1}{\vert H \vert}\sum_{h\in H} (h^{-1}, 1) \cdot \overline{\pi} ( (h, 1)(h_0, k_0) \cdot v)\\
&=&\frac{1}{\vert H \vert}\sum_{h\in H} (h^{-1}, 1) \cdot \overline{\pi} ( (hh_0, 1)(1, k_0) \cdot v)\\
&=&\frac{1}{\vert H \vert}(h_0, 1)\cdot\sum_{h\in H} (hh_0, 1)^{-1} \cdot \overline{\pi} ( (hh_0, 1)(1, k_0) \cdot v)\\
&=& \frac{1}{\vert H \vert}(h_0, 1)\cdot\sum_{h\in H} (h^{-1}, 1) \cdot \overline{\pi} ( (h, 1)(1, k_0) \cdot v)\\
&=& \frac{1}{\vert H \vert}(h_0, 1)\cdot\sum_{h\in H} (h^{-1}, 1) \cdot \overline{\pi} ( (1, k_0)(k_0^{-1}hk_0, 1) \cdot v)\\
&=&\frac{1}{\vert H \vert}(h_0, 1)\cdot\sum_{h\in H} (h^{-1}, 1)(1, k_0) \cdot \overline{\pi} ( (k_0^{-1}hk_0, 1) \cdot v)\quad \text{[since $\overline{\pi}$ is $K-$equivariant]}\\
&=&\frac{1}{\vert H \vert}(h_0, 1)(1, k_0)\cdot\sum_{h\in H} (k_0^{-1}hk_0, 1)^{-1} \cdot \overline{\pi} ( (k_0^{-1}hk_0, 1) \cdot v)\\
&=&(h_0, k_0)\cdot \pi(u).
\end{eqnarray*}
Putting $W \coloneqq \ker \pi$ we now obtain the desired splitting $V = U \oplus W$.
\end{proof}
\noindent We are now in a position to prove Lemma \ref{vFourBarsSplitAsGammaReps}:
\begin{proof}
We are assuming that we have a pair of $\bindih-$representations $U \le V$ and that $U \cong_4 \overline{V}_4^{\oplus k}$ for some $k \ge 1$. By applying Lemma \ref{vFourBarsSplitAsCfourReps} $k$ times, we find a $C_4-$subrepresentation $\overline{W} \le V$ such that $V \cong_4 U \oplus \overline{W}$. Now, since $C_3$ preserves $U$ and $\bindih = C_3 \rtimes C_4$, we can apply Lemma \ref{splitting reps of semidirect products} to find a $\bindih-$subrepresentation $W \le V$ such that $V = U \oplus W$.
\end{proof}
Armed with Lemma \ref{vFourBarsSplitAsGammaReps}, we are now ready to extend our classification of indecomposable $\bindih-$representations to higher dimensions. So let $V$ be a $\bindih-$representation with $V^{C_3} =0$ and $\dim V =6$. We will show that $V$ cannot be indecomposable. We observe that since $\vert O_{C_3}(V) \vert = 1+4+4^2$ is odd, there must exist an orbit $A \in O_{C_3}(V)$ which is fixed by the $C_4-$action on $O_{C_3}(V)$. Then $D_0 \coloneqq \spaN A$ is a two-dimensional $\bindih-$subrepresentation of $V$ and so $D_0 \cong D$ and we have a short exact sequence 
\begin{equation}\label{short exact sequence}
\xymatrix{
0\ar[r] & D_0 \ar[r] & V \ar[r]^{\pi}& V/D_0 \ar[r]& 0.}
\end{equation} 
By the semisimplicity of $R_3$, this is a split sequence of $C_3-$representations and in particular $(V/D_0)^{C_3} = 0$.
Then $\vert O_{C_3}(V/D_0) \vert = 5$ and again there must be an orbit $B \in O_{C_3}(V/D_0)$ which is fixed by the $C_4-$action. Put $D_1 \coloneqq \spaN B \le V/D_0$ and $U \coloneqq \pi^{-1}(D_1) \le V$. We thus obtain a composition series
$$0 \lneq D_0 \lneq U \lneq V$$
with $U/D_0 = D_1\cong D$ and $V/U \cong D$. From our classification of the four-dimensional representations, we now have the following two possibilities
\begin{enumerate}
\item Suppose that $U \cong U_4$ or $V/D_0 \cong U_4$. Then
\begin{enumerate}
\item if $U \cong U_4$ we know by Lemma \ref{vFourBarsSplitAsGammaReps} that $V$ is not indecomposable and in fact $V \cong U_4 \oplus D$
\item if $V/D_0 \cong U_4$ then in particular $V/D_0 \cong_4 \overline{V}_4$ is a free $R_4-$module and hence  \eqref{short exact sequence} splits as a sequence of $C_4-$representations. However, Lemma \ref{splitting reps of semidirect products} then implies that \eqref{short exact sequence} is also a split sequence of $\bindih-$representations, i.e. again $V \cong D \oplus U_4$.
\end{enumerate}
\item Suppose that $U \cong D \oplus D$ and $V/D_0 \cong D \oplus D$. It follows (see Remark \ref{remark on how to classify dihedral reps}) that there exist $\bindih-$equivariant sections $r \colon D_1 = U/D_0 \longrightarrow U$ and $s \colon V/U=(V/D_0)/D_1 \longrightarrow V/D_0$ of the respective quotient maps. Now let $t \colon V/D_0 \longrightarrow V$ be any $\FF_2-$linear section of $\pi\colon V \to V/D_0$, satisfying $\at{t}{D_1}  = r$. These maps fit into the following diagram of $\bindih-$representations, whose rows and columns are exact:
$$
\xymatrix@1{
&&0 \ar[d]& 0 \ar[d] &\\
0 \ar[r]& D_0 \ar@{=}[d] \ar[r]& U \ar[d] \ar[r]& U/D_0 = D_1 \ar@/_/[l]_{r} \ar[d] \ar[r]& 0\\
0 \ar[r]& D_0 \ar[r]& V \ar[d]  \ar[r]_{\pi}& V/D_0 \ar@{.>}@/_/[l]_{t}  \ar[d] \ar[r]& 0\\
&&V/U  \ar[d] \ar@{=}[r]&  (V/D_0)/D_1 \ar@/_/[u]_{s} \ar[d] &\\
&&0 & 0&}$$
\noindent Observe that since $\pi$ and $s$ are $C_4-$equivariant, we have 
\begin{equation}\label{b commutes with ts up to D_0}
\pi(b \cdot ts(x)) = b \cdot \pi(ts(x)) = b \cdot s(x) = s(b\cdot x).
\end{equation}
Consider now the $\FF_2-$linear splitting
\begin{equation}\label{V splits into three parts}
V = D_0 \oplus t(V/D_0) = D_0 \oplus t(D_1 \oplus s(V/U)) = D_0 \oplus r(D_1) \oplus ts(V/U).
\end{equation}
We claim that $C_4$ preserves the summand $\overline{W} \coloneqq D_0 \oplus ts(V/U)$. Indeed, if $v_0 \in D_0$, $x \in V/U$, then by \eqref{b commutes with ts up to D_0} we have 
\begin{eqnarray*}
b \cdot (v_0 + ts(x)) & = & \Big[b \cdot v_0 + b \cdot ts(x) - t(\pi(b \cdot (ts(x)))) \Big] + t(\pi(b \cdot ts(x)))\\
&=&\Big[b \cdot v_0 + b \cdot ts(x) - t(\pi(b \cdot (ts(x)))) \Big] + ts(b\cdot x) \in D_0 \oplus ts(V/U).
\end{eqnarray*}
Now, since $r$ is $C_4-$equivariant we see that \eqref{V splits into three parts} gives rise to the  splitting $V = r(D_1) \oplus \overline{W}$ of $C_4-$representations.
On the other hand, since $r$ is also $C_3-$equivariant, we have that $r(D_1)$ is a $\bindih-$subrepresentation of $V$ and then it follows from Lemma \ref{splitting reps of semidirect products} that $V$ is not indecomposable.
\end{enumerate}

We have seen that if $V$ is a six-dimensional $\bindih-$representation with $V^{C_3} =0$ then we must have $V \cong D^{\oplus 3}$ or $V \cong D \oplus U_4$. In particular, the only faithful six-dimensional $\bindih-$representation with $V^{C_3} =0$ is $D \oplus U_4$.

 Now let $V$ be a faithful $\bindih-$representation with $V^{C_3} =0$ and $\dim V =8$. Since the representation is faithful, there exists $\A \in O_{C_4}(O_{C_3}(V))$ with $\vert \A\vert =4$. Then $\spaN \A$ is a faithful subrepresentation of $V$ and so $\dim(\spaN \A) \in \{4, 6, 8\}$. If $\dim (\spaN \A) =4$ we know that $\spaN \A \cong U_4$. By Lemma \ref{vFourBarsSplitAsGammaReps} we have that $V$ is not indecomposable. If $\dim(\spaN \A) = 6$, then must have $\spaN \A \cong D \oplus U_4$; in particular $U_4 \le V$ and again Lemma \ref{vFourBarsSplitAsGammaReps} shows that $V$ cannot be indecomposable. We are left with the case $\dim(\spaN \A) =8$, i.e. $\spaN \A = V$. We can then write
 $$ \A = \{\{1, c, c^2\}, \{x, cx, c^2x\}, \{x^2, cx^2, c^2x^2\}, \{x^3, cx^3, c^2x^3\}\}$$
 and $\{1, c, x, cx, x^2, cx^2, x^3, cx^3\}$ forms a basis for $V$. Hence
 $$V \cong_3 \S \oplus \S x \oplus \S x^2 \oplus \S x^3$$
and $b \in C_4$ acts as the cyclic permutation $(1, x, x^2, x^3)$. That is, $V \cong U_8 \cong U_4 \oplus U_4$ is not indecomposable.

Finally, we are ready to finish the proof of Proposition \ref{classification of representations}, by showing that if $V$ is a faithful representation of $\bindih$ with $V^{C_3}=0$ and $\dim V > 8$, then $V$ cannot be indecomposable. Indeed, by faithfulness, there must exist $\A \in  O_{C_4}(O_{C_3}(V))$ with $\vert \A \vert =4$. Then $\spaN \A$ is a faithful subrepresentation of $V$. In particular, we have that $(\spaN \A)^{C_3} = 0$ and hence $V \cong_3 \widehat{D}^{\oplus k}$. It follows that $\dim (\spaN \A)$ must be even and for each $A = \{v, c\cdot v, c^2 \cdot v\} \in \A$ we have $v + c \cdot v + c^2 \cdot v =0$. Then  $\dim(\spaN \A) \le 2\vert \A \vert \le 8$. But we have seen that any faithful $\bindih-$representation of dimension at most 8 contains a copy of $U_4$. Hence $U_4 \le \spaN \A \le V$ and Lemma \ref{vFourBarsSplitAsGammaReps} shows that $V$ cannot be indecomposable.


Proposition \ref{classification of representations} is now proved.

\clearpage
\bibliographystyle{amsalpha}
\bibliography{biblio_Lagrangian_Floer_towards_paper}

\end{document}